\documentclass[10.6pt, reqno]{amsart}
\usepackage{caption}
\usepackage{graphicx}
\usepackage{cancel}
\usepackage{mathtools}
\usepackage{subfigure}
\usepackage{slashed}
\usepackage{amsfonts}
\usepackage{amssymb}
\usepackage{accents}
\usepackage{amsmath}
\usepackage{amsxtra}
\usepackage{epstopdf}
\usepackage{latexsym}
\usepackage{mathrsfs}
\usepackage{leftidx}
\usepackage{tensor}
\usepackage{enumitem}

\newtheorem{theorem}{Theorem}[section]
\newtheorem{lemma}[theorem]{Lemma}
\newtheorem{corollary}[theorem]{Corollary}
\newtheorem{proposition}[theorem]{Proposition}
\newtheorem{definition}[theorem]{Definition}

\theoremstyle{remark}
\newtheorem{remark}[theorem]{Remark}

\numberwithin{equation}{section}

\def\sI{{\mathscr{I}}}
\def\fv{\mathfrak{v}}
\def\fZ{\mathscr{Z}}
\def\fW{\mathscr{W}}

\def\rp#1{^{\!(#1)}}

\def\bi{{\bf i}}

\def\bp{\boldsymbol{\partial}}

\def\tmin{{\mbox{\tiny{min}}}}
\def\T{\mathcal{T}}

\def\smu{\slashed{\mu}}

\def\tt{{t'}}

\def\bb{{\mathbf{b}}}

\def\Er{\mbox{Er}}

\def\sn{{\slashed{\nabla}}}
\def\Q{\mathcal{Q}}
\def\sQ{\mathscr{Q}}

\def\zb{{\underline{\zeta}}}
\def\bpi{{\bar{\pi}}}

\def\L{{\mathcal{L}}}
\def\J{{\mathcal{J}}}

\def\M{{\mathcal{M}}}

\def\bT{{\textbf{T}}}

\def\bR{{\textbf{R}}}

\def\bd{{\textbf{D}}}
\def\ti{\tilde}

\def\bg{\mathbf{g}}

\def\hk{{\hat{k}}}
\def\I{{\mathcal I}}

\def\beaa{\begin{eqnarray*}}
\def\eeaa{\end{eqnarray*}}

\def\ba{\begin{array}}
\def\ea{\end{array}}
\def\d{\delta}

\def\be#1{\begin{equation} \label{#1}}
\def \eeq{\end{equation}}

\newcommand{\nn}{\nonumber}

\def\l{\langle}
\def\r{\rangle}

\def\cir{\overset\circ}
\def\nn{\nonumber}
\def\S{{\mathcal S}}

\def\cga{\overset\circ{\ga}}

\def\ud#1{\underline{#1}}
\def\zb{\ud{Z}}

\def\S2{{\mathbb S}^2}

\def\A{\mathcal {A}}

\def\E{{\mathcal E}}

\def\K{{\mathcal{K}}}

\def\ze{{\zeta}}

\def\Lie{{\mathcal L}}

\def\tr{\mbox{tr}}

\def\D{{\mathcal D}}
\def\H{{\mathcal H}}

\def\N{{\mathcal N}}
\def\cga{\overset\circ{\ga}}

\def\La{{\Lambda}}

\def\B{{\mathcal B}}

\def\P{{\mathcal P}}

\def\c{\cdot}

\def\a{\alpha}
\def\b{\beta}

\def\ep{{\epsilon}}

\def\l{\langle}
\def\r{\rangle}
\def\ga{\gamma}
\def\Ga{\Gamma}

\def\p{\partial}
\def\P{{\mathcal P}}

\def\nab{\nabla}
\def\hb{{\ud h}}

\def\C{{{\mathfrak C}}}
\def\CC{{\mathcal{C}}}

\def\Lb{{\underline{L}}}

\def\div{\mbox{\,div\,}}
\def\curl{\mbox{\,curl\,}}
\def\tr{\mbox{tr}}
\def\Tr{\mbox{Tr}}

\def\bE{{\mathbf E}}

\def\tir{{\tilde r}}

\def\f14{\frac{1}{4}}
\def\f12{{\frac{1}{2}}}
\def\t1a{t^{-\frac{1}{a}}}

\def\bm{{\bf m}}

\def\sl{\slashed}
\def\sD{\slashed{\Delta}}

\def\sn{{\slashed{\nabla}}}

\def\zb{{\underline{\zeta}}}
\def\bpi{{\bar{\pi}}}

\def\L{{\mathcal{L}}}
\def\J{{\mathcal{J}}}

\def\M{{\mathcal{M}}}

\def\bT{{\emph{\bf{T}}}}

\def\bR{{\emph{\bf{R}}}}

\def\bd{{\emph{\bf{D}}}}
\def\ti{\tilde}

\def\hk{{\hat{k}}}
\def\I{{\mathcal I}}

\def\beaa{\begin{eqnarray*}}
\def\eeaa{\end{eqnarray*}}

\def\ba{\begin{array}}
\def\ea{\end{array}}
\def\be#1{\begin{equation} \label{#1}}
\def \eeq{\end{equation}}
\def\nn{\nonumber}

\def\l{\langle}
\def\r{\rangle}

\def\cir{\overset\circ}
\def\nn{\nonumber}
\def\S{{\mathcal S}}

\def\cga{\overset\circ{\ga}}

\def\S2{{\mathbb S}^2}

\def\A{\mathcal {A}}

\def\E{{\mathcal E}}

\def\ze{{\zeta}}

\def\Lb{\underline{L}}
\def\tr{\mbox{tr}}

\def\bA{{\emph{\bf{A}}}}
\def\fA{{\mathfrak{A}}}
\def\D{{\mathcal D}}
\def\H{{\mathcal H}}

\def\cga{\overset\circ{\ga}}

\def\La{{\Lambda}}

\def\B{{\mathcal B}}

\def\P{{\mathcal P}}

\def\c{\cdot}

\def\a{\alpha}
\def\b{\beta}

\def\l{\langle}
\def\r{\rangle}
\def\ga{\gamma}

\def\Ga{\Gamma}
\def\la{\lambda}

\def\p{\partial}
\def\P{{\mathcal P}}

\def\nab{\nabla}

\def\Lb{{\underline{L}}}

\def\div{\mbox{\,div\,}}
\def\curl{\mbox{\,curl\,}}

\def\tr{\mbox{tr}}
\def\Tr{\mbox{Tr}}

\def\tir{{\tilde r}}

\def\f14{\frac{1}{4}}
\def\f12{{\frac{1}{2}}}
\def\t1a{t^{-\frac{1}{a}}}

\def\bm{{\bf m}}

\def\sl{\slashed}
\def\sD{\slashed{\Delta}}

\def\ckk{\check}

\newcommand{\bea}{\begin{eqnarray}}
\newcommand{\eea}{\end{eqnarray}}

\def\nn{\nonumber}

\def\ei{\E^{(1)}}

\newcommand{\chih}{\hat{\chi}}
\newcommand{\chib}{\underline{\chi}}

\newcommand{\chibh}{\underline{\hat{\chi}}\,}
\newcommand{\les}{\lesssim}

\def\gac{\stackrel{\circ}\ga}

\def\bN{{\mathbf{N}}}

\def\S{\mathcal{S}}

\def\cir#1{\stackrel{\circ}{#1}}
\def\sV{{\mathscr{V}}}

\def\sX{\mathscr{X}}

\def\sY{\mathscr{Y}}

\def\sF{\mathscr{F}}

\def\sP{{\mathscr{P}}}

\def\sQ{{\mathscr{Q}}}

\def\fC{\mathfrak{C}}

\def\ud#1{\underline{#1}}

\def\fw{\mathfrak{w}}
\def\be{{(e)}}
\def\bi{{(i)}}

\begin{document}
\title[]{Rough solutions of the $3$-D compressible Euler equations}
\author{Qian Wang}
\address{
Oxford PDE center, Mathematical Institute, University of Oxford, Oxford, OX2 6GG, UK}
  \email{qian.wang@maths.ox.ac.uk}
  \date{\today}
\begin{abstract}
We prove the local-in-time well-posedness for the solution of the compressible Euler equations in $3$-D, for the Cauchy data of
the velocity, density
and vorticity $(v,\varrho, \fw) \in H^s\times H^s\times H^{s'}$, $2<s'<s$.

The classical local well-posedness result for the compressible Euler equations in $3$-D  holds for the initial data  $v, \varrho \in H^{s+\f12},\, s>2$.
Due to the works of Smith-Tataru \cite{Tataru} and  Wang \cite{Wangrough}, for the irrotational isentropic case, the local well-posedness can be achieved
 if the data satisfy  $v, \varrho \in H^{s}$, with $s>2$.
  In the incompressible case, the solution is proven  to be ill-posed for the datum  $\fw\in H^\frac{3}{2}$ by Bourgain-Li \cite{Bourgain-Li}. Hence the solution of the compressible Euler equations
  is not expected to be well-posed if the data merely satisfy $v, \varrho\in H^{s}, s>2$  with a general rough  vorticity.

The rough term $\curl \fw$  lowers the regularity of the spacetime geometry, and causes crucial difficulties in
each main building block of the work:  the energy propagation, linearization and the proof of Strichartz estimates.

 By introducing the decomposition of the velocity into the term $(I-\Delta_e)^{-1}\curl \fw$ and a wave function
verifying an improved wave equation, with a series of cancellations for treating the latter,  we achieve the $H^s$-energy bound and complete the linearization for the wave functions
by using the $H^{s-\f12}, \, s>2$  norm for the vorticity. The propagation of energy for  the vorticity typically requires either the data of velocity to be $\f12$-derivative
smoother or $\curl \fw\in C^{0, 0+}$ initially,  stronger than our assumption by $\f12$-derivative. We perform trilinear estimates to gain regularity by observing a $\div$-$\curl$ structure when propagating the energy of the normalized  double-curl of the vorticity,
 and also by spacetime integration by parts.  To prove the Strichartz estimate for the linearized
  wave in the rough spacetime, we encounter a strong Ricci defect which requires  the bound of $\|\curl \fw\|_{L_x^\infty L_t^1}$ on the acoustic null cones since $\curl \fw$ appears in the Ricci tensor. This difficulty is solved by uncovering the cancellation structures
 due to the acoustic metric on the angular derivatives of Ricci and the second fundamental form.
\end{abstract}
\maketitle
\section{Introduction}
\subsection{Basic set-up and the main result}
We consider the compressible Euler equations of $3$ space dimension for a perfect fluid under a barotropic equation of state, that is the pressure $p$ is a function of the density  $\rho:{\mathbb R}^{1+3}\rightarrow (0,\infty)$,
\begin{equation}\label{10.12.1.19}
p=p(\rho).
\end{equation}
We can fix a constant  background  density $\bar\rho>0$.   Define the normalized density
\begin{equation}\label{10.12.2.19}
\varrho=\ln (\rho/\bar \rho)
\end{equation}
 and the sound speed
 \begin{equation*}
  c=\sqrt{\frac{dp}{d\rho}}.
 \end{equation*}
Clearly, due to (\ref{10.12.1.19}),   $c=c(\varrho)$.

Let $v$ be the velocity of the compressible fluid $v:{\mathbb R}^{1+3}\rightarrow {\mathbb R}^3$.
We define the acoustic metric $\bg$ as
\begin{equation}\label{metric}
\bg:=-dt\otimes dt+c^{-2} \sum_{a=1}^3(d x^a-v^adt)\otimes (dx^a-v^a dt),
\end{equation}
and may regard ${\mathbb R}^3\times [0,T]$ with $T>0$ as the acoustic spacetime $(\M, \bg)$.
The inverse metric $\bg^{-1}$ can be written as
\begin{equation*}
\bg^{-1}=-\bT\otimes \bT+c^2 \Sigma_{a=1}^3 \p_a \otimes \p_a,
\end{equation*}
where $\bT$ is the future directed, time-like unit normal of the level set of $t$. And the component of  $\bg^{-1}$ will be denoted by  $\bg^{\a\b}$. \begin{footnote}{We adopt Einstein summation convention in this article. The range of  the indices of Greek letters such as $\a,\b,\mu,\nu$ is $0,\cdots, 3$, and the range of the Latin letters $i,j,k,l,m,n,a,b$ is $1,2,3$. We also fix the convention that $\p_0=\p_t$.}\end{footnote}

 Relative to the Cartesian coordinates, $\bT$ is written as
\begin{equation*}
\bT=\p_t+v^a \p_a.
\end{equation*}
 We can compute directly the  induced metric  $g_{ij}=c^{-2}\delta_{ij}$ on $\Sigma_t=\{t\}\times {\mathbb R}^3$, where $\delta_{ij}$ is the kronecker delta. Define the second fundamental form
 \begin{equation*}
 k_{ij}=-\f12 \Lie_\bT g_{ij}, \qquad \Tr k =g^{ij} k_{ij},
 \end{equation*}
 where $\Lie_X$ denotes the Lie derivative by the vector field $X$.
Let $\cir{k}_{ij}=-\f12 \Lie_\bT \delta_{ij}$.  Thus $\Tr \cir{k}:=\delta^{ij} \cir{k}_{ij}= -\p^i v_i$.

Now we introduce the compressible Euler equations with (\ref{10.12.1.19}) for $\varrho$ and $v$,
\begin{equation}\label{4.23.1.19}
\left\{
\begin{array}{lll}
\bT \varrho=-\div v\\
\bT v^i=-c^2\delta^{ia} \p_a \varrho,
\end{array}
\right.
\end{equation}
where $\div v= \p^i v_i,$  $\varrho$ is the normalized density function in (\ref{10.12.2.19}).

 Let $\tensor{\ep}{_i^j^k}$, $i,j,k=1,2,3$, be the standard volume form on ${\mathbb R}^3$. We define the vorticity to be
\begin{footnote}{ The indices of the tensor field here are lifted and lowered by the Euclidean metric.}\end{footnote}
$
\fw_i=\tensor{\ep}{_i^j_k}\p_j v^k.
$ We
 may employ the normalized vorticity
 $\Omega=e^{-\varrho}\fw$
 for convenience.
There hold for $\Omega$ the equations
\begin{align}
&\div \Omega=-\Omega^a \p_a \varrho,\label{div}\\
& \bT \Omega^i=\Omega^a \p_a v^i, \label{4.10.3.19}
\end{align}
where (\ref{div}) can follow directly due to $\div \fw=0$.

 The compressible Euler equations (\ref{4.23.1.19}) can be reduced to
\begin{align}
& \Box_\bg v^i=-e^\varrho c^2\curl \Omega^i+ \sQ^i, \label{4.10.1.19}\\
& \Box_\bg \varrho= \sQ^0,\label{4.10.2.19}
\end{align}
where $\Box_\bg$ is the Laplace-Beltrami operator of the Lorentzian metric $\bg$, and the two quadratic forms are
\begin{align*}
&\sQ^i:=-(1+c^{-1}c')\bg^{\a\b}\p_\a\varrho\p_\b v^i+2 e^\varrho \tensor{\ep}{^i_a_b}\bT v^a \Omega^b, \\
&\sQ^0:=-3c^{-1}c' \bg^{\a\b}\p_\a \varrho \p_\b \varrho+2\sum_{1\le a<b\le 3}\big(\p_a v^a\p_b v^b-\p_b v^a \p_a v^b\big).
\end{align*}
See the equations from the work of Luk-Speck \cite[Page 13]{Jared_Luk}.

 We can derive by using (\ref{4.10.3.19}) the following transport equation  for $ \C^i=e^{-\varrho} \curl \Omega^i$ that
\begin{equation}\label{4.25.4.19}
\bT \C^i =-2\delta_{jk} \tensor{\ep}{^i^{ab}}  \p_a v^j \p_b \Omega^k e^{-\varrho}+\tensor{\ep}{^{aj}_k} \p_a v^i \p_j \Omega^k  e^{-\varrho},
\end{equation}
with the derivation given in Section \ref{eng_vor}. (See also \cite[(2.3.4.b)]{Jared_Luk}.)

Assume there hold
\begin{equation}\label{10.22.1.19}
|v, \varrho|\le C_1, \quad c>c_0>0,  \,\mbox{ at } t=0
\end{equation}
where $C_1, c_0>0$  are constants. $c_0>0$ is used in particular to ensure the uniform hyperbolicity of the compressible Euler system.
The lower bound $c_0$ can be determined by the bound $C_1$ on $|\varrho(0)|$ if one assumes an explicit form for the pressure, such as the Gamma-law, i.e. $p(\rho)=A\rho^\ga,$ with constants $ \, A, \ga>0$.

Let $\p$ represent the spatial derivative $\p_i, i=1,2,3$ and $\bp$ include $\p$ and $\bT$.
Now we state the main result of this paper.
\begin{theorem}\label{thm1}
Let $s$ and $s'$ be fixed  and $2<s'<s$.  For the given data set of $(v, \varrho, \fw)$  satisfying the assumption of (\ref{10.22.1.19}) and  any $M>0$, there exist positive constants $T_*$ and $M_1$  such that if the initial data  satisfy
$$\|(\p v, \p \varrho, \fw)(0)\|_{H^{s-1}({\mathbb R}^3)\times H^{s-1}({\mathbb R}^3)\times H^{s'}({\mathbb R}^3)}\le M<\infty,
$$ there exists a unique set of the solution with $(\p v,\p \varrho, \fw)\in C(I_*, H^{s-1})\times C(I_*, H^{s-1})\times C(I_*,H^{s'})$ for the $3$-D compressible Euler equations in (\ref{4.23.1.19}) and (\ref{4.10.3.19}), satisfying the estimates
\begin{align}
&\|\bp v, \bp \varrho\|_{L^2_{I_*} L_x^\infty}\le M_1,\nn\\
&\|(\p v, \p \varrho, \fw)(t)\|_{H^{s-1}({\mathbb R}^3)\times H^{s-1}({\mathbb R}^3)\times H^{s'}({\mathbb R}^3)}+ |(v, \varrho)(t)|\le M_1,\quad t\in I_*, \label{eng_f}
\end{align}
where $I_*=[0,T_*]$
 and $0<s'-2<(\frac{s-2}{5})^2$. \begin{footnote}
{The upper bound of $s'$ here is merely chosen for convenience.  One may extend the energy estimate to  $2<s'<s$.  Since the ultimate challenge is to prove the above result for $s'=s-\f12$, such an extension, requiring  higher regularity assumption on data, is not of our interest.}
\end{footnote}
\end{theorem}
\begin{remark}
Theorem \ref{thm1} holds if the first assumption in (\ref{10.22.1.19}) is replaced by the boundedness of $\|v-\bar v\|_{L^2(\Sigma_0)}+ \|\varrho-\bar \varrho\|_{L^2(\Sigma_0)}$ with $\bar v$ and $\bar \varrho$  constant states, since, in view of Sobolev embedding on ${\mathbb R}^3$, the bound on $|v, \varrho|$ in (\ref{10.22.1.19}) can be obtained by combining the $L^2$ assumption with the other assumption in Theorem \ref{thm1}.
\end{remark}
\begin{remark}\label{10.22.2.19}
 The assumption on $\p v$ in Theorem \ref{thm1} is merely on $\div v$, due to the standard elliptic estimates for the Hodge operator $(\div, \curl)$ on vector fields in $\mathbb R^3$.
By direct comparison estimates on the initial slice, the assumption and the result on $\p \varrho$ in Theorem \ref{thm1} can be stated for $\bT v$ instead,
 due to  (\ref{4.23.1.19}) and  (\ref{10.22.1.19}); 
Theorem \ref{thm1} can  also be stated with $\fw$ replaced by $\Omega$.
\end{remark}
\begin{remark}
Since the classical  local well-posedness holds for  $s>\frac{5}{2}$, we focus on the case  $2<s<\frac{5}{2}$.
\end{remark}

\subsection{Motivation of the problem}
The classical local well-posedness of the compressible Euler equations in $n$-D can be obtained by using energy method for the Cauchy data $v, \varrho\in H^s, \quad s>\frac{n}{2}+1$.
See \cite[Chapter 2]{Majda} and also the earlier work  by Kato \cite{Kato}  for the alternative approach. 
For incompressible Euler equations, Kato and Ponce in \cite{Kato-Ponce} proved the classical local well-posedness for data in the general
Sobolev space $v\in W^{s,p}({\mathbb R}^n ) = (I-\Delta_e)^\frac{s}{2}
L^ p({\mathbb R}^n )$ with  $ s > n/p + 1$  and
$1 < p < \infty$, where $\Delta_e$ is the Laplace-Beltrami operator of the Euclidean metric.
In $3$-D for the incompressible case, the assumption of the datum that  $\fw\in H^{s-\f12}$, $s>2$ is sharp, since  in \cite{Bourgain-Li} the solution is proven to be  strongly ill-posed if $s=2$.

Recall the general quasilinear wave equations considered in \cite{KRduke,Tataru,Wangrough}
\begin{equation}\label{4.10.4.19}
\Box_{h(\phi)} \phi=Q(\bp \phi, \bp \phi)
\end{equation}
with $h(\phi)$ a Lorentzian metric depending on $\phi$ and the right hand side being  quadratic forms of $\bp \phi$. By using the energy method and iteration, the classical local well-posedness
holds for $(\phi(0), \p_t\phi(0))\in H^s\times H^{s-1}$ with $s>\frac{5}{2}$ (see \cite{HKM}), which is of the same level as the classical result for the compressible Euler equations.
Starting from the late 90s, based on establishing Strichartz estimates with loss for wave equations with rough
coefficients, there had been vast improvements, due to Smith, Bahouri-Chemin, Tataru, and Klainerman-Rodnianski, to $s>2+\frac{1}{4}$ achieved in \cite{BC1, BC2, T1}, to $s>2+\frac{1}{6}$ in \cite{T3, Kcom},  and to $s>2+\frac{2-\sqrt{3}}{2}$ achieved in \cite{KRduke}, and for $s>2$ for Einstein vacuum equations  in \cite{KREinst, KREins2, KRd} and \cite{WangCMCSH, Wangricci}. Apart from the improvements over the Sobolev exponents,  the commuting vector field approach for Strichartz estimate was introduced in \cite{Kcom}. This physical approach  further showed its power in  \cite{KRduke} where a fundamental decomposition of a Ricci component (proposed in \cite{Kcom}) was used for improving the regularity in the causal geometry. On the other hand, the Fourier method of proving the Strichartz estimate by constructing parametrix and wave packets was developed to rely on the actual null hypersurfaces in the Lorentzian spacetime. With the help of the improvement on the causal geometry,
 the sharp local well-posedness for the solution of (\ref{4.10.4.19}) was proven  in \cite{Tataru},  and by the vector field approach in  \cite{Wangrough}, for
the initial data $(\phi(0), \p_t\phi(0))\in H^s\times H^{s-1}$ for $s>2$.

In comparison with the equation (\ref{4.10.4.19}), we schematically unify (\ref{4.10.1.19}) and (\ref{4.10.2.19}) into the following equation for
 $\Phi=(v, \varrho)$,
\begin{equation}\label{10.13.2.19}
\Box_{\bg(\Phi)} \Phi=Q(\bp \Phi, \bp \Phi)-(e^\varrho c^2 \curl \Omega,0),
\end{equation}
where \begin{footnote}{These quadratic forms are of the same symbolic type for either $\Phi=v$ or $\varrho$. Instead of repeating them in the bracket notation,
we simply write them as one term.}\end{footnote} $Q(\bp \Phi, \bp \Phi)$ represents symbolically the finite sum of
$\N(\Phi) \bp \Phi  \bp\Phi$, with $\N$ smooth functions of their variables, and $\bp \Phi$ representing terms of  $\bp v$ or $\bp \varrho$.

Note that if $\Omega=0$,  the above equation takes the form of
\begin{equation*}
\Box_{\bg(\Phi)} \Phi=Q(\bp \Phi, \bp \Phi).
\end{equation*}
Therefore the local well-posedness for the irrotational (isentropic) compressible Euler flow holds for $(v(0), \varrho(0))\in { H}^{s}\times {H}^{s}$ for any $s>2$ by the results of \cite{Tataru} and \cite{Wangrough}.
Note the proof of the sharp local well-posedness result for (\ref{4.10.4.19}) is based on the energy method, and via a bootstrap argument the key analysis that reestablishes the standard Strichartz estimate for the free wave in the Minkowski space for the linearized wave equation $\Box_{h(\phi)}\psi=0$ within  a short life-span. The latter relies sensitively on the regularity of the spacetime metric. If $\Omega$ (or comparably $\fw$ ) is non-zero and sufficiently smooth, the term $e^\varrho c^2\curl \Omega$ in
 (\ref{10.13.2.19})
  %
  can be treated as a good inhomogeneous term of the linearized wave equation for both the energy method and the Strichartz estimate, which does not in turn influence the regularity of the spacetime metric. This case can be incorporated in the regime either given in \cite{Tataru} or in \cite{Wangrough}.  See \cite{Jared_4} which assumes $\curl \Omega \in H^{s-1}\cap C^{0,\a}$, $0<\a<1$ for the data, and treats the case allowing  dynamical entropy with the same smoothness assumption  on the divergence of the initial entropy gradient. This assumption is smoother than our assumption in Theorem \ref{thm1} by more than $\f12$-derivative.

Clearly, due to the result of \cite{Bourgain-Li}, lowering the regularity of the datum for $\Omega$ below a certain level would significantly change the behaviour of the solution and the regularity of the spacetime metric. 
 As the ultimate challenge, we propose the sharp local well-posedness conjecture  for the  compressible Euler equations in $3$-D, which reaches the sharp results in both the irrotational and incompressbile cases.

{\bf Conjecture.}
The solution $(v, \varrho, \fw)$ of the $3$-D compressible Euler equations (\ref{4.23.1.19}) and (\ref{4.10.3.19}) is well-posed for $t\in (0,T]$ with some $T>0$ if the data satisfy (\ref{10.22.1.19}) and $(\bp v, \fw)\in   H^{s-1} \times H^{s-\f12}$ with $s>2$.

In Theorem \ref{thm1}, the regularity of vorticity datum is set such that the term of  $\curl \Omega$ in (\ref{4.10.1.19}) can not be treated as a source term associated
 to the approach in \cite{Wangrough} for irrotational fluids. This entails a deeper understanding on the velocity and correspondingly a different regime.
 By introducing a decoupling method, for the data verifying the assumption (\ref{10.22.1.19}) and $\bp v\in H^s$, with the bound of $\|\Omega\|_{C^0[0,T] H^{s-\f12}_x}$, $s>2$,
 we reduce the proof of the local well-posedness to recovering  the standard Strichartz estimate of the free wave in Minkowski space to the linear wave equation in the acoustic spacetime.
To solve the conjecture, there remains  the obstruction  for proving the Strichartz estimate due to the insufficiently smooth acoustic null hypersurfaces.
  The regularity of data in Theorem \ref{thm1} reaches the borderline for guaranteeing the acoustic null hypersurfaces to be sufficiently regular.  Due to the rough data, $\curl \Omega$ fundamental influences our analysis.
It forces us to create and uncover a variety of cancellations on the vorticity derivative and the structure of the acoustic metric for completing the proof of Theorem \ref{thm1}.
With lower  regularity of $\Omega(0)$  than in Theorem \ref{thm1}, we encounter the same difficulty from the rough null hypersurfaces  as that stops the improvement
over the resolution of the bounded $L^2$-curvature conjecture for Einstein vacuum equations in  \cite{KJR}. For the latter, there is also a gap of $\f12$-derivative between
the critical Sobolev exponent for the Einstein vacuum equations and the result by Klainerman-Rodnianski-Szeftel in \cite{KJR}.

We remark that our result can be extended to equations with dynamical entropy by treating the divergence of the entropy gradient analogous to $\curl \fw$. In this article we focus on presenting the method for reducing the regularity requirement on data.
\subsection{Main idea of the proof for Theorem \ref{thm1}}

\subsubsection{A quick guide to the main difficulties}\label{10.13.3.19}
$v, \varrho$ and $\Omega$ in the compressible Euler equations exhibit different physical and analytic properties.
 To achieve a result with the data  $v, \varrho\in H^s,\, s>2$, it is necessary to resort to establishing the Strichartz estimate which gains the spatial regularity
  by taking the advantage of  the dispersive property in time-variable of the quantities. It works particularly well for the solutions of  wave equations.
   To have a brief idea of how the vorticity derivative
 is involved in the analysis, we will apply the energy method to the geometric wave equations (\ref{4.10.1.19}) and (\ref{4.10.2.19}). Through this procedure, we can see the regularity requirement on the Cauchy data  by following the approach  in \cite{Wangrough}.

Under the bootstrap assumption that $\|\bp \Phi\|_{L_t^1 [0,T]L_x^\infty}$ is bounded, applying the standard energy argument to (\ref{10.13.2.19}) implies \begin{footnote}{ We call $C$ a universal constant if it depends merely  on the bound of  $M$ in Theorem \ref{thm1}, and the constants $C_1$ and $c_0$
in (\ref{10.22.1.19}). $A\les B$ means there exists a universal constant $C>0$ such that $A\le C B$. We denote $A\approx B$ if $A\les B$ and $B\les A$. }\end{footnote}
\begin{equation}\label{10.12.5.19}
\|\bp \Phi(t)\|_{H^{s-1}_x}\les \|\bp \Phi(0)\|_{H^{s-1}_x}+\|e^\varrho c^2\curl \Omega\|_{L_t^1 H^{s-1}_x}+\|Q(\bp \Phi, \bp \Phi)\|_{L_t^1 H^{s-1}_x}, \quad \, s\ge 2.
\end{equation}
Due to the standard product estimate, the last term on the right hand side can be controlled by the energy bound, provided that one can control $\|\bp \Phi\|_{L_t^1[0,T] L_x^\infty}$
in terms of the initial data. Assuming  the absence of $\curl\Omega$, to gain such control, this was done in \cite{Tataru, Wangrough} by
   recovering the Strichartz estimate of the free wave  in ${\mathbb R}^{3+1}$  for the solution  $\psi$  of the linear wave equation in the acoustic spacetime $({\mathbb R}^3\times [0,T], \bg)$,
\begin{equation}\label{wavelin}
\Box_{\bg(\Phi)} \psi=0,
\end{equation}
which implies  
\begin{equation}\label{10.13.1.19}
\|\bp \Phi\|_{L_t^q[0,T] L_x^\infty}\les \|\bp \Phi(0)\|_{H^{1+\ep}_x}+\|\Box_{\bg(\phi)}\Phi\|_{L_t^1[0,T] H^{1+\ep}_x}, \quad \ep>0,
\end{equation}
where $q>2$ is sufficiently close to $2$.\begin{footnote}{We refer to Theorem \ref{dyastric} for the precise version, with the relation between $\ep, q$ specified therein. }\end{footnote} By choosing a small  life-span $T$, this estimate (\ref{10.13.1.19}) can close the bootstrap argument with the help of the standard product estimate if the last term is quadratic in $\bp \Phi$.
The proof of the Strichartz estimates depends highly
sensitively on the regularity of the spacetime metric. It at least requires the bound on  $\|\bp \Phi\|_{C_t^0[0,T]H^{1+}_x}+\|\bp \Phi\|_{L_t^2[0,T] L_x^\infty}$.

Consider the term of $\curl\Omega$ on the right hand side of (\ref{10.12.5.19}).  Note  that  with a normalization $\curl \Omega$ verifies the transport equation (\ref{4.25.4.19}).
We can at best expect the regularity of $\curl \Omega$ is preserved with time. This means one has to assume  $\curl \Omega\in H^{s-1}_x$ when $t=0$,  in order to complete the estimate of (\ref{10.12.5.19}). However, our assumption of the initial data is that
$\curl \Omega\in H^{s'-1}_x, \, 2<s'<s$, with $s'$ arbitrarily close to $2$, which is insufficient for using (\ref{10.12.5.19}).
  Therefore, in order to achieve (\ref{eng_f}), we need a different strategy from the above straightforward treatment.

The other serious difficulty arises from the energy propagation of the vorticity, or more precisely, in bounding
the norm $\curl \Omega(t)\in H^\a_x, \, \a\ge 1$, for $0<t\le  T$. To understand the issue, we consider the normalized  transport equation  (\ref{4.25.4.19}) for $\curl \Omega$, which symbolically reads
\begin{equation}\label{10.13.4.19}
\bT \fC= \p v \p \Omega, \quad \quad\mbox{  with }\fC= e^{-\varrho} \curl \Omega.
\end{equation}
Applying (\ref{9.07.4.19}) to $F=G=\fC$, and using $\Tr\cir{k}=-\div v$,  we obtain in view of (\ref{10.13.4.19}) that
\begin{equation*}
\|\fC(t)\|_{H^{\a}_x}\les \|\fC(0)\|_{H^{\a}_x}+\int_0^t \|\p \Omega\c\p v\|_{H^{\a}_x} dt', \mbox{ with }1\le \a\le s'-1.
\end{equation*}

By the standard product estimate
\begin{equation*}
\|\p \Omega \c \p v\|_{H^\a_x}\les \|\p \Omega\|_{L_x^\infty} \|\p v\|_{H^\a_x}+\|\p \Omega\|_{H^\a_x}\|\p v\|_{L_x^\infty}, 1\le \a\le s'-1,
\end{equation*}
we derive
\begin{equation}\label{10.16.1.19}
\|\fC(t)\|_{H^{\a}_x}\les \|\fC(0)\|_{H^{\a}_x}+\int_0^t (\|\p \Omega\|_{L_x^\infty}\|\p v\|_{H^\a_x}+\|\p v\|_{L_x^\infty} \|\p \Omega\|_{H^\a_x}) dt', \mbox{ with }1\le \a\le s'-1.
\end{equation}
  In view of (\ref{div}) and the elliptic estimate for the $\div$-$\curl$ Hodge system, we can obtain
 \begin{equation}\label{11.7.1.19}
 \|\p \Omega \|_{H^\a_x}\les \|\fC\|_{H^\a_x}+l.o.t..
 \end{equation}
Thus by applying Gronwall's inequality and the above estimate to (\ref{10.16.1.19}),  to  bound $ \|\p \Omega (t)\|_{H^\a_x},$ with $1\le \a\le s'-1$,  we need the bound  $\|\p \Omega\|_{L_t^1[0,T] L_x^\infty}$.
 The latter can be bounded  by requiring $\curl \Omega(0)\in C^{0,0+}_x$,  by using (\ref{4.25.4.19}) and the elliptic theory for the Hodge operator $ (\div, \curl)$. This assumes an additional $\f12$-derivative than  Theorem \ref{thm1}  in terms of Sobolev embedding.

  Note that without the bound of the vorticity term $\|\curl \Omega\|_{H^1_x}$ or equivalently $\|\fC\|_{H^1_x}$,  we lose the bound of  $\|\Box_{\bg(\Phi)}\Phi\|_{L_t^1 H^1_x}$
    in view of  (\ref{10.13.2.19}). According  (\ref{10.12.5.19}) with $s=2$ and (\ref{10.13.1.19}), this means we lose both the bounds of $\|\bp \Phi\|_{C^0_t[0,T] H^1_x}$ and $\|\bp\Phi\|_{L_t^2[0,T] L_x^\infty}$. Hence, this is neither a chance to establish the Strichartz estimate for the linearized wave (\ref{wavelin}).
Based on the above treatment, in order to close the energy argument, one may have to assume $\curl \Omega\in H^{s-1}\cap C^{0,0+}$ when $t=0$. 

The more crucial obstruction from the rough vorticity  actually lies in the difficulty in establishing the Stricharz estimate for
the linear wave equation (\ref{wavelin}), which is not yet shown in the above  analysis.
  To prove the Stricharz estimate, it heavily relies on the regularity of the optical function of the Lorentzian space time.
  The optical function $u$ in the acoustic spacetime, defined by the solution of the Eikonal equation  $\bg^{\a\b}\p_\a u\p_\b u=0$,  can be constructed by level hypersurfaces,
  formed by generating null geodesic congruences satisfying certain initial conditions. The regularity of the foliation by  the level surfaces of $u$, i.e. the null cones $C_u$, is
  particularly important for obtaining the dispersive estimates of the linear wave equation (\ref{wavelin}) in the acoustic spacetime (see Theorem \ref{decayth}).

 With $\bb^{-1}=\bT(u)$, we  define the null area expansion, $\tr\chi=-\bb \Box_\bg u$, which is
the most important geometric quantity of the null cones. Denote the normalized null geodesic generator by $L=-\bb\bd u$,  where $\bd$ is the Levi-Civita connection of $\bg$. $\tr\chi$ satisfies the Raychaudhuri equation given in (\ref{s1}), where the Ricci component  $\bR_{LL}$
is the highest order term on the right hand side. Note  $\curl \Omega$ appears in the equation (\ref{6.23.2.19}) for $\bR_{LL}$ as one of the main terms. To achieve the crucial ${L_t^2 L_x^\infty}$ control of $\tr\chi$ for proving the decay estimate in Theorem \ref{decayth}, we need the $L^\infty$
bound of $\curl \Omega$ along each null cone $C_u$. If assuming the bounded $C^{0,0+}$ norm
of $\curl\Omega$ initially, with the help of the transport equation, the $L^\infty$ control on $\curl \Omega$ can hold for $t\le T$. In this situation the smooth vorticity term has
 no influence to the analysis of causal geometry in \cite{Wangrough} (See Sections 5 and 6 therein).  However, with merely the control of
$H^{2+}_x$ bound  of  $\Omega(t)$, we can only obtain  the $H^{\f12+}_x$ bound on null cones for $\p \Omega(t)$ by the standard trace inequality. Thus, there lacks a $\f12$-derivative for $\curl\Omega$ on the
null cone for achieving the desired bound on $\tr\chi$. This remains to be  the main issue even if adopting the alternative approach in \cite{Tataru} to prove the Strichartz
estimate for the linear wave.


  The mechanism described above is  based on  treating the velocity identically as in the irrotational case, provided that the vorticity datum is sufficiently smooth.
 We list below the regularity it requires on the vorticity to summarize the above discussion.
\begin{enumerate}
\item[(a)] For achieving the energy estimate for $v, \varrho$ in (\ref{eng_f}), it requires the initial datum of the vorticity to be bounded in $H^s$ norm;
\item[(b)] For bounding the energy of vorticity, it requires the initial datum of $\curl \Omega$ to be bounded in $C^{0,0+}$;
\item[(c)] For controlling the regularity of the optical function $u$, it needs the initial datum of $\curl \Omega$ to be bounded in $C^{0,0+}$ as well.
\end{enumerate}

With the much weaker assumption on the datum of the vorticity in Theorem \ref{thm1}, we will give our approach based on a deeper understanding of the behaviour of the velocity.

\subsubsection{Decoupling method and cancellations}

 $\Omega$ (or $\fw$) apparently verifies a good transport equation (\ref{4.10.3.19}).  As such, we conceptually regard the vorticity as the relatively
stationary part of $\p v$, which does not exhibit the same dispersive property as a free wave. Since $\fw$ is defined in terms of $\p v$,   with the rough data on $\curl \Omega$, $\p v$ is
 not expected to have the same  dispersive property as the  irrotational fluid even within a short life-span, which is demonstrated by the result of Bourgain-Li \cite{Bourgain-Li} in the extremal case
when $\div v=0$.
To trace the difference back from $v$, we note for the incompressible case, that is when $\div v=0$, the equation (\ref{4.10.1.19}) can be written as
\begin{equation*}
\Box_\bg v=-\bT \bT v- c^2 \curl^2 v+l.o.t= -c^2 \curl^2 v+l.o.t.
\end{equation*}
If repeat the procedure in Section \ref{10.13.3.19} to treat the term $c^2 \curl^2 v$ on the right hand side as a source term, there again requires the smoother
data on vorticity as stated in Section \ref{10.13.3.19} (a)-(c). Nevertheless, clearly from the above equation, the vorticity term can
 be cancelled by the same term on the left hand side. In this situation, it is improper to treat $v$ as the solution for the wave equation (\ref{4.10.1.19}). 
 Thus, in our approach the behaviour of the velocity of the compressible fluid is understood as a blend of the key features from both of the irrotational fluid and the incompressible fluid. Merely applying the approach for the irrotational fluid to (\ref{4.10.1.19}) does not match with the physical nature of $v$, consequently,
 does not provide us with the desired result.
  Schematically, to prove Theorem \ref{thm1}, by decomposing the velocity into  two fundamentally  different parts, we can divide and conquer.

 We hence introduce a decomposition for the velocity
 \begin{equation}\label{dcp1}
v^i=v_+^i+\eta^i,
\end{equation}
where $v_+$ is a vector-valued wave function, satisfying an improved  wave equation system; and  $\eta$ is the part of the velocity determined by the vorticity. $\Box_\bg \eta$ is supposed to cancel the vorticity derivative on the right hand side of (\ref{4.10.1.19}).
 We then will treat $v_+$ and $\varrho$ together as wave functions, for which we apply the energy argument together with establishing the Strichartz estimate, since they verify better wave equations than   (\ref{10.13.2.19}); for $\eta$, we will rely on elliptic estimates and the transport equation for the vorticity.

    More precisely, we define the vector field $\eta^i$  by
\begin{equation}\label{4.10.5.19}
\La^2\eta^i:=(I-\Delta_e)\eta^i=\curl\fw^i,
\end{equation}
where  $\Delta_e$ is the Laplace-Beltrami operator of the Euclidean metric. By direct computation,
\begin{equation*}
\curl \fw^n=e^\varrho\curl \Omega^n+\tensor{\ep}{^{nij}}\fw_j \p_i \varrho.
\end{equation*}
Hence,
we can substitute  (\ref{4.10.5.19}) to  the right hand side of (\ref{4.10.1.19})  to derive
\begin{align*}
\Box_\bg v^i=c^2 \Delta_e \eta^i-c^2 \eta^i+c^2 \fw_n \p_m \varrho \ep^{imn}+\sQ^i.
\end{align*}
Since the induced metric $g_{ij}=c^{-2}\delta_{ij}$ is conformally flat, there holds
$$\Delta_g(\eta^i)=c^2 (\Delta_e (\eta^i)-\p^j (\log c)\p_j (\eta^i)).$$
  Hence  in view of (\ref{dcp1}),
 we can obtain the wave equation for $v_+$
\begin{equation}\label{wave1}
\begin{split}
\Box_\bg v_+^i&=-\bT \bT v_+^i+\Tr k \bT v_+^i+\Delta_g v_+^i\\
&=\bT \bT \eta^i- \Tr k\bT \eta^i-c^2(\eta^i-\p^j \log c \p_j (\eta^i))+c^2 \fw_n \p_m \varrho \ep^{imn}+\sQ^i,
\end{split}
\end{equation}
which gives
\begin{equation}\label{wavevpl}
\Box_\bg v_+^i=\bT \bT \eta^i- \Tr k\bT \eta^i+\widetilde\sQ^i-c^{2}\eta^i,
\end{equation}
where the quadratic term
\begin{equation*}
\widetilde\sQ^i=\sQ^i+c^2( \fw_n \p_m \varrho \ep^{imn}+\p^j (\log c) \p_j (\eta^i)).
\end{equation*}
Such reduction transforms the higher order linear term (\ref{4.10.5.19}), which appears on the right hand side of the equation (\ref{4.10.1.19}), to the term $ \bT \bT \eta$. The latter still looks to be a linear higher order term. We manage to cancel this term during the applications of the wave equation (\ref{wavevpl}). As the consequence of the cancellations, there involves merely spatial derivatives of $\bT\eta$ in the analysis of this article, which can be well-controlled by the derivative bound of vorticity with the help the elliptic estimate.

 Due to the appearance of the $\bT \bT \eta$ in (\ref{wavevpl}), we  consider the following equation for wave functions $v_+$ and  $\varrho$,  which better represents the structure of (\ref{wavevpl}),
 \begin{equation}\label{wave_2}
\Box_\bg \Psi=W+\bT Y-\Tr k Y,
\end{equation}
where the pair of functions $(\Psi, Y) $ is either $(v_+, \bT\eta)$ or $(\varrho, 0)$ with the corresponding error terms contained in  $W$.

 Consider the above quasilinear wave equation for $\Psi$, we manage to cancel $\bT Y$ in the following two major applications.
\begin{itemize}
\item For energy estimates and the flux estimates along null cones, we introduce the modified current (\ref{9.13.1.19}), which was originally constructed in \cite[Section 3.2]{WangCMCSH}, to cancel the term $\bT Y$. By  combining the  elliptic estimates for $\eta$,  we can obtain the total energy of $\varrho, v$. (See  Corollary \ref{eng_wave}).  Due to the $\bT\bT \eta$ term in  (\ref{wave1}), we need to avoid commuting $\bT$ with the wave operator $\Box_\bg$,   the equations in (\ref{4.23.1.19}) allow us to transform  from spatial derivatives for $v, \varrho$ to their time derivatives.

\item To bound the  Strichartz norm  $\|\bp \Psi\|_{L_t^2 L_x^\infty}$ with $\Psi=v_+, \varrho$, 
     instead of treating the full right hand side of (\ref{wave_2}) as the inhomogeneous term of (\ref{wavelin}),
   we adapt the method of linearization in \cite[Section 4.2]{WangCMCSH} to reduce  $\bT Y$  carefully into   $Y$ via the Duhamel's principle in a delicate manner. (See (\ref{pmug}).)
    \end{itemize}
By the completion of the above two steps,  we  bound the energy of $\Psi=v_+, \varrho$ up to the highest order in (\ref{5.11.3.19})
and obtain the necessary $\eta$  estimates  by elliptic estimates, provided that the norm  of  $\|\Omega\|_{H^{s-\f12}_x}$ can be bounded. Therefore to solve the difficulty of (a) listed in Section \ref{10.13.3.19}, for the vorticity we only need the bound of  $\|\Omega(t)\|_{H^{s-\f12}_x}$.

\subsubsection{ Trilinear structure for the propagation of the vorticity}
 Due to (c) in Section \ref{10.13.3.19},  we will have to control the higher order norm $\|\Omega(t)\|_{H^{s'}_x}, 2<s'<s$ for $t\le T$,  even though the energy control of $v, \varrho$ and the linearization only rely on the bound of $\|\Omega\|_{H^{s-\f12}_x}$. Therefore we need to solve the difficulty in (b).

To solve this difficulty, we first note that to bound $\|\Omega(t)\|_{H^{s'}_x}$, it suffices to  bound $\|\curl\fC(t)\|_{H^{\a}_x}$ with $0\le \a\le s'-2, \, 0<t\le T,$  in the same manner as in (\ref{11.7.1.19}) by applying elliptic estimates for the  Hodge operator  $(\div, \curl)$  and comparison estimates.

To control  $\|\curl \fC\|_{L^2_x}$, we integrate the following identity in $t'\in [0,t]$.
\begin{equation}\label{11.8.1.19}
\p_t\int_{{\mathbb R}^3} |\curl \fC|^2 dx=  \int_{{\mathbb R}^3 } (2\bT \curl \fC\c   \curl \fC- \Tr \cir{k}|\curl \fC|^2)dx,
\end{equation}
which is obtained by applying (\ref{9.07.4.19}) to $F=G=\curl \fC$. We recall $\Tr \cir{k}=-\div v$ for treating the last term.

For the term $\bT \curl \fC$, we recast the commutation formula (\ref{9.07.1.19}) schematically as
\begin{equation}\label{10.13.5.19}
\bT \curl \fC-\curl \bT \fC=\p v \p \fC.
\end{equation} 
Substituting  it  to the first term on the right hand side of (\ref{11.8.1.19}), and also  using the elliptic estimate for Hodge system to control $\p \fC$ by $(\div \fC, \curl \fC)$ in $L^2_x$ , allow us to focus on treating the following integral
 \begin{equation*}
\I= \int_0^t \int_{{\mathbb R}^3} \curl \bT \fC \curl \fC dx dt',
 \end{equation*}
 since other terms can be treated by Gronwall's inequality  with the expectation of bounding $\|\p v\|_{L_t^1[0,T] L_x^\infty}$.
 
Instead of using the symbolic formula (\ref{10.13.4.19}) to $\bT \fC$ in the above, we observe trilinear structures in $\I$ by deriving the precise formula in (\ref{9.07.8.19}), with the main terms listed below \begin{footnote}
	{$\p^n$ here is $\delta^{nl}\p_l$ instead of $n$-th order derivative. We will clarify whenever such confusion may occur.}
\end{footnote}
\begin{align*}
\I=\int_0^t\int_{\Sigma_{t'}}e^{-\varrho}\{\p^n \p_m v^j \p_n \Omega_j+\p_j \bT \varrho\p_m \Omega^j\}\curl\fC^m d x dt'+\cdots.
\end{align*}
 See the complete formula in  (\ref{9.08.14.19}).

Clearly, since we at most can obtain the bound $\|\p^2 v(t)\|_{H^{s-2}_x}$  for $t\le T$, to close the estimate for $\I$ with the bound of $\|\curl \fC(t)\|_{L_x^2}$, it seems that we need to bound  $\|\p \Omega\|_{L_t^1 L_x^\infty([0,T]\times {\mathbb R}^3 )}$.

Instead of seeking for the direct bounds, for the first term, since the last factor in the integrand is $\curl \fC^m$, and $\p_m (\curl \fC)^m=0$, we calculate the integrand by
\begin{align*}
e^{-\varrho}\p^n \p_m v^j \p_n \Omega_j \curl \fC^m&=\p_m(\p^n v^j\p_n \Omega_j \curl \fC^m e^{-\varrho})-e^{-\varrho}\p^n v^j \p_m \p_n \Omega_j \curl \fC^m\\
&-\p_m (e^{-\varrho}) \p^n v^j  \p_n \Omega_j \curl \fC^m.
\end{align*}
After integration on $\Sigma_{t'}$, the first term  vanishes due to integration by parts, and other terms can be treated by using the elliptic estimate, Sobolev embedding and Gronwall's inequality.

For the second term in the integrand of $\I$, note
\begin{align*}
e^{-\varrho}\p_j \bT \varrho\p_m \Omega^j \curl\fC^m&= e^{-\varrho}(\bT \p_j \varrho+[\p_j, \bT]\varrho) \p_m \Omega^j\curl \fC^m\\
&=\bT(e^{-\varrho}\p_j \varrho \p_m \Omega^j \curl \fC^m)-\p_j \varrho \bT(e^{-\varrho}\p_m \Omega^j \curl \fC^m)\\
&+e^{-\varrho}[\p_j, \bT]\varrho \p_m \Omega^j\curl \fC^m.
\end{align*}
We then carry out the integration by parts in the spacetime to treat the first term
 on the right hand side  with the help of (\ref{9.07.4.19}), which only generates lower order terms.  For the second term, we can substitute  the transport equations of $\p\Omega$ and $\curl \fC$ to save derivatives. The commutator term is relatively of lower order.
See the proof of (\ref{9.07.5.19}) for the full detail.

To control $\|\curl\fC\|_{H^{\a}_x}$ with $\a=s'-2, \, 0<t\le T,$  besides  entailing the above set of integration by parts, for the estimate of the highest order,  we have an additional issue coming from the high low interaction for the product in the first term below
\begin{equation}\label{10.13.6.19}
\int_0^t \|\p v \p\fC(t')\|_{H^{s'-2}_x}\|\curl \fC(t')\|_{H^{s'-2}_x} dt',
\end{equation}
which comes from the commutator term in  (\ref{10.13.5.19}).

Note that with the expectation of bounding \begin{footnote}{The definition of the Besov norm can be found at the end of Section \ref{h2vorticity}.}\end{footnote}
$
\|\p v\|_{L_t^2 B^{s'-2}_{\infty, 2,x}}
$
by establishing Strichartz estimate for $v_+$ and the elliptic estimate for $\eta$,
(\ref{10.13.6.19}) can be treated by using Gronwall's inequality and the elliptic estimate.
There is yet no control on  $ \|\p v\|_{L_t^2 B^{s-2}_{\infty, 2,x}}$ due to our approach in \cite{Wangrough} for (\ref{4.10.4.19}), even with smoother vorticity data than assumed.  The $s$-$s'$ energy hierarchy for $v, \varrho$ and $\fw$ helps crucially to complete the estimate of (\ref{10.13.6.19}).

\subsubsection{Fundamental structures for  the causal geometry of the acoustic spacetime}
At last, we consider the difficulty from  (c) in Section \ref{10.13.3.19}, that is  to control the acoustic null cones. We will focus on treating the null area expansion $\tr\chi$, which is the most important quantity in the geometric analysis.  The regularity from the general derivative of the  metric is by no means sufficient for the purpose even in the previous works for quasilinear wave equations.

Since $\tr\chi$ verifies the Raychaudhuri equation (\ref{s1}), i.e.
\begin{equation*}
L\tr\chi+\f12 (\tr\chi)^2=-\bR_{LL}+l.o.t.,
\end{equation*}
the strategy is to gain regularity by taking advantage  of the structure of the Ricci component $\bR_{LL}$, which works particularly well in Einstein vacuum spacetime due to the vanishing {\bf Ric}.

For the acoustic spacetime, we derive in (\ref{6.23.2.19}) that
\begin{equation}\label{11.11.1.19}
\bR_{LL}=L(\Xi_L)-e^{\varrho}\delta_{ij}\bN^j \curl \Omega^i+l.o.t.,
\end{equation}
where the one form  $\Xi_\mu=\Ga^\eta_{\a\b}(\bg) \bg^{\a\b}\bg_{\eta \mu}$, $\Ga(\bg)$ is the Christoffel symbol of $\bg$ (see (\ref{ricc6.7.2}) and (\ref{9.30.13.19})), and  $\bN$ is the outward unit normal of $S_{t,u}:=\Sigma_t\cap C_u$.
The basic analysis on null hypersurfaces relies on the energy fluxes of $v, \varrho$ and the vorticity controlled in Section \ref{c_flux}.
We can control the double-curl flux of vorticity $ \|\mu^{s'-2}P_\mu\curl \fC\|_{l_\mu^2 L^2(C_u)}+\|\curl \fC\|_{L^2(C_u)} $ \begin{footnote}{We refer to (\ref{LP4.12.1.19}) for the definition of the Littlewood-Paley projector.}\end{footnote} by energy method and the trilinear estimates, nevertheless, can not bound  $\p \fC$ at the same level, since there lacks the trilinear structure if using $\bT \p \fC$ to propagate the energy, and  there is neither a proper Hodge systems for $\fC= e^{-\varrho} \curl\Omega$ available on $S_{t,u}$.  Applying the trace inequality and using the full energy bound of $\|\C (t)\|_{H^{s'-1}_x}$ on $\Sigma_t$ lead to a loss of $\f12 $-derivative due to the restriction to $ S_{t,u}$. We then lose the bound on $\|\curl \Omega\|_{L^\infty(S_{t,u})}$ by $\f12$-derivative.

 For quasilinear wave equations (\ref{4.10.4.19}), one can gain regularity for $\tr\chi+\Xi_L$  as in \cite{Kcom, KRduke, Tataru, Wangrough} due to (\ref{11.11.1.19}) and the absence of  vorticity. Due to the weak regularity, the derivative of  the term $\Xi_L=\Xi_\mu L^\mu$ caused one of the main difficulties for proving the sharp local well-posedness for the solution of  (\ref{4.10.4.19}) via the geometric approach devised in \cite{Kcom} and \cite{KRduke}. The difficulty was solved in \cite{Wangrough} by introducing the geometric normalization via the conformal change of the spacetime metric together with bounding the conformal energy by an un-canonical energy method. The issue from $\Xi_L$  remains the same for the compressible Euler equations, for which we adopt the method in \cite{Wangrough}.

In the sequel, we focus on solving the difficulty from the Ricci defect caused by the rough $\curl \Omega$  in bounding $\|\tr\chi-\frac{2}{\tir}\|_{L_t^2 L_x^\infty(u\ge 0)}$, with $\tir =t-u$ in (\ref{ric1}). The estimate is crucial for proving the Strichartz estimate in Theorem \ref{dyastric}. The region of the norm is, roughly, the domain of influence of a unit ball with the time-span nearly upto the large frequency $\la$ fixed in Theorem \ref{dyastric}. In such region, $0\le u\le t$, with $u=0$ on the boundary of the domain of influence and $u=t$ along the time-axis. (We refer to Section \ref{setupcone} for the set-up of the acoustic null cones $C_u$.)

Now consider the transport equation for $z=\tr\chi+\Xi_L-\frac{2}{\tir}$ in (\ref{lz}), which is recast below
\begin{equation*}
Lz+\frac{2z}{t-u}=e^\varrho \curl \Omega_\bN+\frac{2}{r}(\Xi_L-k_{\bN\bN})+\cdots,
\end{equation*}
where nonlinear terms are omitted. Note the term $\curl \Omega$ on the right hand side is merely bounded in $H^{\f12+}(S_{t,u})$, which is short of $\f12$-spatial-derivative on $S_{t,u}$ to give the bound of  $\|z\|_{L^\infty(S_{t,u})}$ directly. By  (\ref{9.12.1.19}), $\curl \Omega_i\bN^i=\ep^{AB} \sn_A \Omega_B$
where $\sn$ denotes the Levi-Civita connection on $S_{t,u}$ with respect to the induced metric, $\{e_A\}_{A=1}^2$ forms the orthonormal basis and $\ep^{AB}$ is the volume form on $S_{t,u}$.  Since it is an angular derivative, we can not directly write it into $L F+E$ with $F$ and $E$ smoother functions.
   
   To solve the difficulty, we uncover fundamental structures in Section 7 on the angular derivatives of $\bR_{LL}$ and the radial component of the second fundamental form $k_{\bN\bN}$, in the acoustic spacetime.
   
   As the first step, we derive  the $L^\infty$ bound on $\tir^\f12 z$ by applying
 the Sobolev inequality
\begin{equation}\label{10.15.1.19}
|\tir^\f12 z|\les \|\tir^{\f12-\frac{2}{p}} (\tir\sn)\rp{\le 1}z\|_{L^p(S_{t,u})}, \qquad 0<1-\frac{2}{p}<s'-2.
\end{equation}
To bound the right hand side, we consider the transport equation of $\sn z$ by differentiating (\ref{lz}). It is crucial to observe the following trace decomposition derived in (\ref{9.12.2.19}),
 \begin{align*}
 &\sn(e^\varrho \curl \Omega_\bN)=\sn_L(e^\varrho\curl \Omega)_A+e^{\varrho} {e_A}_i \Pi^{ij} \tensor{\ep}{_{jm}^l}(\curl^2\Omega)_l\bN^m+l.o.t.,
 \end{align*}
with $\Pi^{ij}=\bg^{ij}+\bT^i\bT^j-\bN^i \bN^j$, since  this quantity can  not be directly bounded by the double-curl flux for vorticity. 

 This leads to
\begin{align}\label{10.13.7.19}
\sn_L  \sn z+\frac{3}{t-u} \sn z=\sn_L(e^\varrho\curl \Omega)_A+e^{\varrho}{e_A}_i \Pi^{ij} \tensor{\ep}{_{jm}^l}(\curl^2\Omega)_l\bN^m+\frac{2}{\tir}\sn(\Xi_L-k_{\bN\bN})+\cdots.
\end{align}
By combining the first terms on both sides, we can derive the transport equation for $\sn_A z-(e^\varrho\curl \Omega)_A$ in (\ref{ldz_2}) and obtain the following estimate by integrating along the null cone $C_u$,
\begin{align}
\|\tir(\sn_A z-e^\varrho (\curl \Omega)_A)\|_{L_\omega^p(S_{t,u})}&\les \tir^{-1}\int_{t_\tmin}^t\|\tir\sn(\Xi_L-k_{\bN\bN})\|_{L_\omega^p}dt'\label{10.15.3.19}\displaybreak[0]\\
&+\tir^{-1}\int_{t_\tmin}^t \|\tir^2 \curl^2 \Omega  \|_{L_\omega^p} dt'+\cdots\nn,
\end{align}
where  $t_\tmin=\max(u, 0),$  $\omega \in{\mathbb S}^2$ on $S_{t,u}$ is the pull-back coordinate via the null geodesic flow, (see the construction in Section \ref{null_cone_set},) and we omitted the terms of initial data and nonlinear terms.  We can
 control  $\|\tir^\frac{3}{2}(\sn_A z-e^\varrho (\curl \Omega)_A)\|_{L_\omega^p(S_{t,u})}$ by bounding the right hand side of the inequality with  the flux control in Section \ref{c_flux} and its consequences in Proposition  \ref{7.13.2.19}. The desired bound of $\sn z$ in (\ref{10.15.1.19}) follows since the other term $\tir e^\varrho (\curl \Omega)_A$ can be easily bounded in  $L^p_\omega$. Thus we can obtain the pointwise bound on $\tir^\f12 z$. 

As the second step, we need to derive the integral bound $\|z\|_{L_t^2 L_x^\infty(u\ge 0)}$. Note that near the time-axis, the estimate can not follow from the pointwise estimate of $\tir^\f12 z$, since $\tir=t-u$ can be approaching to $0$. In terms of $L^p_\omega$ estimate, we need the bound of
$
\| \sup_{0\le u\le t}\|\tir\sn z\|_{L_\omega^p}\|_{L_t^2}.
$
In view of  (\ref{10.15.3.19}), the term $\frac{2}{\tir}\sn(\Xi_L-k_{\bN\bN})$ in (\ref{10.13.7.19}) is the obstruction for controlling the above bound since it may require the bound of $\| \tir\sn (\Xi_L-k_{\bN\bN})\|_{L_t^2 L_u^\infty L_\omega^p(C_u)}$,  much stronger than the bound on flux.

We manage to cancel such term by further deriving in Proposition \ref{6.14.2.19} that
 \begin{equation}\label{10.15.2.19}
k_{\bN\bN}-\f12 \Xi_L= -\f12\big(L(\log c+\varrho)+2L(v)_\bN\big).
\end{equation}
Based on the above trace decomposition, and the equation $L\log \bb =-k_{\bN\bN}$ in (\ref{lb}), we consider instead of $z$ the normalized quantity $\sY=\bb (\tr\chi+\Xi_L)-\frac{2}{\tir}$. Note $z=\bb^{-1}\sY+2\frac{\bb^{-1}-1}{\tir}$, where the second term on the right hand side and $\bb^{-1}$ are relatively easier to control. Thus, it remains to derive the bound of $\|\sY\|_{L_t^2 L_x^\infty(u\ge 0)}$.    Note the transport equation of $\sn \sY$ takes the form,
\begin{align}
\sn_L (\sn_A \sY-\bb e^\varrho (\curl \Omega)_A)&+\frac{2}{\tir}(\sn_A \sY-\bb e^\varrho (\curl \Omega)_A)\nn\\
&=-\frac{2}{\tir} \sn (\Xi_L-2 k_{\bN\bN})+\bb e^{\varrho}{e_A}_i \Pi^{ij} \tensor{\ep}{_{jm}^l}(\curl^2\Omega)_l\bN^m+\cdots.\label{11.7.2.19}
\end{align}
By using (\ref{10.15.2.19}), we derive the following trace decomposition
\begin{equation*}
2 \sn (\Xi_L-2 k_{\bN\bN})=L\sn\pi_1+l.o.t, \quad \mbox{ with } \pi_1=2\sn(\log c+\varrho)+4 \sn v^i \bN^j \bg_{ij}.
\end{equation*}
 Substituting the above identity to the first term on the right hand side of (\ref{11.7.2.19}), we can see that  the transport equation of  $\sn_A\sY-\bb e^\varrho (\curl \Omega)_A+\frac{\pi_1}{\tir}$ no longer contains the singular higher order term. 
  Thus we manage to control the estimates of $\sY$ and the desired estimate for $z$.  See the details of the proof in Proposition \ref{10.17.1.19}. The estimate of $\tr\chi-\frac{2}{\tir}$ follows due to  the bounded norm of $\|\Xi_L\|_{L_t^2 L_x^\infty}$ by the bootstrap assumption (\ref{BA1}) and the estimate of $z$.

  To reduce the technical baggage for controlling the full null second fundamental form (with $\tr\chi$ being its trace), we manage to improve the treatment for the torsion tensor $\zeta$. One can refer to Section \ref{causal_reg} for the details.

\subsection{Organization of the proof}
The  proof of  Theorem \ref{thm1} consists of three main building blocks: energy and flux propagation, linearization and reduction for Strichartz estimates, control of the causal geometry of the acoustic spacetime.

The energy propagation includes the propagations of $v,\varrho$ and the vorticity from the Cauchy data. They are completed in Section 2 and 3 with the help of the main bootstrap assumption (\ref{BA1}) in ${\mathbb R}^3\times[0,T]$. In Section 2, the energy estimate of $v_+, \varrho$ up to the order of $H^2_x$ can be bounded together with deriving the bounds of the $L^p_x$ norms of $\p \Omega, \p \fw, \p^2\eta$, with $2\le p\le \frac{3}{3-s}$ from the initial data. They are summarized in Corollary \ref{comp_1}. To bound the highest order energy  $\|\p v, \p \varrho\|_{H^{s-1}_x}$,  Proposition \ref{erro_h} shows that it  relies on the norm of $\|\curl \Omega(t) \|_{H^{s-\frac{3}{2}}_x}$. We then bound the stronger norm $\|\Omega\|_{C^0_t[0,T] H^2_x}$ in (\ref{9.07.5.19})  in Section 3.1 by performing the trilinear estimates. This closes the highest order energy control on $v$ and $\varrho$ in Corollary \ref{eng_wave}. We then use this result to obtain the highest order energy bound for the vorticity in Proposition \ref{9.07.16.19}.

To prove (\ref{BA1}), as in \cite{Wangrough}, we take the framework of the physical approach initiated by the work of \cite{Kcom} (see also \cite{KRduke, KREinst}). In Section 4-5 we obtain the Strichartz estimates in  Theorem \ref{main}  based on Theorem \ref{BT}, which proves (\ref{BA1}). The proof consists of a series of reductions. We first prove Theorem \ref{main} by applying  the dyadic Strichartz estimates in Theorem \ref{dyastric}  to the Littlewood-Paley pieces  $P_\la \p v_+$ and $P_\la \p \varrho$, with large $\la$,  which are represented in (\ref{pmug}) by the solution of the linearized wave equation.  By running the $\T \T^*$ argument,\begin{footnote}{See \cite[Section 4]{WangCMCSH} and \cite[Setion 3 and Section 9]{Wangrough} for the $\T \T^*$ argument and also for more details of the reduction.}\end{footnote} 
 we then reduce the proof of  Theorem \ref{str2}, which is the rescaled statement of Theorem \ref{dyastric}, to the decay estimate in Theorem \ref{decayth}.  By proper localizations in the physical space, it is then reduced to the proof of Proposition \ref{lcestimate} and then further reduced to controlling the conformal energy in Theorem \ref{BT} in the domain of influence of a unit ball with the life-span  slightly less than the large frequency $\la$.
 
 To prove Theorem \ref{BT}, we need to run the multiplier approach in \cite[Section 7]{Wangrough}  which also adopts a conformal change of the spacetime metric for the purpose of normalizing the causal geometry. To obtain the necessary geometric control required therein, we carry out geometric analysis of the acoustic null cones in Section \ref{causal_reg} and Section \ref{conf}.  In Section 5-7, we provide preliminaries for controlling the causal geometry. In Section 5, we give the geometric set-up for the acoustic null cones. In Section \ref{c_flux}, we obtain the control on derivatives of the metric and vorticity along the null cones from the energy estimates in the ambient acoustic spacetime, upto the highest order, with the help of the trace estimates and elliptic estimates for derivatives of $\eta$. In Section 7, we derive the important trace decompositions for $\sn(\bR_{LL})$ and the second fundamental form in Proposition \ref{struc} and Proposition \ref{6.14.2.19}, and also provide the structure equations of the null tetrad and the necessary decompositions of Riemann curvature. In Section \ref{causal_reg}, by using the results in Proposition \ref{struc} and Proposition \ref{6.14.2.19},  we prove Proposition \ref{cone_reg} simultaneously with other estimates therein. With the help of them,  in Section \ref{conf}, we further control the renormalized mass aspect function and obtain the derivative control of the conformal factor used for the geometric normalization, in Proposition \ref{6.23.9.19} and Proposition \ref{dcmpsig}.  Thus we complete the full set of geometric estimates for  running the proof in \cite[Section 7]{Wangrough}. This completes the proof of Theorem \ref{BT}.

 \section{Energy estimates for wave functions $(v_+, \varrho)$}\label{eng_sec}
 \subsection{Bootstrap assumptions}
 For the fixed number  $2<s<\frac{5}{2}$, we fix
 \begin{equation*}
 0<\ep_0<\frac{s-2}{5}, \qquad s'-2=\delta_0=\ep_0^2.
 \end{equation*}
Now we make the following preliminary bootstrap assumption in a spacetime slab $[0,T]\times {\mathbb R}^3$ with $T>0$ (or, identically, $(\M,\bg)$) that
\begin{equation}\label{BA1}
\|\bp\varrho, \bp v_+, \bp v\|^2_{L_t^2 L_x^\infty}+\sum_{\la \ge 2}\la^{2\delta_0}\| P_\la (\bp\varrho, \bp {v_+})\|^2_{L_t^2 L_x^\infty}\le 1,
\end{equation}
 where
$P_\la$ is the Littlewood-Paley projector  with
frequency $\la=2^k $ defined for any function $f$ by
\begin{equation}\label{LP4.12.1.19}
P_\la f(x)=f_\la(x)=\int_{{\mathbb R}^3} e^{i x\c \xi} \b(\la^{-1}\xi) \hat f(\xi) d\xi
\end{equation}
with $\b$ a smooth function supported in the shell $\{\xi: \f12\le |\xi|\le 2\}$
satisfying $\sum_{k\in {\Bbb Z}}\b(2^k \xi)=1$ for $\xi\neq 0$. We refer to \cite{Stein2} for detailed properties of Littlewood-Paley decompositions.

We also make an auxiliary bootstrap assumption that
\begin{equation}\label{ABA1}
\|\p \varrho\|_{L_t^2 L_x^3([0,T]\times {\mathbb R}^3)}\le 1,
\end{equation}
which will be improved  in Section \ref{aux_imp}  to
\begin{equation*}
\|\p \varrho\|_{L_t^2 L_x^3([0,T]\times {\mathbb R}^3)}\le C T^\f12
\end{equation*}
where $C$ is a universal constant.

  (\ref{BA1}) will be improved to
\begin{equation}\label{G1}
\| \bp v_+, \bp\varrho,\bp v\|^2_{L_t^2 L_x^\infty}+\sum_{\la \ge 2}\la^{2\delta_0}\|P_\la (\bp \varrho, \bp v_+)\|_{L_t^2 L_x^\infty}^2\le C T^{2\ga_1},
\end{equation}
where $s>2$ and $0<\ga_1\le \ep_0$. With $0<T<1$ sufficiently small, we can have $\max (C T^\f12, C T^{2\ga_1})<1$.  
 The improvement will be achieved in Section \ref{5.07.4.19} by establishing the dyadic Strichartz estimate in Theorem \ref{dyastric} for (\ref{wavelin}), and by a delicate linearization from  (\ref{wave_2}) of $\Psi=(v_+, \varrho)$ to (\ref{wavelin}).

To prove the Strichartz estimates (\ref{G1}), we  need the full energy control for $v, \varrho$  in Corollary \ref{eng_wave}. To complement the energy estimates on $v_+,\varrho$ by using wave equations, the estimates for $\eta$ will be achieved by elliptic estimates combined with the energy estimates for $\curl\rp{\le 2}\Omega,$  completed in Section \ref{eng_vor}.

Since any point $(x,t)\in (\M, \bg)$ can be reached by following the integral curve of $\bT$, denoted by $x(t)$, from the initial slice
\begin{equation}\label{9.20.6.19}
 c(x(t),t)/c(x(0),0)-1=c(x(0),0)^{-1}\int_0^t c'(\varrho)\bT\varrho;\quad  \varrho(x(t),t)=\varrho(x(0), 0)+\int_0^t\bT \varrho.
\end{equation}
By (\ref{10.22.1.19}) at $t=0$, $|\varrho|\le C_1$. By the evolution of $\varrho$ and (\ref{BA1}),
we can obtain  for all $t\in [0,T]$, $|\varrho|\le C_1+1$.  
Since $c, c'$ are smooth functions about $\varrho$, $|c', c|\les 1.$  By using (\ref{BA1})
\begin{equation*}
|\frac{c(x(t),t)}{c(x(0),0)}-1|\les(C_1+1) c(x(0),0)^{-1}T^\f12\le C c_0^{-1} T^\f12<\f12
\end{equation*}
as long as we fix $0<T^\f12<\f12 c_0 C^{-1}$, where $C$ is a universal constant.
 This leads to
$
\f12<c(x(t),t)/c(x(0),0)<\frac{3}{2}
$
and thus  for $t\in[0,T]$,
\begin{equation}\label{9.20.1.19}
\f12 c_0<\f12 c(x(0),0)\le c\le \frac{3}{2}c(x(0),0)\le C_0,
\end{equation}
where $C_0$ is a universal constant.

This fact together with the fact that $ |c', c''|\les 1$ in $[0,T]\times {\mathbb R}^3$ will be frequently used in this paper.
The metric $g_{ij}=c^{-2}\delta_{ij}$ is thus always an conformally flat Riemannian metric on $\Sigma_t$ for $0<t\le T$.

Similar to (\ref{9.20.6.19}), we can also obtain the bound on $v$ due to (\ref{BA1}) and the $C_1$ bound in (\ref{10.22.1.19}). Thus we will frequently use
$
|\varrho, v|\les 1,
$
and
\begin{equation}\label{5.06.7.19}
\|C(f)\|_{L_x^\infty}\les 1, \quad \|\bp( C(f))\|_{L_x^\infty}\les \|\bp f\|_{L_x^\infty}
\end{equation}
with $f= \varrho$ or $v$,  where  $C(y)$ is a smooth function.

\subsubsection{The second fundamental form}
Next we derive the formula for the second fundamental form.
\begin{proposition}
Define the second fundamental form for $\bT$ on $\Sigma_t$ as $k_{ij}:=-\f12 \L_\bT g_{ij}$. Let $g^{ij} k_{ij}=\Tr k$, and $\hk_{ij}=k_{ij}-\frac{1}{3} \Tr k g_{ij}$. There hold
\begin{align}
&k_{ij}=-\f12 c^{-2}(-2\bT(\log c)\delta_{ij}+\p_i v_j+\p_j v_i) \label{k1}\\
&\Tr k=3 \bT \log c-\div v.\label{5.24.1.19}
\end{align}
\end{proposition}
\begin{proof}
By the fact that
\begin{equation}\label{cmu1}
[\bT, \p_i]=-\p_i v^a \p_a,
\end{equation}
we can compute by definition that
\begin{align*}
k_{ij}&=-\f12 \L_\bT g_{ij}=-\f12 (\bT g_{ij}-g([\bT, \p_i], \p_j)-g(\p_i, [\bT, \p_j]))\\
&=-\f12(\bT(c^{-2}) \delta_{ij}+\p_i v^a g_{aj}+\p_j v^a g_{ia})\\
&=-\f12 c^{-2}(-2\bT(\log c)\delta_{ij}+\p_i v_j+\p_j v_i).
\end{align*}
Thus the trace and  the traceless part of $k$ are
\begin{align*}
\Tr k&= g^{ij} k_{ij}=-\f12(-6\bT(\log c)+2\div v)=3\bT \log c-\div v,\\
\hat k_{ij}&=-\f12 c^{-2}(\p_i v_j+\p_j v_i)+\frac{1}{3}c^{-2}\div \delta_{ij}.
\end{align*}
The proof is complete.
\end{proof}

\subsection{Preliminaries for  energy estimates}
To begin with, we give the uniform method to treat the equations (\ref{wavevpl}) and (\ref{4.10.2.19}) without involving the term $\bT \bT \eta$ in analysis.
\subsubsection{Reduction to the first order system}
We note the  main equations (\ref{wavevpl}) and (\ref{4.10.2.19})  take the form of
\begin{equation}\label{wave2}
\Box_\bg \Psi=W+\bT Y-\Tr k Y,
\end{equation}
which can be written as the first order equation system for a pair of functions $(U,V)$
\begin{equation}\label{lu3}
\left\{\begin{array}{lll}
\bT U=V+F_U\\
\bT V=\Delta_g U+F_V+\Tr k V
\end{array}\right.
\end{equation}
with 
\begin{equation}\label{wave3}
U=\Psi, \quad F_U=-Y, \quad F_V=-W.
\end{equation}
It is crucial not to apply $\bT$ derivative to (\ref{wave2}) for energy estimate. The use of (\ref{lu3}) allows us to keep track of the error terms produced by applying the spatial derivatives or Littlewood-Paley projector $P_\la$ with $\la>1$ to (\ref{wave2}).

 (\ref{wave2}) for $\Psi=v_+$ can be written as  (\ref{lu3}) with
\begin{equation}\label{5.03.1.19}
U=v^i_+,\quad V= \bT v^i,\quad F_U=-\bT \eta^i,\quad  F_V=-\widetilde\sQ^i+c^2\eta^i;
\end{equation}
(\ref{wave2}) for $\Psi=\varrho$ can be written as  (\ref{lu3}) with
\begin{equation}\label{5.03.2.19}
U=\varrho,\quad V=\bT \varrho,\quad F_U=0,\quad F_V=-\sQ^0.
\end{equation}
We denote $U_{v_+}$ and $U_\varrho$ the functions $U$ listed above for $\Psi=v_+$ or $\varrho$ respectively. The same convention applies to $V$ and the errors $F_U$ and $F_V$.
We set the vector-valued function ${\textbf U}=(U_{v_+}, U_{\varrho})$ and ${\textbf V}=(V_{v_+}, V_{\varrho})$ to unify both cases. In the same manner
\begin{equation*}
F_{\textbf U}=(F_{U_{v_+}}, 0), \quad F_{\textbf V}=(F_{V_{v_+}}, F_{V_{\varrho}}).
\end{equation*}
By abuse of notation, when using $U$ and $V$ unless specified, we mean the components in the corresponding vectors ${\textbf U}$ and ${\textbf V}$.
We can directly write in view of (\ref{5.03.1.19}) and (\ref{5.03.2.19}) that
\begin{equation}\label{7.24.8.19}
F_{\textbf U}=(-\bT \eta, 0).
\end{equation}
Consolidating the quadratic forms  $\sQ^0$  in (\ref{4.10.2.19}) and $\widetilde{\sQ}^i$ in (\ref{wavevpl}),  we write both components of $F_{\textbf V}$  as
\begin{equation}\label{5.07.1.19}
 F_V= (C(\varrho)+1)\c \fZ+  c^2 (\eta+\fw \c \p \varrho+\p (\log c)\c \p \eta),
\end{equation}
 where $C(\varrho)$ denotes some smooth functions of $\varrho$, which may vary when distributed to each term in the other factor of the product, and
\begin{equation}\label{5.07.2.19}
\fZ=(\p v)^2+\bp\varrho\c (\bp v+\bp \varrho).
\end{equation}
 Here we use  $\bp f$ to denote the component of total derivative $(\p f, \bT f)$ for smooth functions $f$. We may write $F_U$ to represent any of the components in $F_{\textbf U}$, and apply the same convention to $F_{\textbf V}$.

  By using (\ref{cmu1}) we differentiate (\ref{lu3}) to  obtain the equation system for $(U_i\rp{1}, V_i\rp{1})=(\p_i U, \p_i V)$
\begin{equation}\label{lu3_1}
\left\{\begin{array}{lll}
\bT  U\rp{1}=V\rp{1}+F_{U\rp{1}}\\
\bT V\rp{1}=\Delta_g U\rp{1}+ F_{V\rp{1}}+\Tr k V\rp{1},
\end{array}
\right.
\end{equation}
where
\begin{equation}\label{5.02.3.19}
\begin{split}
F_{U\rp{1}}&= \p_i F_U-\p_i v^m \p_m U\\
F_{V\rp{1}}&=\p_i (c^2) \Delta_e U-\p_i v^m \p_m V-\p_i(c^2\p^m(\log c)) \p_m U+V \p_i\Tr k +\p_i F_V.
\end{split}
\end{equation}
For calculating the above error terms, we used the commutator formula for the scalar function $f=U$,
\begin{equation*}
\p_i \Delta_g f-\Delta_g (\p_i f)=\p_i(c^2) \Delta_e f-\p_i\big(c^2\p^j(\log c)\big)\p_j f.
\end{equation*}

 For  $(U,V)$ satisfying (\ref{lu3}), by applying the Littlewood-Paley projection $P_\mu$ with $\mu>1$ to (\ref{lu3}), we can obtain (\ref{lu3}) holds  for  the pair of functions $(U_\mu,V_\mu):=(P_\mu U, P_\mu V)$  with $F_{U_\mu}$ and $F_{V_\mu}$ given by
\begin{equation}\label{puuv}
\left\{\begin{array}{lll}
F_{U_\mu}=-[P_\mu, v^m]\p_m U +P_\mu F_U,\\
F_{V_\mu}=[P_\mu, c^2] \Delta_e U+P_\mu F_V+[P_\mu, \Tr k]V-[P_\mu,c^2\p_l(\log c)]\p^l U -[P_\mu, v^m]\p_m V.
\end{array}\right.
\end{equation}
Define the energy for $(U,V)$ satisfying (\ref{lu3}) by
\begin{equation}\label{9.28.4.19}
\E(t)=\E[U](t):=\frac{1}{2} \int_{\Sigma_t} \left(g^{ij}\p_i U\p_j U+|V|^2\right) d\mu_g,
\end{equation}
which may also be denoted by $\E_U(t)$. With $U\rp{0}=U, \,V\rp{0}=V$, we define
\begin{align*}
&\E\rp{m}(t)=\E[U\rp{m}](t),\quad m=0,1,\\
&\E\rp{1}_\mu(t)=\E[U\rp{1}_\mu](t).
\end{align*}
For the component of a vector-valued function such as $v_+^i$, we denote the sum of the energy of all components of $U^i$ by the same notation whenever no confusion occurs.

The low order  energy control can be  undertaken  for $v$ and $\varrho$ directly by using (\ref{4.10.1.19}) and (\ref{4.10.2.19}) under the assumption of (\ref{BA1}). This will be  given in (\ref{9.21.1.19}) and in Corollary \ref{5.04.12.19}.

For the higher order energy, we will construct the modified current for the equation (\ref{wave2}), which cancels the term $\bT Y$ on the right hand side. Only the spatial derivative of $Y$ will be  involved in our analysis. This  enables us  to control energy (and energy flux in Section \ref{c_flux}) for $v_+$ and $\varrho$ by  (\ref{wave2}) from initial data merely with the help of the norm of $\|\curl \Omega\|_{H^{s-\f12}_x}$ and (\ref{BA1}).

\subsubsection{The modified current}
 We construct the modified energy current of scalar functions $U$ \begin{footnote}{This is adapted from the one introduced  in \cite[Section 3.2]{WangCMCSH}.}\end{footnote} satisfying the equation  (\ref{wave2}), or identically, (\ref{lu3}) with (\ref{wave3}),
\begin{equation}\label{9.13.1.19}
 \sP[U]_\mu=-F_U \bd_\mu U+Q[U]_{\mu\nu}\bT^\nu+\f12 F_U^2  \bd_\mu t,
\end{equation}
where for any scalar function $f$, $Q_{\a\b}:=Q[f]_{\a\b}$
is the standard energy momentum tensor defined by
\begin{equation}\label{gmom}
Q[f]_{\a\b}=\p_\a f\p_\b f-\f12 \bg_{\a\b}\bd^\mu f \bd_\mu f
\end{equation}
with $\bd$ the Levi-Civita connection of $\bg$ and indices lifted or lowered by $\bg$.
Applying the divergence theorem to $\sP_\mu$ in the spacetime region $\bigcup_{0\le t'\le t} \Sigma_t$  yields
\begin{equation}\label{div0}
\int_{\Sigma_t} \sP_\mu \bT^\mu=
\int_{\Sigma_0} \sP_\mu \bT^\mu-\int_0^t \int_{\Sigma_{t'}}\bd^\mu \sP_\mu,
\end{equation}
where we hide the standard volume element on $\Sigma_t$ and $\bigcup_{0\le t'\le t}\Sigma_{t'}$  which are $d\mu_g$ and $d\mu_g dt'$.  $d\mu_g$ is always comparable to $d\mu_e$ due to (\ref{9.20.1.19}).

Now we show
\begin{equation}\label{dpu}
\bd^\mu \sP_\mu  =-F_V\c V
 -\bd^i F_U\bd_i U-(k^{ij} -\f12 \Tr k c^2\delta^{ji})\bd_i U \bd_j U-\f12 \Tr k V^2.
\end{equation}
Indeed, since (\ref{lu3}) implies $V=-F_U+\bT U$, we have from the definition of $\sP_\mu$
that
\begin{align*}
\bd^\mu \sP_\mu &=-\bd^\mu F_U \bd_\mu U -F_U \Box_\bg U+\bd^\mu Q_{\mu\nu}\bT^\nu
 +\f12 Q_{\mu\nu}{}^{(\bT)}\pi^{\mu\nu}\\
&\quad \, +F_U\bd^\mu F_U \bd_\mu t+\f12 F_U^2 \Box_\bg t\\
&=\bd_\bT(F_U) \bd_\bT U-\bd^i F_U \bd_i U- F_U \Box_\bg U+\Box_\bg U \bd_{\bT}U
+\f12 Q_{\mu\nu}{}^{(\bT)}\pi^{\mu\nu}\\
&\quad\, +F_U \bd^\mu F_U  \bd_\mu t+\f12 \Tr k F_U^2\displaybreak[0]\\
&=(-F_U+\bd_\bT U) (\Box_\bg U+\bd_\bT F_U)-\bd^i F_U\bd_i U +\f12 Q_{\mu\nu}{}^{(\bT)}\pi^{\mu\nu}+\f12 \Tr k F_U^2\\
&=\left(\Box_\bg U+\bd_\bT F_U\right) V-\bd^iF_U\bd_i U+\f12 Q_{\mu\nu}{}^{(\bT)}\pi^{\mu\nu}+\f12\Tr k F_U^2.
\end{align*}
Here  $\pi^{(\bT)}:= \Lie_{\bT} \bg$ is
the deformation tensor of $\bT$.

In view of (\ref{wave2}) and (\ref{wave3}), we obtain (\ref{dpu}).
\begin{align}\label{dpu1}
\bd^\mu \sP_\mu & =\left(-F_V+\Tr k F_U\right)V
 -\bd^i F_U\bd_i U+\f12 Q_{\mu\nu}{}^{(\bT)}\pi^{\mu\nu}+\f12 \Tr k F_U^2.
\end{align}

Note  $ \rp{\bT} \pi_{ij}=-2k_{ij}$ is the non-trivial part of $\pi^{(\bT)}$.
We compute
\begin{align*}
-\f12 Q_{\mu\nu}{}^{(\bT)}\pi^{\mu\nu}&=Q_{ij}k^{ij}=k^{ij}(\bd_i U \bd_j U-\f12 \bg_{ij}\bd^\a U \bd_\a U)\\
&=k^{ij} \bd_i U \bd_j U-\f12 \Tr k\bd^i U \bd_i U+\f12 \Tr k(\bd_\bT U)^2\\
&=k^{ij} \bd_i U \bd_j U-\f12 \Tr k\bd^i U \bd_i U+\f12 \Tr k\big((F_U)^2+2 V F_U+V^2\big).
\end{align*}
Substituting the above identity to (\ref{dpu1}) implies (\ref{dpu}) in view of (\ref{wave2}).

We also compute
\begin{align*}
\bT^\mu \sP_\mu&=-F_U \bT U+Q_{\mu \nu}\bT^\mu \bT^\nu +\f12 F_U^2 \bT(t)\\
&=\f12(\bT U \bT U+ \bd^i U \bd_i U)+\f12 F_U^2-F_U \bT U\\
&=\f12\{( \bT U-F_U)^2+ \bd^i U\bd_j U\}\\
&=\f12 (|V|^2+c^2\delta^{ij}\p_i U \p_j U).
\end{align*}
This is identical to the integrand of (\ref{9.28.4.19}). Thus
\begin{equation}\label{9.28.5.19}
\int_{\Sigma_t} \bT^\mu \sP_\mu= \E(t).
\end{equation}

\begin{lemma}[Fundamental inequalities for higher order energies]\label{peng}
If $\|\p v\|_{L_t^1 L_x^\infty([0,T]\times {\mathbb R}^3)}\les 1$, there hold with $\a>0$ that
\begin{align}
& \E(t)^\f12 \les \E(t)^\f12 +\int_0^t\big(\|\p F_U\|_{L_x^2}+\|F_V\|_{L_x^2}\big) dt'\label{4.12.4.19}\\
&\E\rp{1}(t)^\f12\les \E\rp{1}(0)^\f12 +\int_0^t \big(\|\p F_{U\rp{1}}\|_{L_x^2}+\|F_{V\rp{1}}\|_{L_x^2}\big) dt' \label{5.02.2.19}\\
& \|\la^\a \E\rp{1}_\la(t)^\f12\|_{l_\la^2}\les \|\la^\a \E\rp{1}_\la(0)^\f12\|_{l_\la^2}+\int_0^t\big(\|\la^\a \p F_{U\rp{1}_\la}\|_{l_\la^2 L_x^2}+\|\la^\a F_{V\rp{1}_\la}\|_{l_\la^2 L_x^2}\big) dt'\label{4.13.1.19}
\end{align}
\begin{footnote}{In this paper, for any $f$, $\|f\|_{l_\la^2}:=(\sum_{\la>1}|f_\la|^2 )^\f12$ with $\la$ dyadic. }\end{footnote}
\end{lemma}
\begin{remark}
 Since $\int_{\Sigma_t}|\nab f|^2_g d\mu_g\approx \int_{\Sigma_t} |\p f|^2  d\mu_e$, with $d\mu_e$ the volume element of the Euclidean metric,  we will not distinguish the metric used for $\dot{H}_1$ norm. And the assumption $\|\p v\|_{L_t^1 L_x^\infty([0,T]\times {\mathbb R}^3)}\les 1$ is a consequence of (\ref{BA1}).
\end{remark}
\begin{proof}
We first  derive from (\ref{dpu}) that
\begin{equation}\label{9.28.6.19}
\int_0^t \int_{\Sigma_{t'}} |\bd^\a \sP_\a|\les \int_0^t\{ (\|F_V(t')\|_{L_x^2}+\|\p F_U(t')\|_{L_x^2})\E^\f12(t')+\|k(t')\|_{L_x^\infty}\E(t')\} dt'.
\end{equation}
Note   $|c', c^{-1}|\les 1$ due to (\ref{9.20.1.19}) and (\ref{5.06.7.19}). We can bound $|\bT \log c|\les |\bT \varrho|$.
Hence using (\ref{k1}) and the first equation in (\ref{4.23.1.19}), we have
\begin{equation} \label{9.20.2.19}
|k|\les |\p v|.
\end{equation}
In view of (\ref{9.20.2.19}), substituting (\ref{9.28.6.19}) and (\ref{9.28.5.19}) into (\ref{div0}) gives
\begin{equation*}
\E(t)\le \E(0)+ \int_0^t\{ (\|F_V(t')\|_{L_x^2}+\|\p F_U(t')\|_{L_x^2})\E^\f12(t')+\|\p v(t')\|_{L_x^\infty}\E(t')\} dt'.
\end{equation*}
(\ref{4.12.4.19}) follows immediately by using Gronwall's inequality and the bound $\|\p v\|_{L_t^1 L_x^\infty([0,T]\times {\mathbb R}^3)}\les 1$.

Applying (\ref{4.12.4.19}) to $(U\rp{1}, V\rp{1})$  with error terms  verifying (\ref{5.02.3.19}), we can obtain (\ref{5.02.2.19}). Applying (\ref{4.12.4.19})  to $(U\rp{1}_\la, V\rp{1}_\la)$ with error terms  (\ref{puuv}) substituted by $(U\rp{1}, V\rp{1})$ gives (\ref{4.13.1.19}).
\end{proof}
\begin{corollary}\label{eng_1}
Under the assumption of (\ref{BA1}), there hold the  energy estimates for $l=0,1$,
\begin{align}
&\E\rp{l}(t)^\f12\les \E\rp{l}(0)^\f12+\int_0^t\{ \|\p \p\rp{l} F_U, \p\rp{ l} F_V\|_{L_x^2}+l(\|\p v\|_{H^1_x}+\|\p \varrho\|_{H^1_x})(\|\p U\|_{L_x^\infty}+\| V\|_{L_x^\infty})\} dt'.\label{5.02.4.19}
\end{align}
\end{corollary}
\begin{proof}
The case of  $l=0$ in (\ref{5.02.4.19}) is (\ref{4.12.4.19}). We only need to consider the first order estimate, for which need to bound
 the integrand of the right hand side of (\ref{5.02.2.19}). In view of (\ref{5.02.3.19}), we compute  that
\begin{align*}
\|\p F_{U\rp{1}}&\|_{L_x^2}\les \|\p^2 F_U\|_{L_x^2}+\|\p v\|_{L_x^\infty}\|\p^2 U\|_{L_x^2}+\|\p^2 v\|_{L_x^2}\|\p U\|_{L_x^\infty},\\
\|F_{V\rp{1}}\|_{L_x^2}&\les \|\p \varrho\|_{L_x^\infty}\|\Delta_e U\|_{L_x^2}+\|\p v\|_{L_x^\infty} \|\p V\|_{L_x^2}+\|\p \Tr k\|_{L_x^2} \|V\|_{L_x^\infty}\\
&+\|\p^2(c^2)\|_{L_x^2}\|\p U\|_{L_x^\infty}+\|\p F_V\|_{L_x^2}\\
&\les \|\p \varrho, \p v\|_{L_x^\infty}\E\rp{1}(t)+(\|V\|_{L_x^\infty}+\|\p U\|_{L_x^\infty})( \|\p^2 c\|_{L_x^2}+\||\p \Tr k\|_{L_x^2})+\|\p F_V\|_{L_x^2}.
\end{align*}
By using (\ref{5.24.1.19}) and the first equation in  (\ref{4.23.1.19}), $|c,c', c'', c^{-1}|\les 1$, and Sobolev embedding, we bound
\begin{equation*}
\|\p \Tr k\|_{L_x^2}+\|\p^2(c^2)\|_{L_x^2}\les \|\p v\|_{H^1_x}+\|\p \varrho\|_{H^1_x}.
\end{equation*}
 Combining the above two estimates, we can obtain (\ref{5.02.4.19}).
\end{proof}
The purpose of introducing the  modified current is mainly to control the higher order energy. For the low order estimates, we adopt the standard method.
\begin{proposition}[$0$-order energy]
Under the assumption (\ref{BA1}), there holds
\begin{equation}\label{9.21.1.19}
\|\bp v, \bp \varrho\|_{L^2(\Sigma_t)}\les \|\bp v, \bp \varrho\|_{L^2(\Sigma_0)}+\int_0^t \|\curl \Omega(t')\|_{L_x^2} dt'.
\end{equation}
\end{proposition}
\begin{proof}
To prove (\ref{9.21.1.19}), we recall the standard energy approach.

By applying the divergence theorem to the energy current
$\P_\a^{(\bT)}[f]=Q[f]_{\a\b}\bT^\b$, we can obtain the energy identity
\begin{equation}\label{qs}
\int Q_{\a\b}\bT^\a\bT^\b(t)d\mu_g-\int Q_{\a\b}\bT^\a\bT^\b(0) d\mu_g
=-\int_{[0,t]\times{\mathbb R}^3} \left(\Box_{\bg} f \bd_\bT f+\f12 {}^{(\bT)}\pi_{\a\b} Q^{\a\b}\right).
\end{equation}

For any smooth scalar function, since
\begin{align*}
Q[f]_{\bT \bT}&=\f12\big( (\bT f)^2+c^2 \delta^{ij}\p_i  f \p_j f\big)
\end{align*}
there holds for  $0\le t\le T$
\begin{align}\label{norm_com}
\int_{\Sigma_t} Q[f]_{\bT \bT } d\mu_g\approx \int_{\Sigma_t}\big(|\p f|^2+|\bT f|^2\big) d\mu_e=\|\bp f(t)\|_{L_x^2}^2,
\end{align}
where the constant depends on $c_0$ and $C_0$.

 Since the non-trivial component of the deformation tensor $\pi_{ij}^{(\bT)}=-2k_{ij}$, we use (\ref{9.20.2.19}) to bound $|Q^{\a\b}{}\rp{\bT}\pi_{\a\b}| \les |\p v||\bp f|^2$.
 Due to  (\ref{BA1}),
\begin{equation}\label{eng7.10}
\|\bp f(t)\|_{L_x^2}\les \|\bp f(0)\|_{L_x^2}+ \int_0^t \|\Box_\bg f(t')\|_{L_x^2} dt'.
\end{equation}
Applying (\ref{eng7.10}) to  (\ref{4.10.1.19}) for $v$  and to (\ref{4.10.2.19}) for $\varrho$, and using (\ref{4.23.1.19}) lead to
\begin{align*}
&\|\bp v(t)\|_{L_x^2}+\|\bp \varrho(t)\|_{L_x^2}\\
&\les \|\bp v(0)\|_{L_x^2}+\|\bp \varrho(0)\|_{L_x^2}+\int_0^t(\|\Box_\bg v\|_{L_x^2}+\|\Box_\bg \varrho\|_{L_x^2})dt'\\
&\les \|\bp v(0)\|_{L_x^2}+\|\bp \varrho(0)\|_{L_x^2}+\int_0^t \{\|(\bp \varrho+\bp v)^2\|_{L_x^2} +\|\p \varrho \Omega\|_{L_x^2}+\|\curl \Omega(t')\|_{L_x^2}\} dt'\\
&\les \|\bp v(0)\|_{L_x^2}+\|\bp \varrho(0)\|_{L_x^2}+\int_0^t \|\bp \varrho, \bp v\|_{L_x^\infty} \|\bp \varrho, \bp v, \Omega\|_{L_x^2}  +\|\curl \Omega\|_{L_x^2}\} dt'.
\end{align*}
Since $|\Omega|\les |\p v|$, we will incorporate this term as part of $\bp v$ in the last line. The consequence drops out by using Gronwall's inequality and
(\ref{BA1}).
\end{proof}
In order to carry out energy estimate for the dyadic pairs $(U_\mu, V_\mu)$, we need to derive a series of product estimates and commutator estimates with Littlewood-Paley theory. Since the estimates are not limited to the applications in the energy estimates, we provide them in Section \ref{app}.

Now we give the energy inequality of the highest order.
\begin{proposition}\label{4.13.8.19}
With $0<\a<1$, under the assumption of (\ref{BA1}), there holds
\begin{align*}
\|\mu^\a \E\rp{1}_\mu(t)^\f12\|_{l_\mu^2 L_x^2}&\les \|\mu^\a \E\rp{1}_\mu(0)^\f12\|_{l_\mu^2 L_x^2}+T^\f12\sup_{0\le t'\le t}\E\rp{\le 1}(t')^\f12 \\
& +\int_0^t(\|\p v,\p(c^2), \Tr k\|_{H_x^{1+\a}}\|\p U, V\|_{L_x^\infty}+\|\p (\p F_U, F_V)\|_{\dot{H}^\a_x}) dt'.
\end{align*}
\end{proposition}
\begin{proof}
In view of (\ref{4.13.1.19}), we control the integrand with the help  of (\ref{puuv}).  Recall that
\begin{align*}
F_{U\rp{1}_\mu}&=-[P_\mu, v^m]\p_m U\rp{1}+P_\mu F_{U\rp{1}}\\
F_{V\rp{1}_\mu}&=[P_\mu, c^2]\Delta_e U\rp{1}+P_\mu F_{V\rp{1}}+[P_\mu, \Tr k] V\rp{1}-\f12 [P_\mu, \p(c^2)]\p U\rp{1}-[P_\mu, v^m]\p_m V\rp{1}.
\end{align*}
To treat the commutators, we will employ commutator estimates provided in Section \ref{app}.
By applying (\ref{4.13.4.19}) to $(F, G)=(v,U\rp{1})$, we can bound
\begin{equation}\label{9.20.3.19}
\begin{split}
\|\mu^\a \p F_{U\rp{1}_\mu}\|_{l_\mu^2 L_x^2}&\les \|\p[P_\mu, v] \p U\rp{1}\|_{l_\mu^2 L_x^2}+\|\mu^\a\p P_\mu F_{U\rp{1}}\|_{l_\mu^2 L_x^2}\\
&\les\|\mu^\a\p P_\mu F_{U\rp{1}}\|_{l_\mu^2 L_x^2}+\|\p v\|_{L_x^\infty}\|\p U\rp{1}\|_{H^\a_x}+\|\p v\|_{H^{1+\a}_x}\| U\rp{1}\|_{L_x^\infty}.
\end{split}
\end{equation}

To estimate $\|\mu^\a F_{V_\mu\rp{1}}\|_{l_\mu^2 L_x^2}$,
  we first apply (\ref{4.13.3.19}) to $(F,G)=(c^2, \p U\rp{1})$ and $(v, V\rp{1})$  to derive
  \begin{align*}
  \|\mu^\a[P_\mu, c^2] \Delta_e U\rp{1}\|_{l_\mu^2 L_x^2}+\|\mu^\a[P_\mu, v]\p V\rp{1}\|_{l_\mu^2 L_x^2}&\les \|\p \varrho\|_{L_x^\infty}\|\p U\rp{1}\|_{H^\a_x}+\|\p v\|_{L_x^\infty}  \|V\rp{1}\|_{H^\a_x}.
  \end{align*}
Note $V\rp{1}=\p V$.  Applying  (\ref{4.13.2.19}) to $(F,G)=(\p (c^2), U\rp{1})$ and $(\Tr k, V)$  yields
  \begin{align*}
  \|\mu^\a[P_\mu, \p (c^2)]&\p U\rp{1}\|_{l_\mu^2 L_x^2}+\|\mu^\a[P_\mu, \Tr k] V\rp{1}\|_{l_\mu^2 L_x^2}\\
  &\les \|\p \varrho, \Tr k \|_{L_x^\infty}\|\p U\rp{1}, V\rp{1}\|_{H^\a_x}+\|\p (c^2), \Tr k\|_{H^{1+\a}_x} \|U\rp{1}, V\|_{L_x^\infty}.
  \end{align*}
  Combining  the above two estimates in view of the formula of $F_{V_\mu\rp{1}}$ and also using  (\ref{9.20.2.19}) imply
\begin{equation}\label{9.20.4.19}
\begin{split}
\|\mu^\a F_{V\rp{1}_\mu}\|_{l_\mu^2 L_x^2}
&\les\|\mu^\a P_\mu F_{V\rp{1}}\|_{l_\mu^2 L_x^2}+ \|\p \varrho, \p v\|_{L_x^\infty}\|\p U\rp{1}, V\rp{1}\|_{H^\a_x}\\
&+\|\p(c^2), \Tr k \|_{H^{1+\a}_x}\|U\rp{1},V\|_{L_x^\infty}.
\end{split}
\end{equation}
It remains to bound the first terms on the right hand side  of (\ref{9.20.3.19}) and (\ref{9.20.4.19}). In view of (\ref{5.02.3.19}), we apply (\ref{9.20.5.19}) to $(F, G)=(\p v, \p U)$ to derive
\begin{align*}
\|\mu^\a P_\mu \p F_{U\rp{1}}\|_{l_\mu^2 L_x^2}&\les \|\p^2 F_U\|_{{\dot H}^\a_x}+\|\mu^\a P_\mu\p (\p v\c \p U)\|_{l_\mu^2 L_x^2}\\
&\les \|\p^2 F_U\|_{\dot{H}^\a_x}+\|\p^2 v\|_{H^\a}\|\p U\|_{L_x^\infty}+\|\p v\|_{L_x^\infty} \|\p^2 U\|_{H^\a_x}.
\end{align*}
Applying (\ref{5.05.2.19}) to $(F,G)=(V, \Tr k), (\p U, \p (c^2)), (\p v, V)$ and $(\p(c^2), \p U)$ yields
\begin{align*}
\|\mu^\a P_\mu F_{V\rp{1}}\|_{l_\mu^2 L_x^2}&\les \|\p v, \Tr k\|_{H^{1+\a}} \|V\|_{L_x^\infty}+\|\p(c^2)\|_{H^{1+\a}_x}\|\p U\|_{L_x^\infty}\\
&+\|\p v, \p \varrho \|_{L_x^\infty} \| V\rp{1}, \p U\|_{H^{1+\a}_x}+\|\p F_V\|_{\dot{H}^{\a}_x},
\end{align*}
where we also used (\ref{9.20.2.19}) to bound $|\Tr k|\les |\p v|$.

We summarize the above estimates as
\begin{align}
&\|\mu^\a P_\mu \p F_{U\rp{1}}\|_{l_\mu^2 L_x^2}+\|\mu^\a P_\mu F_{V\rp{1}}\|_{l_\mu^2 L_x^2}\label{9.29.2.19}\\
&\les \|\p v, \p(c^2), \Tr k\|_{H^{1+\a}} \|V, \p U\|_{L_x^\infty}+\|\p v, \p \varrho \|_{L_x^\infty} \| V\rp{1}, \p U\|_{H^{1+\a}_x}+\|\p^2 F_U\|_{\dot{H}^\a_x}+\|\p F_V\|_{\dot{H}^{\a}_x}\nn.
\end{align}

Substituting the inequality to (\ref{4.13.1.19}) implies that for $0<\a<1$,
\begin{align*}
\|\mu^\a \E_\mu(t)^\f12\|_{l_\mu^2 L_x^2}&\les \|\mu^\a \E_\mu^\f12(0)\|_{l_\mu^2 L_x^2}+ \int_0^t\|\p v, \p \varrho\|_{L_x^\infty}(\|\mu^\a\E\rp{1}_\mu(t')^\f12\|_{l_\mu^2}+\E\rp{\le 1}(t')) dt'\\
&+\int_0^t(\|\p v, \p(c^2), \Tr k \|_{H_x^{1+\a}}\|\p U, V\|_{L_x^\infty}+\|\p (\p F_U, F_V)\|_{\dot{H}^\a_x}) dt'.
\end{align*}
 Proposition \ref{4.13.8.19} follows by applying the Gronwall's inequality with the help of $\|\p v, \p \varrho\|_{L_t^1 L_x^\infty}\les 1$ due to (\ref{BA1}).
\end{proof}

The main task will be to control $\|\p(\p F_U, F_V)\|_{L_t^1 H^{\a}_x}$ with $0\le \a\le s-2$ for both $U=v_+$ and $U=\varrho$. In particular for $U=v_+$, due to (\ref{5.03.1.19}), we need to provide estimates for $\bp \eta$.  This will be carried out in the following subsection. 

\subsection{Preliminary Estimates for $\Omega$ and $\eta$} We first rely on the definition of  $\eta$ in (\ref{4.10.5.19}) and the equation  (\ref{4.10.3.19})  to prove the following estimates for the vorticity and  $\eta$.
\begin{lemma}\label{vort_1}
(1) For any $p\ge 2$,
\begin{align}
\|\Omega(t)\|_{L_x^p}&\les\|\Omega(0)\|_{L_x^p}\les 1;\label{5.03.4.19}
\end{align}

(2)  Let $0<\ep\le s-2$. For any $2\le p\le \frac{3}{1-\ep}$, there hold
\begin{align}
&\|\p \Omega, \p \fw\|_{L_x^p}\les \|\p \varrho\|_{L_x^p}+1,\label{5.03.5.19}\\
&\|\curl \Omega, \C\|_{L_x^p}\les 1,\label{5.03.8.19}
\end{align}
and
\begin{equation}\label{9.21.2.19}
\|\p^2\eta\|_{L_x^p}\les \|\p \varrho\|_{L_x^p}+1.
\end{equation}
\end{lemma}
Substituting the estimate (\ref{5.03.8.19}) into (\ref{9.21.1.19})  implies the lowest order energy estimate,
\begin{corollary}\label{5.04.12.19}
\begin{align*}
&\|\bp \varrho, \bp v\|_{L^2(\Sigma_t)}\les \|\bp \varrho, \bp v\|_{L^2(\Sigma_0)}+T\les 1.
\end{align*}
\end{corollary}
\begin{proof}[Proof of Lemma \ref{vort_1}]
The first inequality in (\ref{5.03.4.19}) can be obtained by integrating (\ref{4.10.3.19}) with the help of the bound $\|\p v\|_{L_t^1 L_x^\infty}\les 1$  due to (\ref{BA1}). The second one is due to Sobolev embedding $\|\Omega(0)\|_{L^p_x}\les \|\Omega(0)\|_{H^{\frac{3}{2}+}_x}\les 1$ for all $p>2$.

Similarly, by integrating (\ref{4.25.4.19}), for $2\le p\le \frac{3}{1-\ep}$
\begin{align}\label{5.03.6.19}
\|\C(t)\|_{L_x^p}&\les \|\C(0)\|_{L_x^p}+\int_0^t\|\p v\|_{L_x^\infty}\|\p \Omega\|_{L_x^p}.
\end{align}
Recall from (\ref{div}) and the definition of $\C$ that
\begin{equation*}
\div \Omega=-\Omega^a \p_a\varrho, \qquad \curl \Omega=e^\varrho \C
\end{equation*}
The $L^p$ estimate for the above Hodge system gives
\begin{equation*}
\|\p \Omega\|_{L_x^p}\les \|\Omega \c \p \varrho\|_{L_x^p}+\|\C\|_{L_x^p}.
\end{equation*}
Substituting the above estimate into (\ref{5.03.6.19}) implies
\begin{align*}
\|\C(t)\|_{L_x^p}\les \|\C(0)\|_{L_x^p}+\int_0^t \|\p v\|_{L_x^\infty}(\|\Omega \c \p \varrho\|_{L_x^p}+\|\C\|_{L_x^p}) dt'.
\end{align*}
By using (\ref{5.03.4.19}) for the estimate of $\|\Omega\|_{L_x^p}$ and applying  (\ref{BA1}) for $\|\p v, \p \varrho\|_{L_t^2 L_x^\infty}\les 1$, we can obtain
\begin{equation*}
\|\C(t)\|_{L_x^p}\les 1.
\end{equation*}
This gives the estimate of (\ref{5.03.8.19}) and we can bound
\begin{equation*}
\|\p \Omega\|_{L_x^p}\les \|\Omega\|_{L_x^\infty}\|\p \varrho\|_{L_x^p}+1\les \|\p \varrho\|_{L_x^p}+1
\end{equation*}
which is the first estimate in  (\ref{5.03.5.19}).

It is straightforward to compute,
 \begin{align*}
\curl \fw_i&=\curl(\Omega e^\varrho)_i=\tensor{\ep}{_i^{mn}}(\p_m \Omega_n+\Omega_n \p_m \varrho) e^\varrho.
 \end{align*}
We  hence have obtained the Hodge system
\begin{equation}\label{hodge_2}
\div \fw=0, \qquad  \curl \fw_n=e^\varrho\big((\curl \Omega)_n+\tensor{\ep}{_n^{ij}}\Omega_j \p_i \varrho\big).
\end{equation}
It follows by the  $L^p$ estimate for the above Hodge system, (\ref{5.03.8.19}) and (\ref{5.03.4.19}) that
\begin{equation}\label{5.03.7.19}
\|\p \fw\|_{L_x^p}\les\|\C\|_{L_x^p}+\|\p \varrho\|_{L_x^p}\|\Omega\|_{L_x^\infty}\les 1+\|\p \varrho\|_{L_x^p},
\end{equation}
which is the second estimate of (\ref{5.03.5.19}).

Also in view of (\ref{hodge_2}), (\ref{5.03.8.19}) and (\ref{5.03.4.19}), we have
\begin{align*}
\|\p^2\eta\|_{L_x^p}\les \|\p^2 \La^{-2} (\curl \fw)\|_{L_x^p}\les \|\Omega\c\p \varrho\|_{L_x^p}+\|\curl \Omega\|_{L_x^p}\les \|\p \varrho\|_{L_x^p}+1,
\end{align*}
which gives (\ref{9.21.2.19}). The proof of Lemma \ref{vort_1} is complete.
\end{proof}

Next, we give more estimates on $\eta$.
\begin{proposition}\label{eta_est}
There hold the following estimates for  $t\in[0,T]$, 
\begin{align}
&\|\eta\|_{H^2_x}\les 1,\quad \|\eta\|_{L_x^\infty}\les 1,   \label{5.03.9.19}\\
&\|\bT \eta\|_{H^l_x}\les 1+ l\big(\|\p \eta\|_{L_x^\infty}+\|\p v\|_{H^1_x}(\|\p \varrho\|_{L_x^3}+1)\big),\quad l=0,1\label{5.03.3.19_1}\\
&\|\p\bT \eta\|_{H^1_x}\les \|\p v\|_{H^1_x}(\|\p \eta\|_{L_x^\infty}+\|\p \varrho\|_{L_x^3}+1).\label{4.12.5.19_1}
\end{align}
\end{proposition}
\begin{proof}
The first estimate in (\ref{5.03.9.19}) is obtained by using (\ref{9.21.2.19})  and Corollary \ref{5.04.12.19}, and the second one follows immediately as its consequence by Sobolev embedding.

Consider the estimate of $\bT \eta$.
We first derive the symbolic formula
\begin{equation}\label{9.21.4.19}
\bT (\curl \fw)_n=\p\big(\p v\Omega e^\varrho\big)+\C e^{2\varrho} \p v+\p v\c \p \varrho \Omega e^{\varrho}.
\end{equation}
Indeed,  by using the second identity in (\ref{hodge_2}), (\ref{4.10.3.19}), (\ref{4.25.4.19}) and the first equation in (\ref{4.23.1.19})  that
\begin{align*}
\bT (\curl \fw)_n&=\bT(\C_n e^{2\varrho})+\bT\big(\Omega_j \tensor{\ep}{_n^{ij}}\p_i(e^\varrho)\big)\\
&=\bT \C_n e^{2\varrho}+\C_n \bT (e^{2\varrho})+\bT \Omega_j \tensor{\ep}{_n^i^j} \p_i(e^\varrho)+\Omega_j \tensor{\ep}{_n^{ij}}\bT \p_i(e^\varrho)\\
&=(\bT \C_n+2\C_n\bT \varrho) e^{2\varrho}+\Omega^a \p_a v_j \tensor{\ep}{_n^i^j} \p_i(e^\varrho)+\Omega_j \tensor{\ep}{_{n}^{ij}}(\p_i \bT (e^\varrho)+[\bT, \p_i]e^\varrho)\\
&=-2\p_b(\Omega_j \p_a v^j\tensor{\ep}{_{n}^{ab}} e^\varrho)+\p_a v_n \curl \Omega^a e^\varrho+2\C_n e^{2\varrho} \bT \varrho+\p_i(\Omega_j \tensor{\ep}{_n^{ij}}\bT(e^\varrho))\\
&-(\curl \Omega)_n \bT (e^\varrho)+\Omega \p v \p e^\varrho\\
&=-2\p_b(\Omega_j \p_a v^j\tensor{\ep}{_{n}^{ab}} e^\varrho)+\p_i(\Omega_j \tensor{\ep}{_n^{ij}}\bT(e^\varrho))+\C_n e^{2\varrho} \p v+\Omega \p v \p e^\varrho,
\end{align*}
where we used (\ref{cmu1}) and also have written the higher order terms on the right hand side into divergence form. By using the first equation in (\ref{4.23.1.19}) again, we can obtain (\ref{9.21.4.19}).

We next show
\begin{equation}\label{9.21.3.19}
\|\p\rp{m}\La^{-2}\bT \curl \fw\|_{H^1_x}\les  m\|\p v\|_{H^1_x}\|\p \varrho\|_{L_x^3}+\|\p v\|_{H^1_x}^{\max(m-1,0)}, \quad m=0,1,2,
\end{equation}
by using the following standard estimates for scalar functions $F$,
\begin{equation}\label{10.18.1.19}
\|\La^{-1}F\|_{L^2_x}\les \|F\|_{L^\frac{6}{5}_x}, \qquad \|\La^{-2}F\|_{L_x^2}\les \|F\|_{L^1_x},
\end{equation}
which follows directly from the duality argument, Sobolev embedding and $L^2$ estimates for the Calderon-Zygmund operator. The constants in the inequalities are the universal Sobolev constants.

By using (\ref{10.18.1.19}), in view of (\ref{9.21.4.19}), we derive by using (\ref{5.03.4.19}),  (\ref{5.03.8.19}), Corollary \ref{5.04.12.19} and Sobolev embedding that
\begin{align*}
\|\La^{-2} \bT \curl \fw\|_{L_x^2}&\les \|\La^{-2}\p\big(\p v \Omega e^\varrho\big)\|_{L_x^2}+\|\La^{-2} \big(\p v(e^\varrho\C+ \p \varrho\Omega) e^\varrho\big)\|_{L_x^2}\\
&\les \|\p v\c \Omega\|_{L_x^\frac{6}{5}}+\|\p v e^\varrho(\C e^\varrho+\p \varrho \Omega)\|_{L_x^1}\\
&\les \|\p v\|_{L_x^2}\|\Omega\|_{L_x^3}+\|\p v\|_{L_x^2}\big(\|\p \varrho\|_{L_x^2}\|\Omega\|_{L_x^\infty}+\|\C\|_{L_x^2}\big)\les 1;
\end{align*}
and
\begin{align*}
\|\p\La^{-2}\bT \curl \fw\|_{L_x^2}&\les \|\p v\|_{L_x^2} \|\Omega\|_{L_x^\infty}+\|\La^{-1} \big(\p v(e^\varrho\C+ \p \varrho\Omega) e^\varrho\big)\|_{L_x^2}\\
& \les 1+\|\|\p v(e^\varrho\C+ \p \varrho\Omega) e^\varrho\|_{L_x^{\frac{6}{5}}}\\
&\les 1+\|\C\|_{L_x^3}\|\p v\|_{L_x^2}+\|\Omega\|_{L_x^\infty}\|\p v\|_{L_x^6}\|\p \varrho\|_{L_x^3} \\
&\les 1+\|\p v\|_{H^1_x}\|\p \varrho\|_{L_x^3}.
\end{align*}
Similarly, also using (\ref{5.03.5.19})
\begin{align*}
\|\p^2 \La^{-2} \bT \curl \fw\|_{L_x^2}&\les \|\p (\p v \Omega e^\varrho)\|_{L_x^2}+\|\p v (\C e^\varrho +\Omega\p \varrho) e^\varrho\|_{L_x^2}\\
&\les\|\p^2 v\|_{L_x^2}\|\Omega\|_{L_x^\infty}+\|\p v\|_{L_x^6}\|\C,\p \Omega\|_{L_x^3}+\|\Omega\|_{L_x^\infty}\|\p v\|_{L_x^6}\|\p \varrho\|_{L_x^3}\\
&\les \|\p v\|_{H^1_x}(\|\p \varrho\|_{L_x^3}+1).
\end{align*}
Therefore (\ref{9.21.3.19}) is proved.

On the other hand by the definition of $\eta$ in (\ref{4.10.5.19})
\begin{equation*}
\bT(I-\Delta_e) \eta=\bT \curl \fw
\end{equation*}
which implies
\begin{equation*}
\bT \eta=(I-\Delta_e)^{-1}(-[\bT, I-\Delta_e]\eta+\bT \curl \fw).
\end{equation*}
By using (\ref{cmu1}), there holds symbolically that
\begin{equation*}
[\bT, \Delta_e]\eta=\p(\p v\p \eta)+\p v(\p^2 \eta).
\end{equation*}
Thus,
\begin{equation}\label{4.14.1.19}
\bT\eta=\La^{-2} (\p(\p v\p \eta)+\p v \p^2 \eta+\bT\curl \fw ).
\end{equation}

By using (\ref{10.18.1.19}), we first give the base order estimate for $\bT \eta$ with the help of the above identity,
\begin{align*}
\|\bT \eta\|_{L_x^2}&\le \|\La^{-1}(\p v\p \eta)\|_{L_x^2}+\|\La^{-2}(\p v\p^2\eta)\|_{L_x^2}+\|\La^{-2} \bT \curl \fw\|_{L_x^2}\\
&\les \|\p v\p \eta\|_{L_x^\frac{6}{5}}+\|\p v\c\p^2 \eta\|_{L_x^1}+\|\La^{-2} \bT \curl \fw\|_{L_x^2}\\
&\les \|\p v \|_{L_x^2}(\|\p \eta\|_{L_x^3}+\|\p^2 \eta\|_{L_x^2})+\|\La^{-2} \bT \curl \fw\|_{L_x^2}\\
&\les 1+\|\La^{-2} \bT \curl \fw\|_{L_x^2},
\end{align*}
where we employed Corollary \ref{5.04.12.19}, Sobolev embedding and  the first estimate in  (\ref{5.03.9.19}).

For higher order derivatives, in view of (\ref{4.14.1.19}),  using (\ref{5.03.9.19}), Corollary \ref{5.04.12.19}, Sobolev embedding and the first estimate in (\ref{10.18.1.19}), we derive
\begin{align*}
\|\p \bT \eta\|_{L_x^2}&\les \|\p v \c \p \eta\|_{L_x^2}+\|\La^{-1}(\p v \c \p^2 \eta)\|_{L_x^2}+\|\p\La^{-2}\bT \curl \fw\|_{L_x^2}\\
&\les \|\p v\|_{L_x^2}\|\p\eta\|_{L_x^\infty}+\|\p v \c \p^2 \eta\|_{L^{\frac{6}{5}}_x}+\|\p\La^{-2}\bT \curl \fw\|_{L_x^2}\\
&\les  \|\p v\|_{L_x^2}\|\p \eta\|_{L_x^\infty}+\|\p v\|_{L_x^3}\||\p^2 \eta\|_{L_x^2}+\|\p\La^{-2}\bT \curl \fw\|_{L_x^2}\\
&\les \|\p \eta\|_{L_x^\infty}+\|\p v\|_{L_x^3}+\| \La^{-2}\bT \curl \fw\|_{H^1_x};\\
\|\p^2 \bT \eta\|_{L_x^2}&\les \|\p(\p v\p \eta)\|_{L_x^2}+\|\p v \c \p^2 \eta\|_{L_x^2}+\|\p \La^{-2} \bT\curl \fw\|_{H^1_x}\\
&\les \|\p^2 v\|_{L_x^2}\|\p \eta\|_{L_x^\infty}+\|\p v\|_{L_x^6}\|\p^2 \eta\|_{L_x^3}+\|\p \La^{-2} \bT\curl \fw\|_{H^1_x}\\
&\les \|\p v\|_{H^1_x}(\|\p \eta\|_{L_x^\infty}+\|\p \varrho\|_{L_x^3}+1)+\|\p \La^{-2} \bT\curl \fw\|_{H^1_x},
\end{align*}
where we also used (\ref{9.21.2.19}) to derive the last line.

Note by using Corollary \ref{5.04.12.19} and Sobolev embedding
$$\|\p v\|_{L_x^3}\les \|\p v\|_{H^1_x}^\f12\|\p v\|_{L_x^2}^\f12+\|\p v\|_{L_x^2}\les \|\p v\|_{H^1_x}^\f12+1.$$
Applying (\ref{9.21.3.19}) to the above inequalities leads to
\begin{align*}
&\|\bT \eta\|_{H^l_x}\les 1+ l\big(\|\p \eta\|_{L_x^\infty}+\|\p v\|_{H^1_x}(\|\p \varrho\|_{L_x^3}+1)\big),\quad l=0,1\\
&\|\p\bT \eta\|_{H^1_x}\les \|\p v\|_{H^1_x}(\|\p \eta\|_{L_x^\infty}+\|\p \varrho\|_{L_x^3}+1).
\end{align*}
These are (\ref{5.03.3.19_1}) and (\ref{4.12.5.19_1}).

\end{proof}

\subsection{Energy estimates for $v_+$ and $\varrho$}
  We will provide the energy estimates for $v_+$ and $\varrho$ in this subsection. To distinguish the energies with $(U, V, F_U, F_V)$ defined in (\ref{5.03.1.19}) and (\ref{5.03.2.19}), we denote the two sets of energies by $\E\rp{l}_{v_+}(t)$ and $\E\rp{l}_\varrho(t)$ respectively.

\subsubsection{Lower order energy estimates}
We give the first order energy estimate, where the $0$-order energy estimate in Corollary \ref{5.04.12.19} will be frequently used.
\begin{proposition} [First order energies]\label{leng}
 For $(U, V, F_U, F_V)$ given in (\ref{5.03.1.19}) and  (\ref{5.03.2.19}), there hold
\begin{itemize}
\item[(0)]
\begin{align}
&\|F_U\|_{H^1_x}\les   1+ \|\p \eta\|_{L_x^\infty}+\|\p v\|_{H^1_x}(\|\p \varrho\|_{L_x^3}+1);\label{4.14.2.19}\\
&\|\fZ\|_{L_x^2}+\|F_V\|_{L^2_x}\les \|\bp v, \bp \varrho\|_{L_x^\infty}+1+\|\p \varrho\|_{L_x^3}\label{4.14.3.19}
\end{align}
and in particular $F_U\equiv 0$ for $U=\varrho$.
\item[(1)]
\begin{align}
\|\p^2 F_U\|_{L_x^2}&\les (1+\|\p v_+\|_{H^1_x})(\|\p \eta\|_{L_x^\infty}+\|\p \varrho\|_{L_x^3}+1)\label{5.04.5.19}\\
\|\p F_V\|_{L_x^2}&\les (\|\bp v, \bp \varrho, \p \eta\|_{L_x^\infty}+\|\p \varrho\|_{L_x^3}+1)\big(\|\bp \varrho\|_{H_x^1}+\|\p^2 v_+\|_{L_x^2}+1\big).\label{5.04.6.19}
\end{align}
\item[(2)]
 \begin{align}
&\E\rp{\le 1}_{v_+}(t)^\f12+\E\rp{\le 1}_{\varrho}(t)^\f12\les \E\rp{\le 1}_{v_+}(0)^\f12+\E\rp{\le 1}_{\varrho}(0)^\f12+1\les 1\label{5.04.16.19}
 \end{align}
 \end{itemize}
\end{proposition}

\begin{proof}
We first give the $0$-order error estimates in (0).
Recall $F_U=(-\bT \eta,0)$, (\ref{4.14.2.19}) follows from (\ref{5.03.3.19_1}).
Recall from (\ref{5.07.1.19}) and  Corollary \ref{5.04.12.19}, using the first estimate in (\ref{5.03.9.19}), (\ref{5.03.4.19}) and Sobolev embedding we bound
 \begin{align*}
\|\fZ\|_{L_x^2}+ \|F_V\|_{L_x^2}&\les(\|\p v\|_{L_x^2}+\|\bp \varrho\|_{L_x^2})(\|\bp v\|_{L_x^\infty}+\|\bp \varrho\|_{L_x^\infty})+\|\eta, (\Omega, \p \eta)\c \p \varrho\|_{L_x^2}\\
 &\les \|\bp v, \bp \varrho\|_{L_x^\infty}+\|\eta\|_{L_x^2}+\|\Omega\|_{L_x^\infty}\|\p \varrho\|_{L_x^2}+\|\p\eta\|_{L_x^6}\|\p \varrho\|_{L_x^3}\\
 &\les \|\bp v, \bp \varrho\|_{L_x^\infty}+1+\|\p \varrho\|_{L_x^3}.
 \end{align*}
   This gives (\ref{4.14.3.19}).

 The estimate of (\ref{5.04.5.19}) follows from (\ref{4.12.5.19_1}),
$
\|\p^2 F_U\|_{L_x^2}= \|\p^2\bT\eta\|_{L_x^2},
$
 and   the derivative estimates
 \begin{align}\label{5.07.5.19}
 \|\p^m v\|_{L_x^2}&\les \|\p^m \eta\|_{L_x^2}+\|\p^m v_+\|_{L_x^2}\les \|\p^m v_+\|_{L_x^2}+1, \quad m=0,1,2,
 \end{align}
 which is derived by using (\ref{5.03.9.19}).

 Next we consider $\|F_V\|_{H^1_x}$. In view of (\ref{5.07.1.19}),
\begin{equation}\label{5.07.6.19}
\p F_V= \p\big((C(\varrho)+1)\c \fZ\big)+II
\end{equation}
with
$$II=\p \big (c^2 (\eta+\fw \c \p \varrho+\p (\log c)\c \p \eta)\big)$$ and  $\fZ$  given in (\ref{5.07.2.19}). Expanding $\p\fZ$ gives
 \begin{equation}\label{10.19.1.19}
 \p \fZ=(\bp \varrho+\bp v) \c (\p \bp \varrho+\p\bp v).
  \end{equation}
 It follows  by using (\ref{4.23.1.19}) and  (\ref{5.07.5.19})
  \begin{align}\label{5.06.9.19}
  \begin{split}
  \|\p \fZ\|_{L_x^2}&\les \|\bp \varrho,\bp v\|_{L_x^\infty} \c (\|\p\bp \varrho, \p^2 v_+, \p\bT v \|_{L_x^2}+1)\\
  & \les  \|\bp \varrho,\bp v\|_{L_x^\infty}(\|\bp \varrho\|_{H^1_x}+\|\p^2 v_+\|_{L_x^2}+\||\p \varrho|^2\|_{L_x^2}+1)\\
  &\les  \|\bp \varrho,\bp v\|_{L_x^\infty}(\|\bp \varrho\|_{H^1_x}+\|\p^2 v_+\|_{L_x^2}+1),
  \end{split}
  \end{align}
  where we have used Sobolev embedding on ${\mathbb R}^3$ and Corollary \ref{5.04.12.19} to bound
  \begin{equation}\label{9.21.5.19}
  \|\p\varrho\|_{L_x^4}\les \|\p \varrho\|_{H^1_x}^\frac{3}{4}\|\p \varrho\|_{L_x^2}^\frac{1}{4}\les \|\p \varrho\|_{H^1_x}+1.
  \end{equation}
   The estimate for the term $II$ in (\ref{5.07.6.19}) is performed term by term as follows:
  \begin{equation*}
  \|\p(c^2\eta)\|_{L_x^2}\les \|\p \varrho\|_{L_x^2}\|\eta\|_{L_x^\infty}+\|\p \eta\|_{L_x^2}\les 1
  \end{equation*}
 where we used (\ref{5.03.9.19}) and Corollary \ref{5.04.12.19}.
 \begin{align*}
 \|\p(c^2 \fw \p \varrho)\|_{L_x^2}&\les\|\p \varrho\|_{L_x^6}\|\p \fw \|_{L_x^3}+\|\fw\|_{L_x^\infty}\||\p \varrho|^2\|_{L_x^2}\\
&\les \|\p \varrho\|_{H^1_x}(\|\p \varrho\|_{L_x^3}+1)
 \end{align*}
 where we used (\ref{5.03.4.19}) and (\ref{5.03.5.19}) to bound $\fw$, used (\ref{9.21.5.19}) and Sobolev embedding to bound norms of $\p \varrho$. By using the first estimate in (\ref{5.03.9.19}) and (\ref{9.21.5.19}), we can obtain
  \begin{align*}
  \|\p(\p (c^2) \p \eta)\|_{L_x^2}&\les \|\p^2(c^2)\|_{L_x^2}\|\p\eta\|_{L_x^\infty}+\|\p^2 \eta\|_{L_x^2}\|\p(c^2)\|_{L_x^\infty}\\
  &\les (\|\p \varrho\|_{H^1_x}+1)\|\p \eta\|_{L_x^\infty}+\|\p \varrho\|_{L_x^\infty}.
  \end{align*}
  Combining the above three estimates gives
  \begin{equation*}
  \|II\|_{L_x^2}\les  (\|\p \varrho\|_{H^1_x}+1)(\|\p \eta\|_{L_x^\infty}+\|\p \varrho\|_{L_x^3}+1)+\|\p \varrho\|_{L_x^\infty}
  \end{equation*}
In view of  (\ref{5.07.6.19}), substituting the estimate for $\|II\|_{L_x^2}$, (\ref{5.06.9.19}) and the first estimate in (\ref{4.14.3.19})  leads to
\begin{align*}
\|\p F_V\|_{L_x^2}&\les \|\p \fZ\|_{L_x^2}+\|\p C(\varrho)\|_{L_x^\infty}\|\fZ\|_{L_x^2}+\|II\|_{L_x^2}\\
&\les (\|\bp v, \bp \varrho, \p \eta\|_{L_x^\infty}+\|\p \varrho\|_{L_x^3}+1)\big(\|\bp \varrho\|_{H_x^1}+\|\p^2 v_+\|_{L_x^2}+1\big).
\end{align*}
  This gives (\ref{5.04.6.19}).

We now rewrite  the error estimates (\ref{4.14.2.19})-(\ref{5.04.6.19}) in view of and $\p \eta=\p v-\p v_+$ as
\begin{align}\label{9.28.2.19}
\|F_U\|_{H^2_x}+\| F_V\|_{H^1_x}&\les (\|\bp v, \bp \varrho, \p v_+\|_{L_x^\infty}+\|\p \varrho\|_{L_x^3}+1)\big( \E\rp{\le 1}_{v+}(t)^\f12+\E\rp{\le 1}_\varrho(t)^\f12+1\big).
\end{align}

Substituting the above estimates to (\ref{5.02.4.19}) and also applying (\ref{5.07.5.19})  imply
\begin{align}\label{10.18.3.19}
\begin{split}
&\sum_{l=0}^1(\E\rp{l}_{v_+}(t)^\f12+\E\rp{l}_{\varrho}(t)^\f12)\\
&\les \int_0^t (\|\bp v, \p v_+, \bp \varrho\|_{L_x^\infty}+\|\p \varrho\|_{L_x^3}+1)(\E\rp{\le 1}_\varrho(t)^\f12+\E_{v_+}\rp{\le 1}(t)^\f12+1)\\
&+\sum_{l=0}^1(\E_{v_+}\rp{l}(0)^\f12+\E\rp{l}_{\varrho}(0)^\f12)
\end{split}
\end{align}
Similar to (\ref{5.07.5.19}), there holds the derivative estimate
\begin{align*}
\|\p^m v_+\|_{L_x^2}&\les \|\p^m \eta\|_{L_x^2}+\|\p^m v\|_{L_x^2}\les \|\p^m v\|_{L_x^2}+1, \quad m=0,1,2.
\end{align*}
Therefore when $t=0$,
also using  (\ref{4.23.1.19}),   $V_{v_+}=\bT v$ in (\ref{5.03.1.19}) and Sobolev embedding
$
\|\p \varrho\|_{L_x^4}\les \|\p \varrho\|_{H^1_x}$, we have
\begin{equation*}
\E\rp{l}_{v_+}(0)^\f12\les \|\p v(0)\|_{H^{l}_x}+\|\p \varrho(0)\|_{H^{l}_x},\quad l=0,1.
\end{equation*}

Similarly in view of (\ref{5.03.2.19}) and the first equation in (\ref{4.23.1.19}), we derive
\begin{align*}
\E\rp{l}_\varrho(0)^\f12 \les \|\p \varrho(0)\|_{H^l_x}+\|\p v(0)\|_{H^{l}_x},\quad l=0,1.
\end{align*}
Thus  we have the comparison result for the initial data that
\begin{equation*}
\E\rp{l}_{v_+}(0)+\E\rp{l}_\varrho(0) \les 1,\quad l=0,1.
\end{equation*}
Using (\ref{BA1}) and (\ref{ABA1}) and applying Gronwall's inequality to (\ref{10.18.3.19}),  (\ref{5.04.16.19}) can be proven.
 \end{proof}
\subsection{Improvement on (\ref{ABA1})}\label{aux_imp}
As a direct consequence of (\ref{5.04.16.19}), by Sobolev embedding, we obtain $\|\p \varrho\|_{L_x^3}\les \|\p \varrho\|_{H^1_x}\les 1$. Hence
\begin{equation}\label{9.21.6.19}
\|\p \varrho\|_{L_x^3}\les 1,\quad \quad \|\p \varrho\|_{L_t^2 L_x^3}\les T^\f12.
\end{equation}
Thus the bootstrap assumption (\ref{ABA1}) is improved.

We summarize some important estimates below for future reference.
 \begin{corollary}\label{comp_1}
 	Let $0<\ep\le s-2$. There hold
 \begin{align}
&\|\bp v, \p v_+, \bp \varrho\|_{H^1_x}+\|\bp C(\varrho)\|_{H^1_x}\les 1,\label{5.04.17.19}\\
 &\|\fw, \Omega\|_{H^1_x}\les 1, \quad \|\p \Omega, \p \fw, \p^2 \eta\|_{L_x^p}\les 1,\label{9.22.3.19}\\
 &\|\mbox{Ric}(g)\|_{L_x^2}\les 1,\label{10.21.2.19}
 \end{align}
 where $2\le p\le \frac{3}{1-\ep}$ and $C(y)$ are smooth functions; and  there hold the following error estimates
 \begin{equation}\label{9.28.3.19}
 \|F_U\|_{H^2_x}+\| F_V\|_{H^1_x}\les \|\bp v, \bp \varrho, \p v_+\|_{L_x^\infty}+1.
 \end{equation}
 \end{corollary}
 \begin{proof}
The first set of estimates in (\ref{5.04.17.19}) is a consequence of Corollary \ref{5.04.12.19} and (\ref{5.04.16.19}). The second  estimate follows by using the first set of estimates and the smoothness of $C''$ and Sobolev embedding.  Substituting (\ref{9.21.6.19})  and (\ref{5.04.16.19}) into (\ref{9.28.2.19}) implies (\ref{9.28.3.19}). It follows from (\ref{5.04.17.19}) and Sobolev embedding that $\|\p \varrho\|_{L^p_x}\les 1$ if $2\le p\le \frac{3}{1-\ep}$. The derivative estimates in  (\ref{9.22.3.19}) can then be derived by combining this estimate, (\ref{5.03.5.19}) and (\ref{9.21.2.19}). The $L^2_x$ estimates for $\fw$ and $\Omega$ have been obtained in (\ref{5.03.4.19}). Note
$
\mbox{ Ric}=g (\p^2 g+\p g \p g).
$
 (\ref{10.21.2.19}) is a consequence of the derivative estimate of $\varrho$ in (\ref{5.04.17.19}) together with Sobolev embedding, (\ref{9.20.1.19}) and $|\varrho|\le C$.
 \end{proof}

 Using \cite{Ander} or \cite[Theorem 5.4]{Petersen}, there exists a constant $d_0>0$ depending only on the constant bounds in (\ref{9.20.1.19}) and the bound of $\|\mbox{Ric}\|_{L_x^2}$, which is the uniform lower bound of radius of injectivity on $\Sigma_t$ for $t\in [0,T]$. We assume $T$ satisfies  $0<T\le \min(d_0, 1)$ throughout the paper.

 Due to Corollary \ref{5.04.12.19}, and for smooth function $C(y)$, there holds
  $
 \|C(\varrho), C'(\varrho)\|_{L_x^\infty}+\|\p\varrho\|_{H^1_x}\les 1.
 $
 This leads to
 \begin{lemma}\label{comp_dyd}
 Let $0<\a\le \f12$ be fixed. For $C(\varrho)$ a smooth function of $\varrho$, there hold for scalar functions $f$ that
 \begin{align}
 &\|\La^\a(C(\varrho) f)\|_{L_x^2}\les \|\La^\a f\|_{L_x^2}+ \|f\|_{L_x^2}, \label{9.22.4.19}\\
 &\|\La^{\f12+\a}(C(\varrho) f)\|_{L_x^2}\les\|\La^{\f12+\a} f\|_{L_x^2}+\|f\|_{L_x^2}.\label{9.22.5.19}
 \end{align}
 \end{lemma}
 \begin{proof}
For (\ref{9.22.4.19}), we apply (\ref{9.08.1.19}) to $(F, G)=\big(C(\varrho), f\big)$  to obtain
\begin{equation}\label{9.22.6.19}
\|\La^\a(C(\varrho) f)\|_{L_x^2}\le \|C(\varrho)\|_{L_x^\infty}\|\La^\a f\|_{L_x^2}+\|C(\varrho)\|_{B_{\infty, 2}^\a}\| f\|_{L_x^2}.
\end{equation}
Note that due to Bernstein inequality
\begin{equation*}
\|\mu^\a P_\mu C(\varrho)\|_{L_x^\infty}\les \|\mu^{\a+\f12}P_\mu \p C(\varrho)\|_{L_x^2}.
\end{equation*}
 Thus for all $0\le \a\le \f12$
 \begin{align*}
 \|\mu^\a P_\mu C(\varrho)\|_{l_\mu^2L_x^\infty}\les \|\p C(\varrho)\|_{H^1_x}\les \|\p \varrho\|_{H^1_x} \les 1.
 \end{align*}
Combining this inequality with  (\ref{9.22.6.19}) implies (\ref{9.22.4.19}).

For (\ref{9.22.5.19}) and $0<\a<\f12$,  we first derive
\begin{equation*}
\mu^{\f12+\a}\| P_\mu(C(\varrho) f)\|_{L_x^2}\le \mu^{\f12+\a}\|C(\varrho)P_\mu f\|_{L_x^2}+\mu^{\f12+\a}\|[P_\mu, C(\varrho)]f\|_{L_x^2}.
\end{equation*}
For the second term on the right hand side, we apply (\ref{lem4eq}) to $(F, G)=\big(C(\varrho), f)$  to derive
\begin{equation*}
\|\mu^{\f12+\a}[P_\mu, C(\varrho)]f\|_{l_\mu^2 L_x^2}\les \|\p C(\varrho)\|_{H^1_x}\|\La^\a f\|_{L_x^2}\les \|\La^\a f\|_{L_x^2},
\end{equation*}
where we used (\ref{5.04.17.19}). Combining the above  two inequalities gives (\ref{9.22.5.19}) for the case $0<\a<\f12$.

If $\a=\f12$, we can directly obtain (\ref{9.22.5.19}) by using (\ref{5.04.17.19}) and Sobolev embedding.
 \end{proof}

  \begin{proposition}
  	Let $\ep=s-2$. There hold the following estimates
\begin{align}
&\|\p \eta\|_{L^\infty}+\|\p^2\eta\|_{L_x^p}\les 1,\qquad 2\le p\le \frac{3}{1-\ep}\label{5.03.3.19}\\
&\|\p \Omega\|_{\dot{H}^\a_x}\les \|\curl \Omega\|_{\dot{H}^\a_x}+1, \quad 0<\a \le \frac{1}{2}+\ep\label{5.06.2.19}\\
&\|\curl \fw\|_{H^\a_x}\les \|\curl \Omega\|_{H^\a_x}+1, 0<\a \le \frac{1}{2}+\ep\label{9.22.8.19}\\
&\| \eta\|_{H^{2+\a}_x}\les\|\curl \Omega\|_{H^\a_x}+1,\quad  0<\a\le \f12+\ep\label{4.12.2.19}\\
&\|\bT \eta\|_{H^2_x}\les 1,\displaybreak[0] \label{4.12.5.19}\\
&\|\bT \eta\|_{L^\infty_x}\les 1, \label{5.04.1.19}\\
&\|\bT \eta\|_{H^{2+\ep}_x}\les \|\p v, \p \varrho\|_{H^{1+\ep}_x}+ \|\curl\Omega\|_{H^{\frac{1}{2}+\ep}_x}+1.\label{4.12.3.19}
\end{align}
\end{proposition}
\begin{proof} The $L^p_x$ estimate in (\ref{5.03.3.19}) has been included in (\ref{9.22.3.19}). The estimate of $\|\p \eta\|_{L^\infty_x}$ is a consequence of the $L^p_x$ estimate and  (\ref{5.03.9.19}) by using Sobolev embedding.

Next we consider (\ref{5.06.2.19}). Note that  by using (\ref{div})
\begin{align*}
\int_{\Sigma_t} \mu^{2\a} |\p P_\mu \Omega|^2&=\int_{\Sigma_t}\mu^{2\a}\{|\curl P_\mu \Omega|^2+|\div P_\mu \Omega)|^2 \}dx\nn \\
&=\int_{\Sigma_t}\mu^{2\a}\{|P_\mu \curl \Omega|^2+|P_\mu \div \Omega|^2\} dx\nn\\
&=\int_{\Sigma_t}\mu^{2\a}\{|P_\mu \curl \Omega|^2+|P_\mu(\Omega \p \varrho)|^2 \}dx.
\end{align*}
Thus,
\begin{equation}\label{5.11.1.19}
\|\mu^\a \p P_\mu \Omega\|^2_{l_\mu^2 L_x^2}=\|\mu^\a P_\mu \curl \Omega\|^2_{l_\mu^2 L_x^2}+\|\mu^\a P_\mu(\Omega \p \varrho)\|^2_{l_\mu^2 L_x^2}.
\end{equation}
 We consider the second term on the right hand side of (\ref{5.11.1.19}) by applying (\ref{9.22.7.19}) to $(F, G)=(\Omega, \p \varrho)$, which  gives
\begin{equation}\label{9.22.9.19}
\|\mu^\a P_\mu(\Omega \p \varrho)\|_{l_\mu^2 L_x^2}\les \|\Omega\|_{\dot{H}^{\f12+\a}}\|\p\varrho\|_{H^1_x}+\|\Omega\|_{L_x^\infty}\|\p \varrho\|_{\dot{H}^\a_x}, \quad 0<\a\le \f12+\ep.
\end{equation}
If $\a\le \f12$, using (\ref{9.22.3.19}) and (\ref{5.03.4.19}), the right hand side is bounded by $\|\p \varrho\|_{H^1_x}$. In this case, by using (\ref{5.04.17.19}), we can obtain
\begin{equation*}
\|\p \Omega\|_{\dot{H}^\a_x}\les \|\curl \Omega\|_{\dot{H}^\a_x}+\|\p \varrho\|_{H^1_x}\les \|\curl \Omega\|_{\dot{H}^\a_x}+1.
\end{equation*}
 (\ref{5.06.2.19}) is proved.
 If $\a>\f12$, note that with $\theta=\frac{\a-\f12}{\a}$ and using (\ref{9.22.3.19}),
\begin{equation*}
\|\Omega\|_{\dot{H}^{\f12+\a}_x} \le \|\Omega\|_{\dot{H}^{1+\a}_x}^\theta\|\Omega\|_{\dot{H}^1_x}^{1-\theta}\les \|\Omega\|_{\dot{H}^{1+\a}_x}^\theta,
\end{equation*}
where we also used (\ref{5.03.5.19}).
Combining the above estimate with (\ref{5.11.1.19}) and (\ref{9.22.9.19}), by using Young's inequality, we can obtain (\ref{5.06.2.19}) for the case $\f12<\a\le \f12+\ep$.

 Recall from  the second equation of (\ref{hodge_2}), applying Lemma \ref{comp_dyd} and (\ref{9.22.9.19})
\begin{align*}
\|\curl \fw\|_{H^\a_x}&\les \|\curl \Omega\|_{H^\a_x}+\|\Omega \c \p \varrho\|_{H^\a_x}, \quad 0<\a\le \f12+\ep\\
&\les \|\curl \Omega\|_{H^\a_x} +\|\Omega\|_{H^{\f12+\a}_x}+1\\
&\les \|\curl \Omega\|_{H^\a_x}+1,
\end{align*}
where we also used (\ref{5.03.4.19}), (\ref{5.06.2.19}) and (\ref{5.04.17.19}) to derive the last two lines.  Thus (\ref{9.22.8.19}) is proved. (\ref{4.12.2.19}) is a direct consequence of (\ref{9.22.8.19}) and the definition of $\eta$ in (\ref{4.10.5.19}).

(\ref{4.12.5.19}) can be derived by substituting the first estimate in (\ref{5.03.3.19}),  (\ref{5.04.17.19}) and (\ref{9.21.6.19}) into the estimates (\ref{5.03.3.19_1}) and (\ref{4.12.5.19_1}). (\ref{5.04.1.19}) is its consequence due to Sobolev embedding.

Thus we complete the proof of (\ref{5.04.1.19}) in view of the first estimate in (\ref{5.03.3.19}).

 We next consider (\ref{4.12.3.19})  in view of (\ref{4.14.1.19}),
\begin{equation*}
\|\p^2 \bT \eta\|_{H^\ep_x}  \le I+J+K,
\end{equation*}
where
 \begin{align*}
I=\|\p (\p v\c \p \eta)\|_{H^\ep_x}, \quad J=\|\p v\c \p^2 \eta\|_{H^\ep_x}, \quad K=\|\bT \curl \fw\|_{H^\ep_x}.
 \end{align*}
For the first term $I$, applying (\ref{5.05.6.19}) to $F=\p v$, $G=\p \eta$ and $\a=\ep$ gives
\begin{equation*}
I\les \|\p v\|_{H^{1+\ep}_x}\|\p \eta\|_{L_x^\infty}+\|\p v\|_{H^1_x}\|\p \eta\|_{H^{\frac{3}{2}+\ep}_x}.
\end{equation*}
For the second term $J$, we apply (\ref{5.05.7.19}) to $F=\p v$, $G=\p \eta$ and $\a=\ep$ to derive
\begin{align*}
J&\les \|\p v\|_{H^{1+\ep}_x}\|\p \eta\|_{L_x^\infty}+\|\p v\|_{H^1_x}\|\p \eta\|_{H^{\frac{3}{2}+\a}_x}.
\end{align*}
Also by using (\ref{4.12.2.19}), (\ref{5.03.3.19}) and (\ref{5.03.9.19}), we can conclude that
\begin{equation*}
I+J\les \|\p v\|_{H^{1+\ep}_x}+\|\p v\|_{H^1_x} \|\curl \Omega\|_{H^{\f12+\ep}_x}.
\end{equation*}
For $K$, in view of (\ref{9.21.4.19}), we apply (\ref{9.22.4.19}) with $\a=\ep$ and $C(y)=e^y$ to derive
\begin{equation}
\begin{split}
\|\bT \curl \fw\|_{H^\ep_x}&\les \|\p\big(\p v \Omega e^\varrho\big)\|_{\dot{H}^\ep_x}+\|\C e^{2\varrho} \p v\|_{H^\ep_x}+\|\p v\c \p (e^\varrho)\Omega\|_{H^\ep_x}\\
&\les \|\p(\p v \Omega) \|_{H^\ep_x}+\|\curl \Omega \c \p v\|_{H^\ep_x}+\|\p v\c \p \varrho\Omega\|_{H^\ep_x} \label{9.22.1.19}.
\end{split}
\end{equation}
For the first term on the right hand side, we apply (\ref{5.05.6.19}) to $(F, G)=(\p v, \Omega)$ and $\a=\ep$ to derive
\begin{align*}
\|\p\big(\p v \Omega\big) \|_{H^\ep_x}&\les \|\p v\|_{H^{1+\ep}_x}\|\Omega\|_{L_x^\infty}+\|\p v\|_{H^1_x}\|\Omega\|_{H^{\frac{3}{2}+\ep}_x}\\
&\les \|\p v\|_{H^{1+\ep}_x}+\|\p v\|_{H^1_x}(\|\curl \Omega\|_{H^{\f12+\ep}_x}+1),
\end{align*}
where we also used (\ref{5.06.2.19}) with $\a=\f12+\ep$ and (\ref{5.03.4.19}).

For the second term on the right of (\ref{9.22.1.19}), by applying (\ref{5.05.7.19}) to $(F, G)=(\p v, \Omega)$ we derive
\begin{align*}
\|\La^\ep(\curl \Omega \p v)\|_{L_x^2}&\les \|\p v\|_{\dot{H}^{1+\ep}_x} \|\Omega\|_{L_x^\infty}+\|\p v\|_{H^1_x} \|\Omega\|_{H^{\frac{3}{2}+\ep}_x}\\
&\les \|\p v\|_{\dot{H}^{1+\ep}_x}+\|\p v\|_{H^1_x} (\|\curl \Omega\|_{H^{\f12+\ep}_x}+1),
\end{align*}
where we used (\ref{5.03.4.19})  and (\ref{5.06.2.19}).

For the third term on the right of (\ref{9.22.1.19}), we apply (\ref{9.08.17.19}) to $(G_1, G_2, G_3)=(\p v, \p \varrho, \Omega)$ that
\begin{align*}
\|\La^\ep(\p v \c \p \varrho\c \Omega)\|_{L_x^2}&\les \|\La^\ep \p v\|_{H^1_x}\|\p v\|_{H^1_x}\|\Omega \|_{H^1_x}+\|\La^\ep \Omega\|_{H^1_x}\|\p v\|_{H^1_x}\|\p \varrho\|_{H^1_x}\\
&+\|\La^\ep\p \varrho\|_{H^1_x}\|\p v\|_{H^1_x}\|\Omega\|_{H^1_x}.
\end{align*}
Since $\|\Omega\|_{H^1_x}\les 1$ in (\ref{9.22.3.19}), also using  (\ref{5.04.17.19}) and (\ref{5.06.2.19}),
\begin{equation*}
\|\La^\ep(\p v \c \p \varrho\c \Omega)\|_{L_x^2}\les \|\La^\ep(\p v, \p \varrho)\|_{H^1_x}+\|\La^\ep\curl \Omega\|_{L_x^2}+1.
\end{equation*}
Summing up the above estimates for three terms and also using (\ref{5.04.17.19})  imply
\begin{align*}
K=\|\bT \curl \fw\|_{H^\ep_x}\les \|\La^\ep(\p v, \p\varrho)\|_{H^1_x}+\|\curl \Omega\|_{H^{\f12+\ep}_x}+1.
\end{align*}
Combining the above estimate with the estimate for $I+J$ give (\ref{4.12.3.19}).
\end{proof}

 \subsubsection{Highest order energy estimates for $v_+$ and $\varrho$} We will control the highest order energy for $v_+$ and $\varrho$, by  giving the control of  $F_U$ and $F_V$ at the highest order, and using Proposition \ref{4.13.8.19}.

 \begin{proposition}\label{erro_h}
 Let $0<\ep\le s-2$.

(1) For $(U, V, F_U, F_V)$ in (\ref{5.03.1.19}) and (\ref{5.03.2.19}), there hold
\begin{align}
&\|\p^2 F_U\|_{H^\ep_x}\les\|\p v, \bp \varrho\|_{H^{1+\ep}_x}+\|\curl\Omega\|_{H^{\frac{1}{2}+\ep}_x}+1,\label{5.05.3.19}\\
&\|\p F_V\|_{H^\ep_x}\les (\|\p(\bp v, \bp \varrho)\|_{\dot{H}^\ep_x}+\|\curl\Omega\|_{H^\ep_x} +1)(\|\bp v, \bp \varrho\|_{L_x^\infty}+1). \label{5.05.4.19}
\end{align}

(2) Denote $\E\rp{1}_\mu(t):= \E\rp{1}_{v_+,\mu}(t)+\E\rp{1}_{\varrho,\mu}(t)$  for short. There hold the following energy estimate
\begin{equation}\label{5.11.3.19}
\begin{split}
\|\mu^\ep \E\rp{1}_\mu(t)^\f12\|_{l_\mu^2}&\les 1+\int_0^t\|\curl\Omega\|_{H^{\f12+\ep}_x}(\|\bp\varrho, \p v_+\|_{L_x^\infty}+1)dt'\\
&+\|\mu^\ep\E\rp{1}_\mu(0)^\f12\|_{l_\mu^2}.
\end{split}
\end{equation}
 \end{proposition}
 To prove the above result, we need a preliminary comparison result.
 \begin{lemma}\label{comp_2}
 Let $0<\ep\le s-2$.
\begin{align}
&\|\p \bp v\|_{\dot{H}^\ep_x}\les \big(\sum_{\mu>1}\mu^{2\ep}\E\rp{1}_\mu(t)\big)^\f12+\|\curl \Omega\|_{H^\ep_x}+1,\label{9.23.1.19}\\
&\|\p \bp \varrho, \p \Tr k, \p^2(c^2)\|_{\dot{H}^\ep_x}\les \big(\sum_{\mu>1}\mu^{2\ep}\E_{\varrho,\mu}\rp{1}(t)\big)^\f12+1.\label{9.23.2.19}
\end{align}
\end{lemma}
\begin{proof}
Due to (\ref{4.12.2.19}) and $v=v_+ +\eta$, we can obtain
\begin{align*}
&\|\p^2 v\|_{{\dot H}^\ep_x}\les \|\p^2 v_+\|_{{\dot H}^\ep_x}+\|\curl \Omega\|_{H^\ep_x}+1.
\end{align*}
Since $\bT v=-c^2 \p \varrho$ in (\ref{4.23.1.19}), $\p \bT v=c^2(c^{-1} c'(\varrho)\p \varrho \p \varrho+ \p^2 \varrho)$. We apply Lemma \ref{comp_dyd} to derive
\begin{align*}
\|c^2 \p^2 \varrho\|_{H^\ep_x}\les\|\p^2 \varrho\|_{H^\ep_x}.
\end{align*}
By using (\ref{9.08.8.19}) with $F=G=\p \varrho$, (\ref{5.04.17.19}) and Lemma \ref{comp_dyd}, we have
\begin{equation}\label{9.23.3.19}
\|C(\varrho)\bp \varrho \bp \varrho\|_{H^\ep_x}\les \||\bp \varrho|^2\|_{H^\ep_x}\les \|\bp \varrho\|_{H^1_x}^2\les 1.
\end{equation}
Combining the above two estimates yields
\begin{equation*}
\|\p \bT v\|_{\dot{H}^\ep_x}\les \|\p^2 \varrho\|_{H^\ep_x}+1.
\end{equation*}
Hence (\ref{9.23.1.19}) is proved.

The estimate in (\ref{9.23.2.19}) for $\p \bp \varrho$ follows by definition. In view of (\ref{5.24.1.19}) and (\ref{4.23.1.19})
\begin{equation*}
\p \Tr k=3\p ((\log c)'\bT\varrho)+\p \bT\varrho=C(\varrho)(\p \bp \varrho+\p \varrho \bp \varrho),
\end{equation*}
where $C(\varrho)$ represents several smooth functions of $\varrho$. $\p^2 (c^2)$ can also be written in the same symbolic form.
Applying Lemma \ref{comp_dyd} and (\ref{9.23.3.19}) leads to
\begin{equation*}
\|\p \Tr  k, \p^2(c^2) \|_{\dot{H}^\ep_x}\les \|\p\bp \varrho\|_{H^\ep_x}+1.
\end{equation*}
Thus (\ref{9.23.2.19}) is proved.
\end{proof}

 \begin{proof}[Proof of Proposition \ref{erro_h}]
Note $F_U=0$ in (\ref{5.03.2.19}), (\ref{5.05.3.19}) holds trivially in this case.
Recall the formula of $F_U$  from (\ref{5.03.1.19}) and
$
\|\p^2 F_U\|_{H^\ep_x}=\|\p^2\bT \eta\|_{H^\ep_x}.
$
(\ref{5.05.3.19}) follows immediately by using  (\ref{4.12.3.19}).

Recall the definition of $F_V$ in  (\ref{5.07.1.19}) and $\fw=\Omega e^{\varrho}$. We symbolically write
\begin{align*}
\p F_V&=\p (\fZ(C(\varrho)+1))+\p(c^2(\eta+\fw\c \p \varrho+ \p \log c\p\eta))\nn\\
&=(\p \fZ+\p(\eta+\Omega\c \p \varrho+ \p\varrho\p\eta))C(\varrho)+\p\fZ+(\fZ+\eta+\Omega\c \p \varrho+\p \varrho \p \eta)\p C(\varrho),
\end{align*}
where $C(\varrho)$ are smooth functions of $\varrho$ and may vary when  distributed to different factors.

Since in  (\ref{5.04.6.19}) we have obtained the estimate for $\|\p F_V\|_{L^2_x}$, we only consider the highest order estimate. Applying (\ref{9.22.4.19}) twice gives
\begin{align*}
\|\p F_V\|_{\dot{H}^\ep_x}&\les \|\p \eta\|_{H^\ep_x}+ \|\p \fZ\|_{H^\ep_x}+\|\p (\Omega \c\p \varrho+\p \eta\c \p \varrho)\|_{H^\ep_x}+\|(\fZ+\eta+\Omega\p \varrho+\p \eta\p \varrho)\p C(\varrho)\|_{H^\ep_x}\\
&\les 1+ \|\p \fZ\|_{H^\ep_x}+\|\p (\Omega \c\p \varrho+\p \eta\c \p \varrho)\|_{H^\ep_x}+\|(\fZ+\eta+\Omega\p \varrho+\p \eta\p \varrho)\p \varrho\|_{H^\ep_x},
\end{align*}
 where we also employed the first estimate in (\ref{5.03.9.19}).
 For the lower order term, applying (\ref{9.08.8.19}) to $(F, G)=(\eta, \p \varrho)$ yields
 \begin{align*}
 \|\eta\c \p \varrho\|_{H^\ep_x}&\les \|\eta\|_{H^{\f12+\ep}_x}\|\p \varrho\|_{H^1_x}+\|\p \varrho\|_{H^{\f12+\ep}_x}\|\eta\|_{H^1_x}\\
 &\les \|\eta\|_{H^1_x}\|\p \varrho\|_{H^1_x}\les 1,
 \end{align*}
 where we used (\ref{5.04.17.19}) and (\ref{5.03.9.19}).

 Next we estimate the quadratic terms. Recall (\ref{10.19.1.19}) for the form of $\p \fZ$. We  apply (\ref{9.20.5.19}) with $$\a=\ep,\, F=\p v, \bp \varrho; \, G=\bp v, \bp \varrho, \Omega, \p \eta$$ to derive
 \begin{align*}
& \|\p \fZ\|_{H^\ep_x}+\|\p (\Omega \c\p \varrho+\p \eta\c \p \varrho)\|_{H^\ep_x}\\
&\les \|\p(\p v,\bp \varrho)\|_{H^\ep_x}\|\bp v, \bp \varrho, \Omega, \p \eta\|_{L_x^\infty}+\|\p v, \bp \varrho\|_{L_x^\infty}\|\p(\bp v, \bp \varrho, \Omega, \p \eta)\|_{H^\ep_x}.
 \end{align*}
By using (\ref{5.06.2.19}) and (\ref{4.12.2.19}),
\begin{equation}\label{9.22.10.19}
\|\p \Omega\|_{H^\ep_x}+\|\p^2\eta\|_{H^\ep_x}\les \|\curl \Omega\|_{H^\ep_x}+1.
\end{equation}
 Using  the above estimate, (\ref{5.03.4.19}) and (\ref{5.03.3.19}), we obtain
 \begin{align*}
  \|\p \fZ\|_{H^\ep_x}&+\|\p (\Omega \c\p \varrho+\p \eta\c \p \varrho)\|_{H^\ep_x}\\
  &\les (\|\p(\bp v, \bp \varrho)\|_{\dot{H}^\ep_x}+\|\curl\Omega\|_{H^\ep_x} +1)(\|\bp v, \bp \varrho\|_{L_x^\infty}+1).
 \end{align*}
For the cubic terms, we apply (\ref{9.08.17.19}) to $G_1=\p v, \bp \varrho, \Omega, \p \eta$, $G_2=\bp v, \bp \varrho$ and $G_3=\p \varrho$.
\begin{align*}
\|(\fZ&+\Omega\p \varrho+\p \eta\p \varrho)\p \varrho\|_{H^\ep_x}\les \|\p v, \bp \varrho, \Omega, \p \eta\|_{H^{1+\ep}_x}\|\bp v, \bp \varrho\|_{H^1_x}\|\p \varrho\|_{H^1_x}\\
&+\|\p v,  \bp \varrho, \Omega, \p \eta\|_{H^1_x}(\|\bp v, \bp \varrho\|_{H^{1+\ep}_x}\|\p \varrho\|_{H^1_x}+\|\bp v, \bp \varrho\|_{H^1_x}\|\p \varrho\|_{H^{1+\ep}_x}).
\end{align*}
Using (\ref{9.22.10.19}), (\ref{9.22.3.19}), (\ref{5.03.9.19}), and Corollary \ref{5.04.12.19},  we derive
\begin{align*}
\|(\fZ+\Omega\p \varrho+\p \eta\p \varrho)\p \varrho\|_{H^\ep_x}&\les \|\p(\bp v, \bp \varrho)\|_{H^\ep_x}+\|\curl \Omega\|_{H^\ep_x}+1.
\end{align*}
Summing up the estimates for the lower order, quadratic and cubic terms yields (\ref{5.05.4.19}).

We hence combine the estimates of (\ref{5.05.3.19}) and (\ref{5.05.4.19}) with (\ref{9.23.1.19}) to conclude
\begin{equation}\label{10.19.2.19}
\|\p (\p F_U, F_V)\|_{H^\ep_x}\les (\big(\sum_{\mu>1}\mu^{2\ep}\E\rp{1}_\mu(t)\big)^\f12+\|\curl \Omega\|_{H^{\f12+\ep}_x}+1)(\|\bp \varrho, \bp v\|_{L_x^\infty}+1).
\end{equation}
To  prove (\ref{5.11.3.19}), we will substitute (\ref{5.05.3.19}) and  (\ref{5.05.4.19}) to the right hand side of the inequality in Proposition \ref{4.13.8.19}.
Recall from Lemma \ref{comp_2} that
\begin{equation*}
\|\p^2 (c^2), \p\bp v, \p \bp \varrho, \p(\Tr k) \|_{H^\ep_x}\les (\sum_{\mu>1}\mu^{2\ep}\E\rp{1}_\mu(t))^\f12+\|\curl \Omega\|_{H^\ep_x}+1.
\end{equation*}
Moreover, by (\ref{4.23.1.19}) and the first estimate in (\ref{5.03.3.19}),  for both the cases of $(U, V)$ in (\ref{5.03.1.19}) and (\ref{5.03.2.19})
\begin{equation}\label{5.12.3.19}
\|\bp v\|_{L_x^\infty}+\|\p U, V\|_{L_x^\infty}\les \| \bp \varrho, \p v_+\|_{L_x^\infty}+1.
\end{equation}
Hence by using (\ref{5.04.16.19}), (\ref{10.19.2.19}) and (\ref{5.12.3.19}), we derive
\begin{align*}
&\|\mu^\ep \E\rp{1}_\mu(t)^\f12\|_{l_\mu^2 L_x^2}\\
&\les \|\mu^\ep \E\rp{1}_\mu(0)^\f12\|_{l_\mu^2 L_x^2}+T^\f12\sup_{0\le t'\le t}\E\rp{\le 1}(t')^\f12 \\
& +\int_0^t(\|\p v,\p(c^2), \Tr k\|_{H_x^{1+\ep}}\|\p U, V\|_{L_x^\infty}+\|\p (\p F_U, F_V)\|_{H^\ep_x}) dt'\\
&\les\|\mu^\ep \E\rp{1}_\mu(0)^\f12\|_{l_\mu^2 L_x^2}+T^\f12 +\int_0^t(\|\mu^\ep\E\rp{1}_\mu(t')^\f12\|_{l_\mu^2}+\|\curl \Omega\|_{H^{\frac{1}{2}+\ep}_x}+1)(\|\bp \varrho, \p v_+\|_{L_x^\infty}+1)dt'.
\end{align*}
(\ref{5.11.3.19}) can be obtained by Gronwall's inequality and using (\ref{BA1}).
 \end{proof}

\section{$H^{2+\delta}$ energy estimates for vorticity}\label{eng_vor}
Let $0\le \delta\le s'-2$. The purpose of this section is to derive the bound of $\|\Omega\|_{H^{2+\delta}(\Sigma_t)}$ for all $0<t\le T$.   Neither $\p v$ nor $\Omega$ is  sufficiently smooth for  providing such bound by using the transport equation if following the standard product estimates under the bootstrap assumption (see  (\ref{10.16.1.19})).
  Applying the  $\curl$-operator to (\ref{4.25.4.19}) followed by  pairing the resulting equation with a $\curl$-structure leads to a series of crucial cancellations including the $\div$-$\curl$ structure in space and integration by part in spacetime,
   with the help of the Hodge system
\begin{equation}\label{hodge_1}
\div v=-\bT \varrho, \quad \curl v=\Omega e^{\varrho}
\end{equation}
and the transport equations (\ref{4.10.3.19}) and (\ref{4.25.4.19}).

Recall from (\ref{5.11.3.19}) that  we need to obtain the bound of $\|\curl \Omega(t)\|_{H^{s-\frac{3}{2}}_x}$ for all $0<t\le T$ to obtain the highest order energy estimates for the wave functions $(v_+, \varrho)$. Under our assumption on the initial vorticity,
we will bound $\|\C(t)\|_{H^1_x}$ norm to close the energy estimate in (\ref{5.11.3.19}). We then further obtain the bound for $\|\C(t)\|_{H^{1+\delta}_x}$ with the help of the bound on the highest order energy for $(v_+, \varrho)$ and (\ref{BA1}). For both estimates on $\C(t)$, we heavily rely on the
  the particular structure in the  $\curl$-equation of (\ref{4.25.4.19}).

We first give the precise formulae of $\bT\curl F$ for vector fields $F$.
\begin{proposition}
There hold for $\Sigma_t$-tangent vector fields $F$ that \begin{footnote}{We use the Euclidean metric $\delta_{ij}$ to lift and lower the indices of tensor fields.}\end{footnote}
\begin{align}\label{9.07.1.19}
\bT\curl F^m-\bT \varrho \curl F^m=\p_n v^m \curl F^n -\tensor{\ep}{^m^n^i} \p_n v_j \p_i F^j +\curl \bT F^m,
\end{align}
and identically,
\begin{align}\label{9.07.2.19}
\bT(e^{-\varrho} \curl F^m)= e^{-\varrho}\p_n v^m  \curl F^n-\tensor{\ep}{^m^n^i} \p_n v_j \p_i F^j e^{-\varrho}+\curl \bT F^m e^{-\varrho}.
\end{align}
Let $(\p v_j\wedge \p F^j)^m=\ep^{mni}\p_n v_j \p_i F^j$. There holds
\begin{align}
&\curl(e^{-\varrho}\p v_j\wedge \p F^j)_m\label{9.07.3.19}\\
&=e^{-\varrho}\big(\p^n \p_m v_j \p_n F^j+(\curl^2 v_j+\p_j \bT \varrho)\p_m F^j+\p_m v^j(-\curl^2 F^j+\p_j(\div F))\nn\\
&-\p_n v^j \p^n \p_m F_j-\tensor{\ep}{_m^n^i}\tensor{\ep}{_i^{ab}}\p_a v^j \p_b F_j \p_n \varrho\big).\nn
\end{align}
\end{proposition}
\begin{proof}
In view of (\ref{cmu1})
\begin{equation*}
[\bT, \p_j] F^i=-\p_j v^l \p_l F^i,
\end{equation*}
we can derive
\begin{align*}
\tensor{\ep}{^m^j_i}[\bT, \p_j]F^i&=\tensor{\ep}{^m_i^j}\p_j v^l \p_l F^i\\
&=\tensor{\ep}{^m^i^j}\p_j v^l(\p_i F_l+\tensor{\ep}{_l_i^n}\curl F_n)\\
&=\tensor{\ep}{^m^i^j}\p_j v^l\p_i F_l+\p_n v^m \curl F^n- \div v \curl F^m,
\end{align*}
where we calculated for $Z_n=\curl F_n$ that
\begin{align*}
\tensor{\ep}{^m^j_i} \tensor{\ep}{_l^n^i} \p_j v^l Z_n=(\delta_l^m \delta_j^n-\delta_l^j \delta^{mn})\p_j v^l Z_n=\p^n v^m Z_n-\div v Z^m
\end{align*}
to obtain the last line.
(\ref{9.07.1.19}) follows by applying (\ref{4.23.1.19}) to $\div v$.

(\ref{9.07.2.19}) is a consequence of multiplying (\ref{9.07.1.19}) by $e^{-\varrho}$.

Next we prove (\ref{9.07.3.19}). We first directly compute
$
\tensor{\ep}{^i^a^b}\ep_{mni}=\delta_m^a \delta_n^b-\delta_n^a \delta_m^b.
$
With its help, we derive
\begin{align*}
&\tensor{\ep}{_{mn}^i}\p^n(e^{-\varrho}\p v_j\wedge \p F^j)_i=\tensor{\ep}{_{mn}^i}\p^n(\tensor{\ep}{_i_{ab}} \p^a v^j \p^b F_j  e^{-\varrho})\\
 &=e^{-\varrho}\{(\delta_m^a \delta_n^b-\delta_n^a \delta_m^b)\p^n(\p_a v^j \p_b F_j) -\tensor{\ep}{_{mn}^i}\tensor{\ep}{_i_{ab}}\p^n \varrho\p^a v^j \p^b F_j \} \\
 &=e^{-\varrho}\{(\delta_m^a \delta_n^b-\delta_n^a \delta_m^b)\big(\p^n \p_a v^j \p_b F_j +\p_a v^j \p^n \p_b F_j\big)-\tensor{\ep}{_{mn}^i}\tensor{\ep}{_i^{ab}}\p^n \varrho\p_a v^j \p_b F_j \} \\
 &=e^{-\varrho}\big(\p^n \p_m v^j \p_n F_j-\Delta_e v^j\p_m F_j+\p_m v^j \Delta_e F_j-\p_n v^j\p^n \p_m F_j- \tensor{\ep}{_{mn}^i}\tensor{\ep}{_i^{ab}}\p^n \varrho\p_a v^j \p_b F_j\big).
\end{align*}
Noting for a vector-valued function $G$ there holds
\begin{equation*}
\Delta_e (G^j)=-\curl^2 (G^j)+\p_j(\div G)
\end{equation*}
and using the first equation in (\ref{4.23.1.19}) when $G=v$, we can obtain
\begin{align*}
&\ep_{mni}\p^n(e^{-\varrho}\p v_j\wedge \p F^j)^i\\
&=e^{-\varrho}\big(\p^n\p_m v^j \p_n F_j+(\curl^2 v_j+\p_j \bT\varrho)\p_m F^j+\p_m v^j(-\curl^2 F_j +\p_j(\div F))-\p_n v^j \p^n \p_m F_j\big)\\
&\,- e^{-\varrho}\tensor{\ep}{_{mn}^i}\tensor{\ep}{_i^{ab}}\p^n \varrho\p_a v^j \p_b F_j.
\end{align*}
This is (\ref{9.07.3.19}).
\end{proof}
\begin{proof}[Proof of (\ref{4.25.4.19})]
(\ref{4.25.4.19}) is a direct consequence of applying (\ref{9.07.2.19}) to $F=\Omega$. Indeed,
\begin{align*}
\bT(e^{-\varrho} \curl \Omega^m)= \p_n v^m  (e^{-\varrho}\curl \Omega^n)-\tensor{\ep}{^m^n^i} \p_n v_j \p_i \Omega^j e^{-\varrho}+\curl \bT \Omega^m e^{-\varrho}.
\end{align*}
By using (\ref{4.10.3.19}), (\ref{div}) and $\Omega=\fw e^{-\varrho}$, we calculate
\begin{align*}
\curl(\bT\Omega^m)&=\ep^{mij} \p_i (\bT \Omega_j)=\ep^{mij} \p_i (\Omega^a \p_a v_j)\\
&=\ep^{mij}\p_i\Omega^a \p_a v_j+\p_a \fw^m \Omega^a=\ep^{mij} \p_i \Omega^a(\p_j v_a+\tensor{\ep}{_{aj}^l}\fw_l)+\p_a \fw^m \Omega^a \\
&=\ep^{mij} \p_i \Omega^a \p_a v_j-(\delta_a^m\delta^{il}-\delta^{ml} \delta_a^i) \fw_l \p_i\Omega^a+\p_a \fw^m \Omega^a\\
&=\ep^{mij} \p_i \Omega^a \p_a v_j-(\p_i \Omega^m \fw^i -\fw^m \p_a \Omega^a)+\p_a \fw^m \Omega^a\\
&=\ep^{mij} \p_i \Omega^a \p_j v_a-(\p_i \fw^m e^{-\varrho} \fw^i+\fw^m \p_i (e^{-\varrho})\fw^i +\fw^m \Omega^a \p_a \varrho)+\p_a \fw^m \Omega^a\\\
&=\ep^{mij} \p_i \Omega^a \p_j v_a.
\end{align*}
Combining the above two calculations gives (\ref{4.25.4.19}).
\end{proof}
Recall that for any vector fields  $Z$ tangent to $\Sigma_t$ and any scalar function $f$ there holds
$\int_{\Sigma_t} \Lie_Z (f d\mu_e) =\int_{\Sigma_t} \mbox{div}(f Z) d\mu_e=0$.
To  derive the $H^2$ energy of vorticity,  we derive the energy formula  for $\Sigma_t$-tangent vector fields $F, G$
\begin{align}\label{9.07.4.19}
\begin{split}
\p_t \int_{\Sigma_t} \l F, G\r_e d\mu_e&=\int_{\Sigma_t} (\l F, \bT G\r_e+\l G, \bT F\r_e)  d\mu_e+ \int_{\Sigma_t} \l F, G\r (\p_t+\Lie_v) d\mu_e\\
&=\int_{\Sigma_t} (\l F, \bT G\r_e+\l G, \bT F\r_e)  d\mu_e+ \int_{\Sigma_t}-\l F, G\r \Tr\cir{k} d\mu_e,
\end{split}
\end{align}
where we used $\Lie_\bT d\mu_e=-\delta^{ij}\cir{k}_{ij} d\mu_e=-\Tr \cir{k}d\mu_e$, and $\l\cdot,\cdot\r_e$ means the contraction by using the Euclidean metric.  
\subsection{$H^2$ estimate for vorticity}\label{h2vorticity}
We first derive the energy estimate for vorticity in $H^2_x$, which is sufficient to complete the energy estimate for the wave function $(v_+,\varrho)$ given in (\ref{5.11.3.19}).
\begin{proposition}[$H^2$ bound of vorticity]
There hold for $0< t\le T$ that
\begin{align}
&\|\curl \C\|_{L^2(\Sigma_t)}+\|\p \C\|_{L^2(\Sigma_t)}+\|\p^2 \Omega\|_{L^2(\Sigma_t)}+\|\p^2 \fw\|_{L^2(\Sigma_t)}\les 1\label{9.07.5.19}.
\end{align}
\end{proposition}
\begin{proof}
We first reduce the proof of (\ref{9.07.5.19}) to showing the first estimate therein.

  By $\C^i=e^{-\varrho}\curl \Omega^i$, we derive
\begin{equation*}
\div \C=\p_i (e^{-\varrho})\curl \Omega^i.
\end{equation*}
 By the elliptic estimate for hodge system, Sobolev embedding on ${\mathbb R}^3$, (\ref{5.04.17.19}) and (\ref{5.03.8.19}), we can obtain  that
\begin{align*}
\|\p\C\|_{L^2(\Sigma_t)}&\les \|\p(e^{-\varrho})\|_{L^6_x}\|\curl \Omega\|_{L^3(\Sigma_t)}+\|\curl \C\|_{L^2(\Sigma_t)}\\
    &\les \|\p \varrho\|_{L_x^6}\|+\|\curl \C\|_{L^2(\Sigma_t)}\les \|\curl \C\|_{L^2(\Sigma_t)}+1.
\end{align*}
Moreover, by the elliptic estimate and the formula
\begin{equation*}
\Delta_e \Omega^j= -\curl^2 \Omega^j +\p^j \div \Omega,
\end{equation*}
we derive by using (\ref{5.04.17.19}), (\ref{9.22.3.19}), (\ref{5.03.4.19}), (\ref{div}) and the  Sobolev embedding that
\begin{equation}\label{9.07.9.19}
\begin{split}
\|\p^2 \Omega\|_{L^2(\Sigma_t)}&\les \|\curl(e^\varrho \C^i)\|_{L^2(\Sigma_t)}+\|\p (\Omega \p \varrho)\|_{L^2(\Sigma_t)}+\|\p \varrho\c \curl\Omega\|_{L^2(\Sigma_t)}\\
&\les \|\curl \C\|_{L^2(\Sigma_t)}+\|\p \varrho\|_{L^6(\Sigma_t)}\|\C, \p \Omega\|_{L^3(\Sigma_t)}+\|\Omega\|_{L_x^\infty}\|\p^2 \varrho\|_{L^2(\Sigma_t)}\\
&\les \|\curl \C\|_{L^2(\Sigma_t)}+1.
\end{split}
\end{equation}

Next, we recall that $\fw =\Omega e^{\varrho}$, which gives
\begin{equation}\label{9.24.1.19}
\begin{split}
\p^2 \fw&=\p^2(\Omega e^\varrho)=e^\varrho(\p^2 \Omega+\p \Omega \p \varrho)+\p(e^\varrho \p \varrho) \Omega\\
&=e^\varrho\{\p^2 \Omega+\p \Omega \p \varrho+\Omega\big((\p\varrho)^2+\p^2 \varrho\big)\}.
\end{split}
\end{equation}
Note that by using  (\ref{9.22.3.19}) and (\ref{5.04.17.19}) we can derive by using Sobolev embedding
\begin{equation}\label{9.23.4.19}
\|\p \Omega \c(\bp \varrho, \bp v)\|_{L_x^2}+\|\Omega (\bp \varrho)^2\|_{L_x^2}\les \|\p \Omega\|_{L_x^3}\|\bp \varrho, \bp v\|_{L_x^6}+\|\Omega\|_{H^1_x}\|\bp \varrho\|_{H^1_x}^2\les 1.
\end{equation}
Hence taking $L^2_x$ norm  of the expression for  $\p^2 \fw$  yields
\begin{align*}
\|\p^2 \fw\|_{L^2_x}&\les \|\p^2 \Omega\|_{L_x^2}+\|\p \Omega \p \varrho\|_{L_x^2}+\|\Omega (\p \varrho)^2\|_{L_x^2}+\|\Omega \p^2 \varrho\|_{L_x^2}\\
&\les \|\p^2 \Omega\|_{L_x^2}+\|\Omega\|_{L_x^\infty}\|\p^2 \varrho\|_{L_x^2}+1\les \|\p^2 \Omega\|_{L_x^2}+1,
\end{align*}
where we used (\ref{5.04.17.19}), (\ref{9.22.3.19})  and (\ref{5.03.4.19}).  Combined with (\ref{9.07.9.19}), we have
\begin{equation*}
\|\p^2 \fw\|_{L_x^2}\les \|\curl\C\|_{L_x^2}+1,
\end{equation*}
as desired. Therefore
\begin{equation}\label{9.07.7.19}
\|\p \fC\|_{L^2(\Sigma_t)}+\|\p^2 \Omega\|_{L^2(\Sigma_t)}+\|\p^2 \fw\|_{L^2(\Sigma_t)}\les \|\curl \fC\|_{L^2(\Sigma_t)}+1.
\end{equation}
 To prove (\ref{9.07.5.19}), it suffices to consider the first estimate.

Now we apply (\ref{9.07.4.19}) to $F=G=\curl \C$. By using Gronwall's inequality, $\Tr \cir{k}=-\div v=\bT\varrho$ and (\ref{BA1}),
\begin{align}\label{9.07.14.19}
\int_{\Sigma_t}|\curl \C|^2 d\mu_e\les\|\curl \C\|_{L^2(\Sigma_0)}^2+|\int_0^t \int_{\Sigma_{t'}}\l \curl \C, 2\bT \curl \C-\bT \varrho \curl \C\r_e d\mu_e dt'|.
\end{align}
To treat the last term on the right hand side, which will be denoted as $I$, we apply (\ref{9.07.1.19}) to $F=\C$, followed with using (\ref{9.07.7.19}) and the first equation in (\ref{4.23.1.19}) to derive
\begin{align*}
I&\les |\int_0^t\int_{\Sigma_{t'}}\l \curl \C, \curl \bT \C\r_e d\mu_e dt'|+\int_0^t\|\p v\|_{L_x^\infty}\|\p \C\|_{L^2_x}\|\curl \C\|_{L^2_x}dt'\\
&\les |\int_0^t \int_{\Sigma_{t'}}\l \curl \C, \curl \bT \C\r_e d\mu_e dt'|+\int_0^t\|\curl \C\|_{L^2_x}(1+\|\curl \C\|_{L^2_x})\|\p v\|_{L^\infty_x} dt'.
\end{align*}
 The second term on the right will be treated by Gronwall's inequality and using (\ref{BA1}) when integrating in $t$.
 We focus on the first term on the right hand side, denoted by $|I^1|$ with
 \begin{equation}\label{9.11.7.19}
 I^1=\int_0^t\int_{\Sigma_{t'}}\l \curl \C, \curl \bT \C\r_e d\mu_e dt'.
 \end{equation}
 By using (\ref{4.25.4.19}), we compute
  \begin{align}\label{9.07.8.19}
  \curl \bT \C_m&=-2\curl(e^{-\varrho}\p v_j\wedge \p \Omega^j )_m+\curl(\p_a v \C^a)_m.
  \end{align}
  Note that
  \begin{align}\label{9.08.12.19}
  \begin{split}
  \curl(\p_a v \C^a)_m&=\ep_{mni} \p^n(\p_a v^i \curl \Omega^a e^{-\varrho})\\
  &=\ep_{mn i} \p^n \p_a v^i \C^a+\ep_{mn i}\p_a v^i \p^n \C^a\\
  &=\p_a \fw_m \C^a+\ep_{mn i} \p_a v^i \p^n \C^a.
  \end{split}
  \end{align}
  Thus by using (\ref{9.07.7.19}) and (\ref{9.22.3.19}), we can directly bound
  \begin{align*}
  \|\curl(\p_a v \C^a)\|_{L^2(\Sigma_t)}&\les \|\p \fw\|_{L^3_x} \|\C\|_{L^6_x}+\|\p_a v\|_{L^\infty_x}\|\p \C\|_{L^2_x}\\
  &\les \|\C\|_{H^1_x}(\|\p \fw \|_{L_x^3}+\|\p v\|_{L_x^\infty})\les( \|\curl \C\|_{L_x^2}+1)(\|\p v\|_{L_x^\infty}+1).
  \end{align*}
  Hence the second term in (\ref{9.07.8.19}) has been treated.

  It only remains to bound $\|\cdot \|_{L_x^2}$ norm of the first term on the right hand side of (\ref{9.07.8.19}). For this purpose, we apply (\ref{9.07.3.19}) to $F=\Omega$ and  observe that the first term on the right takes the form of $\p^2 v\c \p_j \Omega$, which is only expected to be in $L^\frac{3}{2}_x$ instead of the favourable $L^2_x$.  To solve this difficulty, we derive carefully the integral below,
  \begin{align}\label{9.08.14.19}
  \begin{split}
&I^1_+= \int_0^t \int_{\Sigma_t'}\curl(e^{-\varrho}\p v_j\wedge \p \Omega^j)_m \curl \C^m d\mu_e dt'\\
&=\int_0^t\int_{\Sigma_t'}\{e^{-\varrho}(\p^n \p_m v^j \p_n \Omega_j+(\p_j \bT \varrho+\curl^2 v_j)\p_m\Omega^j+\p_m v^j(-\curl^2 \Omega_j+\p_j(\div \Omega))\\
&-\p_n v^j \p^n \p_m \Omega_j-\ep_{mni}\ep^{iab}\p_a v^j \p_b \Omega_j \p^n \varrho)\}\curl \C^m d\mu_e dt'.
\end{split}
  \end{align}
  The above integral can be decomposed into two parts. For the following part, denoted by $I^1_{+,1}$, we will employ integration by parts in ${\mathbb R}^3$ and in spacetime,
  \begin{align}\label{9.08.15.19}
I^1_{+,1} &=\int_0^t \int_{\Sigma_{t'}}e^{-\varrho}(\p^n \p_m v^j \p_n \Omega_j+\p_j \bT \varrho \p_m\Omega^j )\curl \C^m d\mu_e dt',
  \end{align}
while the remaining terms are denoted by $I^1_{+,2}$. 
 Using the second equation in (\ref{hodge_2}), (\ref{5.03.4.19}) and (\ref{9.23.4.19}), we carry out the direct estimate below,
\begin{align*}
|I^1_{+,2}|&\les \int_0^t \int_{\Sigma_{t'}}\big( |\curl \fw| |\p \Omega|+|\p v|(|\p^2 \Omega|+|\p\Omega\c \p \varrho|)\big) |\curl \C| d\mu_e dt'\\
&\les \int_0^t\|\curl \C\|_{L_x^2}\{\||\p \Omega|(|\p \Omega|+|\Omega \p \varrho|)\|_{L_x^2}+ \|\p v\|_{L_x^\infty}(\|\p^2 \Omega\|_{L_x^2}+\|\p\Omega \c \p \varrho\|_{L_x^2})\}\\
&\les \int_0^t \|\curl \C\|_{L_x^2}\big(\||\p \Omega|^2\|_{L_x^2}+1+\|\p v\|_{L_x^\infty}(\|\p^2 \Omega\|_{L_x^2}+1)\big) dt'.
\end{align*}
By using Sobolev embedding, (\ref{9.22.3.19}) and (\ref{9.07.9.19})
$$\||\p \Omega|^2\|_{L_x^2}\les \|\p \Omega\|_{H_x^1}\|\p \Omega\|_{L_x^3}\les \|\curl \C\|_{L_x^2}+1.$$
Substituting the above estimate in the estimate for $|I^1_{+,2}|$,
also using (\ref{9.07.9.19}) again, we can obtain
\begin{align*}
|I^1_{+,2}|&\les \int_0^t \|\curl \C\|_{L_x^2}(\|\curl \C\|_{L_x^2}+1)(1+\|\p v\|_{L_x^\infty}) dt',
\end{align*}
which can be further treated by Gronwall's inequality.

For the first term in (\ref{9.08.15.19}), we observe the $\div$-$\curl$ structure and perform integration by parts on $\Sigma_t$,
\begin{align}
I^1_{+,1,1}&=\int_0^t \int_{\Sigma_{t'}}e^{-\varrho}\p^n \p_m v^j \p_n\Omega_j \curl \C^m d\mu_e dt'\nn\\
&=\int_0^t \int_{\Sigma_{t'}}\{ \p_m(e^{-\varrho} \p^n v^j \p_n \Omega_j \curl \C^m)-\p_m(e^{-\varrho} \p_n\Omega_j \curl \C^m)\p^n v^j \} d\mu_e dt'.\label{9.30.3.19}
\end{align}
The first term is a boundary term, denoted by $I^1_{+,1,1,b}$ which vanishes identically.  Due to  $\p_m (\curl \C)^m=0$, we obtain
\begin{align*}
I^1_{+,1, 1}&=-\int_0^t \int_{\Sigma_{t'}}\p_m(e^{-\varrho} \p^n \Omega_j\curl \C^m)\p_n v^j  d\mu_e dt'\\
&=-\int_0^t \int_{\Sigma_{t'}} \big(\p_m(e^{-\varrho}) \p_n \Omega_j  +e^{-\varrho} \p_m \p_n \Omega_j\big)\curl\C^m  \p^n v^j  d\mu_e dt'.
\end{align*}
 We then  use  (\ref{9.07.9.19}) and the first estimate in (\ref{9.23.4.19})   to derive
 \begin{align*}
 |I^1_{+, 1,1}|&\les\int_0^t (1+\|\p^2 \Omega\|_{L_x^2})\|\curl \C\|_{L_x^2}\|\p v\|_{L_x^\infty} dt'\\
 &\les \int_0^t (1+\|\curl\C\|_{L_x^2})\|\curl \C\|_{L_x^2}\|\p v\|_{L_x^\infty} dt'.
 \end{align*}
   The above term can be then treated by Gronwall's inequality.

 Next, we consider the second term of $I^1_{+,1}$ in (\ref{9.08.15.19}), which is denoted by $I^1_{+,1,2}$ below.  In view of (\ref{9.07.4.19}), we integrate by parts in spacetime to obtain
  \begin{align}\label{9.08.19.19}
  \begin{split}
 I^1_{+,1,2}&= \int_0^t \int_{\Sigma_{t'}} \p_j \bT \varrho \p_m \Omega^j e^{-\varrho} \curl \C^m d\mu_e dt'\\
 &=\int_0^t \int_{\Sigma_{t'}}  ([\p_j, \bT]\varrho +\bT \p_j\varrho ) \p_m \Omega^j e^{-\varrho} \curl \C^m   d\mu_e dt'\\
 &=\int_0^t \int_{\Sigma_{t'}}(\p_j v^a \p_a \varrho+\Tr \cir{k}\p_j \varrho) \p_m \Omega^j e^{-\varrho} \curl \C^m d\mu_e dt'\\
 &+\int_0^t \p_t \int_{\Sigma_{t'}}\p_j \varrho \p_m\Omega^j e^{-\varrho} \curl \C^m d\mu_e dt'-\int_0^t \int_{\Sigma_{t'}} \p_j\varrho \bT( \p_m \Omega^j e^{-\varrho} \curl \C^m) d\mu_e dt',
 \end{split}
 \end{align}
 where we used (\ref{cmu1}) to treat the commutator.

By using (\ref{9.23.4.19}) and (\ref{BA1}), we control the first term on the right hand side
 \begin{align*}
|I^1_{+,1,2,1}|&= |\int_0^t \int_{\Sigma_{t'}}(\p_j v^a \p_a \varrho+\Tr \cir{k}\p_j \varrho) \p_m \Omega^j e^{-\varrho} \curl \C^m d\mu_e dt'|\\
&\les\int_0^t \|\p v\|_{L_x^\infty}\|\p \varrho\c \p \Omega\|_{L^2_x}\|\curl \C\|_{L_x^2} dt' \\
&\les \int_0^t \|\p v\|_{L_x^\infty}\|\curl\C\|_{L_x^2} dt'\les  \sup_{0\le t'\le t}\|\curl \C\|_{L^2(\Sigma_{t'})}.
 \end{align*}
Similarly, we bound the  boundary term by
\begin{align*}
|I^1_{+,1,2,2}|&\les \sup_{0\le t'\le t}\{\|\p \varrho\c\p \Omega\|_{L_x^2} \|\curl \C\|_{L_x^2}\}\les \sup_{0\le t'\le t}\|\curl \C\|_{L^2(\Sigma_{t'})}.
\end{align*}
For the  last term in $I^1_{+,1,2}$, we first show the following preliminary estimates
\begin{align}
 \|\bT (\p \Omega\c (e^{-\varrho}+1))\|_{L^2_x} \les 1\label{9.07.10.19}.
 \end{align}
Indeed, 
we  compute by using  (\ref{4.10.3.19}), (\ref{4.23.1.19}) and (\ref{cmu1}) that
\begin{align}\label{9.08.22.19}
\bT \p \Omega&= \p \Omega \p v+\Omega \p^2 v,\quad \bT (\p \Omega e^{-\varrho})=e^{-\varrho}(\p \Omega \p v+\Omega \p^2 v).
\end{align}
Therefore, it follows by using (\ref{9.23.4.19}), (\ref{5.03.4.19}) and  (\ref{5.04.17.19}) that
\begin{align*}
\|\bT \p \Omega (e^{-\varrho}+1)\|_{L_x^2}&\les \|\p \Omega \c\p v\|_{L_x^2}+\|\Omega\|_{L_x^\infty} \|\p^2 v\|_{L_x^2}\les 1.
\end{align*}
This gives (\ref{9.07.10.19}).

In view of  (\ref{4.25.4.19}), we can symbolically write
\begin{equation}\label{9.07.13.19}
\bT \C=e^{-\varrho}\p v\c \p\Omega; \qquad \curl \bT \C=\p( e^{-\varrho}\p v\c\p\Omega).
\end{equation}
Using the above symbolic formulas, (\ref{cmu1}) and $\fC= e^{-\varrho} \curl \Omega$, we derive
\begin{align*}
I^1_{+,1,2,3}&=\int_0^t\int_{\Sigma_{t'}} \{\p_j\varrho \p_m \Omega^j e^{-\varrho} (\curl \bT \C^m+[\bT, \curl]\C^m)+\p_j \varrho\bT( \p_m \Omega^j e^{-\varrho})\curl \C^m\} d\mu_e dt'\\
&=\int_0^t \int_{\Sigma_{t'}} \p \varrho \{\p\Omega e^{-\varrho} \p(e^{-\varrho}\p v\c \p\Omega) + \p v\p \C\p \Omega e^{-\varrho}+\bT(\p \Omega e^{-\varrho}) \curl \C\}d\mu_e dt'\\
&=\int_0^t \int_{\Sigma_{t'}}\{ \p \varrho \p\Omega e^{-2\varrho}(\p^2 v \p \Omega+\p v\p \varrho\p \Omega+\p v\p^2 \Omega)+\p\varrho \bT(\p \Omega e^{-\varrho}) \curl \C\} d\mu_e dt'.
\end{align*}
By using (\ref{9.07.10.19}), we can estimate the last term in the last line,
\begin{equation*}
\|\p\varrho \bT(\p \Omega e^{-\varrho}) \curl \C\|_{L_x^1}\les \|\p \varrho\|_{L_x^\infty}\|\curl \C\|_{L_x^2}.
\end{equation*}
For the remaining terms in the last line, we employ Sobolev embedding, (\ref{5.04.17.19}) and (\ref{9.22.3.19}) to derive
\begin{align*}
&\|\p_j \varrho \p\Omega e^{-2\varrho}(\p^2 v \p \Omega+\p v\p \varrho\p \Omega+\p v\p^2 \Omega)\|_{L_x^1}\\
&\les \|\p \varrho\|_{L_x^6}\big(\|\p^2 v\|_{L_x^2}\|\p \Omega\|^2_{L_x^6}+\|\p\Omega\|_{L_x^6}\|\p \Omega\|_{L_x^3}\|\p v\|_{L_x^6}\|\p \varrho\|_{L_x^6}+\|\p^2 \Omega\|_{L_x^2}\|\p \Omega\|_{L_x^6}\|\p v\|_{L_x^6}\big)\\
&\les (\|\curl \C\|_{L_x^2}+1)(\|\curl \C\|_{L_x^2}+1),
\end{align*}
where we used Sobolev embedding and (\ref{9.07.9.19}) to treat $\|\p \Omega\|_{L_x^6}+\|\p^2 \Omega\|_{L_x^2}$ in the last step.

Combining the estimates for the two parts gives
\begin{align*}
|I^1_{+,1,2,3}|&\les \int_0^t (\|\curl \C\|_{L_x^2}+1)(\|\p \varrho\|_{L_x^\infty}+\|\curl \C\|_{L_x^2}+1) dt',
\end{align*}
which then can be treated by Gronwall's inequality.

We now summarize the above calculations as
\begin{align}\label{9.11.8.19}
|I|+|I^1|&\les \int_0^t  (\|\curl\C\|_{L^2(\Sigma_t)}+1)^2 (\|\p v, \p \varrho\|_{L_x^\infty}+1) dt'+\sup_{0\le t'\le t} \| \curl \C\|_{L^2(\Sigma_t')}.
\end{align}
Substituting the above estimate in (\ref{9.07.14.19}) yields
\begin{align*}
\int_{\Sigma_t}|\curl \C|^2 d\mu_e&\les\|\curl \C\|_{L^2(\Sigma_0)}^2+\int_0^t (\|\curl \C\|_{L_x^2}+1)^2(\|\p v, \p \varrho\|_{L_x^\infty}+1) dt'\\
&+\sup_{0\le t'\le t}\|\curl \C\|_{L^2(\Sigma_{t'})},
\end{align*}
which implies the first estimate in (\ref{9.07.5.19}) by using Gronwall's inequality.
\end{proof}
\begin{corollary}\label{eng_wave}
For $0\le \ep\le s-2$, there hold
\begin{align}
&\|\Omega\|_{H_x^2}+\|\curl\Omega\|_{H_x^1}\les 1\label{4.25.1.19}\\
&\|\bp v, \bp\varrho, \Tr k, \p(c^2) \|_{H^{1+\ep}_x}\les 1, \quad \| \mu^\ep\E\rp{1}_\mu(t)^\f12\|_{l_\mu^2}\les 1\label{eng_3}\\
&\|\bT^2 v, \bT^2 \varrho\|_{H^{\ep}_x}\les 1.\label{eng_2}
\end{align}
\end{corollary}
\begin{proof}
(\ref{4.25.1.19}) is a consequence of (\ref{9.07.5.19}) and (\ref{9.22.3.19}).
The second estimate in (\ref{eng_3}) is derived by substituting (\ref{9.07.5.19}) and (\ref{BA1}) to (\ref{5.11.3.19}). The $\dot{H}^{1+\ep}$ bound in the first  estimate in (\ref{eng_3}) follows as the  consequence of the second estimate with the help of Lemma \ref{comp_2} and  (\ref{4.25.1.19}), while the lower order bound can be obtained by (\ref{5.04.17.19}).

To show (\ref{eng_2}), we first derive by using (\ref{4.23.1.19}) and (\ref{cmu1}) that
   \begin{align*}
   \bT^2 (v)&= \bT(-c^2 \p \varrho)=c c' \bT \varrho \p \varrho+c^2 \p v \p \varrho+c^2 \p \bT \varrho\\
   \bT^2 \varrho&=\bT (\p v)=\p \bT v+(\p v)^2.
    \end{align*}
    Hence, symbolically, we can write
    \begin{equation*}
    \bT^2 v, \bT^2 \varrho=C(\varrho) (\bp (v, \varrho))^2+C(\varrho)\p \bT (\varrho, v).
    \end{equation*}
    By using (\ref{9.22.4.19}), and (\ref{9.08.8.19})
    \begin{align*}
    \|\bT^2 v, \bT^2 \varrho\|_{H^\ep_x}&\les \|(\bp (v, \varrho))^2\|_{H^\ep_x}+\|\bp\bT(\varrho, v)\|_{H^\ep_x}\\
    &\les \|\bp(v, \varrho)\|_{H^{\ep+\f12}_x}\|\bp(v, \varrho)\|_{H^1_x}+\|\p\bT(\varrho, v)\|_{H^\ep_x}.
    \end{align*}
    Therefore $\|\bT^2 v, \bT^2\varrho\|_{H^\ep_x}\les 1$ follows by using (\ref{eng_3}).
\end{proof}

Let  $0<\a<1$ be fixed. Define for functions $f$ the following Besov norm
\begin{equation*}
\|f\|_{B_{\infty,2,x}^\a}=\|f\|_{L_x^\infty}+\|\mu^\a P_\mu f\|_{l_\mu^2 L_x^\infty}.
\end{equation*}
\begin{corollary}\label{comp_3}
Let $0<\delta\le s'-2$. There hold the following estimates
\begin{align*}
&\|\bp v\|_{B_{\infty,2,x}^\delta}\les \|\bp  v_+\|_{B_{\infty, 2,x}^\delta}+1\\
&\|\bp v\|_{L_x^\infty}\les \|\bp v_+\|_{L_x^\infty}+1
\end{align*}
\end{corollary}
\begin{proof}
By using $v=v_++\eta$, Bernstein inequality, the estimate for $\eta$ in (\ref{4.12.2.19}) and  (\ref{9.07.5.19})
\begin{align*}
\|\p v\|_{B_{\infty,2,x}^\delta}&\les \|\p v_+\|_{B_{\infty, 2,x}^\delta}+\|\curl \Omega\|_{H^{\f12+\delta}_x}+1\\
&\les \|\p v_+\|_{B_{\infty, 2,x}^\delta}+1.
\end{align*}
Similarly, by using (\ref{4.12.5.19}),
\begin{align*}
\|\bT v\|_{B_{\infty,2,x}^\delta}&\les \|\bT \eta\|_{B_{\infty,2,x}^\delta}+\|\bT v_+\|_{B_{\infty,2,x}^\delta}\\
&\les \|\bT \eta\|_{H^{\frac{3}{2}+\delta}}+\|\bT v_+\|_{B_{\infty,2,x}^\delta}\\
&\les 1+\|\bT v_+\|_{B_{\infty,2,x}^\delta}.
\end{align*}
The $L_x^\infty$ estimate can be derived in the same way.
\end{proof}
\subsection{$H^{2+\delta}$ bound for vorticity}
Let $0<\delta\le s'-2$ be fixed.
In this subsection, we derive the highest order energy bound for the vorticity. Such bound for the vorticity, together with the flux on $\curl \C$ which will be bounded in Section \ref{flux_C},  is used particularly for controlling the geometry of the acoustic cones.
\begin{proposition}[Highest-order energy control for vorticity]\label{9.07.16.19}
Let $0<\delta\le s'-2$. There hold the following estimates
\begin{align}
&\|\curl \C\|_{H^\delta(\Sigma_t)}+\|\p \C\|_{H^\delta(\Sigma_t)}+\|\p^2 \fw\|_{H^\delta(\Sigma_t)}+\| \p^2 \Omega\|_{H^\delta(\Sigma_t)}\les 1.\label{9.07.15.19}
\end{align}
\end{proposition}
Since the $L^2_x$ estimates have been proved in (\ref{9.07.5.19}), we will focus on proving the highest order estimates in the above result.
We first prove two sets of preliminary estimates.
\begin{lemma}
Let $0<\delta\le s'-2$ and $C(y)$ be any smooth function. There hold
\begin{align}
\|\mu^\delta[\bT, P_\mu \curl] F^m\|_{l_\mu^2 L_x^2}&\les \|\p v\|_{B_{\infty, 2, x}^\delta}\|\p F\|_{L_x^2}+\|\p v\|_{L_x^\infty}\|\p F\|_{H^\delta_x},\label{9.08.2.19}\\
\| C(\varrho) F \|_{H^{1+\delta}_x}\les \|F\|_{H^{1+\delta}_x}.\label{9.08.21.19}
\end{align}
\end{lemma}
\begin{proof}
In view of  (\ref{9.07.1.19}) and using (\ref{4.23.1.19}) to treat $\bT \varrho$ therein, symbolically, we can write
\begin{align}\label{9.08.5.19}
[\bT, P_\mu \curl] F^m=[\bT, P_\mu]\curl F^m+P_\mu(\p v\c \p F).
\end{align}
Since $$[P_\mu, \bT] \curl F^m=[P_\mu, v] \p \curl F^m,$$
by  applying (\ref{4.13.3.19}) to $F=v$ and $G=\curl F$, we derive
\begin{equation}\label{9.08.4.19}
\|\mu^\delta[P_\mu, \bT]\curl F\|_{l_\mu^2 L_x^2}\les \|\p v\|_{L_x^\infty}\|\curl F\|_{H^\delta_x}.
\end{equation}
For the other term in (\ref{9.08.5.19}), we apply (\ref{9.08.1.19}) with $F=\p v$ and $G=\p F$ to have
\begin{align*}
\|\p v\c \p F\|_{H^\delta_x}\les\|\p v\|_{B_{\infty, 2, x}^\delta}\|\p F\|_{L_x^2}+\|\p v\|_{L_x^\infty}\|\p F\|_{H^\delta_x}.
\end{align*}
We combine the above estimate with (\ref{9.08.4.19}) to derive (\ref{9.08.2.19}).

Next we prove (\ref{9.08.21.19}). It suffices to consider the highest order estimate. We apply (\ref{9.22.4.19}) to $\a=\delta$, $C(\varrho)$ and  $f=\p F$;  and apply (\ref{9.08.8.19}) to $\a=\delta$, $F$ and $G=\p (C(\varrho))$ to derive
\begin{align*}
\|\La^\delta\p(F C(\varrho))\|_{L_x^2}&\les \|\La^\delta(\p F C(\varrho))\|_{L_x^2}+\|\La^\delta( F \p( C(\varrho)))\|_{L_x^2}\\
  &\les \|\p F\|_{H^\delta_x}+\|F\|_{H^{\f12+\delta}_x}\|\p (C(\varrho))\|_{H^1_x}+\|F\|_{H^1_x}\|\p( C(\varrho))\|_{H^{\f12+\delta}_x}\\
  &\les \|\p F\|_{H^\delta_x}+\|F\|_{H^1_x},
\end{align*}
where we used $\|\p \varrho, \p \big(C(\varrho)\big)\|_{H^1_x}\les 1$ in (\ref{5.04.17.19}). Thus (\ref{9.08.21.19}) is proved.
\end{proof}
Next we provide a comparison result.
\begin{lemma}
Let $0<\delta\le s'-2$. There hold the following estimates,
\begin{align}\label{9.08.6.19}
\|\p \C\|_{H^\delta_x}+\|\p^2 \Omega\|_{H^\delta_x}+\|\p^2 \fw\|_{H^\delta_x}\les \|\curl \C\|_{H^\delta_x}+1.
\end{align}
\end{lemma}
\begin{proof}
We first show the following preliminary estimates
\begin{align}
&\|(e^{\pm\varrho}+1)\p \varrho\c (\C+\p \Omega)\|_{H^\delta_x}\les 1,\label{9.08.10.19}\\
&\|\Omega\c \p^2 \varrho\|_{H^\delta_x}+\|\p(\Omega \c \p \varrho)\|_{H^\delta_x}\les 1, \label{9.09.1.19}\\
&\|(e^{\pm\varrho}+1) (\p \varrho)^2\Omega\|_{H^\delta_x}\les 1.\label{9.24.2.19}
\end{align}
Indeed,
applying (\ref{9.08.8.19}) to $F=\p \varrho$ and $G=\C+\p \Omega$, and using the result $\|\p \varrho, \C, \p \Omega\|_{H^1_x}\les 1$ which is from (\ref{5.04.17.19}) and (\ref{9.07.5.19}), we have
\begin{align*}
\|(e^{\pm\varrho}+1)\p \varrho\c (\C+\p \Omega)\|_{H^\delta_x}&\les\|\p \varrho\c (\C+\p \Omega)\|_{H^\delta_x}\\
&\les \|\p \varrho\|_{H^1_x}(\|\C\|_{H^1_x}+1)\les 1,
\end{align*}
where for the first inequality we applied Lemma \ref{comp_dyd}. This gives (\ref{9.08.10.19}).

By using (\ref{9.08.1.19}), (\ref{4.25.1.19}), (\ref{5.03.4.19}) and (\ref{eng_3}),
\begin{align*}
\|\Omega \c \p^2 \varrho\|_{H^\delta_x}&\les \|\Omega\|_{B_{\infty,2,x}^\delta}\|\p^2 \varrho\|_{L_x^2}+\| \Omega\|_{L^\infty_x}\|\p^2 \varrho\|_{H^\delta_x}\\
&\les \|\Omega\|_{H^{\frac{3}{2}+\delta}_x}\|\p^2 \varrho\|_{L_x^2}+\| \Omega\|_{L^\infty_x}\|\p^2 \varrho\|_{H^\delta_x}\les 1.
\end{align*}
Due to $\p (\Omega \p \varrho)=\Omega\p^2\varrho+\p \Omega \p \varrho$, combining the above estimate with (\ref{9.08.10.19}) yields the second estimate in (\ref{9.09.1.19}).

For (\ref{9.24.2.19}),  we first use (\ref{9.22.4.19}), and then apply (\ref{9.08.17.19}) with $G_1=G_2=\p \varrho$ and $G_3=\Omega$ to derive
\begin{align*}
\|(e^{\pm\varrho}+1) (\p \varrho)^2\Omega\|_{H^\delta_x}&\les \|(\p \varrho)^2 \Omega\|_{H^\delta_x}\les \|\p \varrho\|_{H^1_x}^2\|\Omega\|_{H^{1+\delta}_x}+\|\Omega\|_{H^1_x}\|\p \varrho\|_{H^{1+\delta}_x}\|\p \varrho\|_{H^1_x}\\
&\les 1,
\end{align*}
where we used (\ref{eng_3}) and (\ref{4.25.1.19}).

Now we consider the estimate in (\ref{9.08.6.19}).  By using the equation
\begin{equation}\label{divc}
\div \C=-\p_i \varrho  \C^i
\end{equation}
and the elliptic estimate for the $\div$-$\curl$ system, we derive
\begin{equation}\label{9.08.7.19}
\|\p \C\|_{H^\delta_x}\les \|\p_i \varrho \c \C^i\|_{H^\delta_x}+\| \curl \C\|_{H^\delta_x}.
\end{equation}
Applying (\ref{9.08.10.19}) to the first term gives the first estimate in (\ref{9.08.6.19}).

For the second estimate in (\ref{9.08.6.19}), similar to (\ref{9.07.9.19}), we derive
\begin{align}
\|\p^2 \Omega\|_{H^\delta_x}&\les \|\curl^2 \Omega\|_{H^\delta_x}+\|\p(\Omega\c \p \varrho)\|_{H^\delta_x}\nn\\
&\les \|\curl(e^\varrho \C)\|_{H^\delta_x}+\|\p (\Omega \c \p \varrho)\|_{H^\delta_x}\nn\\
&\les \|e^\varrho\curl \C \|_{H^\delta_x}+\|e^\varrho \p \varrho\c \C\|_{H^\delta_x}+\|\p(\Omega\c \p \varrho)\|_{H^\delta_x}.\label{9.08.9.19}
\end{align}
Applying Lemma \ref{comp_dyd} to the first term, (\ref{9.08.10.19}) to the second term and (\ref{9.09.1.19}) to the third term leads to
\begin{equation}\label{9.24.3.19}
\|\p^2 \Omega\|_{H^\delta_x}\les \|\curl\C\|_{H^\delta_x}+1.
\end{equation}
 Finally, we estimate $\|\p^2 \fw\|_{H^\delta_x}$ in view of the symbolic formula (\ref{9.24.1.19}) for $\p^2 \fw$
 \begin{align*}
 \|\p^2 \fw\|_{H^\delta_x}&\les \|e^\varrho\{\p^2 \Omega+\p \Omega \p \varrho+\Omega\big((\p\varrho)^2+\p^2 \varrho\big)\} \|_{H^\delta_x}\\
 &\les \|\p^2 \Omega\|_{H^\delta_x}+\|\p\Omega \p \varrho\|_{H^\delta_x}+\|\Omega \p^2 \varrho\|_{H^\delta_x}+\|\Omega(\p \varrho)^2\|_{H^\delta_x}\\
 &\les \|\p^2 \Omega\|_{H^\delta_x}+1\les \|\curl \C\|_{H^{\delta}_x}+1,
 \end{align*}
 where we used (\ref{9.22.4.19}), (\ref{9.08.10.19})-(\ref{9.24.2.19})  and the estimate (\ref{9.24.3.19}). Using the $L^2_x$ bound in (\ref{9.07.5.19}),
 (\ref{9.08.6.19}) is proven.
 \end{proof}
  For  estimates in the fractional Sobolev spaces,  we recall the standard trichotomy of the Littlewood-Paley theory. For smooth functions $F$ and $G$,
\begin{equation}\label{9.08.3.19}
P_\mu(F\c G)=P_\mu[F\c G]_{HL}+P_\mu[F\c G]_{L H}+P_\mu[F\c G]_{HH}
\end{equation}
where
\begin{align*}
&P_\mu[F\c G]_{HL}=P_\mu(P_\mu F P_{\le \mu}G), \\
&P_\mu[F\c G]_{LH}=P_\mu(P_{\le \mu} F P_\mu G),\\
&P_\mu[F\c G]_{HH}=P_\mu(\sum_{\la>\mu} P_\la F P_\la G).
\end{align*}
\begin{lemma}\label{9.08.11.19}
Let $0<\delta\le s'-2$ and
let $\I_\mu=\mu^{2\delta}\int_0^t \int_{\Sigma_{t'}}P_\mu (\curl \bT \C^m)P_\mu \curl \C_m d\mu_e dt'$. There holds the following estimate, 
\begin{align*}
\sum_{\mu>1}|\I_\mu|&\les \int_0^t (\|\curl \C\|_{\dot{H}^\delta_x}+1)\{\|\p v\|_{B_{\infty, 2,x}^\delta}+(1+\|\p v\|_{L_x^\infty})(\|\curl \C\|_{\dot{H}^\delta_x}+1)\} dt'\\
& +\sup_{0\le t'\le t}\|\La^\delta\curl \C(t')\|_{L_x^2}.
\end{align*}
\end{lemma}
\begin{proof}
In view of (\ref{9.07.8.19}),
\begin{align}\label{9.30.7.19}
\begin{split}
\I_\mu&=\mu^{2\delta}\int_0^t\int_{\Sigma_{t'}} (-2P_\mu\curl(e^{-\varrho}\p v_j \wedge\p \Omega^j)_m+P_\mu\curl(\p_a v \C^a)_m)\c P_\mu \curl \C^m d\mu_e\\
& =\I^1_\mu+\I^2_\mu.
\end{split}
\end{align}
For the term $\I^2_\mu$, we recall the formula in (\ref{9.08.12.19}). We first bound the second term in (\ref{9.08.12.19}) by applying (\ref{9.08.1.19}) to $F=\p v$ and $G=\p \C$,
\begin{align*}
\|\mu^\delta P_\mu(\p v\c \p \C)\|_{l_\mu^2L_x^2}&\les \|\p v\|_{B_{\infty, 2, x}^\delta}\|\p \C\|_{L^2_x}+\|\p \C\|_{H^\delta_x}\|\p v\|_{L_x^\infty}\\
&\les \|\p v\|_{B_{\infty,2,x}^\delta}+(\|\curl \C\|_{\dot{H}^\delta_x}+1)\|\p v\|_{L_x^\infty},
\end{align*}
where we used (\ref{9.07.5.19}) and (\ref{9.08.6.19}) to derive the last line.

For the first term in (\ref{9.08.12.19}), we apply (\ref{9.08.8.19}) to $F=\p \fw$ and $G=\C$, and also use (\ref{9.07.5.19}) to derive
\begin{align*}
\|\mu^\delta P_\mu(\p \fw\c \C)\|_{l_\mu^2 L_x^2}&\les \|\p \fw\|_{H^{\f12+\delta}_x}\|\C\|_{H^1_x}+\|\p \fw\|_{H^1_x}\|\C\|_{H^{\f12+\delta}_x}\les 1.
\end{align*}
Combining the above two estimates implies
\begin{align}\label{9.08.13.19}
\sum_{\mu>1}|\I^2_\mu|&\les \int_0^t\|\curl \C\|_{\dot{H}^\delta_x}(\|\p v\|_{B_{\infty,2,x}^\delta}+(\|\curl \C\|_{\dot{H}^\delta_x}+1)\|\p v\|_{L_x^\infty}+1) dt'.
\end{align}
For the hard term $\I^1_\mu$ \begin{footnote}{We dropped the constant $-2$ which does not influence the analysis.}\end{footnote}, applying (\ref{9.07.3.19}) to $F=\Omega$, similar to (\ref{9.08.14.19}), we write
\begin{align*}
&\I^1_{\mu}= \mu^{2\delta}\int_0^t \int_{\Sigma_t'}\curl P_\mu(e^{-\varrho}\p v_j\wedge \p \Omega^j)_m P_\mu\curl \C^m d\mu_e dt'\\
&=\mu^{2\delta}\int_0^t\int_{\Sigma_t'}P_\mu\{e^{-\varrho}(\p^n \p_m v^j \p_n \Omega_j+(\p_j \bT \varrho+\curl^2 v_j)\p_m\Omega^j+\p_m v^j(-\curl^2 \Omega_j+\p_j(\div \Omega))\\
&-\p_n v^j \p^n \p_m \Omega_j-\tensor{\ep}{_m^{ni}}\tensor{\ep}{_i^{ab}}\p_a v^j \p_b \Omega_j \p_n \varrho)\}P_\mu\curl \C^m d\mu_e dt'.
\end{align*}
Let us write the hard term below, which will be treated by integration by parts
\begin{align}\label{9.08.20.19}
\I^1_{+,\mu} &=\mu^{2\delta}\int_0^t \int_{\Sigma_t'}P_\mu\{e^{-\varrho}(\p^n \p_m v^j \p_n \Omega_j+\p_j \bT \varrho \p_m\Omega^j )\}P_\mu\curl \C^m d\mu_e dt'.
  \end{align}
The remaining terms are denote by $\I^1_{-,\mu}$
\begin{align*}
\I^1_{-,\mu}=\mu^{2\delta}\int_0^t \int_{\Sigma_{t'}} P_\mu (e^{-\varrho}\A_m) P_\mu\curl \C^m d\mu_e dt'
\end{align*}
where
\begin{align*}
\A_m&=\curl^2 v_j\p_m\Omega^j+\p_m v^j(-\curl^2 \Omega_j+\p_j(\div \Omega))-\p_n v^j \p^n \p_m \Omega_j-\tensor{\ep}{_m^{ni}}\tensor{\ep}{_i^{ab}}\p_a v^j \p_b \Omega_j \p_n \varrho\\
&=\curl\fw \p \Omega+ \p v\p^2\Omega+\p v \c\p \Omega\c \p \varrho\\\
&=e^{\varrho}(\curl \Omega+\Omega \p \varrho)\p \Omega+\p v\p^2\Omega+\p v \c\p \Omega\c \p \varrho.
\end{align*}
and we used  the second equation in  (\ref{hodge_2}) to derive the last line. To control this term, we first  directly obtain
 \begin{equation}\label{9.08.16.19}
\sum_{\mu>1} |\I^1_{-,\mu}|\les \int_0^t \|e^{-\varrho}\A\|_{\dot{H}^\delta_x}\|\curl\C\|_{\dot{H}^\delta_x} dt'.
 \end{equation}
 We then estimate by using Lemma \ref{comp_dyd} that
 \begin{align*}
 \| e^{-\varrho} \A\|_{\dot{H}^\delta_x}&\les \|\curl \Omega \c \p\Omega\|_{\dot{H}^\delta_x}+\|\Omega \p \varrho\p \Omega\|_{H^\delta_x}+\|\p v\p^2\Omega\|_{H^\delta_x}+\|\p v \c\p \Omega\c \p \varrho\|_{H^\delta_x}.
 \end{align*}
 By using (\ref{9.08.8.19}) and (\ref{9.07.5.19})
 \begin{align*}
  \|\p \Omega \c \p\Omega\|_{\dot{H}^\delta_x}&\les \|\p \Omega\|_{H^{\frac{1}{2}+\delta}_x}\|\p \Omega\|_{H^1_x}\les \|\p \Omega\|_{H^1_x}^2\les 1.
 \end{align*}
By using (\ref{9.08.1.19}), (\ref{9.07.5.19}) and (\ref{9.08.6.19}), we can obtain
\begin{align*}
\|\p v \p^2 \Omega\|_{H^\delta_x}&\les \|\p v\|_{B_{\infty,2,x}^\delta} \|\p^2 \Omega\|_{L_x^2}+\|\p v\|_{L_x^\infty} \|\p^2 \Omega\|_{H^\delta_x}\\
&\les \|\p v\|_{B_{\infty, 2,x}^\delta}+\|\p v\|_{L_x^\infty}(\|\curl \C\|_{\dot{H}^\delta_x}+1).
\end{align*}
For the cubic terms, we employ (\ref{9.08.17.19}) to obtain
\begin{align}\label{9.09.2.19}
\begin{split}
\|(\Omega+\p v+\p \varrho) \c\p \Omega \c\p \varrho\|_{H^\delta_x}&\les \big(\|\La^\delta\Omega\|_{H^1_x}+\|\La^\delta (\p v, \p \varrho)\|_{H^1_x}\big)\|\p \Omega\|_{H^1_x}\|\p \varrho\|_{H^1_x}\\
&+\|\Omega, \p v,\p \varrho\|_{H^1_x}(\|\p \Omega\|_{H^1_x}\|\La^\delta\p \varrho\|_{H^1_x}+\|\La^\delta\p \Omega \|_{H^1_x}\|\p \varrho\|_{H^1_x})\\
&\les \|\La^\delta \p \Omega\|_{H^1_x}+1\les \|\curl \C\|_{\dot{H}^\delta}+1,
\end{split}
\end{align}
where we employed Corollary \ref{eng_wave}, (\ref{9.07.5.19}) and (\ref{9.08.6.19}) to get the last line.

Substituting  the above three estimates to (\ref{9.08.16.19}) yields
\begin{equation}\label{9.08.18.19}
\sum_{\mu>1} |\I^1_{-,\mu}|\les \int_0^t \|\curl \C\|_{\dot{H}^\delta_x}\{\|\p v\|_{B_{\infty, 2,x}^\delta}+(1+\|\p v\|_{L_x^\infty})(\|\curl \C\|_{\dot{H}^\delta_x}+1)\} dt'.
\end{equation}

Next we consider the  term of (\ref{9.08.20.19}). For the two terms therein, we will integrate by parts in two different ways. Therefore we first separate them as
\begin{align}\label{9.11.9.19}
\begin{split}
\J_\mu&=\mu^{2\delta}\int_0^t \int_{\Sigma_{t'}}P_\mu [\p^n \p_m v^j \c \p_n \Omega_j e^{-\varrho}]P_\mu \curl \C^m d\mu_e dt',\\
\K_\mu&= \mu^{2\delta}\int_0^t \int_{\Sigma_{t'}}P_\mu [\p_j \bT \varrho \c \p_m\Omega^j e^{-\varrho}]P_\mu\curl \C^m d\mu_e dt'.
\end{split}
\end{align}
For $\J_\mu$, using $\p^m P_\mu \curl \C_m=0$, we derive
\begin{align*}
\J_\mu&=\mu^{2\delta}\int_0^t \int_{\Sigma_{t'}}P_\mu [ \p_m(\p^n v^j \c \p_n \Omega_j e^{-\varrho})-\p^n v^j \p_m(\p_n \Omega_j e^{-\varrho})]P_\mu \curl \C^m d\mu_e dt'\\
&=\mu^{2\delta}\int_0^t \int_{\Sigma_{t'}}\big\{\p_m \big(P_\mu(\p^n v^j \c \p_n \Omega_j e^{-\varrho})P_\mu \curl \C^m\big)\\
&\quad-P_\mu[\p^n v^j \p_m (\p_n \Omega_j \c e^{-\varrho})] P_\mu \curl \C^m \big\}d\mu_e dt'.
\end{align*}
Again, since the boundary term vanishes, we only need to treat the last line. By trichotomy in (\ref{9.08.3.19}), we decompose $\J_\mu=\J_{\mu,HL}+\J_{\mu,LH}+\J_{\mu,HH}$,
\begin{align*}
\J_{\mu, HL}&=\mu^{2\delta}\int_0^t \int_{\Sigma_{t'}}P_\mu[ \p^n v_\mu^j \c P_{\le \mu}\p_m(e^{-\varrho}\p_n \Omega_j)]P_\mu \curl \C^m d\mu_e dt',\\
\J_{\mu, LH}&=\mu^{2\delta}\int_0^t \int_{\Sigma_{t'}}P_{\mu}[P_{\le\mu}\p^n v_\mu^j \c P_\mu\p_m(e^{-\varrho}\p_n \Omega_j)]P_\mu \curl \C^m d\mu_e dt',\\
\J_{\mu, HH}&=\mu^{2\delta}\int_0^t \int_{\Sigma_{t'}}\sum_{\la\ge \mu}P_\mu[P_\la( \p^n v_\mu^j) \c P_\la \p_m(e^{-\varrho}\p_n \Omega_j)]P_\mu \curl \C^m d\mu_e dt',
\end{align*}
where we neglected the $-$ sign, since it does not influence the estimate; and $F_\mu^i=P_\mu (F^i)$. By using H\"{o}lder inequality and Cauchy-Schwarz inequality, we derive
\begin{equation}\label{9.24.4.19}
\begin{split}
&\sum_{\mu>1}|\J_{\mu,HL}|\les \int_0^t \|\p v\|_{B_{\infty,2,x}^\delta} \|\p(e^{-\varrho} \p \Omega)\|_{L^2_x}\|\curl \C\|_{\dot{H}^\delta_x}dt';\\
&\sum_{\mu>1}(|\J_{\mu,LH}|+|\J_{\mu,HH}|)\les \int_0^t\|\p v\|_{L_x^\infty}\|\p(e^{-\varrho}\p\Omega)\|_{\dot{H}^\delta_x}\|\curl \C\|_{\dot{H}^\delta_x} dt'.
\end{split}
\end{equation}
By using (\ref{9.08.21.19}) with $F=\p\Omega$ and (\ref{9.08.6.19}), we have
\begin{align*}
\|\p(e^{-\varrho}\p\Omega)\|_{\dot{H}^{\delta}_x}\les \|\curl\C\|_{\dot{H}^\delta_x}+1.
\end{align*}
Note the lower order estimate holds due to (\ref{5.04.17.19}), the first estimate in (\ref{9.23.4.19}) and  (\ref{9.07.5.19})
\begin{align*}
\|\p (e^{-\varrho}\p \Omega)\|_{L^2_x}\les \|\curl\C\|_{L^2_x}+1\les 1.
\end{align*}
Substituting the above two estimates to (\ref{9.24.4.19}) yields
\begin{align}\label{9.10.1.19}
\sum_{\mu>1} |\J_\mu|&\les \int_0^t \{\|\p v\|_{B_{\infty,2,x}^\delta}+\|\p v\|_{L_x^\infty}(\|\curl \C\|_{\dot{H}^\delta_x}+1)\} \|\curl \C\|_{\dot{H}^\delta_x}dt'.
\end{align}
Next we treat $\K_\mu$ in (\ref{9.11.9.19}) by performing the spacetime integration by parts by virtue of (\ref{9.07.4.19}).
\begin{align*}
&\K_\mu= \mu^{2\delta}\int_0^t \int_{\Sigma_{t'}} P_\mu(\p_j \bT \varrho \p_m \Omega^j e^{-\varrho})P_\mu \curl \C^m d\mu_e dt'\\
 &=\mu^{2\delta}\int_0^t \int_{\Sigma_{t'}}  P_\mu\big(([\p_j, \bT]\varrho +\bT \p_j\varrho ) \p_m \Omega^j e^{-\varrho}\big)P_\mu \curl \C^m   d\mu_e dt'\\
 &=\mu^{2\delta}\int_0^t \int_{\Sigma_{t'}}P_\mu[\bT(\p_j \varrho  \p_m \Omega^j e^{-\varrho})-\p_j \varrho\bT(\p_m \Omega^j e^{-\varrho})+\p_j v^n \p_n \varrho \p_m \Omega^j e^{-\varrho}]P_\mu \curl \C^m d\mu_e dt'\\
 &=\mu^{2\delta}\int_0^t \{\bT P_\mu (\p_j \varrho \p_m\Omega^j e^{-\varrho})+[P_\mu, \bT] (\p_j \varrho \p_m \Omega^j e^{-\varrho})+P_\mu[-\p_j \varrho \bT(\p_m \Omega^j e^{-\varrho})\\
 &\quad \quad+\p_j v^n \p_n \varrho \p_m \Omega^j e^{-\varrho}]\}P_\mu \curl \C^m d\mu_e dt'\displaybreak[0]\\
 &=\mu^{2\delta}\int_0^t \int_{\Sigma_{t'}}\big(\bT\{P_\mu(\p_j \varrho \p_m \Omega^j e^{-\varrho})\c P_\mu \curl \C^m\}-P_\mu(\p_j \varrho \p_m \Omega^j e^{-\varrho}) \bT P_\mu \curl \C^m \big)d\mu_e dt'\\
 &+\mu^{2\delta}\int_0^t \int_{\Sigma_{t'}}\{[P_\mu, \bT] (\p_j \varrho \p_m \Omega^j e^{-\varrho})+P_\mu[-\p_j \varrho\bT(\p_m \Omega^j e^{-\varrho})+\p_j v^n \p_n \varrho \p_m \Omega^j e^{-\varrho}]\} P_\mu \curl \C^m d\mu_e dt'\displaybreak[0]\\
 &=\mu^{2\delta}\{\int_0^t \int_{\Sigma_{t'}}\Tr\cir{k}P_\mu(\p_j \varrho \p_m \Omega^j e^{-\varrho})P_\mu \curl \C^m d\mu_e dt'+\int_0^t \p_t \int_{\Sigma_{t'}}P_\mu(\p_j \varrho \p_m\Omega^j e^{-\varrho})P_\mu \curl \C^m d\mu_e dt'\}\\
 &+\mu^{2\delta}\int_0^t \int_{\Sigma_{t'}}\{[P_\mu, \bT] (\p_j \varrho \p_m \Omega^j e^{-\varrho})+P_\mu\big(-\p_j \varrho\bT(\p_m \Omega^j e^{-\varrho})+\p_j v^n\p_n \varrho \p_m \Omega^j e^{-\varrho}\big)\} P_\mu \curl \C^m d\mu_e dt'\\
 &-\mu^{2\delta}\int_0^t\int_{\Sigma_{t'}}P_\mu(\p_j \varrho \p_m \Omega^j e^{-\varrho}) \bT P_\mu \curl \C^m d\mu_e dt'.
 \end{align*}
 Let us denote the last line by $\K^-_\mu$ and the remaining four terms together by $\K^+_\mu$. We further decompose $\K^-_\mu$ as follows
 \begin{align*}
\K^-_{\mu}&= \mu^{2\delta}\int_0^t\int_{\Sigma_{t'}}P_\mu(\p_j \varrho \p_m \Omega^j e^{-\varrho}) \bT P_\mu \curl \C^m d\mu_e dt'\\
&=\mu^{2\delta} \int_0^t \int_{\Sigma_{t'}}P_\mu(\p_j \varrho\p_m \Omega^j e^{-\varrho})([\bT, P_\mu \curl]\C^m +P_\mu  \curl\bT \C^m)d\mu_e dt'\\
&=\K^-_{1, \mu}+\K^-_{2,\mu},
 \end{align*}
 where the $-$ sign is dropped without influencing the result. We can bound the second term  by using  (\ref{9.07.13.19}) and the finite band property
 \begin{align*}
 |\K^-_{2, \mu}|&\le \int_0^t \int_{\Sigma_{t'}} \mu^{2\delta}P_\mu(\p \varrho \p \Omega e^{-\varrho})P_\mu\p(e^{-\varrho}\p v\c \p \Omega) d\mu_e dt|\\
 &\les \mu^{1+2\delta}\int_0^t \|P_\mu(\p \varrho e^{-\varrho}\c \p \Omega)\|_{L_x^2}\|P_\mu(e^{-\varrho}\p v \p \Omega)\|_{L_x^2} dt'.
 \end{align*}
 By using (\ref{9.22.5.19})
 \begin{align*}
\Er:&=\|\La^{\f12+\delta}(\p \varrho e^{-\varrho}\c \p \Omega)\|_{L_x^2}+\|\La^{\f12+\delta}(e^{-\varrho}\p v\c \p \Omega)\|_{L_x^2}\\
&\les\|\La^{\f12+\delta}(\p \varrho  \p \Omega)\|_{L_x^2}+\|\La^{\f12+\delta}(\p v \p \Omega)\|_{L_x^2}.
 \end{align*}
  For the $\dot{H}^{\f12+\delta}$ norms, applying (\ref{9.08.8.19}) to $(F, G)=(\p \varrho, \p \Omega)$ and $(F,G)=(\p v, \p \Omega)$ leads to
 \begin{align*}
 \Er&\les \|\p \varrho ,\p v\|_{H^{1+\delta}_x}\|\p \Omega\|_{H^1_x}+\|\p \varrho, \p v\|_{H^1_x}\|\p \Omega\|_{H^{1+\delta}_x}\\
&\les \|\p \varrho, \p v\|_{H^{1+\delta}_x}+\|\p \Omega\|_{H^{1+\delta}_x},
 \end{align*}
 where we used Corollary \ref{comp_1} to obtain the last inequality. Hence by using Corollary \ref{eng_wave} and (\ref{9.08.6.19}),
 \begin{align*}
 \Er\les \|\curl \C\|_{\dot{H}^\delta_x}+1.
 \end{align*}
 We thus conclude that
 \begin{align*}
 \sum_{\mu>1}|\K^-_{2, \mu}|&\les \int_0^t\|\La^{\f12+\delta}(e^{-\varrho}\p \varrho \p \Omega)\|_{L_x^2}\|\La^{\f12+\delta}(e^{-\varrho}\p v\c \p \Omega)\|_{L_x^2}dt'\\
 &\les \int_0^t (\|\curl \C\|_{\dot{H}^\delta_x}+1)^2 dt'.
 \end{align*}
 For the other term, we first derive
 \begin{align*}
 \sum_{\mu>1}|\K^-_{1,\mu}|&\les\int_0^t  \|\La^\delta( e^{-\varrho}\p \varrho\p\Omega)\|_{L_x^2}\|\mu^\delta[\bT, P_\mu\curl]\C\|_{l_\mu^2 L_x^2} dt'.
 \end{align*}
 Due to (\ref{9.08.10.19})
 \begin{equation}\label{9.09.3.19}
 \| e^{-\varrho}\p \varrho\p\Omega\|_{H^\delta_x}\les 1.
 \end{equation}
 Combining  the above estimate with (\ref{9.08.2.19}) for $F=\fC$, (\ref{9.07.5.19}) and (\ref{9.08.6.19}) yields
 \begin{align*}
 \sum_{\mu>1}|\K^-_{1,\mu}|&\les\int_0^t ( \|\p v\|_{B_{\infty, 2, x}^\delta}\|\p \C\|_{L_x^2}+\|\p v\|_{L_x^\infty}\|\p \C\|_{H^\delta_x}) dt'\\
 &\les \int_0^t \{ \|\p v\|_{B_{\infty, 2, x}^\delta}+\|\p v\|_{L_x^\infty}(\|\curl \C\|_{\dot{H}^\delta_x}+1)\} dt'.
 \end{align*}
 Hence
 \begin{align}\label{9.24.5.19}
 \sum_{\mu>1}|\K^-_\mu|&\les \int_0^t\{  \|\p v\|_{B_{\infty, 2, x}^\delta}+\big(\|\curl \C\|_{\dot{H}^\delta_x}+1+\|\p v\|_{L_x^\infty}\big)(\|\curl \C\|_{\dot{H}^\delta_x}+1)\} dt'.
 \end{align}
 We go back to consider $\K_\mu^+$, which has four terms, denoted by $\K_{i, \mu}^+$ with $i=1,\cdots, 4$. By using (\ref{9.08.22.19}), we symbolically rewrite the term below
 \begin{align*}
 \K_{1,\mu}^+&=\int_0^t \int_{\Sigma_{t'}} \mu^{2\delta}P_\mu(-\p_j \varrho\bT(\p_m \Omega^j e^{-\varrho})+\p_j v^n \p_n \varrho \p_m \Omega^j e^{-\varrho}) P_\mu \curl \C^m d\mu_e dt'\\
 &=\int_0^t \int_{\Sigma_{t'}} \mu^{2\delta}P_\mu\{(\p \varrho(\p \varrho+\p v) \p \Omega+\Omega \p^2 \varrho) e^{-\varrho}\}P_\mu \curl \C d\mu_e dt'.
 \end{align*}
Thus it follows by using Lemma \ref{comp_dyd}, (\ref{9.09.2.19}) and (\ref{9.09.1.19}) that
\begin{align*}
\sum_{\mu>1}|\K^+_{1,\mu}|&\les \int_0^t \|\p \varrho(\p \varrho+\p v) \p \Omega+\Omega \p^2 \varrho\|_{H^\delta_x}\|\La^\delta\curl \C\|_{L_x^2} dt'\\
&\les \int_0^t (\|\curl \C\|_{\dot{H}^\delta}+1)^2 dt'.
\end{align*}
For the term
 \begin{align*}
\K^+_{2,\mu}=\mu^{2\delta} \int_0^t \int_{\Sigma_{t'}} \Tr \cir{k} P_\mu (\p \varrho \p \Omega e^{-\varrho}) P_\mu \curl \C d\mu_e dt'
\end{align*}
by using (\ref{9.09.3.19}) and $\Tr \cir{k}=-\div v$ we estimate
\begin{align*}
\sum_{\mu>1}|\K^+_{2,\mu}|&\les \int_0^t \|\Tr \cir{k}\|_{L_x^\infty} \|\La^\delta(\p \varrho \p \Omega e^{-\varrho})\|_{L_x^2}\|\La^\delta\curl \C\|_{L_x^2} dt'\\
&\les \int_0^t \|\p v\|_{L_x^\infty} \|\La^\delta \curl \C\|_{L_x^2} dt'.
\end{align*}
For the term
\begin{align*}
\K^+_{3,\mu}=\mu^{2\delta} \int_0^t \int_{\Sigma_{t'}} [P_\mu, \bT] (\p \varrho \p \Omega e^{-\varrho}) P_\mu \curl C d\mu_e dt'
\end{align*}
we recall that $[P_\mu, \bT] =[P_\mu, v]\p $  and apply (\ref{4.13.3.19}) to $F=v$ and $G=\p \varrho \p \Omega e^{-\varrho}$ with $\a=\delta$,
\begin{align*}
\sum_{\mu>1}|\K^+_{3,\mu}|&\les \int_0^t \|\mu^\delta[P_\mu, v]\p(\p \varrho \p \Omega e^{-\varrho})\|_{l_\mu^2 L_x^2}\|\La^\delta \curl \C\|_{L_x^2} dt'\\
 &\les \int_0^t \|\p v\|_{L_x^\infty}\|\p \varrho \p \Omega e^{-\varrho}\|_{H^\delta_x}\|\La^\delta \curl \C\|_{L_x^2}\\
 &\les \int_0^t \|\p v\|_{L_x^\infty}\|\La^\delta \curl \C\|_{L_x^2} dt',
\end{align*}
where we used (\ref{9.09.3.19}) to obtain the last line.

For the boundary term,
\begin{align*}
\K^+_{4,\mu}=\mu^{2\delta}\int_0^t \p_t\int_{\Sigma_{t'}} P_\mu(\p\varrho \p \Omega e^{-\varrho})P_\mu \curl \C d\mu_e dt',
\end{align*}
by using (\ref{9.09.3.19}) again, we derive
\begin{align*}
\sum_{\mu>1}|\K^+_{4,\mu}|&\les \sup_{0\le t'\le t}\|\La^\delta(\p \varrho \p\Omega e^{-\varrho})\|_{L_x^2}\|\La^\delta \curl \C\|_{L_x^2} \\
&\les \sup_{0\le t'\le t}\|\La^\delta\curl \C\|_{L_x^2}(t').
\end{align*}
Hence we summarize the above estimates for the terms in $\K_\mu^+$, and combine the estimate of (\ref{9.24.5.19}) to conclude
\begin{align*}
\sum_{\mu>1} |\K_\mu|&\les   \int_0^t (\|\p v\|_{L_x^\infty}+\|\La^\delta \curl \C\|_{L_x^2}+1)(\|\La^\delta \curl \C\|_{L_x^2}+1) dt'\\
&+\sup_{0\le t'\le t}\|\La^\delta\curl \C\|_{L_x^2}(t')+\int_0^t  \|\p v\|_{B_{\infty, 2, x}^\delta}dt'.
\end{align*}

Combining the above estimate with (\ref{9.10.1.19}) and (\ref{9.08.18.19}) implies the result in Lemma \ref{9.08.11.19}.
\end{proof}
\begin{proof}[Proof of Proposition \ref{9.07.16.19}]
We apply (\ref{9.07.4.19}) to $F=G=P_\mu \curl \C$ to obtain
\begin{align*}
\sum_{\mu>1}\mu^{2\delta}&\int_{\Sigma_t} |P_\mu\curl \C|^2 d\mu_e \\
&\les \|\La^\delta\curl \C(0)\|_{L_x^2}^2+\sum_{\mu>1}|\int_0^t \mu^{2\delta}\bT P_\mu \curl \C \c P_\mu \curl \C d\mu_e dt'|\\
&+\int_0^t \|\p v\|_{L_x^\infty}\|\La^\delta\curl \C\|^2_{L_x^2} dt'\displaybreak[0]\\
&\les \|\La^\delta\curl \C(0)\|_{L_x^2}^2+\sum_{\mu>1}|\int_0^t \mu^{2\delta} ([\bT, P_\mu \curl]+P_\mu \curl \bT ) \C \c P_\mu \curl \C d\mu_e dt'|\\
&+\int_0^t \|\p v\|_{L_x^\infty}\|\La^\delta\curl \C\|^2_{L_x^2} dt'\\
&\les \|\La^\delta\curl \C(0)\|_{L_x^2}^2+\sum_{\mu>1}|\int_0^t \mu^{2\delta}  [\bT, P_\mu \curl] \C \c P_\mu \curl \C d\mu_e dt'|\\
&+\int_0^t \|\p v\|_{L_x^\infty}\|\La^\delta\curl \C\|^2_{L_x^2} dt'+\sum_{\mu>1}|\I_\mu|\\
&\les \int_0^t (\|\curl \C\|_{\dot{H}^\delta_x}+1)\{\|\p v\|_{B_{\infty, 2,x}^\delta}+(1+\|\p v\|_{L_x^\infty})(\|\curl \C\|_{\dot{H}^\delta_x}+1)\} dt'\\
& +\sup_{0\le t'\le t}\|\La^\delta\curl \C\|_{L_x^2}(t')+\|\La^\delta\curl \C(0)\|_{L_x^2}^2,
\end{align*}
where we used  $\Tr \cir{k}=-\div v$, and  applied (\ref{9.08.2.19}) to $F=\C$ together with (\ref{9.08.6.19}), and Lemma  \ref{9.08.11.19} to obtain the last inequality. Proposition \ref{9.07.16.19} then follows by using Gronwall's inequality and using (\ref{9.08.6.19}).
\end{proof}

 \section{Reduction to the Strichartz estimates for the linearized wave equation}\label{5.07.4.19}

 Due to Corollary \ref{comp_3}, the main task from now on is to improve  the bootstrap assumption in (\ref{BA1}) by establishing Strichartz  estimate for the wave function $\Psi=(v_+, \varrho)$, which is restated below.
\begin{theorem}[Main estimate]\label{main}
Let $s>2$ and let $\Psi=(v_+, \varrho)$ be the pair of solutions of (\ref{wavevpl}) and (\ref{4.10.2.19}). If $\Psi$ satisfies (\ref{BA1}), then there holds with a number $8\delta_0<\ga_0<s-2$
\begin{equation}\label{7.11.300}
\|\bp \Psi\|_{L^2_{[0,T]}L_x^\infty}^2+\sum_{\la\ge 2}\la^{2\ga_0}\|\bar P_\la\bp \Psi\|_{L^2_{[0,T]} L_x^\infty}^2
\les T^{2\ga_1}
\end{equation}
where $\bar P_\la$ denote the Littlewood-Paley projections with $\sum_{\la} \bar P_\la=Id$ in $L^2({\mathbb R}^3)$ and $0<\ga_1\le \ep_0$.
\end{theorem}
\begin{remark}
Different from \cite[Section 4]{WangCMCSH} and \cite[Section 3]{Wangrough},  $\bT$-derivative estimates in (\ref{7.11.300}) will be derived by using (\ref{4.23.1.19}) and the spatial derivative estimate in (\ref{7.11.300}) due to a potential issue from commutations.
\end{remark}

 We first reduce the proof of Theorem \ref{main} to the proof of Strichartz estimates on small time intervals.
Let $\la$ be a fixed large dyadic number and let $0<\ep_0<\frac{s-2}{5}$ and $\delta_0=\ep_0^2$ be the fixed numbers as mentioned in the introduction.
By using the bootstrap assumption (\ref{BA1}), we can partition $[0, T ]$ into disjoint union of sub-intervals
$I_k := [t_{k-1}, t_k]$ of total number $\les \la^{8\ep_0}$ with the properties that $|I_k|\le \la^{-8\ep_0} T$ and
\begin{equation}\label{BA2}
\|\bp \Psi\|_{L_{I_k}^2 L_x^\infty}^2+\sum_{\mu\ge 2}\mu^{2\delta_0}\|\bar P_\mu\bp \Psi\|_{L^2_{I_k} L_x^\infty}^2\les \la^{-8\ep_0}.
\end{equation}
The total number of intervals $I_k$ depends on $\la$, which is denoted by $k_\la$.

Let $\H$ be the space of pairs of functions $\vartheta=(\vartheta_0, \vartheta_1)$, with the norm
\begin{equation*}
\|\vartheta\|^2_{\H}=\|\p \vartheta_0\|^2_{L_x^2}+\|\vartheta_1+v^i \p_i \vartheta_0\|^2_{L_x^2}.
\end{equation*}
For the pair of functions
$
f[t]:=(f(t), \p_t f(t)),
$
\begin{equation*}
\|f[t]\|^2_\H= \|\p f(t)\|^2_{L^2_x}+\|\bT f(t)\|^2_{L_x^2}.
\end{equation*}

On each time interval $I_k$ we will show the following dyadic Strichartz estimate.
\begin{theorem}[Dyadic Strichartz estimates for the linearized wave equation]\label{dyastric}
 Fix $\la\ge\La$ with $\La$ a large constant. Let $\bg$ be the acoustic metric given in \eqref{metric}, and $\psi$ be a solution of
 \begin{equation}\label{eq2}
 \Box_\bg\psi=0
 \end{equation}
on the time interval $I_k$. Then for any $q>2$ sufficiently close to $2$ there holds
 \begin{equation*}
\|P_\la \bp \psi\|_{L^q_{I_k}L_x^\infty}\les \la^{\frac{3}{2}-\frac{1}{q}} \|\psi[t_k]\|_{\H},
 \end{equation*}
 where $\bp=\p,\bT$.
\end{theorem}

\subsection{Proof of Theorem \ref{main} assuming Theorem  \ref{dyastric}}
By the reproducing property, we can write $\bar P_\mu=P_\mu^2$ with $P_\mu$ the Littlewood-Paley projection associated to a different symbol.
We now recall that for the solution $U$ of (\ref{wave2}) verifies the first order system (\ref{lu3}) with correspondence given in (\ref{wave3}). We also obtained that the corresponding $(U_\mu, V_\mu)=(P_\mu U, P_\mu V)$ verifies (\ref{lu3}) with $F_{U_\mu}$ and $F_{V_\mu}$ given in (\ref{puuv}).   Recast the equation for $U_\mu$ into a form of (\ref{wave2}) by using (\ref{wave3}), we have
\begin{equation*}
\Box_\bg P_\mu U= -\bT F_{U_\mu}+\Tr k F_{U_\mu}-F_{V_\mu}.
\end{equation*}

We will apply the above calculation to $U$ being either of the function in $\Psi$ according to (\ref{5.03.1.19}) and (\ref{5.03.2.19}). The first term  on the right hand side will be shortly transformed into the lower order term $-F_{U_\mu}$ which are set into the functions of initial data  in Duhamel's principle.

To be precise, we define
$\fW(t,s)$ to be the operator defined on $\H$ such that, for each $\vartheta:=(\vartheta_0, \vartheta_1)\in \H$ on $\Sigma_s$,
$\phi(t,s,x)=\fW(t,s)(\vartheta)$ is the unique solution of the initial value problem such that
\begin{equation}\label{eqn23}
\Box_\bg \phi=0
\end{equation}
with
\begin{equation*}
 \phi(t;s,x)=\vartheta_0,\quad  \p_t \phi(t;s,x)=\vartheta_1, \mbox{ at } t=s.
\end{equation*}

Then,  by an adaption from \cite[Section 4]{WangCMCSH}, we derive the representation formula by the Duhamel's principle for $t\in I_k=[t_{k-1}, t_k]$ that
\begin{align}
P_\mu U(t)&=\fW(t,{t_{k-1}})\left(P_\mu U(t_{k-1}), \p_t P_\mu U(t_{k-1})- F_{U_\mu}(t_{k-1})\right)\nn\\
&+\int_{t_{k-1}}^t \{\fW(t,s)(0, - R_\mu(s)) +\fW(t,s)(F_{U_\mu}(s), 0)\}ds,\label{pmug}
\end{align}
where
\begin{equation}\label{5.12.2.19}
R_\mu=-F_{V_\mu}+v^i \p_i F_{U_\mu}.
\end{equation}

Now we apply $P_\mu$ to the both sides and take the spatial derivative.
\begin{align}
P_\mu^2 \p_m U(t)&=\int_{t_{k-1}}^t \left\{\p_m P_\mu \fW(t,s)(0, - R_\mu(s))
  +\p_m P_\mu \fW(t,s)(F_{U_\mu}(s), 0)\right\}ds\nn\\
&\quad \, +\p_m P_\mu \fW(t,{t_{k-1}})\left(P_\mu U(t_{k-1}),
  \p_t P_\mu U(t_{k-1})-F_{U_\mu}(t_{k-1})\right).\label{ppg}
\end{align}


By using Theorem \ref{dyastric}, we have for any one-parameter family of data $\vartheta(s):=(\vartheta_0(s), \vartheta_1(s))\in \H$ with
$s\in I_k:=[t_{k-1}, t_k]$ that
$$
\|P_\mu \p \fW(t,s)(\vartheta(s))\|_{L_{[s,t_k]}^{q}L_x^\infty}
\les \mu^{\frac{3}{2}-\frac{1}{q}} \|\vartheta(s)\|_{\H}.
$$
In view of the Minkowski inequality we then obtain
\begin{align*}
\left\|\int_{t_{k-1}}^t P_\mu \p \fW(t,s)(\vartheta(s))ds\right\|_{L_{I_k}^2 L_x^\infty}
&\les \int_{t_{k-1}}^{t_k} \|P_\mu \p \fW(t,s)(\vartheta(s))\|_{L^2_{[s,t_k]} L_x^\infty} ds\\
&\les |I_k|^{\frac{1}{2}-\frac{1}{q}}\mu^{\frac{3}{2}-\frac{1}{q}}\int_{I_k}\|\vartheta(s)\|_{\H} ds.
\end{align*}
Since $|I_k|\les T \mu^{-8\ep_0}$, it follows that
\begin{equation*}
\left\|\int_{t_{k-1}}^t P_\mu \p \fW(t,s)(\vartheta(s))ds\right\|_{L_{I_k}^2 L_x^\infty}
\les T^{\f12-\frac{1}{q}}\mu^{(\frac{1}{2}-\frac{1}{q})(1-8\ep_0)}
\int_{I_k} \mu \|\vartheta(s)\|_{\H} ds.
\end{equation*}
Applying the above inequality to (\ref{ppg}) gives, with
$\delta_2:=(\frac{1}{2}-\frac{1}{q})(1-8\epsilon_0)$, that
\begin{align}
\|\bar P_\mu \p_m U\|_{L^2_{I_k} L_x^\infty}
&\les T^{\f12-\frac{1}{q}}\mu^{\delta_2} \|\mu(F_{U_\mu},-R_\mu)\|_{L^1_{I_k} \H} + T^{\f12-\frac{1}{q}} B_\mu(t_{k-1}), \label{sumu}
\end{align}
where
\begin{align*}
B_{\mu}(t):=\mu^{\delta_2}\|\mu\left(P_\mu U(t), \p_t P_\mu U(t)-F_{U_\mu}(t)\right)\|_{\H}, 
\end{align*}
In the following we will give the estimates on $R_\mu$, $F_{U_\mu}$ and $B_\mu(t_{k-1})$.
Positive indices $\ep_0, q, \delta$ are chosen such that $4\ep_0+\delta_2+\delta<s-2$, and $\delta_2+\delta<4\ep_0.$
\begin{remark}\label{9.25.2.19}
{For convenience, we choose $\frac{1}{2}-\frac{1}{q}=\ep_0$.  Hence $0<\delta<\min(s-2-4\ep_0-\ep_0(1-8\ep_0), 4\ep_0-\ep_0(1-8\ep_0))$. Since $s-2>5\ep_0>0$,  this allows us to achieve $\delta> 8\ep_0^2=8\delta_0$ in the range.}
\end{remark}

\subsubsection{ Estimates for $ R_\mu$, $F_{U_\mu}$ }
To treat the first term on the right of (\ref{sumu}), we note
\begin{equation}\label{5.12.4.19}
\begin{split}
\|(F_{U_\mu}(s), - R_\mu(s))\|^2_{\H}&=\|(-R_\mu+v^m\p_m F_{U_\mu})(s)\|^2_{L_x^2}+\|\p F_{U_\mu}(s)\|^2_{L_x^2}\\
&=\|F_{V_\mu}(s)\|^2_{L_x^2}+\|\p F_{U_\mu}(s)\|^2_{L_x^2}.
\end{split}
\end{equation}
\begin{lemma}\label{rmd.1}
For any $\delta_1>\delta>0$ satisfying  $b:=\delta_2+\delta_1<4\epsilon_0$,  there holds
\begin{align}
\left(\sum_{\mu>\La}\sum_{k=1}^{\kappa_\mu} \|\mu^{1+\delta_2+\delta} (F_{U_\mu}(s), - R_\mu(s))\|_{L^1_{I_k}\H}^2\right)^\f12
&\les 1.\label{rem1.1}
\end{align}
\end{lemma}

\begin{proof}
Recall from the definitions of (\ref{puuv}),
\begin{align*}
F_{V_\mu}&=[P_\mu, c^2] \Delta_e U-[P_\mu, v^m]\p_m V+P_\mu F_V+[P_\mu, \Tr k]V-[P_\mu,c^2\p_l(\log c)]\p^l U,\\
\p_i F_{U_\mu}=&
-[P_\mu, v^m]\p_i\p_m U-[P_\mu, \p_i v^m] \p_m U+ \p_i P_\mu F_U.
\end{align*}
These terms will be divided into two types and thus treated differently,
\begin{align*}
&\I_1:=\sum_{\mu>\La}\sum_{k=1}^{\kappa_\mu}\|\mu^{1+\delta_2+\delta}([P_\mu, c^2]\Delta_e U, [P_\mu, v]\p V, [P_\mu, v]\p^2 U) \|_{L_{I_k}^1 L_x^2}^2, \\
&\I_2:=\sum_{\mu>\La}\sum_{k=1}^{\kappa_\mu}\|\mu^{1+\delta_2+\delta}\ckk R_\mu\|_{L_{I_k}^1 L_x^2}^2,
\end{align*}
where symbolically $$\ckk R_\mu=[P_\mu, \Tr k]V+[P_\mu, \p_i v^m] \p_m U+[P_\mu,\p_l(c^2)]\p^l U+P_\mu F_V+\p P_\mu F_U.$$
It suffices to show that
\begin{align}
\I_1^{\f12}&\les \| \p^2 U, \p V\|_{L_I^\infty L_x^2}, \label{rem1.2} \\
\I_2^{\f12}&\les \|\bp v_+,\bp \varrho\|_{L_I^1L_x^\infty}+\|\mu^{1+b}\p P_\mu F_U\|_{L_I^1 l_\mu^2 L_x^2}
+\|\mu^{1+b} P_\mu F_V\|_{L_I^1 l_\mu^2 L_x^2}. \label{gnr.1}
\end{align}
By (\ref{BA2}) and Lemma \ref{comp_3}, we have  $\|\p \varrho, \p v\|_{L_{I_k}^1 L_x^\infty} \les \mu^{-8\epsilon_0}T^\f12$. We can apply
(\ref{5.12.1.19}) to obtain
\begin{align*}
\|\mu^{1+\delta+\delta_2}& ([P_\mu, c^2]\Delta_e U, [P_\mu, v]\p V, [P_\mu, v]\p^2 U)\|_{L^1_{I_k} L_x^2}\\
&\les \mu^{\delta+\delta_2}\|\p \varrho, \p v\|_{L^1_{I_k}L_x^\infty}\|\p^2 U, \p V\|_{L_t^\infty L_x^2}\\
&\les \mu^{\delta+\delta_2-8\epsilon_0} \|\p^2 U, \p V\|_{L_t^\infty L_x^2}.
\end{align*}
Recall also that $\kappa_\mu\les \mu^{8\epsilon_0}$. We can obtain
\begin{align*}
\sum_{k=1}^{\kappa_\mu}&\|\mu^{1+\delta+\delta_2}([P_\mu, c^2]\Delta_e U,[P_\mu, v]\p V, [P_\mu, v]\p^2 U)\|_{L^1_{I_k}L_x^2}^2
\le C\mu^{2(\delta+\delta_2-4\ep_0)}\|\p^2 U, \p V\|_{L_t^\infty L_x^2}^2.
\end{align*}
Since $0<\delta<\delta_1$ and $b:=\delta_2+\delta_1<4\epsilon_0$, we have
$$\I_1\les \La^{2(b-4\epsilon_0)} \|\p^2 U, \p V\|^2_{L_t^\infty L_x^2}\les \|\p^2 U, \p V\|_{L_t^\infty L_x^2}^2,$$
which gives (\ref{rem1.2}).

Next we prove (\ref{gnr.1}). Since $0<\delta<\delta_1$, we observe that for any function $a_\mu$ there holds
\begin{align}
\sum_{\mu>\La} \sum_{k=1}^{\kappa_\mu} \|\mu^\delta a_\mu\|_{L_{I_k}^1 L_x^2}^2
&\le \sum_{\mu>\La} \|\mu^\delta a_\mu\|_{L_I^1 L_x^2}^2
\le \left(\int_I \sum_{\mu>\La}\|\mu^\delta a_\mu\|_{L_x^2}\right)^2 \nn\\
&\les \left(\int_I \|\mu^{\delta_1} a_\mu\|_{l_\mu^2 L_x^2} \right)^2.\label{amu1}
\end{align}

The first three terms in $\ckk R_\mu$ can be treated directly by using (\ref{amu1}). We first note that
\begin{equation*}
[P_\mu, f]G=P_\mu (fG)- fP_\mu G.
\end{equation*}
Applying (\ref{4.14.11.19}) to the first term yields
\begin{equation*}
\|\mu^{1+b}[P_\mu, f] G\|_{l_\mu^2 L_x^2}\les \|f\|_{L_x^\infty} \|G\|_{H^{b+1}_x}+\|f\|_{H^{1+b}_x}\|G\|_{L_x^\infty}.
\end{equation*}
With the help of the above estimate and (\ref{amu1}),
\begin{align*}
\|\mu^{1+b}&\big([P_\mu, \Tr k] V, [P_\mu, \p v]\p U, [P_\mu,\p_l(c^2)]\p^l U\big)\|_{l_\mu^2 L_x^2}\\
&\les \| \p v, \Tr k, \p(c^2)  \|_{L_x^\infty}\| V, \p U\|_{ H_x^{b+1}}+\|\p v, \Tr k , \p (c^2)\|_{H^{1+b}_x}\| V, \p U\|_{ L_x^\infty}.
\end{align*}
Using (\ref{5.12.3.19}) and  taking $L_I^1$ norm gives
\begin{align*}
\|\mu^{1+b}\big([P_\mu, \Tr k] V&, [P_\mu, \p v]\p U, [P_\mu,\p_l(c^2)]\p^l U\big)\|_{l_\mu^2 L_x^2}\\
&\les \|\bp \varrho, \p v_+ \|_{L_I^1L_x^\infty}\|V, \p U, \p v, \Tr k, \p(c^2)\|_{L_t^\infty H^{1+b}_x}.
\end{align*}
 (\ref{gnr.1}) then follows by applying  Corollary \ref{eng_wave}.

 Thus
\begin{align*}
 \I_1+\I_2&\le\|\bp v_+, \bp \varrho\|_{L_I^1 L_x^\infty} +\|\mu^{1+b} P_\mu F_U\|_{L_I^1 l_\mu^2 H^1_x}\nn\\
&+\|\mu^{1+b} P_\mu F_V\|_{L_I^1 l_\mu^2 L_x^2}+\| \p^2 U, \p V\|_{L_I^\infty L_x^2}.
\end{align*}
We now combine Corollary \ref{eng_wave}, (\ref{5.05.3.19}), (\ref{5.05.4.19}) and Corollary \ref{comp_3} to obtain
\begin{equation*}
\|\mu^{1+b} P_\mu F_U\|_{L_I^1 l_\mu^2 H^1}+\|\mu^{1+b} P_\mu F_V\|_{L_I^1 l_\mu^2 L_x^2}\les\|\bp v_+, \bp \varrho\|_{L_I^1 L_x^\infty}+1.
\end{equation*}
 Using the above estimate, $\|\bp v_+, \bp \varrho\|_{L_I^1 L_x^\infty}\les 1$ due to (\ref{BA1}) and Corollary \ref{comp_3}, we can conclude (\ref{rem1.1}).
\end{proof}

\subsubsection{Estimate for  $B_{\mu}(t_{k-1})$.}
Recall  (\ref{lu3}) with $U$ being $P_\mu U$. By  the definition of $\H$, we derive
\begin{equation}\label{5.12.5.19}
B_\mu(t)^2=\mu^{2\delta_2+2}\big(\|\bT P_\mu U(t)-F_{U_\mu}(t)\|^2_{L_x^2}+\|\p P_\mu U(t)\|_{L_x^2}^2)=\mu^{2\delta_2+2}\big(\|P_\mu V(t)\|^2_{L_x^2}+\|\p P_\mu U(t)\|^2_{L_x^2}\big).
\end{equation}
Hence, by using (\ref{eng_3}), we can obtain directly
\begin{lemma}\label{rmd.111}
If $4\ep_0+\delta+\delta_2<s-2$, there holds
\begin{align*}
 \sum_{\mu>\La}\sum_{k=1}^{\kappa_\mu} \mu^{2\delta} B_{\mu}(t_{k-1})^2\les 1.
\end{align*}
\end{lemma}
Indeed since $\kappa_\mu<\mu^{8\ep_0}$, we can derive in view of  the energy bound in (\ref{eng_3}) that
\begin{align*}
 \sum_{\mu>\La}\sum_{k=1}^{\kappa_\mu} \mu^{2\delta} B_{\mu}(t_{k-1})^2 &\les \sum_{\mu>\La} \mu^{2\delta+2\delta_2+8\ep_0} \sup_{t\in I} \ei_\mu(t)\\
 &\les \sum_{\mu>\La} \mu^{2\delta+2\delta_2+8\ep_0-2(s-2)} \c \sup_{t\in I}\sup_{\mu>\La} \mu^{2(s-2)} \ei_\mu(t)\les 1.
\end{align*}

In view of (\ref{sumu}), Lemma \ref{rmd.1}, Lemma \ref{rmd.111}
and writing
$$
\sum_{\mu>\La} \|\mu^{\delta} P_\mu \p_m U\|_{L_{I}^2 L_x^\infty}^2
=\sum_{\mu>\La} \sum_{k=1}^{\kappa_\mu} \|\mu^{\delta} P_\mu \p_m U\|_{L_{I_k}^2 L_x^\infty}^2,
$$
we can obtain the following result.

\begin{proposition}\label{smau}
For any $q>2$ sufficiently close to $2$ and any $\delta>0$ sufficiently small such that
$4\epsilon_0+\delta_2+\delta<s-2$, where $\delta_2:=(\f12-\frac{1}{q})(1-8\epsilon_0)$, and $\delta_2+\delta<4\ep_0$,
 for  $(U, V)$ in (\ref{5.03.1.19}) and (\ref{5.03.2.19}) satisfying (\ref{lu3}) there holds
\begin{align*}
&\sum_{\mu>\La}  \|\mu^{\delta} P_\mu \p U\|_{L_{I}^2 L_x^\infty}^2\les T^{1-\frac{2}{q}}.
\end{align*}
\end{proposition}
Therefore for $\Psi=(v_+,\varrho)$,  we have obtained
\begin{equation*}
\|\p \Psi\|_{L_I^2 L_x^\infty}^2
+\sum_{\mu>\La} \|\mu^{\delta} P_\mu \p\Psi\|_{L_{I}^2 L_x^\infty}^2
  \le C T^{1-\frac{2}{q}}.
\end{equation*}

We need the following result for obtaining the control on  $\bT\Psi$.
\begin{lemma}\label{10.26.3.19}
There holds the following commutator estimate
\begin{equation*}
\|[P_\mu, G]\p f\|_{L_x^\infty}\les \mu^{-\f12} \|\p G\|_{L_x^\infty} \|\p^2 f\|_{L_x^2}, \quad \mu>1.
\end{equation*}
\end{lemma}
Indeed,  from (\ref{5.12.1.19}), (\ref{4.13.5.19}) and using Bernstein inequality, we derive
\begin{align*}
 \|[P_\mu, G]\p_m f\|_{L_x^\infty}&\les \|[P_\mu, G] \p f_{\le \mu}\|_{L_x^\infty}+\|P_\mu(\sum_{\la>\mu}(G)_\la \p  f_\la)\|_{L_x^\infty}\\
 &\les \mu^{-1}\|\p G\|_{L_x^\infty}\|\p f_{\le\mu}\|_{L_x^\infty}+\sum_{\la>\mu}\la^{-1}\| \p (G)_\la \|_{L_x^\infty}\| \p f_\la\|_{L_x^\infty}\\
 &\les\|\p G\|_{L_x^\infty} \{\mu^{-1}\sum_{l\le \mu}\|l^{\frac{3}{2}}\p f_l\|_{L_x^2}+\mu^{-\f12}\sum_{\la>\mu}(\frac{\mu}{\la})^\f12\|\la P_\la \p f\|_{L_x^2}\}\\
& \les \mu^{-\f12}\|\p G\|_{L_x^\infty}\|\p^2 f\|_{L_x^2}.
\end{align*}
This implies Lemma \ref{10.26.3.19}.

\begin{corollary}\label{10.21.1.19}
With the same choices of $q,\delta_2, \delta$ as in Proposition \ref{smau},
for $\Psi=(v_+,\varrho)$,  there hold
\begin{equation*}
\|\bT \Psi\|_{L_I^2 L_x^\infty}^2
+\sum_{\mu>\La} \|\mu^{\delta} \bar P_\mu \bT\Psi\|_{L_{I}^2 L_x^\infty}^2
  \le C T^{1-\frac{2}{q}}.
\end{equation*}
\end{corollary}
\begin{proof}
By using (\ref{4.23.1.19}) and (\ref{dcp1}), we can derive
\begin{align*}
&\bT v_+= \bT v-\bT \eta=-c^2\p \varrho -\bT \eta,\\
&\bT \varrho=-\div v=-(\p v_++\p \eta).
\end{align*}
By using (\ref{5.03.3.19}),  (\ref{4.12.5.19}) and Sobolev embedding,
\begin{align*}
&\|\bT v_+\|_{L^\infty}\le \|\p \varrho\|_{L_x^\infty}+\|\bT \eta\|_{L_x^\infty}\les \|\p \varrho\|_{L_x^\infty}+1; \\
&\|\bT \varrho\|_{L_x^\infty}\les \|\p v_+\|_{L_x^\infty}+\|\p \eta\|_{L_x^\infty}\les \|\p v_+\|_{L_x^\infty}+1.
\end{align*}
Hence the $L_I^2 L_x^\infty$ estimates in Corollary \ref{10.21.1.19} follows immediately as a consequence of the $L_I^2 L_x^\infty$ estimates in Proposition \ref{smau}.

By Sobolev embedding and (\ref{4.12.5.19}), we derive
\begin{equation*}
\|\mu^\delta \bar P_\mu \bT \eta\|_{l_\mu^2 L_x^\infty}\les\|\bT \eta\|_{H^2_x}\les 1
\end{equation*}
Applying Lemma \ref{10.26.1.19} to $G=c^2$,  together with using (\ref{5.04.17.19}) implies
\begin{equation*}
\|\mu^\delta \bar P_\mu(c^2 \p \varrho)\|_{l_\mu^2 L_I^2 L_x^\infty}\les \|\mu^\delta \bar P_\mu \p \varrho\|_{L_I^2 l_\mu^2 L_x^\infty}+\|\p \varrho\|_{L_I^2  L_x^\infty}.
\end{equation*}
These two estimates imply  $\sum_{\mu>\La} \|\mu^{\delta} \bar P_\mu \bT v_+\|_{L_{I}^2 L_x^\infty}^2
  \le C T^{1-\frac{2}{q}},$ by using the estimate of $\p\varrho$ in  Proposition \ref{smau}.

Note by Berntein inequality,  (\ref{4.12.2.19}) and (\ref{4.25.1.19}), we derive
\begin{equation*}
\|\mu^\delta \bar P_\mu \p \eta\|_{l_\mu^2 L_x^\infty}\les \|\mu^{\delta+\frac{3}{2}}\bar P_\mu \p \eta\|_{l_\mu^2 L_x^2}\les\|\curl \Omega\|_{H^{\f12+\delta}_x}+1\les 1.
\end{equation*}
Combining  the above estimate with the estimate of $\p v_+$ in Proposition \ref{smau} gives
\begin{equation*}
\|\mu^\delta \bar P_\mu \bT \varrho\|_{L_I^2 l_\mu^2 L_x^\infty}\les T^\f12+\|\mu^\delta \bar P_\mu \p v_+\|_{L_I^2 l_\mu^2 L_x^\infty}\les T^{\f12-\frac{1}{q}}.
\end{equation*}
Hence the proof of Corollary  \ref{10.21.1.19} is complete.
\end{proof}

Rename the choice of $\delta, q$ in Remark \ref{9.25.2.19} by $\ga_0=\delta$, and $\ga_1=\f12-\frac{1}{q}$. The proof of Theorem \ref{main} is complete.

\subsection{Prove Theorem \ref{dyastric} using dispersive estimate}

In order to prove Theorem \ref{dyastric} on each spacetime slab $I_k\times {\Bbb R}^3$, we consider the coordinate change
$(t,x)\rightarrow(\la ( t-t_k), \la x)$. The interval $I_k$ becomes $I_*=[0, \tau_*]$ with $\tau_*\le \la^{1-8\ep_0}T$.
Under the rescaled coordinates the function $\Phi=(\varrho, v)$ and the metric component $\bg$  in (\ref{metric}) become
\begin{align}\label{7.26.1}
\Phi(\la^{-1}t+t_k, \la^{-1}x) \quad \mbox{and} \quad \bg(\Phi(\la^{-1}t+t_k, \la^{-1}x))
\end{align}
which are still denoted as $\Phi$ and $\bg$. In view of (\ref{BA2}), Corollary \ref{comp_3} and $|I_k|\le \la^{-8\ep_0} T$,
we have $$\|\bp \bg\|_{L_{I_*}^1 L_x^\infty}\les \la^{-8\ep_0}.$$ Therefore, to prove Theorem \ref{dyastric}, it is equivalent to show
the following Strichartz estimate on $I_*$ with respect to Littlewood-Paley projection $P$ on the frequency domain $\{1/2\le |\xi|\le 2\}$.
Here we fix the convention that $P=P_1$, which is the Littlewood-Paley projection $P_\la$ with $\la=1$.

\begin{theorem}\label{str2}
If there is a large number  $\La$  such that for $\la\ge \La$,
 on the time interval $I_*=[0,\tau_*]$, there holds
\begin{equation}\label{smallas}
\|\bp v_+, \bp \varrho \|_{L_{I_*}^2 L_x^\infty}^2+\la^{2\delta_0}\sum_{\mu \ge 2}\mu^{2\delta_0}\|\bar P_\mu \bp (v_+, \varrho)\|_{L_{I_*}^2 L_x^\infty}^2\les \la^{-1-8\ep_0}
\end{equation}
then for any solution $\psi$ of $\Box_{\bg} \psi=0$ on the time interval $I_*$ and $q>2$ sufficiently close to $2$, there holds
\begin{equation}\label{str1}
\|P \bp \psi\|_{L_{I_*}^q L_x^\infty}\les \|\psi[0]\|_{\H}.
\end{equation}
\end{theorem}
The proof of Theorem \ref{str2} crucially relies on the following decay estimate.

\begin{theorem}[Decay estimate]\label{decayth}
Let $0<\ep_0<\frac{s-2}{5}$ be a fixed number. There exists a large number $\La$ such that for
any $\la\ge \La$ and any solution $\psi$ of the equation
\begin{equation}\label{wave.4}
\Box_\bg \psi=0
\end{equation}
on the time interval $I_*=[0, \tau_*]$ with $\tau_*\le \la^{1-8\ep_0} T$, there is a function $\mathfrak{d}(t)$
satisfying
\begin{equation}\label{correccondi}
 \|\mathfrak{d}\|_{L^{\frac{q}{2}}}\les 1, \mbox{ for } q>2 \mbox{ sufficiently close to }2
 \end{equation}
such that for any $0\le t\le \tau_*$ there holds
\begin{equation}\label{decay}
 \|P\bT \psi(t) \|_{L_x^\infty}\le \left(\frac{1}{{(1+t)}^{\frac{2}{q}}}+\mathfrak{d}(t)\right)
\left(\sum_{m=0}^3\|\p^m \psi(0)\|_{L_x^1}+\sum_{m=0}^2 \|\p^m \bT \psi(0)\|_{L_x^1}\right).
\end{equation}
\end{theorem}

Assuming Theorem \ref{decayth}, we can prove Theorem \ref{str2} by running the $\T \T^*$ argument. See \cite[Section 4]{WangCMCSH} and \cite[Section 9]{Wangrough}.

To prove  Theorem \ref{decayth}, we carry out a further localization of the solution for (\ref{wave.4}) in  physical space  in  $I_*\times {\mathbb R}^3$ with $I_*=[0, \tau_*]$ under the
rescaled coordinates, where $\tau_*\le\la^{1-8\ep_0} T$. For each $t$ denote by $g(t)$ or $g$ the induced Riemannian metric of (\ref{7.26.1}) on  $\Sigma_t =\{t\}\times {\mathbb R}^3$. Given $d>0$ and ${\bf p}\in \Sigma_t$ we use
$B_d({\bf p})$ and $B_d({\bf p}, g)$ to denote the Euclidean ball and the geodesic ball on $\Sigma_t$ with respect to $g$.
We can find $R>0$ such that
\begin{equation}\label{8.7.1}
B_R({\bf p}) \subset B_{\f12}({\bf p}, g(t)), \quad \forall \,{\bf p}\in \Sigma_t \mbox{ and } 0\le t\le \tau_*.
\end{equation}
This is achievable due to the ellipticity condition from (\ref{9.20.1.19}). Now we fix $t_0=1$. We take a sequence of Euclidean
balls $\{B_J\}$ with radius $R$ such that their union covers ${\mathbb R}^3$ and any ball in this collection intersect
at most $20$ other balls. Let $\{\sX_J\}$ be a partition of unity subordinate to the cover $\{B_J\}$. We may assume that
$\sum_{m=1}^3|\p^m\sX_{J}|_{L_x^\infty}\le C_1$ holds uniformly in $J$. By using this partition of unity and a standard argument
we can reduce the proof of Theorem \ref{decayth} to establishing the following dispersive estimate for the solution of
$\Box_{\bg} \psi =0$ with $\psi[t_0]:=(\psi(t_0), \p_t \psi(t_0))$ supported on an Euclidean ball of radius $R$.

\begin{proposition}\label{lcestimate}
There is a large constant $\La$ such that for any $\la\ge \La$ and
any solution $\psi$ of
\begin{equation*}
\Box_\bg \psi=0
\end{equation*}
on the time interval $[0, \tau_*]$ with $\tau_*\le \la^{1-8\ep_0} T$ and with $\psi[t_0]$ supported in the Euclidean
ball $B_{R}$  of radius $R$, there exists a function $\mathfrak{d}(t)$ satisfying
\begin{equation}\label{correc}
 \|\mathfrak{d}\|_{L^\frac{q}{2}[0,\tau_*]}\les 1 \quad  \mbox{for  } q>2 \mbox{ sufficiently close to }2
 \end{equation}
 such that for all $t_0\le t\le \tau_*$,
 \begin{equation}\label{decaycp}
 \|P \bT \psi(t)\|_{L_x^\infty}\les \left(\frac{1}{{(1+|t-t_0|)}^{\frac{2}{q}}}+\mathfrak{d}(t)\right)(\|\psi[t_0]\|_\H+\|\psi(t_0)\|_{L^2}).
 \end{equation}
\end{proposition}

In order to show  Theorem \ref{decayth}  by using Proposition \ref{lcestimate}, we  first need the standard energy control.
\begin{lemma}\label{geoeg}
Under the bootstrap assumption (\ref{BA1}), there holds, for any solution $\psi$ of (\ref{wave.4}), the standard energy estimate
\begin{equation}\label{geoge}
\|\psi[t]\|_\H\les \| \psi[0]\|_\H,
\end{equation}
and for $0<t\le t_0$,
\begin{equation*}
\|\psi(t)\|_{L_x^2}\les  \|\psi[0]\|_\H+\|\psi(0)\|_{L_x^2}.
\end{equation*}
\end{lemma}

\begin{proof}
(\ref{geoge}) follows from (\ref{eng7.10}). The other inequality then can be obtained by using the fundamental theorem of calculus.
\end{proof}

\begin{proof}[Proof of Theorem \ref{decayth} assuming Proposition \ref{lcestimate}]
For $0<t<t_0$, we may use the Bernstein inequality of LP projections to obtain
\begin{equation*} 
\|P \bT \psi(t)\|_{L_x^\infty}\les \|\bT \psi(t)\|_{L_x^2}.
\end{equation*}
Also using Lemma \ref{geoeg} and the Sobolev embedding $W^{2,1} \hookrightarrow L^2$ in ${\mathbb R}^3$,
we can obtain the desired inequality. Thus we can assume $t_0\le t\le \tau_*$. By using the partition of unity $\{\sX_J\}$ we decompose
$
\psi =\sum_J \psi_J,
$
with $\psi_J$  the solution of $\Box_{\bg} \psi_J=0$ satisfying the initial conditions
\begin{equation*}
\psi_J(t_0)=\sX_J \psi(t_0),\qquad  \p_t \psi_J(t_0)=\sX_J  \p_t \psi(t_0).
\end{equation*}
By using (\ref{decaycp}) in Proposition \ref{lcestimate}, we have
\begin{equation*} 
 \|P \bT \psi_J(t)\|_{L_x^\infty}\les \left(\frac{1}{{(1+|t-t_0|)}^{\frac{2}{q}}}+\mathfrak{d}(t)\right)(\|\psi_J[t_0]\|_\H+\|\psi_J(t_0)\|_{L^2}).
\end{equation*}
In view of Lemma \ref{geoeg} and the Sobolev embedding $W^{2,1}\hookrightarrow L^2$ in ${\mathbb R}^3$, we obtain
\begin{equation*} 
\|P \bT \psi_J(t)\|_{L_x^\infty}\les \left(\frac{1}{{(1+|t-t_0|)}^{\frac{2}{q}}}+\mathfrak{d}(t)\right)
\left(\sum_{m=0}^3 \|\p^m \psi_J(0)\|_{L_x^1} + \sum_{m=0}^2 \|\p^m \bT  \psi_J(0)\|_{L_x^1}\right).
\end{equation*}
Summing over $J$ and using the fact that any ball $B_J$ intersects with at most $20$ other balls, we can conclude
the desired estimate.
\end{proof}

\section{Reduction to the boundedness of conformal energy}\label{null_cone_set}
The main task now is to prove Proposition \ref{lcestimate}, which will be reduced further to controlling conformal energy. To define the conformal energy, we set up a foliation of  the acoustic spacetime by outgoing acoustic null cones, which  covers the domain of influence of $ B_{\f12}({\bf p}, g(t_0))$.

\subsection{Construction of foliations of acoustic null cones}\label{setupcone}
 Let  $\bf p$ be the center of $B_R$ in Proposition \ref{lcestimate} at  $t=t_0$.
We  denote by $\Ga^+$ the time axis passing through $\bf p$ which
is defined to be the integral curve of the forward unit normal $\bT$ with $\Ga^+(t_{\bf p})=\bf p$. We similarly extend the integral curve of $-\bT$ from $t=t_{\bf p}$ till $t=0$, which is still denoted as $\Ga^+$ by abuse of notation, and  denote the point $\Ga^+(0)={\bf o}$.   We will only consider the segment of $\Ga^+(t)$ with $t\in[0,\tau_*]$.

Define the optical function $u$ to be the solution of the Eikonal equation
\begin{equation}\label{optical}
\bg^{\a\b}\p_\a u \p_\b u=0.
\end{equation}
We refer to the level sets of the optical function $u$, denoted by $C_u$, as the acoustic cones. The global optical function $u$ is constructed as follows.

Let $p$ be any point on $\Ga^+$, and $L_\omega, \omega \in {\mathbb S^2}$ be the family of null vectors in $\T_p\M$. For each  $\omega\in {\mathbb S}^2$,  define the vector field $L'$  to be the generator of  the null geodesic $\Upsilon_\omega$ in $(\M, \bg)$ by
\begin{equation}\label{6.29.2.19}
\bd_{L'} {L'}=0, \quad \frac{d}{ds}\Upsilon_\omega(s)=L',\quad  L'(s)=1,
\end{equation}
and when $s=0$, $L'=L_\omega$. The null vector $L_\omega$ can be decomposed as $L_\omega=\bT+\bN_\omega$, where $\bN_\omega, \omega\in {\mathbb S}^2$ is the family of the unit vectors in $\T_p \Sigma_{t_p}$.
 We set $u=t$ at $p\in \Ga^+$. The ruled hypersurface formed by $\{\Upsilon_\omega, \omega\in{\mathbb S}^2\}$ is the level set of $u$, which is denoted by $C_u$. This immediately yields $L'(u)=0$. One can direct check via  (\ref{6.29.2.19}) and Eikonal equation  that $L'=-\bd u$, and clearly
\begin{equation}\label{9.26.1.19}
\bT u=1, \qquad \mbox{on }\Ga^+(t_p).
\end{equation}

As such, the optical function $u$ has been defined in the causal future of $\bf o$, denoted by $\D^+$, with the level sets $C_u$ being  the outgoing null (acoustic) cones
with vertex on $\Ga^+$ at $t=u$. With $S_{t,u}=C_u\cap \Sigma_t$, which is a smooth surface diffeomorphic to ${\mathbb S}^2$,  we denote the two solid cones
$$
\D^+_0=\bigcup_{\{t\in [t_0,\tau_*],0\le u\le t\}} S_{t,u} \quad \mbox{and}\quad
\D^+=\bigcup_{\{t\in [0,\tau_*], 0\le u\le t\}} S_{t,u}.
$$

Next we extend the foliation of spacetime by the null cones to a neighbourhood of  $\D^+$ in $\bigcup_{t\in[0, \tau_*]} \Sigma_t$.
 Recall that $\tau_*\le \la^{1-8\ep_0} T$ and $0<T\le d_0$, where $d_0>0$ is the radius of injectivity on $\Sigma_t$ for $t\le [0,T]$.
Let $\fv_*=\frac{4}{5} \tau_*$. We can guarantee that there is a neighbourhood ${\mathscr O}$ of ${\bf o}$ contained in the geodesic ball
$B_{\tau_*}({\bf o}, g)$ of radius $\tau_*$ on $\{t=0\}$ such that it can be foliated by the level sets $S_\fv$ of a function $\fv$
taking all values in $[0, \fv_*]$ with $\fv({\bf o})=0$ and with each $S_\fv$, for $\fv>0$, diffeomorphic to ${\mathbb S}^2$;
see Proposition \ref{exten} shortly for various important properties.
Let $a^{-1}=|\nab \fv|_g$ be the lapse function on $\{t=0\}$ and we know $\lim_{\fv \rightarrow 0}a=1$ in Proposition \ref{exten}. Then, in ${\mathscr O}$ the metric $g$ can be written as
\begin{equation}\label{radial}
ds^2=a^{2} d\fv^2+\ga_{AB} d\omega^A d\omega^B,
\end{equation}
where $\omega^A$, $A=1,2$, denote the angular variables on ${\mathbb S}^2$ and $\ga$ is the induced metric on $S_\fv$. At $t=0$, we denote by $\bN=\bN(\fv,\omega)$ the outward unit normal of the foliation of $S_\fv$ with $[0,\fv_*]$ and we also note that $\bN\rightarrow \bN_\omega$ as $\fv\rightarrow 0$ for $\omega\in {\mathbb S}^2$.

With the initial datum  at $S_\fv$  given by
\begin{equation}\label{6.29.1.19}
L'=a^{-1}(\bT+\bN), \quad 0<\fv\le \fv_*,\,  t=0,
\end{equation}
 we define $L', \Upsilon_\omega$ to be the vector field and the null geodesic  satisfying (\ref{6.29.2.19}).  By setting  $u=-\fv$, the level set of $u$ which is $C_u$ is  the ruled hypersurface of the null geodesics $\{\Upsilon_\omega, \omega\in {\mathbb S}^2\}$, and thus $L'(u)=0$. And we can check $L'=-\bd u$.

This gives an extension
of $u$ satisfying (\ref{optical}) in the causal future of $\cup_{0\le \fv\le \fv_*}S_\fv$, denoted by $\widetilde {\D^+}$. It is foliated by the null cones $C_u$, initiating
from $S_\fv$ with  $u=-\fv$ at $t=0$, and those initiating from the time axis $\Ga^+$ with $u=t$. We still denote $S_{t, u}=C_u\cap \Sigma_t$  for $-\fv_*\le u<0$ and $t\in [0, \tau_*]$. $\widetilde{\D^+}$ can be written as
$$
\widetilde{\D^+}=\bigcup_{\{t\in [0,\tau_*], -\fv_*\le u\le t\}} S_{t,u}.
$$

Define in $\widetilde{\D^+}$ for all $0<t\le \tau_*$ that
$$
\bb^{-1}:=\bT u=-\l L',\bT\r.
$$
By (\ref{6.29.1.19}) and (\ref{9.26.1.19}), we can see  the initial and the boundary value of $\bb$ verify $\bb=a \mbox{ at }t=0$  and $\bb=1$ on $\Ga^+(t)$.

Next we directly compute
\begin{equation*}
\nab u=\bd u+\l \bd u, \bT\r \bT=-L'+\bb^{-1} \bT
\end{equation*}
and thus $|\nab u|=\bb^{-1}$. Let $\bN$ be the outward unit normal of the radial foliation $S_{t,u}$ with $-\fv_*\le u\le t$ on $\Sigma_t$. By the above calculation
\begin{equation*}
\bN=\bb L'- \bT.
\end{equation*}
Moreover, also using  $L'=-\bd u$,
\begin{equation*}
-\bN(u)=\bb^{-1}=\bT u=|\nab u|.
\end{equation*}
We can directly see that
\begin{align*}
&\bN\rightarrow \bN_\omega \big(\Ga^+(t)\big), \quad \bb\rightarrow 1 \mbox{ as }u\rightarrow t.
\end{align*}

In Lemma \ref{6.17.1}, we will show that $|\bb-1|\le \frac{1}{4}$. Assuming this property, we can prove as in \cite[Section 4]{Wangrough} that
\begin{equation*}
B_{\f12}({\bf p}, g(t_0)) \subset \D_0^+\cap \Sigma_{t_0},
\end{equation*}
which implies $B_R(\bf p)\subset\D_0^+\cap \Sigma_{t_0}$. 

For convenience in $\widetilde{\D^+}$, we introduce the pair of null frames
\begin{equation}\label{9.11.4.19}
L=\bb L'=\bT+\bN, \quad \Lb= \bT-\bN,
\end{equation}
and define the projection tensor as
\begin{equation}\label{metric_2}
\Pi^{\mu \nu}=\bg^{\mu\nu}+\bT^\mu \bT^\nu-\bN^\mu \bN^\nu
\end{equation}
or identically
\begin{equation*}
\Pi^{\mu\nu}=\bg^{\mu\nu}+\f12 (L^\mu \Lb^\nu+L^\mu \Lb^\nu).
\end{equation*}
Since $\bN$ is the unit normal to $S_{t,u}$ in $\Sigma_t$, $\Pi_{\mu\nu}$ gives the induced metric on $S_{t,u}$, which will be denoted as $\ga$, with $\sn$ the corresponding Levi-Civita connection on $S_{t,u}$. We can write down the line element of the metric $g$ as
\begin{equation}\label{metric_II}
ds^2=\bb^2 du^2+\ga_{AB}d\omega^A d\omega^B.
\end{equation}

Let $\{e_A, e_B\}$ with $A,B=1,2$  be the orthonormal basis of the tangent bundle on $S_{t,u}$.
We now introduce the connection coefficient on $C_u$ by using the null pair  $e_4=L,\, e_3=\Lb$. The null second fundamental forms $\chi$ and $\chib$,
the torsion $\zeta$, and the Ricci coefficient $\zb$ of the foliation $S_{t,u}$ are defined by
\begin{equation}\label{ricc_def}
\begin{split}
\chi_{AB}=\bg(\bd_A e_4, e_B), &\qquad \chib_{AB}=\bg (\bd_A e_3, e_B),\\
\zeta_A=\f12 \bg(\bd_3 e_4, e_A), &\qquad \zb_A=\f12 \bg (\bd_4 e_3, e_A).
\end{split}
\end{equation}
 We denote by $\tr\chi$ and $\chih$ the trace and traceless part of $\chi$ by the metric $\ga$, and apply the same convention to $\chib$.

 On each $C_u \cap \widetilde{\D^+}$, with $t_\tmin =\max\{u, 0\}$,  we write $\Upsilon(t,\omega):=\Upsilon(s(t, \omega),\omega)$  for $t_\tmin\le t\le \tau_*$ by change of parameter. The pull-back  coordinates $(t, \omega^1, \omega^2)$ induced by the null geodesic flow $\Upsilon(t,\omega), \omega\in {\mathbb S}^2 $ along $C_u$ together with the function $u$  define a complete system of coordinates on $\widetilde{\D^+}$. Indeed,
 the short-time well-posedness of the null geodesic family $
 \{\Upsilon(s(t), \omega, u)\}$, for all $\omega\in {\mathbb S}^2, -\fv_*\le u\le t$ can follow from the standard ODE theory. The semi-global behaviour throughout $\widetilde{\D^+}$, or more precisely,  the global diffeomorphism of the null exponential map,   $\Upsilon: [t_\tmin, \tau_*] \times {\mathbb S}^2\rightarrow C_u\cap \widetilde{\D^+}$ for all $C_u$ contained in $\widetilde{\D^+}$, can be guaranteed   by running the same continuity  argument for proving \cite[Theorem 1.2]{Wang10}, based on  (\ref{comp2}), Proposition \ref{ricpr}, (\ref{bb_4}) and (\ref{BA1}).

Hence for each $p\in \widetilde{\D^+}$, there exists a unique triplet $(t,\omega,u )$  such that  $p=\Upsilon(t,\omega,u)$.
Corresponding to the case that $t=0$ in the triplet, in view of (\ref{radial}), we note on each $S_\fv$, $\{\p_{\omega^1}, \p_{\omega^2}\}$ is the pair of pull-back coordinate frame   by the diffeomorphism  ${\mathbb S}^2\rightarrow S_\fv$ with $0< \fv\le \fv_*$. (See more details in the proof in \cite{roughgeneral_online} for Theorem \ref{exten}.)

For $t>0$, we introduce the transport coordinate on $C_u$ by
\begin{equation}\label{trscoord}
\frac{d}{dt} x^\mu\big(\Upsilon(t, \omega)\big)= L^\mu(t, \omega), t>0
\end{equation}
  and adopt the pull-back coordinate frame $\{\p_{\omega^A}, A=1,2\}$ on $S_{t,u}$ defined by the diffeomorphism $\Upsilon(t,\cdot, u):{\mathbb S}^2\rightarrow S_{t,u}, t>0$.   Along the cone $C_u$,  $L=\p_t$\begin{footnote}{This is the partial differentiation with $u$ and $\omega$ fixed, instead of the $\p_t$ in the cartesian frame.}\end{footnote} together with the
the pull-back coordinate frame $\{\p_{\omega^1}, \p_{\omega^2}\}$ forms a set of coordinate frame on $C_u$.
We can derive
\begin{equation}\label{trscoord2}
\begin{split}
\frac{d}{dt}\ga(\p_{\omega_A}, \p_{\omega_B})&=\ga(\bd_L \p_{\omega_A},\p_{\omega_B})+\ga(\p_{\omega_A}, \bd_L \p_{\omega_B})\\
&=\ga(\bd_{\p_{\omega_A}}L ,\p_{\omega_B})+\ga(\p_{\omega_A}, \bd_{\p_{\omega_B}} L)\\
&=2\chi(\p_{\omega_A},\p_{\omega_B}).
\end{split}
\end{equation}

The second fundamental form of $S_{t,u}$ in $\Sigma_t$ for $0\le t\le\tau_*$ is given by
\begin{equation}\label{7.15.7.19}
\theta(X, Y)=\l \nab_X \bN, Y\r
\end{equation}
for any vector fields $X, Y$ tangent to $S_{t,u}$. The trace of $\theta$ is defined by $\tr\theta=\ga^{AB} \theta_{AB}$, and the traceless part of $\theta$ is denoted by $\hat \theta$.

Let
\begin{equation*}
v_t=\frac{\sqrt{|\ga|}}{\sqrt{|\ga\rp{0}|}},
\end{equation*}
where $\ga\rp{0}$ is the canonic round metric on ${\mathbb S}^2$. Since $v_t$ is defined on $S_{t,u}$, we may write it as $v_{t,u}$ in particular when $u$ is also varying.
 By  $\Lie_L\ga=2\chi$,
\begin{equation}\label{lv}
L(v_t)= v_t \tr\chi.
\end{equation}

We define  $S_{t,u}$-tangent tensor field $F$ if $F$ verifies $i_L F=0$ and $i_\Lb F=0$. For such tensor fields,
 $|F|$ is the norm of $F$ under the  induced metric $\ga$. We will use the two norms
$$
\|F\|_{L_x^q(S_{t, u})}^q= \int_{S_{t, u}} |F|^q d\mu_\ga\quad \mbox{and} \quad
\|F\|_{L_\omega^q(S_{t, u})}^q = \int_{{\mathbb S}^2} |F|^q(\omega) d \mu_{{\mathbb S}^2}.
$$

For $S_{t, u}$-tangent tensor field $F$ on $C_u$, we introduce the mixed norms
\begin{align*}
\|F\|_{L_\omega^q L_t^\infty(C_u)}^q&= \int_{{\mathbb S}^2} \sup_{\Upsilon_\omega}|F|^q   d\mu_{{\Bbb S}^2}\quad
\mbox{and} \quad \|F\|_{L_x^q L_t^\infty(C_u)}^q=\int_{{\Bbb S}^2} \sup_{t\in \Upsilon_\omega}(v_{t} |F|^q) d\mu_{{\Bbb S}^2}
\end{align*}
and
\begin{align*}
 &\|F\|_{L_t^p L_x^q(C_u)}^p=\int_{t_1}^{t_2}(\int_{{\Bbb S}^2} |F|^q v_{t'} d\mu_{{\Bbb S}^2})^\frac{p}{q} dt'.
 \end{align*}
For tensor fields $F$ defined on $\Sigma_t\cap \widetilde{\D^+}$ we use the norms
\begin{align*}
& \|F\|_{L_u^p L_x^q}^p=\int_{u_1}^{u_2}(\int_{{\Bbb S}^2} |F|^q v_{t, u'} d\mu_{{\Bbb S}^2})^\frac{p}{q} du',\\
&\|F\|_{L_x^q L_u^\infty}=\left(\int_{{\Bbb S}^2}(\sup_{u}(v_t |F|^q ))(\omega) d\mu_{{\Bbb S}^2}\right)^\frac{1}{q}.
 \end{align*}
 We may  write $d \mu_{{\Bbb S}^2}$ as $d\omega$ for convenience. The ranges $t_1$, $t_2$ and $u_1$ and $u_2$ are determined by the integral regions.

Finally, we recall that the existence of the $\fv$-foliation with the desired properties in a neighbourhood of ${\bf o}$ on $\{t=0\}$ is guaranteed by the following result.
The proof of the following result depends on the Ricci curvature of the induced metric $g$ on $\Sigma_t$,  which is of the same regularity level as in \cite{Wangrough}. See the proof in \cite[Section 10]{roughgeneral_online}.

\begin{proposition}\label{exten}
On $\{t=0\}$ there exists a function $\fv$ with $0\le \fv\le \fv_*=\frac{4}{5}\tau_*$ such that each level set $S_\fv$ is diffeomorphic to ${\Bbb S}^2$ and
\begin{equation}\label{amc}
{\emph\tr} \, \theta + k_{\bN\bN}=\frac{2}{a\fv} + {\emph\Tr} \, k-\Xi_4, \qquad a({\bf o})=1.
\end{equation}
Let $\ga^{(0)}$ be the canonical round metric on ${\Bbb S}^2$ and $\ga$ be the induced metric of $g$ on $S_\fv$. Let $\cga=\fv^{-2} \ga$
and $\ckk\ga=\fv^2\ga^{(0)}$. Then on $\cup_{0\le \fv\le \fv_*}S_\fv$ there hold
\begin{align}
&|a-1|\les \la^{-4 \ep_0} <\frac{1}{4},\quad \|\fv^{\f12-\frac{2}{q_*}}(\hat\theta, \sn \log a) \|_{L^{q_*}(S_\fv)}\les \la^{-\f12},\label{a_3} \\
&\|\sn \log a\|_{L_\fv^2 L^\infty_{S_\fv}} +\| \chih\|_{L_\fv^2 L^\infty_{S_\fv}}\les\la^{-\f12}, \label{a_4}\\
&|\cga-\ga^{(0)}|+\|\p_\omega(\cga-\ga^{(0)})\|_{L_\omega^{q_*}(S_\fv)} \les \la^{-4\ep_0}, \label{a_5}\\
&\|\fv^{\f12-\frac{2}{q_*}}\sn (\log \sqrt{|\ga|}-\log \sqrt{|\ckk\ga|})\|_{L^{q_*}(S_\fv)}\les \la^{-\f12}, \label{4a_6}
\end{align}
where $0<1-\frac{2}{q_*}< s-2$ and for scalar functions $f$, $\|f\|_{L_\omega^{q_*}(S_\fv)}^{q_*}:=\int_{{\Bbb S}^2}|f|^{q_*}(\omega) d\mu_{{\Bbb S}^2}$.
 Moreover
\begin{equation}\label{w8.1.1}
|a-1|\les \la^{-4\ep_0},\qquad \|\fv^{-\f12}(a-1)\|_{L^\infty}\les \la^{-\f12}, \qquad\frac{\sqrt{|\ga|}}{\sqrt{|\ga^{(0)}|}}\approx \fv^2
\end{equation}
and there holds the inclusion
\begin{equation}\label{w7.13.1}
\cup_{0\le \fv\le \fv_*} S_\fv\subset B_{\tau_*}({\bf o}).
\end{equation}
\end{proposition}

\subsection{Reduction from the dyadic Strichartz estimates to the boundedness of conformal energy}\label{10.30.1.19}

In order to give the definition of our conformal energy, we take two smooth cut-off functions $\underline{\varpi}$ and $\varpi$
depending only on two variables $t,u$; for $t>0$ they are defined in a manner such that
\begin{equation*}
\underline{\varpi}=\left\{\begin{array}{lll}
1  & \mbox{ on }\,  0\le u\le t\\
0  & \mbox { on } u \le -\frac{t}{4} ,
\end{array}\right.
\quad \mbox{and} \quad
 \varpi=\left\{\begin{array} {lll}
 1  & \mbox { on } 0\le\frac{u}{t}\le \f12 \\
 0  & \mbox{ if } \frac{u}{t}\ge \frac{3}{4} \mbox { or } u\le  -\frac{t}{4}.
\end{array}\right.
\end{equation*}
We may define $\varpi$ and $\underline{\varpi}$ such that they coincide in the region $\cup_{\{t\in [t_0, \tau_*],-\frac{t}{4}<u\le 0 \}}S_{t,u}$.

\begin{definition}\label{cfen}
For any scalar function $\psi$ vanishing outside $\D^+$,  we define the conformal energy $\CC[\psi]$ of $\psi$ by
 \begin{equation*}
\CC[\psi](t)=\CC[\psi]^\bi(t)+\CC[\psi]^\be(t),
 \end{equation*}
 where
 \begin{align}\label{confen}
 &\CC[\psi]^\bi(t)=\int_{\Sigma_t} (\underline{\varpi}-\varpi)t^2\left(|\bd \psi|^2+|\tir^{-1} \psi|^2\right) d\mu_g,\\
 &\CC[\psi]^\be(t)=\int_{\Sigma_t}\varpi \left(\tir^2|\bd_L \psi|^2+\tir^2|\sn \psi|^2+|\psi|^2\right) d\mu_g.
 \end{align}
\end{definition}

 We will prove the following boundedness theorem  for the conformal energy in the rest of this paper, combined with  \cite[Section 7]{Wangrough}.

\begin{theorem}[Boundedness theorem]\label{BT}
Let (\ref{smallas}) hold. Let $\psi$ be any solution of $\Box_\bg \psi=0$ on $I_*=[0, \tau_*]$ with $\psi[t_0]$ supported in
$B_R\subset \D^+\cap \Sigma_{t_0}$.  Then, for $t\in [t_0, \tau_*]$, the conformal energy of $\psi$ satisfies the estimate
$$
\CC[\psi](t)\les(1+t)^{2\ep}\left(\|\psi[t_0]\|_{\H}^2+\|\psi(t_0)\|_{L^2(\Sigma)}^2\right),
$$
where $\ep>0$ is an arbitrarily small number.
\end{theorem}

Under the assumption (\ref{smallas}), to show Theorem \ref{BT} implies Proposition \ref{lcestimate}, we refer to \begin{footnote}{See also \cite[Section 4.3]{WangCMCSH}.}\end{footnote} \cite[Section 4.1]{Wangrough}, for which we need the following results on $\widetilde{\D^+}$,
\begin{equation}\label{pba2}
\|\varpi\big(\chih, \sn \log \bb, \tr\chi-\frac{2}{\tir}\big)\|_{(L^{\frac{q}{2}}[0,\tau_*] L^\infty_x\cap \widetilde{\D^+})}\le C \la^{\frac{2}{q}-1-4\ep_0(\frac{4}{q}-1)},
\end{equation}
where $\tir=t-u$, $q>2$ and is sufficiently close to $2$;
\begin{equation}\label{Pba}
|\bb-1|\le \f12,\,\quad \, \|\tr\theta-\frac{2}{\tir}\|_{L^3(\Sigma_t\cap \widetilde{\D^+})}\le C,
\end{equation}
and the estimate
 \begin{equation}\label{wras}
C^{-1}\ga^{(0)}(X,X) \le \tir^{-2}\ga(X,X)\le C\ga^{(0)}(X,X), \quad  v_t/\tir^2\approx 1,
 \end{equation}
 where  $X$ is any $S_{t,u}$ tangent vector field with $S_{t,u}$ contained in $\widetilde{\D^+}$ and $C>0$ is a universal constant.

The assumptions (\ref{Pba}) and (\ref{wras})  are used to prove the scaling-invariant inequalities in Lemma \ref{trace2} and Proposition \ref{basic1} which are involved in proving Proposition \ref{lcestimate}.
  (\ref{wras}) ensures that the area element $v_t$ and $\tir^2$ are comparable,  and the first assumption in (\ref{Pba}) implies $\bb$ can be regarded as a positive constant away from zero.
Thus for any tensor $F$ on $S_{t, u}$ and $1\le q<\infty$ we have
$$
\|F\|_{L^q(S_{t, u})}\approx \|\tir^{\frac{2}{q}} F\|_{L_\omega^q(S_{t, u})}.
$$
 With the series of reduction, the proof of the main theorem, Theorem \ref{thm1}, has been reduced to the proof of  Theorem \ref{BT}.

The proofs of Theorem \ref{BT} had always been the most important part of the series of works on rough solutions for quasilinear wave equations \cite{KRduke, KREinst,WangCMCSH, Wangrough}.   In particular, due to the optimal regularity assumption on the data in \cite{Wangrough}, the proof of the boundedness theorem in \cite{Wangrough} is completely different from the previous works and the result has a growth with time, i.e. $(t+1)^{2\ep}$. The reason of such harmless loss is due to the weak regularity of the spacetime metric and the null hypersurfaces therein. The proof in \cite[Section 7]{Wangrough} contains two major ingredients: one is to use the conformal method to normalize the null cones; the other is to  adapt the hierarchic approach in \cite{Mih_rod} to the rough spacetime to obtain the weighted energy flux together with the weighted  energy, so as to compensate the weak control in particular on the normalized mass aspect function (see (\ref{6.23_mu})). The analysis in the proof is hard to be relaxed further.

 In our situation, the appearance of the rough vorticity derivative in (\ref{4.10.1.19}) lowers the regularity of the Ricci curvature significantly, which makes it much harder to control the null hypersurfaces of the spacetime time. Our task is to gain the complete set of the geometric control required in \cite[Section 7]{Wangrough} by utilizing the geometric structures of the acoustic spacetime derived in Section \ref{fundstr}. The first set of estimates will be achieved in  Proposition \ref{cone_reg} and Proposition \ref{ricpr}, which will complete the proof of (\ref{pba2})-(\ref{wras}), (see Remark \ref{10.26.4.19}). We will control the normalized mass aspect function and the conformal factor for applying the conformal method in Section \ref{conf}.

Next, we provide a set of analytic tools under the original coordinates in the region of $\widetilde{\D^+}$, for which, we rely on bootstrap assumptions (\ref{Pba}) and (\ref{wras}).

 \begin{proposition}\label{basic1}
 Under the assumption (\ref{wras}),
 there hold the following Sobolev inequalities

(i) For $2\le q<\infty$ and any $S_{t,u}$-tangent tensor $F$, there hold
\begin{equation}\label{sob.12}
\|\tir^{1-2/q}F\|_{L^q(S_{t,u})}\les\|\tir\sn F\|_{L^2(S_{t,u})}^{1-2/q}\|F\|_{L^2(S_{t,u})}^{2/q}+\|F\|_{L^2(S_{t,u})},
\end{equation}
and
\begin{equation}\label{sob_9.18}
\tir^\f12\| F\|_{L^\infty(S_{t,u})}\les \|\tir \sn F\|_{L^4(S_{t,u})}+\|F\|_{L^4(S_{t,u})}.
\end{equation}
(ii) For any $\delta\in (0,1)$, any $q\in (2,\infty)$ and any scalar function $f$
there hold
\begin{align*}
\sup_{S_{t,u}}|f|\les  \tir^{\frac{2\delta(q-2)}{2q+\delta(q-2)}}
&\left(\int_{S_{t,u}} \left(|\sn f|^2+\tir^{-2} |f|^2\right)\right)^{\f12-\frac{\delta q}{2q+\delta(q-2)}}
\left(\int_{S_{t,u}} \left(|\sn f|^q+\tir^{-q}|f|^q\right) \right)^{\frac{2\delta}{2q+\delta(q-2)}}.
\end{align*}
 \end{proposition}

 We refer to \cite{KRduke}, \cite{KREins2}, \cite{WangCMCSH} and \cite{Wangrough} for the  above inequalities. We will also need a collection of  trace inequalities for future reference.
\begin{lemma}\label{trace2}
Under the assumptions (\ref{Pba}) and (\ref{wras})  there hold on $S_{t,u}$ for scalar functions $F$ the following trace inequalities
\begin{align*}
&\int_{S_{t,u}}|F|^2\les \big(\|F\|_{\dot{H}^1(\Sigma_t\cap \{u'\ge u\})}+\|F\|_{L^6(\Sigma_t\cap \{u'\ge u\})}\big)\|F\|_{L^2(\Sigma_t\cap \{u'\ge u\})},\\
&\|F\|_{L^4(S_{t,u})}+\|\tir^{-\f12} F\|_{L^2(S_{t,u})}\les \| F\|_{\dot {H}^1(\Sigma_t\cap \{u'\ge u\})}+\|F\|_{L^6(\Sigma_t\cap \{u'\ge u\})},
\end{align*}
where $-\fv_*\le u\le t$.
\end{lemma}
The above result will be always used together  with Sobolev embedding on $\Sigma_t$.
\begin{proof}
This results can be obtained by  slightly adapting the original proof in \cite[Sectioon 7.2]{Wang10online}.
\end{proof}

We note that in $\widetilde{\D^+}$,  $0\le \tir\le \frac{9}{5}\tau_*$. Thus, back to the coordinate before rescaling, there holds $0\le \tir\les \la^{-8\ep_0}T\les T$. This fact will be constantly used in the rest of this section and Section \ref{c_flux}.

  \begin{lemma}[Dyadic trace inequality]
 Let $0<\a<\f12$.  Under the assumptions of (\ref{Pba}) and (\ref{wras}), there hold the following estimates for scalar functions $F$
 \begin{align}
 &\|\mu^\a[P_\mu, v] F\|_{l_\mu^2 L^2(S_{t,u})}\les \|\p v\|_{H^1_x}\|F\|_{H^\a(\Sigma_t)},\label{9.10.3.19}\\
 &\|\mu^\a P_\mu F\|_{l_\mu^2 L^2(S_{t,u})}\les \| F\|_{H^{\f12+\a}(\Sigma_t)}, \label{9.10.4.19}\\
&\|F\|_{L^2(S_{t,u})}\les \|F\|_{H^{\f12+}(\Sigma_t)},\label{trace1}
 \end{align}
 where $P_\mu$ is a Littlewood-Paley projector with the smooth symbol supported in a dyadic shell $\{C^{-1}<|\xi|<C, \xi \in {\mathbb R}^3\} $.
 \end{lemma}
 \begin{proof}
 To prove (\ref{9.10.3.19}), we  apply Lemma \ref{trace2} and Lemma \ref{lem2} to obtain
\begin{align*}
\|\mu^\a[P_\mu, v] F\|_{L^2(S_{t,u})}&\les \|\mu^{\a-\f12}[P_\mu, v]F\|^\f12_{H^1(\Sigma_t)}\|\mu^{\a+\f12}[P_\mu, v]F\|^\f12_{L^2(\Sigma_t)}\\
&\les\|\p v\|_{L^6_x}(\sum_{\la \le \mu}(\frac{\la}{\mu})^{\frac{1}{2}-\a}+\sum_{\la>\mu}(\frac{\mu}{\la})^{\f12+\a})\|\la^\a F_\la \|_{L_x^2}.
\end{align*}
Since $\|\p v\|_{L^6_x}\les \|\p v\|_{H^1_x}\les 1$ due to (\ref{5.04.17.19}), taking $l_\mu^2$ norm gives (\ref{9.10.3.19}).

 By using Lemma \ref{trace2} and the finite band property, 
 \begin{align*}
\|P_\mu F\|^2_{L^2(S_{t,u})}&\les \|P_\mu F\|_{H^1(\Sigma_t)}\|P_\mu F\|_{L^2(\Sigma_t)}\les \|\mu^\f12 P_\mu F\|^2_{L^2(\Sigma_t)}.
 \end{align*}
Multiplying the above inequality by $\mu^{2\a}$, followed with taking $l_\mu^2$ norms on both sides,  gives (\ref{9.10.4.19}).

By applying $F=\sum\bar P_\la F$, with $\sum_\la \bar P_\la=Id$ in $L^2({\mathbb R}^3)$,  we apply the first inequality in  Lemma \ref{trace2}  to derive
\begin{align*}
\|F_{\le 1}\|_{L^2(S_{t,u})}&\les\|F_{\le 1}\|^\f12_{H^1(\Sigma_t)}\|F_{\le 1}\|^\f12_{L^2(\Sigma_t)}\les \|F\|_{L^2(\Sigma_t)}.
\end{align*}
And for $\|\sum_{\la >1} \bar P_\la F\|_{L^2(S_{t,u})}$, we apply (\ref{9.10.4.19}). (\ref{trace1}) follows as  a consequence.
 \end{proof}


\begin{lemma}
Let $s-2\ge \delta>1-\frac{2}{p}$. Under the assumptions (\ref{Pba}) and (\ref{wras}), there hold for scalar functions $F$,
\begin{align}
&\|\tir F\|_{L_t^2 L^p_\omega(C_u)}\les \|\mu^\delta \ti{P}_\mu F\|_{l_\mu^2 L^2(C_u)}+T^\f12\sup_{0\le t\le T}\|F\|_{L^2(\Sigma_t)},\label{9.17.1.19}\\
&\|\tir F\|_{L_u^2 L^p_\omega(u\ge u_0)}\les \|\La^\delta F\|_{L^2(\Sigma_t\cap \{u\ge u_0\})}+\|F\|_{L^2(\Sigma_t\cap\{u\ge u_0\})}, \label{10.2.5.19}\\
&\tir^\f12\| F\|_{L^{2p}_\omega(S_{t,u})}\les \|\La^\delta F\|_{H^1(\Sigma_t)}+\|F\|_{H^1(\Sigma_t)},\label{9.18.1.19}
\end{align}
where  $\ti P_\mu$ is a Littlewood-Paley projector in ${\mathbb R}^3$  under the original coordinates, which may have slightly different smooth symbol from either $\bar P_\mu$ or $P_\mu$.
\end{lemma}
\begin{proof}
For the Littlewood-Paley projectors $\sum_{\la}\bar P_\la =Id$, in view of the reproducing property $\bar P_\la=P^2_\la$,
we can decompose $F=\sum_{0<\mu\le 1}P^2_\mu F+\sum_{\mu>1}P^2_\mu F$.

If $\mu>1$, by using (\ref{sob.12})
\begin{align*}
\|\tir^{1-2/p}P^2_\mu F\|_{L^p(S_{t,u})}&\les\|\tir\sn P^2_\mu F\|_{L^2(S_{t,u})}^{1-2/p}\|P_\mu^2 F\|_{L^2(S_{t,u})}^{2/p}+\|P_\mu^2 F\|_{L^2(S_{t,u})}\\
&\les \mu^{1-\frac{2}{p}}\|\tir\ti P_\mu P_\mu F\|_{L^2(S_{t,u})}^{1-\frac{2}{p}}\|P_\mu^2 F\|_{L^2(S_{t,u})}^\frac{2}{p}+\|P_\mu^2 F\|_{L^2(S_{t,u})}\\
&\les ((\mu \tir)^{1-\frac{2}{p}}+1)\|\ti P_\mu F\|_{L^2(S_{t,u})}\les (\mu ^{1-\frac{2}{p}}+1)\|\ti P_\mu F\|_{L^2(S_{t,u})},
\end{align*}
where we used $|\sn f|\les |\p f|$. In the last line above,  we have regarded both $P_\mu^2 F, \ti P_\mu P_\mu F$ as $\ti P_\mu F$, which are the $3$-D Littlewood-Paley projection associated to some smooth symbols. For the lower frequency term, applying  Lemma \ref{trace2} leads to
\begin{align*}
\|\tir^{1-\frac{2}{p}}P^2_{\le 1}F\|_{L^p(S_{t,u})}\les \|P^2_{\le 1} F\|_{H^1(\Sigma_t)}\les \|F\|_{L^2(\Sigma_t)}.
\end{align*}
Therefore we have obtained for $\delta>1-\frac{2}{p}$
\begin{equation*}
\|\tir^{1-\frac{2}{p}}F\|_{L^p(S_{t,u})}\les \|\mu^{\delta} \ti{P}_\mu F\|_{l_\mu^2 L^2(S_{t,u})}+\|F\|_{L^2(\Sigma_t)}.
\end{equation*}
 Integrating the inequality along $C_u$ with $L_t^2$ gives (\ref{9.17.1.19}). Integrating in $u$ from $t$ to $u_0$ gives (\ref{10.2.5.19}). 

To prove (\ref{9.18.1.19}), we first derive for $\mu>1$
\begin{align*}
\|P_\mu^2 F\|_{L_\omega^{2p}}\les \|P_\mu^2F\|_{L_\omega^4}^\frac{2}{p}\|P_\mu^2 F\|_{L_\omega^\infty}^{1-\frac{2}{p}}.
\end{align*}
Thus by using (\ref{sob_9.18})
\begin{align*}
 \tir^\f12\|P_\mu^2 F\|_{L_x^\infty}&\les \tir\| \sn P_\mu^2 F\|_{L^4(S_{t,u})}+\|P^2_\mu F\|_{L^4(S_{t,u})}\\
&\les \tir\mu\|\ti P_\mu P_\mu F\|_{L^4(S_{t,u})}+ \|P^2_\mu F\|_{L^4(S_{t,u})}.
\end{align*}
We then apply the $L^4$ estimate in Lemma \ref{trace2} to derive
\begin{align}\label{9.28.1.19}
\begin{split}
\tir^\f12\|P_\mu^2 F\|_{L_\omega^{2p}}&\les \|P_\mu^2 F\|_{H^1_x}^\frac{2}{p} (\tir\mu\|\ti P_\mu P_\mu F\|_{L^4(S_{t,u})}+ \|P^2_\mu F\|_{L^4(S_{t,u})})^{1-\frac{2}{p}}\\
&\les \|P_\mu^2 F\|_{H^1_x}^\frac{2}{p}(\mu\|\ti P_\mu F\|_{H^1_x}+\|P_\mu^2 F\|_{H^1_x})^{1-\frac{2}{p}}\\
&\les \|\ti P_\mu F\|_{H^1_x}(\mu^{1-\frac{2}{p}}+1).
\end{split}
\end{align}
For $F_{\le 1}=\sum_{0<\mu\le 1}P_\mu^2 F$, by the same procedure, we can obtain
\begin{equation*}
\tir^\f12 \|F_{\le 1}\|_{L_\omega^{2p}}\les \|F\|_{H^1_x}
\end{equation*}
We sum the estimate (\ref{9.28.1.19}) for $\mu>1$, then combine the result with the above estimate to obtain (\ref{9.18.1.19}) with $\delta>1-\frac{2}{p}$.
\end{proof}
  
\section{Control of flux}\label{c_flux}
In order to understand the analytic property of the acoustic null cones,  we will  control the energy flux for derivatives of $v$, $\varrho$, and  for $\curl \C$ along null cones.
\subsection{Flux for $\p v$ and $\bp \varrho$}

In order to control the flux of $(v, \varrho)$ along null cones $C_u$, we  apply the divergence theorem to $\sP_\mu$ in (\ref{9.13.1.19}) and (\ref{wave2}) in the spacetime region $\widetilde{\D^+}\cap \{ u'\ge u\}\cap\{ t_\tmin\le t'\le t\}$, where $t_\tmin=\max\{u,0\}$. This leads to
\begin{equation}\label{div1}
\begin{split}
&\int_{C_u\cap \{t_{\tmin}\le t'\le t\}} L^\mu  \sP_\mu d\mu_\ga dt'\\
&=\int_{\Sigma_t\cap\{ u'\ge u\}}\sP_\mu \bT^\mu
-\int_{\Sigma_{t_\tmin}\cap\{ u'\ge u\}} \sP_\mu \bT^\mu+\int_{\widetilde{\D^+}\cap \{ u'\ge u\}\cap\{ t_{\tmin}\le t'\le t\}}\bd^\mu \sP_\mu,
\end{split}
\end{equation}
where we hide the volume elements on $C_u$, which is $d\mu_\ga dt'$ and on $\Sigma_t$ which is $d\mu_g$.  The latter is always comparable to $d\mu_e$. In particular if $t_\tmin=u$, then $\Sigma_{t_\tmin}\cap\{ u'\ge u\}$ is only the point $\Ga^+(u)$. In this situation, the corresponding integral  vanishes.

 We compute
\begin{align*}
L^\mu  \sP_\mu&=-F_U L U+Q_{\mu\nu}\bT^\nu L^\mu+\f12 F_U^2 L^\mu\bd_\mu t\\
&=- F_U L U+\f12 \big((LU)^2+(\sn U)^2\big)+\f12 F_U^2 L(t)\\
&=\f12 \big((LU-F_U)^2+|\sn U|^2\big),
\end{align*}
where we used $ Q(L, \bT)[f]=\f12\big((L f)^2+|\sn f|^2\big)$.
Substituting the above identities to   (\ref{div1}) implies the following result.

\begin{lemma}[Fundamental estimate for flux]\label{fflux}
Let
\begin{equation*}
\sF[U](C_u)=\int_{C_u\cap \{t_\tmin\le t'\le t\}} \big(|L U|^2+|\sn U|^2\big)
\end{equation*}
and $C_u$ be a short-hand notation for $C_u \cap \{t_\tmin\le t'\le t\}$.
There holds on $\widetilde{\D^+}$ for $(U,V, F_U, F_V)$ satisfying (\ref{lu3}),
\begin{equation}\label{du1}
\sF[U](C_u)\les \int_{C_u}|F_U|^2+\E[U](t)+\E[U](t_{\tmin})+\left|\int_{\widetilde{\D^+}\cap \{ u'\ge u\}\cap\{ t_{\tmin}\le t'\le t\}} \bd^\mu \sP_\mu\right|,
\end{equation}
where the integrand of the last term can be found in   (\ref{dpu}), and the term $\E[U](t_{\tmin})$ vanishes if $u>0$.
\end{lemma}
In Section \ref{causal_reg}, we need the flux control for the metric components $\bg$. Since $v=v_++\eta$, we will apply Lemma \ref{fflux} to wave functions $(v_+, \varrho)$, and  use the trace inequalities and   elliptic estimates to control derivatives of $\eta$.
\begin{proposition}[$H^2$ flux for $(v, \varrho)$]\label{7.24.4.19}
Under the assumptions (\ref{Pba}) and (\ref{wras}) on $\widetilde{\D^+}$,  for the density $\varrho$ and the component of velocity $v^i$, there holds
\begin{equation*}
\sF[\bp \varrho](C_u)+\sF[\bp v](C_u)\les 1.
\end{equation*}
\end{proposition}
\begin{proof}
We will use the equation (\ref{lu3_1}) with (\ref{5.03.1.19}), (\ref{5.03.2.19}), and recall $(U_i\rp{1}, V_i\rp{1})=(\p_i U, \p_i V)$. Since $(U_i\rp{1}, V_i\rp{1})$ involves merely the   spatial derivatives of $(U,V)$, we will obtain the flux of  time derivatives  by using (\ref{4.23.1.19}).  Since   $U=v_+$ or $\varrho$ in   (\ref{lu3_1}), to recover the full control on $v$, we will employ the trace inequality to control  $\|L \p \eta\|_{L^2(C_u)}$ and $\|\sn \p \eta\|_{L^2(C_u)}$. Note  that $\|F_U\|_{L^2(C_u)}$ appears on the right hand side of (\ref{du1}). Since it also contains the terms of $\eta$, we will treat such term by virtue of trace inequalities. To this end, we first show
\begin{equation}\label{7.24.5.19_1}
\|L (\p\eta)\|_{L^2(S_{t,u})}+\|\sn (\p \eta) \|_{L^2(S_{t,u})}+\|F_{U\rp{1}}\|_{L^2(S_{t,u})}\les 1,
\end{equation}
which immediately implies
\begin{equation}\label{7.24.5.19}
\|L (\p\eta)\|_{L^2(C_u)}+\|\sn (\p \eta) \|_{L^2(C_u)}+\|F_{U\rp{1}}\|_{L^2(C_u)}\les T^\f12.
\end{equation}

 $L=\bT+\bN$ in (\ref{9.11.4.19}) will be  frequently used in the proof.
To see the first estimate in (\ref{7.24.5.19_1}), we derive  by using (\ref{trace1}) that
\begin{align*}
\|L \p \eta\|_{L^2(S_{t,u})}&\les \|\bN \p \eta\|_{L^2(S_{t,u})}+\|\bT \p \eta\|_{L^2(S_{t,u})}\\
&\les \|\p^2 \eta\|_{H^{\f12+}(\Sigma_t)}+\|\bT \p \eta\|_{H^{\f12+}(\Sigma_t)}.
\end{align*}
For the second term, by (\ref{cmu1}), $\bT\p \eta=\p \bT \eta-\p v^m \p_m \eta$.
By using (\ref{4.12.5.19}), Sobolev inequality, (\ref{5.03.3.19}) and (\ref{5.04.17.19}), we derive
\begin{align*}
\|\bT \p \eta\|_{H^{\f12+}_x}&\les \|\p \bT \eta\|_{H^1_x}+\|\p v\c \p \eta\|_{H^1_x}\\
&\les 1+\|\p^2 v\p \eta\|_{L^2_x}+\|\p v\p^2 \eta\|_{L^2_x}+\|\p v\c \p\eta\|_{L_x^2}\\
 &\les 1+\|\p v\|_{H^1_x}\|\p \eta\|_{L^\infty_x}+\|\p^2 \eta\|_{L^3_x}\|\p v\|_{L^6_x}\\
 &\les\|\p v\|_{H^1_x}+1 \les 1.
\end{align*}
For the first term,  by using (\ref{4.12.2.19}) and (\ref{4.25.1.19}),
 \begin{equation}\label{7.25.5.19}
 \|\p^2 \eta\|_{H^{\f12+}(\Sigma_t)}\les \|\curl \Omega\|_{H^{\f12+}_x}\les\|\curl \Omega\|_{H^1_x} \les 1.
 \end{equation}
Hence,
\begin{equation}\label{7.24.6.19}
\|L (\p \eta)\|_{L^2(S_{t,u})}\les 1.
\end{equation}
In view of (\ref{trace1}), (\ref{7.25.5.19}) also implies
\begin{equation*}
\|\sn (\p \eta)\|_{L^2(S_{t,u})}\les \|\p^2 \eta\|_{L^2(S_{t,u})}\les \|\p^2 \eta\|_{H^{\f12+}(\Sigma_t)}\les 1.
\end{equation*}
 Thus the first two estimates in (\ref{7.24.5.19_1}) are proven.

For the last estimate in (\ref{7.24.5.19}), we recall from (\ref{5.02.3.19}),
\begin{align*}
\|F_{U\rp{1}}\|_{L^2(S_{t,u})}&\les \|\p_i F_U\|_{L^2(S_{t,u})}+\|\p v \c \p U \|_{L^2(S_{t,u})}\\
&\les \|\p_i \bT \eta\|_{L^2(S_{t,u})}+\|\p v\|_{L^4(S_{t,u})}\|\p U\|_{L^4(S_{t,u})}.
\end{align*}
By using (\ref{trace1}) and Lemma \ref{trace2}, the energy estimates (\ref{5.04.16.19}), (\ref{5.04.17.19}),  and (\ref{4.12.5.19}), we have
\begin{equation*}
\|F_{U\rp{1}}\|_{L^2(S_{t,u})}\les \|\p \bT \eta\|_{H^{\f12+}_x}+\|\p v\|_{H^1_x}\|\p  U\|_{H^1_x}\les 1,
\end{equation*}
as desired in (\ref{7.24.5.19_1}). Thus the proof of (\ref{7.24.5.19_1}) is completed.

Next we apply (\ref{du1}) to $U\rp{1}, V\rp{1}, F_{U\rp{1}}, F_{V\rp{1}}$.  Similar to  (\ref{9.28.6.19})
\begin{align*}
\|\bd^\a \sP_\a\|_{L^1(\widetilde{\D^+})}&\les \big(\|F_{V\rp{1}}\|_{L_t^1 L_x^2}+\|\p F_{U\rp{1}}\|_{L_t^1 L_x^2}\big)\sup_{t'\le t}\E\rp{1}_U(t')^\f12+\|k\|_{L_t^1 L_x^\infty}\sup_{t'\le t} \E\rp{1}_U(t').
\end{align*}
Recall from (\ref{5.02.3.19}) and the calculation in Corollary \ref{eng_1}
\begin{align*}
\|\p F_{U\rp{1}}\|_{L^2_x}&\les \|\p^2 v\|_{L_x^2}\|\p U\|_{L_x^\infty}+\|\p v\|_{L_x^\infty} \|\p^2 U\|_{L_x^2}+\|\p^2 F_U\|_{L_x^2}\\
&\les \|\p U,\p v\|_{L_x^\infty}+\|\p^2 F_U\|_{L_x^2},\\
\|F_{V\rp{1}}\|_{L_x^2}&\les \|\p \varrho,\p v \|_{L_x^\infty}\E\rp{1}_U(t)^\f12+\|V, \p U\|_{L_x^\infty}(\|\p v\|_{H^1_x}+\|\p \varrho\|_{H^1_x})+\|\p F_V\|_{L_x^2}\\
&\les \|\p \varrho, \p v, V, \p U\|_{L_x^\infty}+\|\p F_V\|_{L_x^2},
\end{align*}
where we also used (\ref{5.04.16.19}) and (\ref{5.04.17.19}). Substituting (\ref{9.28.3.19}) into the above inequalities yields
\begin{equation*}
\|\p F_{U\rp{1}}\|_{L_x^2}+\|F_{V\rp{1}}\|_{L_x^2}\les \|\bp v, \bp \varrho, \p v_+\|_{L_x^\infty}+1
\end{equation*}
for $U=v_+$ and $U=\varrho$.

By (\ref{9.20.2.19}), $\|k\|_{L_t^1 L_x^\infty}\les T^\f12 $ by (\ref{BA1}). By using the boundedness of energy in (\ref{5.04.16.19}), we  summarize the above calculations and derive with the help of (\ref{BA1}) that
\begin{equation}\label{7.25.1.19}
\|\bd^\a \sP_\a\|_{L^1(\widetilde{\D^+})}\les \|\bp v, \bp \varrho, \p v_+\|_{L_t^1 L_x^\infty(\widetilde{\D^+})}+T\les T^\f12.
\end{equation}

Substituting (\ref{7.25.1.19}), the last estimate in (\ref{7.24.5.19}) and the boundedness of energy (\ref{5.04.16.19})  to (\ref{du1}) yields
\begin{equation*}
\|L \p U, \sn \p U\|_{L^2(C_u)}\les 1, \mbox{ for } U=v_+, \varrho.
\end{equation*}
Using the first two estimates in (\ref{7.24.5.19}), $v=v_++\eta$, and the first equation in (\ref{4.23.1.19}), we can conclude
\begin{equation}\label{7.25.6.19}
\|L \p (v, \varrho), \sn \p (v, \varrho), L \bT \varrho, \sn \bT \varrho\|_{L^2(C_u)}\les 1.
\end{equation}
We further apply the second equation in  (\ref{4.23.1.19})
\begin{equation}\label{10.2.2.19}
|L \bT v|+|\sn \bT v|\les |L (c^2 \p \varrho)|+|\sn (c^2 \p \varrho)|\les |L\p \varrho|+|\sn\p \varrho|+|\bp \varrho|^2.
\end{equation}
By using Lemma \ref{trace2} and (\ref{5.04.17.19}), we can bound
\begin{equation*}
\||\bp\varrho|^2\|_{L^2(C_u)}\les\sup_{t'\le t}\|\bp \varrho\|_{L^4(S_{t',u})}^2\c T^\f12 \les \|\bp \varrho\|_{L_t^\infty H^1_x}^2 T^\f12\les T^\f12.
\end{equation*}
Hence by using (\ref{7.25.6.19}) we have
\begin{equation*}
\|L \bT v, \sn \bT v\|_{L^2(C_u)}\les 1.
\end{equation*}
Thus the proof of Proposition \ref{7.24.4.19} is complete.
 \end{proof}

\begin{proposition}[$H^{2+\ep}$ flux for $(v, \varrho)$]\label{9.10.2.19}
Let $0<\ep\le s-2$. Under the assumption (\ref{Pba}) and (\ref{wras}) on $\widetilde{\D^+}$, there holds
\begin{equation*}
\|\mu^\ep \sF^\f12[P_\mu \p v](C_u)\|_{l_\mu^2}+ \|\mu^\ep\sF^\f12[P_\mu \bp \varrho](C_u)\|_{l_\mu^2}+\|\mu^\ep \sF^\f12[P_\mu \p v_+](C_u)\|_{l_\mu^2} \les 1.
\end{equation*}
\end{proposition}
\begin{remark}
{
To avoid unnecessary technical baggage, we will not derive the flux control for $\bT v$, since  only the weaker control in (\ref{6.18.5.19}) is required in Section \ref{causal_reg}.  By using the second equation of (\ref{4.23.1.19}), $\bT v=-c^2 \p \varrho$, we can directly get the bound for $\bT v$ in (\ref{6.18.5.19}) by using the above result with the help of the trace inequality and energy estimates. We will give the thorough detail to prove  (\ref{6.18.5.19}) in Section \ref{causal_reg}.}
\end{remark}
\begin{proof}
We apply (\ref{du1}) to $(U_\mu\rp{1}, V_\mu\rp{1})=(P_\mu U\rp{1}, P_\mu V\rp{1})$ with
\begin{equation*}
(U\rp{1}, V\rp{1})=(\p v_+, \p \bT v), \quad (U\rp{1}, V\rp{1})=(\p \varrho, \p \bT \varrho)
\end{equation*}
to obtain
\begin{align}\label{9.10.6.19}
\begin{split}
\sum_{\mu>1}\mu^{2\ep}\sF[U\rp{1}_\mu](C_u)
&\les \sum_{\mu>1}\{ \int_{C_u\cap \{t'\le t\}}\mu^{2\ep}|F_{U\rp{1}_\mu}|^2+ \mu^{2\ep}(\E[U\rp{1}_\mu](t)+\E[U\rp{1}_\mu](t_\tmin))\}\\
&+\sum_{\mu>1}\mu^{2\ep}\left|\int_{\widetilde{\D^+}\cap \{ u'\ge u\}\cap \{t_\tmin\le t'\le t\}} \bd^\a \sP_\a[U\rp{1}_\mu]\right|.
\end{split}
\end{align}
Substituting $(U_\mu\rp{1}, V_\mu\rp{1})$ into (\ref{9.28.6.19}) implies
\begin{align}
&\sum_{\mu>1}\mu^{2\ep}\|\bd^\a \sP_\a[U\rp{1}_\mu]\|_{L^1(\widetilde{\D^+}\cap \{ u'\ge u\}\cap \{t_\tmin\le t'\le t\})}\label{9.10.7.19}\\
&\les \int_0^t \|\E[U\rp{1}_\mu]^\f12(t')\|_{l_\mu^2}\{\|\mu^\ep F_{V\rp{1}_\mu}(t')\|_{l_\mu^2 L_x^2}+\|\mu^\ep \p F_{ U\rp{1}_\mu}(t')\|_{l_\mu^2 L_x^2}+ \|\E[U\rp{1}_\mu]^\f12(t')\|_{l_\mu^2}\|k(t')\|_{L_x^\infty}\}dt',\nn
\end{align}
where the formulas of $F_{U_\mu\rp{1}}$ and $F_{V\rp{1}_\mu}$  can be found in  (\ref{puuv}).
To control the right hand side, we will rely on the estimates of $F_{U_\mu\rp{1}}, F_{V\rp{1}_\mu}$ in the proof of Proposition \ref{4.13.8.19} and the results of Corollary \ref{eng_wave}.

On $\Sigma_{t}$ with $0< t\le T$, we derive from (\ref{9.29.2.19}) that
\begin{align}
&\|\mu^\ep \p F_{U\rp{1}_\mu}\|_{l_\mu^2 L_x^2}+\|\mu^\ep F_{V\rp{1}_\mu} \|_{l_\mu^2 L_x^2}\nn\\
&\les \|\p v, \Tr k\|_{H^{1+\ep}} \|V, \p U\|_{L_x^\infty}+\|\p v, \p \varrho \|_{L_x^\infty} \| V\rp{1}, \p U\|_{H^{1+\ep}_x}+\|\p^2 F_U\|_{H^\ep_x}+\|\p F_V\|_{\dot{H}^{\ep}_x}\nn\\
&\les \|\bp v, \bp \varrho, \p v_+\|_{L_x^\infty}\|\p U, V\rp{1}, \p v, \Tr k\|_{H^{1+\ep}_x}+\|\p^2 F_U\|_{H^\ep_x}+\|\p F_V\|_{\dot{H}^{\ep}_x}\nn\\
&\les (\E\rp{\le 1}(t)^\f12+\|\mu^\ep \E\rp{1}_\mu(t)^\f12\|_{l_\mu^2}+\|\curl \Omega\|_{H^{\f12+\ep}_x}+1)(\|\p v_+, \p v, \bp \varrho\|_{L_x^\infty}+1)\label{9.29.3.19}\\
&\les \|\p v_+, \p v, \bp \varrho\|_{L_x^\infty}+1\nn,
\end{align}
where we employed (\ref{5.05.3.19}) and (\ref{5.05.4.19}) together with Lemma \ref{comp_2} to derive the line of (\ref{9.29.3.19}), and   used  Corollary \ref{eng_wave} and (\ref{9.07.5.19}) to derive the last line. We then substitute the above estimate to (\ref{9.10.7.19}). By using (\ref{BA1}) and  Corollary \ref{eng_wave}, we can conclude that
\begin{equation}\label{10.22.3.19}
\sum_{\mu>1}\mu^{2\ep}\|\bd^\a \sP_\a[U\rp{1}_\mu]\|_{L^1(\widetilde{\D^+})}\les T^\f12.
\end{equation}

Next, we bound $\|\mu^\ep F_{U\rp{1}_\mu}\|_{l_\mu^2 L^2(C_u)}$  in (\ref{9.10.6.19}). Using (\ref{5.02.3.19}) and (\ref{puuv}) directly implies
\begin{align}
\|\mu^\ep F_{U\rp{1}_\mu}\|_{l_\mu^2 L^2(C_u)}&\les\|\mu^\ep[P_\mu, v] \p U\rp{1}\|_{l_\mu^2 L^2(C_u)}+\|\mu^\ep P_\mu F_{U\rp{1}}\|_{l_\mu^2 L^2(C_u)}\nn\\
   &\les T^\f12 (\|\p  v\|_{L_t^\infty H^1_x}\|\p U\rp{1}\|_{L_t^\infty H^\ep_x}+\|\La^{\f12+\ep} F_{U\rp{1}}\|_{L_t^\infty L_x^2})\label{9.10.5.19},
\end{align}
where we applied  (\ref{9.10.3.19}) to  $\p U\rp{1}$,  and  (\ref{9.10.4.19}) to $F_{U\rp{1}}$ to obtain the last inequality.  For the second term on the right hand side, we recall (\ref{7.24.8.19}) and use (\ref{9.08.8.19}) with $\a=\f12+\ep$ to obtain
\begin{align*}
\|\La^{\f12+\ep} F_{U\rp{1}}\|_{L^2_x}&\les \|\La^{\f12+\ep}\p F_U\|_{L^2_x}+\|\La^{\f12+\ep}(\p v \p U)\|_{L^2_x}\\
&\les \|\La^{\f12+\ep} \p\bT \eta\|_{L^2_x}+\|\p v\|_{H^{1+\ep}_x}\|\p U\|_{H^1_x}+\|\p U\|_{H^{1+\ep}_x}\|\p v\|_{H^1_x})\\
&\les \|\p \bT\eta\|_{H^1_x}+ \|\p U\|_{H^{1+\ep}_x} \les 1,
\end{align*}
where we used (\ref{4.12.5.19}) and Corollary \ref{eng_wave} to derive the last inequality.

Substituting the above estimate  to (\ref{9.10.5.19}) and using Corollary \ref{eng_wave} to treat the first term in the last line of (\ref{9.10.5.19}) imply
\begin{align*}
\|\mu^\ep F_{U\rp{1}_\mu}\|_{l_\mu^2 L^2(C_u)}\les T^\f12.
\end{align*}
This controls the first term on the right hand side of (\ref{9.10.6.19}).  Substituting this estimate and  (\ref{10.22.3.19}) into (\ref{9.10.6.19}) and using Corollary \ref{eng_wave} imply
\begin{equation}\label{9.10.8.19}
\|\mu^\ep \sF^\f12[U_\mu\rp{1}](C_u)\|_{l_\mu^2}^2\les 1+\sup_{t'\le t}\sum_{\mu>1}\mu^{2\ep}\E\rp{1}_\mu(t')\les 1.
\end{equation}
This together with Proposition \ref{7.24.4.19} immediately gives the control of the  flux for $(\p v_+, \p \varrho)$  to the highest order.
Similar to the $H^2$-case, in view of $v=v_++\eta$, we need to show
\begin{equation}\label{9.11.1.19}
\|\mu^\ep (L P_\mu \p \eta, \sn P_\mu \p \eta)\|_{l_\mu^2 L^2(C_u)}\les T^\f12\les 1.
\end{equation}
For simplicity, we denote by $S=S_{t',u}$ with $t_\tmin\le t'\le t$.
We first apply (\ref{9.10.4.19}) to $F=\p^2\eta$.   By using (\ref{4.12.2.19}) and (\ref{9.07.5.19}),
\begin{equation}\label{9.29.4.19}
\|\La^\a\p^2 \eta\|_{L_t^\infty L_x^2}\les 1,\quad 0<\a\le \f12+\ep.
\end{equation}
 Hence, by using (\ref{wras})
\begin{align}\label{9.11.2.19}
\|\mu^\ep \sn P_\mu \p \eta\|_{l_\mu^2 L^2(S)}\les \|\mu^\ep \p P_\mu \p \eta\|_{l_\mu^2 L^2(S)}\les \|\La^{\f12+\ep}\p^2 \eta\|_{L_t^\infty  L^2_x}\les 1.
\end{align}
For the first estimate in (\ref{9.11.1.19}), we derive
\begin{align*}
\|\mu^\ep L P_\mu \p \eta\|_{l_\mu^2 L^2(S)}&\les \|\mu^\ep \bN^i\p_i P_\mu \p \eta\|_{l_\mu^2 L^2(S)}+\|\mu^\ep \bT P_\mu \p \eta\|_{l_\mu^2 L^2(S)}\\
&\les \|\mu^\ep P_\mu \p^2 \eta\|_{l_\mu^2 L^2(S)}+\|\mu^\ep[\bT, P_\mu]\p \eta\|_{l_\mu^2 L^2(S)}+\|\mu^\ep P_\mu \bT \p \eta\|_{l_\mu^2 L^2(S)}.
\end{align*}
The estimate of the first term on the right hand side of the last line is already included in (\ref{9.11.2.19}). It suffices to show
\begin{align}\label{9.11.3.19}
\|\mu^\ep[\bT, P_\mu]\p \eta\|_{l_\mu^2 L^2(S)}+\|\mu^\ep P_\mu \bT \p \eta\|_{l_\mu^2 L^2(S)}\les 1.
\end{align}
For the first estimate, we apply (\ref{9.10.3.19})  to obtain
\begin{align*}
\|\mu^\ep[P_\mu,v]\p^2 \eta\|_{l_\mu^2 L^2(S)}\les \|\p v\|_{H^1_x}\|\p^2 \eta\|_{H^\ep_x}\les 1,
\end{align*}
where we used (\ref{9.29.4.19}) and (\ref{5.04.17.19}) to get the last bound. For the other term,  we first  use (\ref{9.29.4.19}), (\ref{5.03.9.19}), (\ref{9.08.8.19}), Corollary \ref{eng_wave} to bound
\begin{align*}
\|\La^{\f12+\ep} [\bT, \p] \eta\|_{L_x^2}&\le\|\La^{\f12+\ep}(\p v\c \p \eta)\|_{L_x^2}\\
&\les \|\p v\|_{H^{1+\ep}_x}\|\p \eta\|_{H^1_x}+\|\p v\|_{H^1_x}\|\p \eta\|_{H^{1+\ep}_x}\\
&\les 1.
\end{align*}
Next we use (\ref{9.10.4.19}) with $\a=\ep$,  (\ref{4.12.5.19}) together with the above estimate  to derive
\begin{align*}
\|\mu^\ep P_\mu \bT \p \eta\|_{l_\mu^2 L^2(S)}&\les \|\La^{\f12+\ep} ([\bT, \p]+\p \bT) \eta\|_{L_t^\infty L_x^2}\les 1.
\end{align*}
Thus (\ref{9.11.3.19}) is proved. We have completed the proof of
\begin{equation*}
\|\mu^\ep L P_\mu \p \eta\|_{l_\mu^2 L^2(C_u)}\les T^\f12.
\end{equation*}
The proof of (\ref{9.11.1.19}) is thus completed.

Combining (\ref{9.10.8.19}) with (\ref{9.11.1.19}), and applying $v=v_++\eta$ imply
\begin{align*}
\|\mu^\ep \sF^\f12[v_\mu\rp{1}](C_u)\|_{l_\mu^2}&\les\| \mu^\ep\sF^\f12[P_\mu v_+\rp{1}](C_u)\|_{l_\mu^2}+\|\mu^\ep(\sn_L P_\mu \p \eta, \sn P_\mu \p \eta)\|_{l_\mu^2 L^2(C_u)}\les 1.
\end{align*}
 Using $\bT\varrho=-\div v$ in (\ref{4.23.1.19}), the above estimate together  with (\ref{9.10.8.19}) leads to
\begin{equation*}
\|\mu^\ep\sF^\f12[\p v, \p v_+, \bp \varrho](C_u)\|_{l_\mu^2}\les 1.
\end{equation*}
Thus Proposition \ref{9.10.2.19} is proved.
\end{proof}

\subsection{Flux of $\curl \C$}\label{flux_C}

In this subsection, we derive the flux control of $\curl\C$ up to the highest order along $C_u$, which will be crucial for Section \ref{causal_reg}.

Along $C_u$, as a consequence of Proposition \ref{9.07.16.19}, Proposition \ref{9.07.5.19} and the trace inequalities (\ref{9.10.4.19}) and (\ref{9.18.1.19}), we can  bound
\begin{equation}\label{10.26.1.19}
\|\mu^{\delta+\f12}P_\mu\p  \Omega\|_{l_\mu^2L^2(S_{t,u})}+\tir^\f12 \|\p \Omega\|_{L^{2p}_\omega(S_{t,u})}\les 1,
\end{equation}
 with $0\le 1-\frac{2}{p}<\delta\le s'-2$.
 Nevertheless,  the analysis in  Section \ref{causal_reg} requires us to gain nearly $\f12$ derivative more than the above bounds. The goal is  achieved in this subsection merely for $\curl \C$, whereas the full control  for $\p^2 \Omega$ may not actually hold under our assumption on the initial data.

We will introduce an energy argument identical to  Section \ref{eng_vor}, with the flux achieved as the boundary term on $C_u$. Recall that the energy bound of $\curl \fC$ in Section \ref{eng_vor} is derived by taking advantage of the trilinear structure due to the pairing of  $\bT \curl \C\c \curl \C$, with the help of a series of integration by parts. We lose the structure if propagation the  general $\p \fC$ directly. The elliptic estimates for the Hodge system on $\Sigma_t$ allows us to obtain the $H^\delta_x$ bound for $\p^2\Omega, \p \C, \p^2 \fw$ whence $\|\curl \C\|_{H^\d_x}$ is bounded. (See  Proposition \ref{9.07.16.19} and Proposition \ref{9.07.5.19}.) Restricted on the $C_u$, we can only obtain the corresponding flux bound for $\curl \C$ without loss instead of for $\p \C$, since there lacks a reasonable Hodge system on $S_{t,u}$ which allows us to bound $\p \fC$ by $\curl \C$ in $L^2(S_{t,u})$.
 Fortunately, we manage to use merely the highest-order flux of  $\curl \C$  to obtain  the sufficient regularity of the acoustic cones in   Section \ref{causal_reg} by uncovering  a series of  geometric structures of the acoustic spacetime.

We first give the fundamental  inequality for bounding the energy flux of $\curl \fC$.  
\begin{lemma}
Let $F$ and $G$ be one-tensor fields. There holds on $\widetilde{\D^+}$ that
\begin{align}\label{9.11.5.19}
\begin{split}
|\int_{C_{u_0}\cap \{t_\tmin\le t'\le t\}} c^3 \l F, G\r_e d\mu_\ga  dt&+ \int_{\widetilde{\D^+}\cap \{u\ge u_0\}\cap \{t_\tmin\le t'\le t\}} (\l \bT F,G\r_e+\l\bT G, F\r_e) d\mu_e dt|\\
&\les \int_0^t \|\p v(t')\|_{L_x^\infty}\int_{\Sigma_{t'}} | \l F, G\r| d\mu_e dt'\\
&+|\int_{\Sigma_t} \l F, G\r_e  d\mu_g|+|\int_{\Sigma_{t_\tmin}} \l F, G\r_e d\mu_g|.
\end{split}
\end{align}
where $t_\tmin(u_0)=\max(u_0, 0)$, and thus the last term vanishes if $u_0\ge 0$.
\end{lemma}
\begin{proof}
Let $\sV^\mu=\l F, G\r_e \bT^\mu$. By applying the divergence formula, we have
\begin{equation}\label{9.30.1.19}
\begin{split}
\int_{C_{u_0}\cap \{t'\le t\}} \sV^\mu \bd_\mu u \bb d\mu_\ga dt'&+\int_{\Sigma_{t_\tmin}} \sV^\mu \bd_\mu t-\int_{\Sigma_t} \sV^\mu \bd_\mu t\\
&=-\int_{\widetilde{\D^+}\cap \{u\ge u_0\}\cap \{t_\tmin\le t'\le t\}} \bd_\mu \sV^\mu d\mu_g dt.
\end{split}
\end{equation}
With $\bar \Pi^\nu_\mu=\delta^\nu_\mu+\bT^\nu \bT_\mu$,
\begin{equation*}
{e_i}_\nu\bd_i(\sV^\mu \bar\Pi_\mu^\nu)={e_i}_\nu\bd_i \sV^\mu \bar \Pi^\nu _\mu +\sV^\mu  \bd_i \bar\Pi^\nu_\mu= \bd_i \sV^j \bar \Pi^i_j+{e_i}_\nu \sV^\mu \bd_i (\bT_\mu \bT^\nu),
\end{equation*}
where $e_i=c^{-1}\p_i$ is the orthonormal basis in $(\Sigma_t, g)$, the component of $\sV(e_i)=\sV^i$  and $\bd$ denotes the covariant derivative in the spacetime.

By the definition of $\sV$, $\sV^\mu \bar\Pi_\mu^\nu=0$. Thus  we can derive from the above identity that
\begin{equation*}
\bd_i \sV^i=\l \sV, \bT\r   \Tr k.
\end{equation*}
Using the above identity and the facts that  $\l \sV, \bd u\r= \l F, G\r_e \bb^{-1}$, $\l \sV, \bd t\r= \l F, G\r_e$, we compute
\begin{align*}
\bd_\mu \sV^\mu&=-\bT \l \sV, \bT\r+\bd_i \sV^i=\bT (\l F, G\r_e)-\l F,G\r_e \Tr k\\
&=\l \bT F,G\r_e+\l\bT G, F\r_e-\l F, G\r_e \Tr k.
\end{align*}
Substituting the above identity to (\ref{9.30.1.19}) yields
\begin{align}
\int_{C_{u_0}\cap \{t'\le t\}} &\bb^{-1} \l F, G\r_e \bb d\mu_\ga  dt= \int_{\Sigma_t} \l F, G\r_e  d\mu_g-\int_{\Sigma_{t_\tmin}} \l F, G\r_e d\mu_g\label{9.30.4.19}\\
&-\int_{\widetilde{\D^+}\cap \{u\ge u_0\}\cap\{t_\tmin\le t'\le t\}} \{\l \bT F,G\r_e+\l F, \bT G\r_e-\l F, G\r_e \Tr k\} d\mu_g dt'.\nn
\end{align}
 Since $d\mu_g=c^{-3} d\mu_e$, we can replace  $F$ by $c^\frac{3}{2}F$  and $G$ by $c^\frac{3}{2}G$ in the above identity. This implies
\begin{align*}
\int_{C_{u_0}\cap \{t'\le t\}} & c^3\l F,G\r_e d\mu_\ga  dt= \int_{\Sigma_t}\l F, G\r_e  d\mu_e-\int_{\Sigma_{t_\tmin}}\l F, G\r d\mu_e\\
&-\int_{\widetilde{\D^+}\cap \{u\ge u_0\}\cap\{t_\tmin\le t'\le t\}}  \{\l \bT F,G\r_e+\l F, \bT G\r_e-\l F, G\r_e (\Tr k-3 \bT(\log c) )\} d\mu_e dt.
\end{align*}
By (\ref{k1}), $\Tr k-3 \bT(\log c)=-\div v$. (\ref{9.11.5.19}) follows by  substituting this formula to the above calculation.
\end{proof}

Next, we prove the main result of this subsection.
\begin{proposition}\label{9.30.15.19}
Under the assumptions (\ref{Pba}) and (\ref{wras}), there hold on  $C_u \cap \widetilde{\D^+}$ that,
\begin{align}
&\|\curl \C\|_{L^2(C_u)}\les 1,\label{9.07.6.19}\\
&\|\mu^\delta P_\mu \curl \C\|_{l_\mu^2 L^2(C_u)}\les 1, \quad 0\le \delta\le s'-2,\label{9.11.6.19}\\
&\|\tir\curl^2 \Omega\|_{L_t^2 L_\omega^p(C_u)}\les 1,\quad 0\le 1-\frac{2}{p}<s'-2.\label{9.17.2.19}
\end{align}
\end{proposition}
\begin{proof}
Firstly, it is direct to compute $\ep_{mij} \p^m \C^i=e^{-\varrho}\big(\ep_{mij}\p^m(\curl \Omega)^i-\ep_{mij} \p^m \varrho \C^i\big)$.
By using (\ref{9.18.1.19}) with $s'-2>\delta>1-\frac{2}{p}$, we derive
\begin{align*}
\|\tir\curl^2 \Omega\|_{L^2_t L^p_\omega(C_u)}&\les \|\tir \p \varrho\c \C\|_{L^2_t L^p_\omega(C_u)}+\|\tir\curl \C\|_{L_t^2 L_\omega^p(C_u)}\\
&\les \|\tir^\f12\p \varrho\|_{L_t^2 L^{2p}_\omega(C_u)} \|\tir^\f12 \C\|_{L^\infty_t L^{2p}_\omega(C_u)}+\|\tir\curl \C\|_{L_t^2 L_\omega^p(C_u)}\\
&\les T^\f12\|\p \varrho\|_{L_t^2 L_x^\infty}\|\La^{\delta}\C\|_{L_t^\infty H^1_x}+\|\tir\curl \C\|_{L_t^2 L_\omega^p(C_u)}\nn\\
&\les T^\f12+\|\tir\curl \C\|_{L_t^2 L_\omega^p(C_u)},
 \end{align*}
 where we used (\ref{BA1}) and (\ref{9.07.15.19}) to obtain the last line.

To derive (\ref{9.17.2.19}),  we apply (\ref{9.17.1.19}) to $F=\curl^2 \Omega$ by assuming  (\ref{9.11.6.19}) to derive
\begin{align*}
\|\tir\curl\C\|_{L^2_t L^p_\omega(C_u)}\les \|\mu^\delta \ti{P}_\mu \curl\C\|_{l_\mu^2 L^2(C_u)}+\|\curl\C\|_{L_t^\infty L^2_x}\les 1
\end{align*}
where we also used (\ref{9.07.5.19}).

Next we prove (\ref{9.07.6.19}) and (\ref{9.11.6.19}). The analysis overlaps with Section \ref{eng_vor}, except that all the boundary terms along $C_u$ vanished therein.

 To prove (\ref{9.07.6.19}), we apply (\ref{9.11.5.19}) to $F=G=\curl \C$.  Consider the following integral on $\widetilde{\D^+}\cap \{u\ge u_0\}\cap \{t_\tmin\le t'\le t\}$, and the range of $t'$ will be hidden for short,
 \begin{equation}\label{9.30.2.19}
 \begin{split}
\int_{\widetilde{\D^+}\cap \{u\ge u_0\}}\l \bT \curl \C, \curl \C\r_e d\mu_e dt'&=\int_{\widetilde{\D^+}\cap \{u\ge u_0\}}\l \curl \bT \C, \curl \C\r_e d\mu_e dt'\\
&+\int_{\widetilde{\D^+}\cap \{u\ge u_0\}}\l [\bT, \curl] \C, \curl \C\r_e d\mu_e dt'.
\end{split}
 \end{equation}

By using (\ref{9.07.1.19}) and the first equation in  (\ref{4.23.1.19}), the last line can be bounded by
$$\int_0^t \|\p v\|_{L_x^\infty} \|\p \C\|_{L_x^2}^2 dt'\les 1$$ with the help of  (\ref{9.07.5.19}) and (\ref{BA1}).
Thus
\begin{equation}\label{9.30.5.19}
\big|\int_{\widetilde{\D^+}\cap \{u\ge u_0\}}\l \bT \curl \C, \curl \C\r_e d\mu_e dt'\big|\les \big|\int_{\widetilde{\D^+}\cap \{u\ge u_0\}}\l \curl \bT \C, \curl \C\r_e d\mu_e dt'\big|+1.
\end{equation}

 The first term on the right hand side of (\ref{9.30.5.19}) has the same integrand as the term (\ref{9.11.7.19}). The  only difference is that the integral is on the spacetime domain enclosed by the boundary $\{t'=t_\tmin\}, \{t'=t\}$ and $\{u=u_0\}$. Therefore whenever undertaking integration by parts, we need to keep track of the additional  boundary terms along $C_{u_0}$ compared with the treatment for (\ref{9.11.7.19}).

We employ the same  integration by parts as in the proof of (\ref{9.07.5.19}), which leads to the non-trivial boundary terms in $I^1_{+, 1,1}$ in (\ref{9.30.3.19}) and $I^1_{+,1,2}$ in (\ref{9.08.19.19}) on the cone $C_{u_0}$,  denoted by $I^1_{+,1,1,b}$ and $I^1_{+,1,2,b}$ respectively. Let us compute $I^1_{+,1,1,b}$ first.
\begin{align}\label{9.30.6.19}
\begin{split}
I^1_{+, 1, 1, b}&=\int_{t_{\tmin}}^t \int_{\Sigma_{t'}\cap \{u\ge u_0\} } \p_m(e^{-\varrho} \p^n v_j \p_n \Omega^j \curl \C^m) d\mu_e dt'\\
&=\int_{t_{\tmin}}^t \int_{\Sigma_{t'}\cap \{u\ge u_0\} } \{\div_g(e^{-\varrho} \p^n v^j \p_n \Omega_j \curl \C)\\
&\qquad+3 \p_m \log c\c e^{-\varrho} \p^n v^j \p_n \Omega_j \curl \C^m\}  c^{3} c^{-3} d\mu_e dt' \\
&=\int_{t_{\tmin}}^t \int_{\Sigma_{t'}\cap \{u\ge u_0\} } \div_g(c^3e^{-\varrho} \p^n v^j \p_n \Omega_j \curl \C) d\mu_g dt',
\end{split}
\end{align}
where we calculated $\div_{g} X=\div_e X-3  X(\log c)$ for $\Sigma$ tangent vector field $X$, since $g_{ij}=c^{-2}\delta_{ij}$.
\begin{equation*}
I^1_{+, 1, 1, b}=\int_{C_{u_0}} c^3 e^{-\varrho} \p^n v^j \p_n \Omega_j \curl \C^m \bN^l g_{ml} d\mu_\ga dt'.
 \end{equation*}
By using Lemma \ref{trace2} and (\ref{4.25.1.19})
\begin{equation*}
\|\p \Omega\|_{L^4(S_{t,u})}+\|\p \Omega\|_{L^2(S_{t,u})}\les 1.
\end{equation*}
Thus by using the above inequality and (\ref{BA1}), we have
\begin{align*}
|I^1_{+, 1, 1, b}|&\les \|\p v\|_{L_t^2 L_x^\infty}\|\p \Omega\|_{L^2(S_{t,{u_0}})}\|\curl \C\|_{L^2(C_{u_0})}\les \|\curl \C\|_{L^2(C_{u_0})}.
\end{align*}

For the term in $I^1_{+,1,2}$, when integrated in $\widetilde{\D^+}\cap \{u\ge u_0\}\cap \{t_{\tmin}\le t'\le t\}$, the term
$$\int_{t_\tmin}^t \p_t \int_{\Sigma_{t'}\cap \{u\ge u_0\}}\p_j \varrho \p^m\Omega^j e^{-\varrho} \curl \C_m d\mu_e dt'$$  contributes the additional boundary term
\begin{equation*}
I^1_{+,1,2,b}=\int_{C_{u_0}}c^3 e^{-\varrho}\p_j \varrho \p^m \Omega^j \curl \C_m d\mu_\ga dt'.
\end{equation*}
Alternatively,  if we consider  $I^1_{+,1,2}$ on $\widetilde{\D^+}\cap \{u\ge u_0\}$ by applying (\ref{9.30.4.19}) to $F=c^\frac{3}{2}\p_j \varrho $ and $G=c^\frac{3}{2}\p_m \Omega^j e^{-\varrho} \curl \C^m$, we can get the same additional boundary term as above.
Similar to the estimate for $I^1_{+,1,1,b}$,
\begin{align*}
|I^1_{+, 1, 2, b}|&\les \|\p \varrho\|_{L_t^2 L_x^\infty}\|\p \Omega\|_{L^2(S_{t,{u_0}})}\|\curl \C\|_{L^2(C_{u_0})}\\
&\les \|\curl \C\|_{L^2(C_{u_0})}
\end{align*}
The control of the  rest of the terms  can be found in the estimate for $|I^1|$ in (\ref{9.11.8.19}). Therefore, from (\ref{9.30.5.19}) and the above boundary estimates,  we conclude
\begin{align*}
&|\int_{\widetilde{\D^+}\cap \{u\ge u_0\}}\l \bT \curl \C, \curl \C\r_e d\mu_e dt|\\
&\les |I^1_{+, 1, 1, b}|+|I^1_{+, 1, 2, b}|+1+\int_0^t  (\|\curl\C\|_{L^2(\Sigma_t)}+1)^2 (\|\p v, \p \varrho\|_{L_x^\infty}+1)+\sup_{0\le t'\le t} \| \curl \C\|_{L^2(\Sigma_{t'})}\\
&\les \|\curl \C\|_{L^2(C_{u_0})}+1,
\end{align*}
where we used $\|\p v, \p \varrho\|_{L_t^1 L_x^\infty}\les T^\f12$ and the estimate (\ref{9.07.5.19}) to obtain the last line.

Thus in view of (\ref{9.11.5.19}), we have
\begin{align*}
\|\curl \C\|_{L^2(C_{u_0})}^2\les 1+\|\curl \C\|_{L^2(C_{u_0})},
\end{align*}
which gives $\|\curl \C\|_{L^2(C_{u_0})}\les 1$. This shows (\ref{9.07.6.19}).

To prove (\ref{9.11.6.19}), we apply (\ref{9.11.5.19}) to $F=G=P_\mu \curl \C_i$ to bound
\begin{align}
\sum_{\mu>1}|\int_{C_{u_0}} c^3 \mu^{2\delta}|P_\mu \curl \C|_e^2  d\mu_\ga  dt|&\les \sum_{\mu>1}\big|\int_{\widetilde{\D^+}\cap \{u\ge u_0\}} 2\mu^{2\delta}\l \bT P_\mu \curl \C, P_\mu \curl \C\r_e d\mu_e dt\big|\nn\\
&+ \sup_{t'\le t}\int_{\Sigma_{t'}}\sum_{\mu>1} \mu^{2\delta} |P_\mu \curl \C|^2  d\mu_e,\label{9.30.8.19}
\end{align}
where we also used (\ref{BA1}) for deriving the term in the last line.

For the first term on the right hand side,
\begin{align*}
\B_\mu&=\int_{\widetilde{\D^+}\cap \{u\ge u_0\}} 2\mu^{2\delta}\l \bT P_\mu \curl \C, P_\mu \curl \C\r_e d\mu_e dt'\\
&=\int_{\widetilde{\D^+} \cap \{u\ge u_0\}}2\mu^{2\delta} \l [\bT, P_\mu \curl]\C+\curl P_\mu\bT \C , P_\mu \curl \C\r_e d\mu_e dt'.
\end{align*}
Then
\begin{align}\label{9.11.10.19}
\begin{split}
\sum_{\mu>1}|\B_\mu|&\les \sum_{\mu>1}\int_0^t\|\mu^\delta [\bT, P_\mu \curl ]\C\|_{l_\mu^2 L_x^2}\|\mu^\delta P_\mu \curl \C\|_{l_\mu^2 L_x^2}\\
&+\sum_{\mu>1}\big|\int_{\widetilde{\D^+} \cap \{u\ge u_0\}}\mu^{2\delta} \l \curl P_\mu\bT \C , P_\mu \curl \C\r_e d\mu_e dt'\big|.
\end{split}
\end{align}
The first term on the right can be bounded by using (\ref{9.08.2.19}), (\ref{9.08.6.19}) and (\ref{BA1}),
\begin{align}\label{9.11.13.19}
\begin{split}
\sum_{\mu>1}\int_0^t \|\mu^\delta [\bT, P_\mu \curl ]&\C\|_{l_\mu^2 L_x^2}\|\mu^\delta P_\mu \curl \C\|_{l_\mu^2 L_x^2}\\
&\les \int_0^t (\|\p v\|_{B_{\infty, 2, x}^\delta}+\|\p v\|_{L^\infty_x})\|\p \C\|^2_{H^\delta_x} dt'\les 1.
\end{split}
\end{align}

The last line of (\ref{9.11.10.19}) has the same integrand as $\sum_{\mu>1}|\I_\mu|$ with $\I_\mu$ defined in Lemma \ref{9.08.11.19}.   We will repeat the procedure in the proof of Lemma \ref{9.08.11.19}. Due to the integral region is changed to $\widetilde{\D^+}\cap \{u\ge u_0\}$,  the only difference is to control the two additional boundary terms on $C_u$ generated by the integration by parts in (\ref{9.11.9.19}), which are detailed below
\begin{align*}
\J_{\mu,b}&=\mu^{2\delta}\int_{t_{\tmin}}^t \int_{\Sigma_{t'}\cap \{u\ge u_0\}}\p_m (P_\mu(\p_n v^j \c \p^n \Omega_j e^{-\varrho})P_\mu \curl \C^m)d\mu_e dt',\\
\K_{\mu,b}&=\int_{C_{u_0}}c^3\mu^{2\delta}P_\mu(\p_j \varrho \p_m\Omega^j e^{-\varrho})P_\mu \curl \C^m d\mu_\ga dt',
\end{align*}
where the second term is contributed by $\K_\mu$ in (\ref{9.11.9.19}).

For the first term, in the exactly same calculation as in (\ref{9.30.6.19}), we carry out integration by parts on $\Sigma_{t'}\cap \{u\ge u_0\}$,
\begin{align*}
\J_{\mu,b}&= \int_{\widetilde{\D^+}\cap\{u\ge u_0\}} \mu^{2\delta}\div_g(c P_\mu(\p_n v^j \c \p^n \Omega_j e^{-\varrho})P_\mu \curl \C) d\mu_g dt'\\
&=\mu^{2\delta}\int_{C_{u_0}}c^3  P_\mu(\p_n v^j \c \p^n \Omega_j e^{-\varrho})P_\mu \curl \C^m \bN^n c^{-2}\delta_{mn} d\mu_\ga dt'.
\end{align*}
 Hence
\begin{align}
&\sum_{\mu>1}( |\J_{\mu,b}|+|\K_{\mu,b}|)\les  \|\mu^\delta P_\mu \curl \C\|_{l_\mu^2 L^2(C_{u_0})}\|\mu^\delta P_\mu\big((\p v, \p\varrho)  \c \p \Omega e^{-\varrho}\big)\|_{l_\mu^2 L^2(C_{u_0})}.\label{9.11.12.19}
\end{align}
We then use (\ref{9.10.4.19}) combined with (\ref{9.08.8.19}) to treat the last term
\begin{align}\label{9.11.11.19}
\begin{split}
\|\mu^\delta P_\mu\big((\p v, \p\varrho)  \c \p \Omega e^{-\varrho}\big)\|_{l_\mu^2 L^2(C_{u_0})}&\les T^\f12 \|\La^{\f12+\delta}\big((\p v, \p\varrho)  \c \p \Omega e^{-\varrho}\big)\|_{L_t^\infty L^2_x}\\
&\les T^\f12( \|\La^{1+\delta}(\p v, \p \varrho)\|_{L_t^\infty L_x^2} \|\p \Omega e^{-\varrho}\|_{L_t^\infty H^1_x}\\
&+\|\p v, \p \varrho\|_{L_t^\infty H^1_x}\|\La^{1+\delta}(\p \Omega e^{-\varrho})\|_{L_t^\infty L_x^2}).
\end{split}
\end{align}
Note that for the last term in the above inequality, we can apply (\ref{9.08.21.19}) to $F=\p \Omega$ to bound
\begin{equation*}
\|\La^{1+\delta}(\p \Omega e^{-\varrho})\|_{L_t^\infty L_x^2}\les \|\p \Omega\|_{H^{1+\delta}_x}.
\end{equation*}
Substituting the above inequality to (\ref{9.11.11.19}) gives
\begin{align*}
&\|\mu^\delta P_\mu\big((\p v, \p\varrho)  \c \p \Omega e^{-\varrho}\big)\|_{l_\mu^2 L^2(C_u)}\\
&\les T^\f12( \|\La^{1+\delta}(\p v, \p \varrho)\|_{L_t^\infty L_x^2} \|\p \Omega e^{-\varrho}\|_{L_t^\infty H^1_x}+\|\p v, \p \varrho\|_{L_t^\infty H^1_x}\|\p \Omega \|_{L_t^\infty H^{1+\delta}_x})\\
&\les 1,
\end{align*}
where we employed Corollary \ref{eng_wave}, (\ref{9.23.4.19}), (\ref{9.07.15.19}) and (\ref{5.04.17.19}).

Hence combining the above inequality with  (\ref{9.11.12.19}),  we conclude
\begin{equation}\label{10.25.1.19}
\sum_{\mu>1}( |\J_{\mu,b}|+|\K_{\mu,b}|)\les \|\mu^\delta P_\mu \curl \C\|_{l_\mu^2 L^2(C_{u_0})}.
\end{equation}
Note the change of the integral region of $\I_\mu$ in (\ref{9.30.7.19}) to $\widetilde{\D^+}\cap \{u\ge u_0\}$ only requires us to provide the above two additional estimates.  The remaining estimate for the term is carried out exactly as in Lemma \ref{9.08.11.19}, and controlled by the same bound. 
 Then by the estimate in Lemma \ref{9.08.11.19},  in view of (\ref{9.11.10.19}), (\ref{9.11.13.19}) and using (\ref{9.08.6.19}), we derive
\begin{align*}
\sum_{\mu>1}|\B_\mu|&\les \sum_{\mu>1}(|\J_{\mu,b}|+|\K_{\mu,b}|)+1+ \int_0^t (\|\curl \C\|_{\dot{H}^\delta_x}+1)\{\|\p v\|_{B_{\infty, 2,x}^\delta}\\
&+(1+\|\p v\|_{L_x^\infty})(\|\curl \C\|_{\dot{H}^\delta_x}+1)\} dt' +\sup_{0\le t'\le t}\|\La^\delta\curl \C\|_{L_x^2}(t')\\
&\les 1+\|\mu^\delta P_\mu \curl \C\|_{l_\mu^2 L^2(C_{u_0})},
\end{align*}
where we used (\ref{10.25.1.19}), Corollary \ref{comp_3}, (\ref{BA1}) and (\ref{9.07.15.19}) in the above.

Recall from (\ref{9.30.8.19}), we can conclude by  using (\ref{9.07.15.19}) that
\begin{align*}
\sum_{\mu>1}&|\int_{C_{u_0}} c^3 \mu^{2\delta}|P_\mu \curl \C|_e^2  d\mu_\ga  dt|\\
&\les 1+\|\mu^\delta P_\mu \curl \C\|_{l_\mu^2 L^2(C_{u_0})}+ \sup_{t'\le t}\int_{\Sigma_{t'}} \sum_{\mu>1}\mu^{2\delta} |P_\mu \curl \C|^2  d\mu_e\\
&\les 1+\|\mu^\delta P_\mu \curl \C\|_{l_\mu^2 L^2(C_{u_0})}.
\end{align*}
This implies (\ref{9.11.6.19}). Hence the proof of Proposition \ref{9.30.15.19} is complete.
\end{proof}

\section{Fundamental structures for  the causal geometry of the acoustic spacetime}\label{fundstr}

The Hessian of the optical function $u$ defined in (\ref{optical}) is usually decomposed into the connection coefficients  (\ref{ricc_def}) of  the null tetrads $\{L, \Lb, e_A, A=1,2\}$.
  Among them, we focus on controlling the null second fundamental form $\chi$ and the torsion $\zeta$. $\tr\chi$ is usually controlled by using the Raychaudhuri equation (see (\ref{s1})), which contains the Ricci component $\bR_{LL}$; while for  $\zeta$ and $\chih$, one may rely on the Hodge system on $S_{t,u}$ (see (\ref{dze}), (\ref{dcurl}) and (\ref{dchi})), due to the limited regularity on the Riemann curvature. The estimates on them are coupled together via a bootstrap argument.  The control of $\tr\chi$ plays a more fundamental role since it is crucially used to guarantee the coordinate system by $t, u$ and $\omega \in {\mathbb S}^2$ to be well-defined in $\widetilde{\D^+}$ as explained in Section \ref{null_cone_set}. To control $\tr\chi$, it typically relies on the specific structure in $\bR_{LL}$ to gain regularity over the general control from the spacetime metric, c.f. \cite{KRduke}-\cite{KRd} and \cite{Tataru, Wangricci, Wangrough}.

Due to the decoupling method and a series of cancellations, the regularity of the general spacetime metric is  maintained at the same level as in \cite{Wangrough},  under our assumption of the data. In the acoustic spacetime, the main defect caused by the rough vorticity derivative actually occurs on the acoustic null cone, since $\curl \Omega$ is the main uncontrollable term in $\bR_{LL}$, appearing in the Raychaudhuri equation, which at the first glance fails the estimate on $\tr\chi$, 
and thus collapses any further analysis on the cone.

In order to gain the sufficient regularity for proving Theorem \ref{BT} as well as showing (\ref{pba2})-(\ref{wras}), we have to investigate in a much more delicate level on the specific structure of the acoustic metric. Here we uncover two fundamental structures: one is on the $\sn (\curl \Omega_\bN)$ given in Proposition \ref{struc}; the other is on $k_{\bN\bN}-\f12 \Xi_L$ given in Proposition \ref{6.14.2.19}, where $\Xi$ is a one-form defined in (\ref{ricc6.7.2}).  Both of the structures  are the crucial ingredients for the analysis of $\tr\chi$ carried out in Section \ref{causal_reg}.

We start with recalling the basic calculations by virtue of the null tetrads $\{L, \Lb, e_A, A=1,2\}$, which appeared in \cite{KRduke,Wangrough} for the spacetime therein.
\begin{proposition}\label{6.7con}
\begin{align}
\bd_A e_4=\chi_{AB} e_B-k_{A\bN} e_4, &\qquad  \bd_A e_3=\chib_{AB}e_B+k_{A\bN} e_3, \label{7.21.3.19}\\
\bd_4 e_4=-k_{\bN\bN} e_4,   &\qquad  \bd_4 e_3=2\zb_A e_A+  k_{\bN\bN} e_3,\label{6.29.5.19} \\
\bd_3 e_4=2\zeta_A e_A+k_{\bN\bN} e_4, & \qquad \bd_4 e_A=\sn_L e_A+\zb_A e_4, \label{6.29.6.19}\\
\bd_B e_A=\sn_B e_A+\f12 \chi_{AB} e_3+\f12 \chib_{AB} e_4& \qquad \bd_3 e_3 = (-2\zeta_A+2k_{\bN A}) e_A -k_{\bN\bN} e_3\label{6.29.7.19}\\
\chi_{AB}=\theta_{AB}-k_{AB},\quad\zb^A=-k^A_\bN, &\qquad  \zeta^A=\sn\log \bb+k^A_\bN.\label{3.19.1}
\end{align}
where $e_4=L$ and $e_3=\Lb$.
\end{proposition}
As a direct consequence of (\ref{6.29.5.19})-(\ref{6.29.7.19}), the following decompositions hold under the null tetrad.
\begin{corollary}\label{decom_wave}
Let $h=\f12 \tr\chi$ and $\hb=\f12 \tr\chib$. For a scalar function $f$, there holds
\begin{equation}\label{6.30.1.19}
\Box_\bg f=\sD f-\Lb L f-(\hb-k_{\bN\bN}) L f-h\Lb f+2\zeta^A \sn_A f;
\end{equation}
and
\begin{equation}\label{6.30.2.19}
\Box_\bg f=\sD f-L \Lb f-(h-k_{\bN\bN})\Lb f-\hb L f+2\zb^A \sn_A f.
\end{equation}
\end{corollary}

\subsubsection{Commutation formulas}
We recall the following commutation relations used in \cite[Section 5]{Wangrough} (see also in \cite{KRduke,KREinst}).
\begin{proposition}
(1) There holds for the scalar functions $f$ that
\begin{equation}\label{3.19.2}
[L,\bT]f=\f12[L, \Lb]f=(\zb^A-\zeta^A)\sn_A f-k_{\bN\bN} \bN f;
\end{equation}
%
(2) There holds for $S_{t,u}$-tangent $m$-covariant tensor fields $U_A$ that
\begin{equation}\label{cmu2}
\begin{split}
&\sn_L\sn_B U_A-\sn_B \sn_L U_A  \\
&=-\chi_{BC}\c \sn_C U_A+\sum_{i}(\chi_{A_i B} \zb_C-\chi_{BC} \zb_{A_i}
+\bR_{{A_i}C4 B})U_{A_1\cdots\ckk C\cdots A_m}
\end{split}
\end{equation}
and for any scalar function $f$ there holds
\begin{equation}\label{cmu_2}
[L,\sn_A] f=- \chi_{AB} \sn_B f.
\end{equation}
Consequently, for any scalar function $f$ there holds
\begin{align}
L\sD f+  \tr\chi \sD f= \sD L f-2\chih\c \sn^2 f-\sn_A\chi_{AC}\sn_C f
+(\tr\chi \zb_C-\chi_{AC}\zb_{A}-\delta^{AB}\bR_{CA4B}) \sn_C f.\label{tran1}
\end{align}
\end{proposition}

\subsubsection{Null Structure Equations}

 We will rely heavily on the following structure equations for the connection coefficients
on null hypersurfaces $C_u$ in $\widetilde{\D^+}$ (see \cite[Chapter 7]{CK}, \cite{KRduke} and \cite[Section 5]{Wangrough}):
\begin{proposition}[Transport equations and Hodge systems for connection coefficients]
\begin{align}
&L \bb=-\bb { k}_{\bN\bN}, \label{lb}\\
&L\tr\chi+\f12 (\tr\chi)^2=-|\chih|^2-{ k}_{\bN\bN} \tr\chi-\bR_{44}, \label{s1} \displaybreak[0]\\
&\sn_L \chih_{AB}+\f12 \tr\chi \chih_{AB}=-{k}_{\bN\bN} \chih_{AB}-(\bR_{4A4B}-\f12 \bR_{44} \delta_{AB}), \label{s2} \displaybreak[0]\\
&L \tr\chib+\f12 \tr\chi \tr\chib=2\div \zb+k_{\bN\bN} \tr\chib-\chih\c \chibh+2|\zb|^2+\delta^{AB}\bR_{A34B}, \label{mub} \displaybreak[0]\\
&\sn_L \zeta+\f12\tr\chi \zeta=-(k_{B\bN}+\zeta_B) \chih_{AB}-\f12 \tr\chi k_{A\bN}-\f12 \bR_{A4 43}, \label{tran2} \displaybreak[0]\\
&(\sl{\div} \chih)_A+\chih_{AB}\c k_{B\bN}=\f12(\sn \tr\chi+k_{A\bN} \tr\chi)+\bR_{B4BA}, \label{dchi} \displaybreak[0]\\
&\sl{\div} \zeta=\f12(\mu-k_{\bN\bN} \tr\chi-2|\zeta|^2-|\chih|^2-2k_{AB}\chih_{AB})-\f12\delta^{AB}\bR_{A43B}, \label{dze} \displaybreak[0]\\
&\sl{\curl} \zeta=-\f12 \chih\wedge \chibh+\f12 \ep^{AB}\bR_{B43A}, \label{dcurl} \displaybreak[0]\\
&\sn_\Lb \chih_{AB}+\f12 \tr\chib \chih_{AB}=-\f12 \tr\chi\chibh_{AB}+2\sn_A\zeta_B-\div \ze \delta_{AB}+k_{\bN\bN} \chih_{AB} \label{3chi} \\
&\quad\quad\quad\quad\quad \quad\quad+(2\zeta_A\ze_B-|\ze|^2\delta_{AB})+\bR_{A43B}-\f12 \delta^{CD}\bR_{C43D}\delta_{AB},\nn
\end{align}
where the mass aspect function $\mu:=\Lb \tr\chi+\f12 \tr\chi \tr\chib$; and for an $S_{t,u}$-tangent tensor field $F$,  $\sn_L F:=L^\mu\bd_\mu F$, with $\bd$ the covariant derivative of $(\M,\bg)$.

There holds the following Hodge system on $S_{t,u}$ contained in $\widetilde{\D^+}$
\begin{equation}\label{6.17.7.19}
\sl{\div}(\zeta-\zb)=|\zb|^2-|\zeta|^2-(\zb-\zeta)\sn \varphi, \quad\sl{\curl} (\zeta-\zb)=-2\sl{\curl}\zb
\end{equation}
with $\varphi= \log \sqrt{|\ga|}-\log\sqrt{|\ckk\ga|}$ with $\ckk\ga=(t-u)^{-2}\ga\rp{0}$.
\end{proposition}
The schematic form of (\ref{6.17.7.19}) was used in \cite[Section 5]{Wangrough} crucially to provide the control of $\zeta$,  since the regularity of null cones in \cite{Wangrough} is much weaker than the previous works.
   The explicit form in  (\ref{6.17.7.19}) shows clearly  the relation of $\zeta$ and $\zb$, which simplifies the control on $\zeta$ in Section \ref{causal_reg}.
\begin{proof}
The majority of the above equations has appeared for a couple of times in literature.
 We only prove the new formula (\ref{6.17.7.19}).

 We first note that the $\sl{\curl}$ equation is a direct consequence of
$
\zeta+\zb=\sn \log \bb,
$
since the right hand side of this identity vanishes after taking $\sl{\curl}$. The identity itself can be obtained by using the last two equations in (\ref{3.19.1}).

Note  $L(t-u)=1$ and $\Lb(t-u)=1-2\bb^{-1}$.
 To show the first equation,  we use the definition  $\varphi=\log \sqrt{|\ga|}-\log\sqrt{|\ckk \ga|}$ to compute
\begin{equation}\label{9.30.9.19}
 L \varphi=\tr\chi-\frac{2}{\tir}, \quad \Lb \varphi=\tr\chib+(2\bb^{-1}-1)\frac{2}{\tir},\quad \bN \varphi=\tr\theta-\frac{2\bb^{-1}}{\tir},
 \end{equation}
where the last one can be derived by the first two by using $2\bN=L-\Lb$. Hence, due to (\ref{3.19.2}),
\begin{equation}\label{9.30.10.19}
[L, \Lb]\varphi=2(\zb^A-\zeta^A)\sn_A \varphi-2 k_{\bN\bN} \bN \varphi=2(\zb_A-\zeta_A)\sn_A \varphi-2 k_{\bN\bN} (\tr\theta-\frac{2\bb^{-1}}{\tir}).
\end{equation}
On the other hand, by using (\ref{9.30.9.19}), we can derive
\begin{equation}\label{9.30.11.19}
L\tr\chib-\Lb \tr\chi+L((2\bb^{-1}-1)\frac{2}{\tir})+\Lb(\frac{2}{\tir})=-2k_{\bN\bN} (\tr\theta-\frac{2\bb^{-1}}{\tir})+2(\zb-\zeta) \c \sn \varphi.
\end{equation}
Since we can directly check
$$
L((2\bb^{-1}-1)\frac{2}{\tir})+\Lb(\frac{2}{\tir})=4\tir^{-1} L(\bb^{-1})=4\tir^{-1}\bb^{-1} k_{\bN\bN},
$$
substituting the identity to (\ref{9.30.11.19}) leads to a cancellation with the right hand side,
\begin{equation}\label{9.30.12.19}
L\tr\chib-\Lb\tr\chi=-2k_{\bN\bN} \tr\theta+2(\zb-\zeta)\c \sn \varphi.
\end{equation}
Next, we substitute (\ref{mub}) and (\ref{dze}) to the left hand side of the above identity. Using
\begin{equation}\label{10.31.2.19}
\chib-\chi=-2\theta,\qquad \chib_{AB}+\chi_{AB}=-2k_{AB}
\end{equation}
we can obtain
\begin{equation*}
2\div (\zb-\zeta)-2\tr\theta k_{\bN\bN}+2(|\zeta|^2-|\zb|^2)=L \tr\chib-\Lb \tr\chi=-2k_{\bN\bN} \tr\theta+2(\zb-\zeta)\c \sn \varphi.
\end{equation*}
This implies the first equation in (\ref{6.17.7.19}).
 \end{proof}
\subsubsection{Structures of  Ricci curvature in the acoustic spacetime}
As seen in (\ref{pba2}), we need to provide control on $L_x^\infty$ norm of $\tr\chi-\frac{2}{\tir}$. The fundamental structure uncovered in \cite{Kcom} and \cite{KRduke} is  an important  decomposition for $\bR_{LL}$, based on the
 following formula of Ricci component under the cartesian coordinates,
\begin{equation}\label{ricc6.7.1}
\bR_{\a\b}=-\f12 \Box_\bg (\bg_{\a\b})+\f12 (\bd_\a \Xi_\b+\bd_\b \Xi_\a)+S_{\a\b},
\end{equation}
where $\Xi$ is a 1-form defined by
\begin{equation}\label{ricc6.7.2}
\Xi_\ga=(\Ga_{\a\b}^\eta-{\hat \Ga}_{\a\b}^\eta)\bg^{\a\b}\bg_{\ga\eta},
\end{equation}
with $\hat\Ga $ being the Christoffel symbol of a smooth reference metric $\hat \bg$.
For convenience, $\hat\bg$ is chosen to be the Minkowski metric $\bm$. Under this choice, the symmetric two-tensor field
$S_{\a\b}$ is quadratic in $\bp \bg$.
Since
\begin{equation*}
\Ga_{\a\b\ga}=\f12(\p_\b \bg_{\a\ga}+\p_\a \bg_{\b \ga}-\p_\ga \bg_{\a \b}),
\end{equation*}
 we  can directly compute
\begin{equation}\label{9.30.13.19}
\Xi_\ga=\bg^{\a\b}(\p_\a \bg_{\b\ga}-\f12 \p_\ga \bg_{\a\b}).
\end{equation}

Now by adopting the decomposition of $\bR_{\a\b}$ in (\ref{ricc6.7.1}), we will show in (\ref{6.23.2.19}) that
 there holds for the component of Ricci curvature $\bR_{44}$ of the acoustic metric that
\begin{equation}\label{9.30.14.19}
\bR_{44}=L(\Xi_L)-e^{\varrho}\delta_{ij}\bN^j \curl \Omega^i+\Q(\bp \bg, \bp \bg),
\end{equation}
where  $\Xi_L=\Xi_\mu L^\mu$, (alternatively, $\Xi_L=\Xi_4$). For the angular derivative of $\curl \Omega^i \bN^j \delta_{ij}$, we will obtain a trace decomposition.  


Recall from (\ref{metric}) that under the Cartesian coordinate frame $\p_t=\p_0, \p_i, i=1,2,3$
\begin{align}\label{metric_3}
\begin{split}
&\bg_{00}=-1+c^{-2}|v|^2, \quad \bg_{0i}=-c^{-2} v_i, \quad \bg_{ij}=c^{-2}\delta_{ij}\\
&\bg^{00}=-1, \quad \bg^{0i}=-v^i \quad \bg^{ij}=c^2 \delta^{ij}-v^i v^j.
\end{split}
\end{align}

We denote by $\ti \pi=f(\bg) \bp \bg$ with $f$ a smooth function, and $\pi=\ti \pi\c X$, that is   the contraction to $\ti\pi$ with the metric $\bg$ by the tensor fields $L, \Lb$ or $\Pi$ in (\ref{metric_2}) denoted in general by $X$.

We first prove the following  decompositions by direct calculations.
\begin{proposition}[Decompositions of Ricci components]\label{struc}
\begin{align}
&\bR_{34}=-c^{-2}\bd^\a(v^i)\bd_\a(v^j)\delta_{ij}+\f12 c^2 \Box_\bg (c^{-2})+\f12 (\bd_L\Xi_{\Lb}+\bd_\Lb \Xi_L) +S_{\bT \bT}-S_{\bN\bN},\label{5.27.2.19}\\
&\bR_{44}=-\exp \varrho \bN^j \curl \Omega_j+\delta^{ij} c^{-2} \bN^j \sQ^i-\f12 c^2 \Box_\bg (c^{-2})+\bd_L \Xi_L +S_{44},\nn\\
&\quad\quad\quad -c^{-2} \bd^\a(v^i) \bd_\a(v^j) \delta_{ij}+2\delta_{ij}\bN^j \bd^\a(c^{-2}) \bd_\a(v^i),\label{6.23.2.19}\displaybreak[0]\\
  &\Pi^{ij} \p_j (\bN^m \curl \Omega_m)=\Pi^i_l\bd_L(\Pi^{jl} \curl\Omega_j)+\Pi^{ij}\tensor{\ep}{_{jm}^l}\curl^2 \Omega_l \bN^m+ \p \Omega\c (\chi+\pi)\c X,
   \label{9.12.2.19}\\
&\Pi^{ij}\bR_{ij}=-c^2 \Box_\bg (c^{-2})+\sn_A \Xi^A+\tr\theta \Xi_\bN+\pi\c \pi,\label{6.16.1.19}
\end{align}
where $\Pi$ is defined in (\ref{metric_2}), $\c X$ means contracted by the combination of  null vector fields $L, \Lb$, or by $\Pi$.
\end{proposition}
\begin{remark}
Note that for a $\Sigma_t$ tangent tensor $F$
\begin{equation}\label{10.01.2.19}
\nab_A F_B= \sn_A F_B+\theta_{AB} F_\bN.
\end{equation}
Thus, by using ${}\rp{g}\Ga$ to represent the Christoffel symbol of the Riemiannian metric $g$, we can derive
\begin{align}\label{6.30.4.19}
\begin{split}
\bN^i\curl F_i&=\bN^i\ep_i^{\ mn} \p_m(F_n)=\bN^i \ep_i^{\ AB} (\nab_A F_B+{}\rp{g}\Ga_{mn}^l F_l e_A^m e_B^n)\\
&=\ep^{AB}(\sn_A F_B+\theta_{AB} F_\bN +{}\rp{g}\Ga_{mn}^l F_l e_A^m e_B^n)=\ep^{AB} \sn_A F_B,
\end{split}
\end{align}
where other terms are cancelled since $\ep^{AB}$ is the volume form of $(S_{t,u}, \ga)$, anti-symmetric about $A, B=1,2$.

Therefore the first term on the right of (\ref{6.23.2.19}) is
\begin{equation}\label{9.12.1.19}
\exp \varrho \bN^i \curl \Omega_i= \exp \varrho \ep^{AB} \sn_A \Omega_B,
\end{equation}
which does not directly take the form of $\sn_L P+E$ with the scalar functions $P$ and $E$ verifying good estimates.
\end{remark}
\begin{proof}[Proof of Proposition \ref{struc}]
 We first compute by using (\ref{ricc6.7.1})
\begin{align*}
\bR_{34}&=\bR_{\bT \bT}-\bR_{\bN \bN}\\
&=-\f12 (\bT^\a \bT^\b-\bN^\a \bN^\b) \Box_\bg \bg_{\a\b}+\f12 (\bd_L\Xi_{\Lb}+\bd_\Lb \Xi_L)+S_{\bT \bT}-S_{\bN \bN}.
\end{align*}
In view of (\ref{metric_3}), we compute
\begin{align*}
\bT^\a \bT^\b \Box_\bg \bg_{\a\b}&= \bT^0 \bT^0 \Box_{\bg} \bg_{00}+2\bT^0 \bT^i \Box_\bg \bg_{0i}+\bT^i \bT^j \Box_\bg \bg_{ij}\\
&=\Box_\bg (-1+c^{-2}|v|^2)-2 v^i \Box_\bg (c^{-2}v_i)+v^i v^j \delta_{ij}\Box_\bg (c^{-2})\\
&=2c^{-2}\bd^\a (v^i) \bd_\a (v^j)\delta_{ij};
\end{align*}
and
\begin{equation*}
\bN^\a \bN^\b \Box_\bg \bg_{\a\b}=\bN^i \bN^j\Box_\bg \bg_{ij}=\bN^i \bN^j \delta_{ij}\Box_\bg(c^{-2})=c^2\Box_\bg (c^{-2}).
\end{equation*}
Thus we can obtain (\ref{5.27.2.19}).

Next we calculate $\bR_{44}$.  Noting that $L=\bT+\bN$ gives $L^i=\bT^i+\bN^i=v^i+\bN^i$, we have
\begin{align}
&L^\a L^\b \Box_\bg \bg_{\a\b}\nn\\
&=L^0 L^0 \Box_\bg \bg_{00}+2 L^0 L^i \Box_\bg \bg_{0i}+ L^i L^j \Box_\bg \bg_{ij}\nn\\
&=\Box_\bg \bg_{00}+2(v^i+\bN^i) \Box_\bg \bg_{0i}+(v^i+\bN^i)(v^j+\bN^j)  \Box_\bg \bg_{ij}\nn\\
&= \Box_\bg(-1+c^{-2}|v|^2)+2(v^i +\bN^i)\Box_\bg(-c^{-2}v_i)+|v+\bN|_e^2 \Box_\bg(c^{-2})\nn\\
&=\Box_\bg(c^{-2}|v|^2)-2 v^i \Box_\bg(c^{-2}v_i)-2\bN^i \Box_\bg(c^{-2}v_i)+|v+\bN|^2_e \Box_\bg(c^{-2})\nn\\
&= -2\delta_{ij}\bN^j \Box_\bg(c^{-2}v^i)+\delta_{ij}(2v^i \bN^j+\bN^i \bN^j)\Box_\bg (c^{-2})+2c^{-2} \bd^\a(v^i) \bd_\a(v^j) \delta_{ij} \nn\\
&=-2\delta_{ij}c^{-2}\bN^j \Box_\bg v^i+c^2 \Box_\bg (c^{-2})+2c^{-2} \bd^\a(v^i) \bd_\a(v^j) \delta_{ij}-4\delta_{ij}\bN^j \bd^\a(c^{-2}) \bd_\a(v^i)\nn.
\end{align}
Substituting (\ref{4.10.1.19}) to the first term on the right hand side yields
\begin{align}
L^\a L^\b \Box_\bg \bg_{\a\b}&= 2\exp\varrho\delta_{ij} \bN^j \curl \Omega^i-2\delta_{ij} c^{-2} \bN^j \sQ^i\nn\\
&+c^2 \Box_\bg (c^{-2})+2c^{-2} \bd^\a(v^i) \bd_\a(v^j) \delta_{ij}-4\delta_{ij}\bN^j \bd^\a(c^{-2}) \bd_\a(v^i)\nn.
\end{align}
(\ref{6.23.2.19})  follows by substituting the above identity to the formula   (\ref{ricc6.7.1}) for $\bR_{LL}$.

Now we prove (\ref{9.12.2.19}).
 \begin{align*}
 \Pi^{ij} \p_j (\bN^m \curl \Omega_m)&=\Pi^{ij} (\p_j \curl \Omega_m\bN^m+\curl \Omega_m \p_j \bN^m)\\
 &=\Pi^{ij} (\p_m (\curl \Omega)_j \bN^m)+\Pi^{ij} \tensor{\ep}{_{jm}^l} \curl^2\Omega_l \bN^m+\Pi^{ij} \curl \Omega_m \p_j \bN^m\\
 &= \Pi^{ij}( L(\curl\Omega_j)-\bT(\curl \Omega_j))+\Pi^{ij} \tensor{\ep}{_{jm}^l} \curl^2\Omega_l \bN^m+\Pi^{ij} \curl \Omega_m \p_j \bN^m.
 \end{align*}
 By using (\ref{4.25.4.19}) and the first equation in (\ref{4.23.1.19}), we have
 \begin{equation}\label{9.12.3.19}
 \bT\curl \Omega_j=\bT \varrho \curl\Omega_j +\p v \p \Omega=\p v\p \Omega.
 \end{equation}
 For the other term,
 \begin{align}\label{10.01.1.19}
 \begin{split}
 \Pi^i_l\Pi^{jl} L (\curl \Omega)_j&= \Pi^i_l\Pi^{jl}\big(\bd_L(\curl \Omega)_j+ {}\rp{\bg}\Ga\c \curl \Omega\big)\\
 &=\Pi^i_l\bd_L\big(\Pi^{jl} \curl\Omega_j\big)-\Pi^i_l \curl \Omega_j \bd_L \Pi^{jl}+ \Pi {}\rp{\bg}\Ga\c \curl \Omega\\
 &=\Pi^i_l\bd_L(\Pi^{jl} \curl\Omega_j)+\curl\Omega_\bN k_{\bN j}\Pi^{ij}  +\pi\c \curl \Omega.
 \end{split}
 \end{align}
 For deriving the last line, by using (\ref{6.29.5.19}) and (\ref{3.19.1}) we computed
 \begin{align*}
 \Pi^{ij}\curl \Omega_j \bd_L \Pi^{jl}&=\Pi^i_\nu \curl \Omega_\mu(\bd_L \bT^\mu \bT^\nu +\bT^\mu \bd_L \bT^\nu-\bd_L \bN^\mu \bN^\nu-\bN^\mu \bd_L \bN^\nu)\\
 &=-\Pi^i_\nu \curl \Omega_l \bN^l \bd_L \bN^\nu=-\curl\Omega_\bN k_{\bN j}\Pi^{ij}.
 \end{align*}
 Also using $\Pi^{ij}\p_j \bN^m=\Pi^{ij}(\nab_j \bN^m-{}\rp{g}\Ga\c \bN)=(\chi+\pi)\c X$,
 (\ref{9.12.2.19}) can then be derived by combining (\ref{10.01.1.19}) with (\ref{9.12.3.19}).

At last we prove (\ref{6.16.1.19}) by using (\ref{ricc6.7.1}) and (\ref{10.01.2.19}),
\begin{align*}
\Pi^{ij} \bR_{ij}&=-\f12 \Pi^{ij} \Box_\bg \bg_{ij}+\Pi^{ij} \bd_i \Xi_j+\Pi^{ij}S_{ij}\\
 &=-\f12 \Pi^{ij}c^{-2} \delta_{ij} c^2 \Box_\bg (c^{-2})+\sn_A \Xi^A+\tr\theta \Xi_\bN+\pi\c \pi\\
 &=-c^2\Box_\bg(c^{-2})+\sn_A \Xi^A+\tr\theta \Xi_\bN+\pi\c \pi,
\end{align*}
which gives (\ref{6.16.1.19}).
\end{proof}

Besides the structure of Ricci curvature, we give an important cancellation between $k_{\bN\bN}$ and $\Xi_4$.
\begin{proposition}\label{6.14.2.19}
\begin{align}
&\Xi_\mu\bT^\mu=\Tr k,\label{6.21.9.19}\\
&\Xi_j=\p_j(\log c-\varrho)\label{6.14.1.19},\\
&k_{\bN\bN}=\f12 \big(\Xi_L-L(\log c+\varrho)-2L(v)_\bN\big),\label{7.04.8.19}
\end{align}
where we denote, for any vector field $Y$,  $Y(v)_\bN=Y(v^i) \bN^j \bg_{ij}$.

\end{proposition}
\begin{proof}
We first compute $\Xi_\ga \bT^\ga$ by using (\ref{9.30.13.19}).
\begin{align*}
\Xi_{\bT}&=\bg^{\a\b}(\p_\a \bg_{\b \ga}-\f12 \p_\ga \bg_{\a\b}) \bT^\ga\\
&=\bg^{\a\b} \p_\a \bg_{\b0} \bT^0+\bg^{\a\b} \p_\a \bg_{\b i} \bT^i-\f12 \bg^{\a\b} \bT(\bg_{\a\b}).
\end{align*}
The last term on the right hand side can be computed as follows
\begin{align*}
\bg^{\a\b} \bT (\bg_{\a\b})&=2(-v^i) \bT(-c^{-2} v_i)+(c^2\delta^{ij}-v^i v^j )\bT(c^{-2}\delta_{ij})-\bT(-1+c^{-2}|v|^2)\\
&=2 v^i \bT(c^{-2} v_i) +c^2 \delta^{ij} \bT(c^{-2})\delta_{ij}-|v|^2 \bT(c^{-2})-\bT(c^{-2}|v|^2)\\
&=-6 \bT\log c.
\end{align*}
Now we compute the remaining terms
\begin{align*}
&\bg^{\a\b} \p_\a \bg_{\b0} \bT^0+\bg^{\a\b} \p_\a \bg_{\b i} \bT^i\\
&=\bg^{00} \p_0\bg_{00}+\bg^{0i}(\p_0 \bg_{i0}+\p_i \bg_{00})+\bg^{ij}\p_i \bg_{j0}\\
&+\bg^{00} \p_0 \bg_{0i}v^i+\bg^{0j}(\p_0 \bg_{ij}+\p_j \bg_{0i}) v^i +\bg^{lj}\p_l \bg_{ji} v^i\\
&=-\p_0(-1+c^{-2}|v|^2)-v^i (\p_0(-c^{-2}v_i)+\p_i(-1+c^{-2}|v|^2))+(c^2\delta^{ij}-v^i v^j) \p_i (-c^{-2} v_j)\\
&+(-1)\p_0(-c^{-2} v_i) v^i +(-v^j)v^i (\p_0 (c^{-2}\delta_{ij})+\p_j (-c^{-2}v_i))+(c^2\delta^{lj}-v^l v^j )\p_l (c^{-2}\delta_{ij})v^i \\
&=-\bT(c^{-2}|v|^2)+v^i \bT(c^{-2}v_i)+c^2 \delta^{ij}\p_i(-c^{-2}v_j)+\bT(c^{-2}v_i)v^i-v^j \bT(c^{-2}\delta_{ij})v^i\\
&\quad+c^2 \delta^{lj}\p_l (c^{-2}\delta_{ij})v^i\\
&=- \div v.
\end{align*}
Combining the above two calculations yields (\ref{6.21.9.19}) in view of (\ref{5.24.1.19}).

We now consider (\ref{6.14.1.19}).
\begin{align*}
\bg^{\a\b} \p_\a \bg_{\b j}&=\bg^{00} \p_0 \bg_{0j}+\bg^{0i} \p_0 \bg_{ij}+\bg^{i0} \p_i \bg_{0j}+\bg^{il} \p_i \bg_{lj}\\
&=- \p_0(-c^{-2} v_j)+(-v^i) [\p_0 (c^{-2} \delta_{ij})+\p_i(-c^{-2} v_j)]+\bg^{il} \p_i(c^{-2}\delta_{lj})\\
&=\p_0(c^{-2} v_j)-v^i[\p_0(c^{-2}) \delta_{ij}-\p_i(c^{-2} v_j)]+(c^2 \delta^{il}-v^j v^l) \p_i(c^{-2}) \delta_{lj}\\
&=c^{-2} \p_0 v_j+v^i \p_i (c^{-2} v_j)+ (c^2 \delta^{il}-v^i v^l) \p_i(c^{-2}) \delta_{lj}\\
&=c^{-2}\bT v_j-2\p_j \log c,
\end{align*}
\begin{align*}
\bg^{\a\b} \p_j \bg_{\a\b}&=2 \bg^{0i} \p_j \bg_{0i}+\bg^{00}\p_j \bg_{00}+\bg^{il}\p_j \bg_{il}\\
&=-2v^i \p_j (-c^{-2} v_i)-\p_j(-1+c^{-2}|v|^2)+(c^2 \delta^{il}-v^i v^l)\p_j \bg_{il}\\
&=2 v^i \p_j (c^{-2} v_i)-\p_j (c^{-2} |v|^2)+(c^2 \delta^{il}-v^i v^l) \p_j (c^{-2} \delta_{il})\\
&=-6 \p_j \log c.
\end{align*}
Combining the above calculations with (\ref{9.30.13.19}) implies
\begin{equation*}
\Xi_j= \bg^{\a\b} (\p_\a\bg_{\b j}-\f12 \p_j \bg_{\a\b})=c^{-2}\bT v_j+\p_j \log c
\end{equation*}
and the second equation in  (\ref{4.23.1.19}) gives (\ref{6.14.1.19}).

From (\ref{6.21.9.19}) and (\ref{6.14.1.19}) we have
\begin{equation}\label{6.22.3.19}
\Xi_\mu L^\mu= \Tr k+\bN(\log c-\varrho).
\end{equation}
Combining (\ref{6.22.3.19}) with (\ref{5.24.1.19}) yields
\begin{equation}\label{6.23.10.19}
\Xi_L= 2\bT (\log c+\varrho)+L (\log c-\varrho).
\end{equation}
We derive by using (\ref{k1}) that
\begin{equation*}
k_{\bN \bN}=-c^{-2}\delta_{ij}\bN (v^i) \bN^j+\bT \log c=-\bg_{ij}\bN(v^i) \bN^j+\bT \log c .
\end{equation*}
By using the second equation in (\ref{4.23.1.19}) and (\ref{6.23.10.19}), we have
\begin{align*}
k_{\bN\bN}&= \bT\log c-\bN \varrho-L(v)_\bN=\bT(\log c+\varrho)-L \varrho-L(v)_\bN\\
&=\f12(\Xi_L-L(\log c+\varrho))-L(v)_\bN.
\end{align*}
This gives (\ref{7.04.8.19}). The proof of Proposition \ref{6.14.2.19} is complete.
\end{proof}

Finally, we recall a result from the previous works \cite{KREins2, Wangricci}, and \cite[Lemma 5.12]{Wangrough}.
\begin{lemma}[Decomposition of Riemann curvature]\label{6.23.17.19}
We denote by $\ti \pi=f(\bg) \bp \bg$ with $f$ a smooth function, and $\pi=\ti \pi\c X$.
Let $\bA= \chih, \tr\chi-\frac{2}{\tir}, \pi$, and $\bE =  \bA\c \pi + \tr \chi \c \pi$.
 \begin{enumerate}
\item[(i)] Let $\D_*=(\sn, \sn_L)$. There hold
$$
{\emph \bR}_{4A4B}, {\emph\bR}_{A443}, {\emph \bR}_{44}, {\emph \bR}_{4A}=\D_* \pi+\bE.
$$

\item[(ii)]  There exist scalar $\pi$, 1-form $\bE$ and $S_{t,u}$ tangent 2-vector $\pi_{AB}$  such that
\begin{equation*}
\delta^{AB} {\emph \bR}_{CA4B}=\sn_C \pi+\sn^B\pi_{CB}+\bE_C \quad \mbox{ and } \quad
{\emph \bR}_{CA4B}=\sn \pi+\bE.
\end{equation*}

\item[(iii)]There exists 1-form $\pi$ and scalar $\bE$ that ${\emph \bR}_{ABAB}={\sl{\div}} \pi+\bE$.

\item[(iv)] There exist 1-forms $\pi$ and scalar $\bE$ such that
\begin{equation*}
\delta^{AB} {\emph\bR}_{B43A}={\sl{\div}} \pi+\bE, \quad \ep^{AB} {\emph\bR}_{A43B}={\sl{\curl}} \pi+\bE.
\end{equation*}
\end{enumerate}
\end{lemma}
This result follows from the similar argument in \cite[Section 4]{KREins2}.
 We  recall the argument in  \cite[Proposition 4.1]{KREins2}. There holds under the coordinate frame $e_\a, e_\b, e_\ga, e_\d$ in $(\M, \bg)$ the following decomposition,
\begin{equation*}
\bR_{\a\b\ga \d}=\bd_\a \cir{\pi}_{\b\d\ga}+\bd_\b \cir{\pi}_{\a\ga\d}-\bd_\a \cir{\pi}_{\b\ga\d}-\bd_\b \cir{\pi}_{\d\a\ga}+E_{\a\b\ga\d}
\end{equation*}
with $E=\bg\c \ti \pi \c\ti \pi$ and $\cir{\pi}_{\a\b\ga}=\bp_\ga \bg_{\a\b}$. We contract the above identity by the null tetrad, and use Proposition \ref{6.7con} for the covariant derivatives on $L, \Lb, \Pi$. This gives the results in (i)-(iii).

 The proof of (iv) needs a minor change due to the change of the spacetime metric. Let us  compute with the help of Bianchi identity that
\begin{equation*}
\delta^{AB}\bR_{B43A}=\delta^{AB}(\bR_{AB}-\delta^{CD}\bR_{ACBD}), \quad \ep^{AB} \bR_{AB43} =-2 \ep^{AB} \bR_{A43B}.
\end{equation*}
 For $\delta^{AB}\bR_{AB}$ we use
$
\delta^{AB}\bR_{AB}=\sn_A \Xi_A+\bE
$
which follows from (\ref{6.16.1.19})  together with (\ref{4.10.2.19}). We then can obtain (iv) by using (iii) and the above calculations.

\section{Causal geometry of the acoustic spacetime}\label{causal_reg}

In this section, we establish a set of crucial estimates on connection coefficients in $\widetilde{\D^+}$ set up in Section \ref{null_cone_set} under the rescaled coordinates, that is   $(t, x) \to (\la(t-t_k), \la x)$ as
 done in (\ref{7.26.1}). Here $\la\ge\La>1$ with $\La$ sufficiently large and fixed. Recall that $\widetilde{\D^+}$ is contained in  $I_*\times {\mathbb R}^3$ with $I_*= [0,\tau_*]$ and $\tau_*\le \la^{1-8\ep_0}T$.
According to (\ref{BA2}) and Corollary \ref{comp_3}, the rescaled components of the metric $\bg$ satisfy the estimates
\begin{equation}
\| \ti\pi\|_{L_t^2L_x^\infty(I\times {\mathbb R}^3)}+\la^{\delta_0}\left(\sum_{\mu\ge 2 }\mu^{2\delta_0}
\| P_\mu\ti\pi\|^2_{L_t^2 L_x^\infty(I\times {\mathbb R}^3)}\right)^{\f12}\les
\la^{-1/2-4\ep_0}, \label{pi.2}
\end{equation}
where $\delta_0=s'-2$, $P_\mu$ is the Littlewood-Paley projection in (\ref{BA2}),   and $\ti \pi$ denotes the collection of terms taking the form of $f(\bg) \bp \bg$, with $f$ being a smooth function of its variables. To derive the last inequality in (\ref{pi.2}) we combined (\ref{BA2}) with applying Lemma \ref{10.26.3.19} to $G=f$  and using (\ref{5.04.17.19}), followed with rescaling. In the following sections we will work under the condition (\ref{pi.2}).

We fix the convention that
\begin{eqnarray}
\tir = t-u, && \widetilde{\tr\chi}=\tr\chi+\Xi_4, \qquad \mho=\frac{\bb^{-1}-1}{\tir},\nn\\
 z=\widetilde{\tr\chi}-\frac{2}{t-u}, && \sY=\bb(\widetilde{\tr\chi}-\frac{2}{\bb(t-u)})=\bb(z-2\mho). \label{11.3.1.19}
\end{eqnarray}

Similar to \cite[Lemma 5.1]{Wangrough}, we have the following results for the initial data along the null cones for the geometric quantities.
\begin{lemma}\label{inii}
Let $t_\tmin =\max\{u, 0\}$.
 \begin{enumerate}
\item[(i)]On any null cone $C_u$ initiating from a point on the time axis $\Ga^+$ at $t=u\ge 0$, there hold
\begin{align*}
&\tir z,\bb-1, \sn \bb,  \tir\sn z, \tir^2\mu\rightarrow 0 \mbox{ as }
t\rightarrow u,\quad \lim_{t\rightarrow u}\|\chih,\zeta,\zb,k, \curl \Omega\|_{L^\infty(S_{t,u})}<\infty.
\end{align*}
Along any null cone $C_u$ in $\widetilde{\D^+}$ ,
\begin{equation}\label{10.2.1.19}
 \sY, \sn \sY\rightarrow 0  \mbox{ as } t\rightarrow t_{\tmin}.
\end{equation}
\item[(ii)] Let $\gac:=(t-u)^{-2}\ga$ be the rescaled metric on $S_{t,u}$ and let ${\gamma}^{(0)}$ denote the
canonical metric on ${\Bbb S}^2$. Then, relative to the pull-back coordinates by the null geodesic flow $\Upsilon(t,\cdot, u): {\mathbb S}^2\rightarrow S_{t,u}$,  there hold
\begin{equation}
\lim_{t\rightarrow t_\tmin}\stackrel{\circ}
\gamma_{ab}={\gamma}_{ab}^{(0)},\qquad  \lim_{t\rightarrow
t_\tmin}\p_c\!\!\stackrel{\circ}\ga_{ab}=\p_c{\ga}_{ab}^{(0)},\label{8.1.1}
\end{equation}
where $a, b, c=1,2$.

\item[(iii)]
On $\bigcup_{\fv\in (0,\fv_*]}S_\fv$ there hold $\bb-a\rightarrow 0$, $|\fv z|\les \la^{-4\ep_0}$
and $\|\fv^\frac{3}{2} \sn z\|_{L_\fv^\infty L_\omega^p}+\|\fv^\f12 z\|_{L^\infty}\les \la^{-\f12}$,
where $\|F\|_{L_\fv^\infty L_\omega^p} = \sup_{\fv\in (0, \fv_*]} \left(\int_{S_\fv} |F|^p d\omega\right)^{1/p}$
for any tensor field $F$.
\end{enumerate}
\end{lemma}

 When $u\ge 0$, the proof is based on the local expansion of the geometric quantities at the vertex of the cone $C_u$. The items (i) and (ii) in Lemma \ref{inii} can be found from \cite{Wangthesis, Wang09} and  \cite[Section 2]{Wangricci}, if $u>0$.
If $u<0$, the results are based on Proposition \ref{exten}.
The item (iii)  also follows from Proposition \ref{exten}.

Now we state the main result of this section.
\begin{proposition}\label{cone_reg}
Let $p$ be a fixed number satisfying $0<1-\frac{2}{p}<s'-2$. Let $\D_*=(\sn, \sn_L)$. Under the assumption (\ref{BA2}), there hold on $\widetilde{\D^+}\subset[0,\tau_*]\times \Sigma$ the estimates,
\begin{align}
&\tir \widetilde{\tr\chi}\approx 1, \|\tir^\f12 z\|_{L^\infty(\widetilde{\D^+})}\les \la^{-\f12}\label{comp2}\\
&\|\tir^\frac{3}{2}\sn z\|_{L_t^\infty L_u^\infty L_\omega^p(\widetilde{\D^+})}\les \la^{-\f12}\label{ricp}\\
&\|\tir \sn ( \chih,  z)\|_{L_t^2 L_\omega^p(C_u\cap \widetilde{\D^+})}\les \la^{-\f12}\label{sna}\\
&\| z, \chih, \tr\chi-\frac{2}{\tir}, \zeta\|_{L_t^{\frac{q}{2}}L_x^\infty(\widetilde{\D^+})}\les \la^{\frac{2}{q}-1-4\ep_0(\frac{4}{q}-1)},\, 2<q<4\label{ric1.1}\\
&\|\frac{\bb^{-1}-1}{\tir}\|_{L_t^2 L_\omega^\infty(C_u\cap \widetilde{\D^+})}+\|\frac{\bb^{-1}-1}{\tir^\f12}\|_{L_\omega^{2p}(C_u\cap \widetilde{\D^+})}+\|\tir \D_*(\frac{\bb^{-1}-1}{\tir})\|_{L_t^2 L_\omega^p(C_u\cap \widetilde{\D^+})}\les \la^{-\f12}\label{ric4}
\end{align}
and there holds in $\D^+$ (that is where  $0\le u\le t\le \tau_*$),
\begin{equation}\label{ric1}
\| z, \chih, \zeta, \tr\chi-\frac{2}{\tir},\frac{\bb^{-1}-1}{\tir}\|_{L_t^2 L_x^\infty}\les \la^{-\f12-4\ep_0}.
\end{equation}
\end{proposition}

\begin{proposition}\label{ricpr}
Let $p$ be as fixed in Proposition \ref{cone_reg}. On the null cone $C_u$ contained in $\widetilde{\D^+}$, there hold
\begin{align}
&\|z\|_{L_t^2 L_\omega^\infty(C_u\cap \widetilde{\D^+})}+\|\chih\|_{L_t^2 L_\omega^\infty(C_u\cap \widetilde{\D^+})}+\|\zeta\|_{L_t^2 L_\omega^\infty(C_u\cap \widetilde{\D^+})}\les\la^{-\f12}, \label{ric3.18.1}\\
 &\|\p_\omega( \stackrel{\circ}{\ga}-\ga^{(0)})\|_{L_\omega^p L_t^\infty(C_u\cap \widetilde{\D^+})}
 \le \la^{-4\ep_0}, \quad \|\stackrel{\circ}{\ga}-\ga^{(0)}\|_{L^\infty}\les \la^{-4\ep_0}, \label{8.0.3}
\end{align}
where $\stackrel{\circ}{\ga}=(t-u)^{-2}\ga$. \begin{footnote}{We may hide the range for $u,t$ in $\D^+$ or $\widetilde{\D^+}$ for short and refer to Section \ref{setupcone} for their definitions.}\end{footnote}
\end{proposition}

As a consequence of (\ref{ric3.18.1}) and (\ref{pi.2})
\begin{equation}\label{10.11.3.19}
\|\bA\|_{L_t^2 L_\omega^\infty(C_u\cap \widetilde{\D^+})}\les \la^{-\f12}.
\end{equation}

The proof of the above results rely on a bootstrap argument. We make the bootstrap assumption on any $C_u$ contained in  $\widetilde {\D^+}$,
\begin{align}
&\|\chih\|_{L_t^2 L_\omega^\infty(C_u)}+\|z\|_{L_t^2 L_\omega^\infty(C_u)}+\|\zeta\|_{L_t^2 L_\omega^\infty(C_u)}\le \la^{-\f12 +\ep_0},\label{ba3.18.1}\\
 &\|\p_\omega(\stackrel{\circ}{\ga}-\ga^{(0)})\|_{L_t^\infty L_\omega^p}\le \la^{-\ep_0}, \,\|\stackrel{\circ}{\ga}-\ga^{(0)}\|_{L^\infty}\le \la^{-\ep_0}\label{ba3},\\
\end{align}
We also assume that on any $S_{t,u}\subset\widetilde{\D^+}$, there hold
\begin{align}
& \|\tr\theta-\frac{2}{\tir}\|_{L^3(\Sigma_t\cap \widetilde{\D^+})}\le 1, \quad \| (\la\tir)^\f12(\sn \bb, \chih)\|_{L_\omega^p}\le \la^{2\ep_0} \label{aux_1}, \displaybreak[0]\\
& |\bb-1|\le \f12, \label{bb_3}
\end{align}
where $0<1-\frac{2}{p}<s-2$ is fixed. 

By repeating the proof in \cite[Lemma 5.4]{Wangrough} with the help of the transport equations (\ref{lb}), (\ref{lv}), and the data in Lemma \ref{inii}, we can derive as a direct consequence of the estimate of $z$ in (\ref{ba3.18.1}), (\ref{pi.2}) and  (\ref{bb_3})  the following result
\begin{lemma}\label{6.17.1}
On $\widetilde{\D^+}$ there holds
\begin{align}
v_t &\approx (t-u)^2, \label{comp_3_27}\\
|\bb-1| &\les \la^{-4\ep_0}<\frac{1}{4} \label{bb_4}.
\end{align}
\end{lemma}

\begin{remark}\label{10.26.4.19}
In view of (\ref{pi.2}) and (\ref{3.19.1}), (\ref{pba2}) can be proved by  (\ref{ric1.1}).  The first assumption in (\ref{wras}) and the second assumption in (\ref{Pba}) are included in  (\ref{8.0.3}) and (\ref{aux_1}).  The first assumption in  (\ref{Pba}) is  proven in Lemma \ref{6.17.1}.
\end{remark}

In what follows, we will frequently use Lemma \ref{6.17.1} without explicit mention.
Next we recall important inequalities for carrying out analysis.

Using Lemma \ref{6.17.1} and the second assumptions in (\ref{ba3}), we can show that on $\widetilde{\D^+}$ there hold the following Sobolev inequalities
and trace inequalities:

\begin{enumerate}
\item[$\bullet$] For any scalar function or $S_{t,u}$-tangent tensor field $F$ (see \cite{Wang10}), there holds
\begin{align}
&\|F\|_{L_u^2 L_\omega^2}\les \|\tir \sn_N F\|_{L^2_u L_\omega^2}+\|\tir^\f12 F\|_{L_u^\infty L_\omega^2}. \label{trc_2}
\end{align}

\item[$\bullet$] For any scalar function or $S_{t,u}$-tangent tensor field $F$ and $2<q<\infty$ there hold (see \cite{CK,KRsurf,KRduke}
\begin{align}
&\|F\|_{L_\omega^q(S_{t, u})} \les \|\tir \sn F\|_{L_\omega^2(S_{t, u})}^{1-\frac{2}{q}}\|F\|_{L_\omega^2(S_{t, u})}^{\frac{2}{q}}
+\|F\|_{L_\omega^2(S_{t, u})}, \label{sob}\\
&\|F\|_{L_\omega^\infty(S_{t,u})} \les \|r \sn F\|_{L_\omega^q(S_{t,u})}+\|F\|_{L_\omega^2(S_{t,u})}. \label{sobinf}
\end{align}

\item[$\bullet$]
For $q\ge 2$ and any scalar functions or tensor fields $F$ defined on $\widetilde{\D^+}\cap \Sigma_t$ there holds
\begin{equation}\label{tran_sob}
\|\tir^{\f12-\frac{1}{q}}F\|^2_{L_x^{2q} L_u^\infty}\les \|F\|_{L_\omega^\infty L_u^2} \left(\|\tir \sn_N F\|_{L_\omega^q L_u^2}+\|F\|_{L_\omega^q L_u^2}\right).
\end{equation}
For  $\widetilde{\D^+}\cap C_u$,
\begin{equation}\label{7.04.20.19}
\|\tir^{\f12-\frac{1}{q}}F\|^2_{L_x^{2q} L_t^\infty(C_u)}\les \|F\|_{L_\omega^\infty L_t^2} \left(\|\tir \sn_L F\|_{L_\omega^q L_t^2}+\|F\|_{L_\omega^q L_t^2}\right),
\end{equation}
see \cite[Lemma 2.13]{Wangricci}, and the proof in  \cite[Section 8]{Wang10}.
\end{enumerate}

\subsection*{Hardy-Littlewood maximal function} For a scalar function $f(t)$ defined on $[0, \tau_*]$, its Hardy-Littlewood maximal function is defined by
\begin{equation*}
\M(f)(t)=\sup_{0\le t'\le \tau_*} \frac{1}{|t-t'|}\int_{t'}^t |f(\tau)| d\tau.
\end{equation*}
It is well-known that for any $1<q<\infty$ there holds
\begin{equation}\label{hlm}
\|\M(f)\|_{L_t^q} \les \|f\|_{L_t^q}.
\end{equation}

Using Lemma \ref{6.17.1} and the second assumptions in (\ref{ba3}), we also obtain the following control along the null cones.
\begin{proposition}[$L^p$ Control of the flux]\label{7.13.2.19}
Let $1-\frac{2}{p_*}\le \delta_0=s'-2$,
 there hold for $2\le q<p_*$ on $\widetilde{\D^+}$ the following inequalities
\begin{align}
 &\| \curl^2 \Omega\|_{L^2(C_u\cap \widetilde{\D^+})}+\|\tir^{1-\frac{2}{q}} \curl^2\Omega\|_{L_t^2 L_x^q(C_u\cap \widetilde{\D^+})}\les \la^{-\frac{3}{2}}\label{9.13.2.19}\\
 &\|\tir^\f12\p \Omega\|_{L^{2q}_\omega(\Sigma_t\cap \widetilde{\D^+})}\les \la^{-\frac{3}{2}}.\label{9.13.3.19}
\end{align}
 Under the assumption (\ref{aux_1}),  with $\D_*=(\sn, \sn_L)$, there hold for $0\le 1-\frac{2}{p}< s-2$ that
\begin{align}
&\|\tir^{1-\frac{2}{q}} \bp \ti\pi\|_{L_u^2 L_x^p(\Sigma_t\cap \widetilde{\D^+})}\les \la^{-\f12}\label{6.18.3.19}\\
 &\|\ti\pi\|_{L_u^2 L_\omega^p(\Sigma_t\cap \widetilde{\D^+})}+\|\tir^\f12 \ti\pi\|_{L^\infty L_\omega^{2p}(\Sigma_t \cap \widetilde{\D^+})}\les \la^{-\f12}\label{6.18.4.19}\\
 &\|\D_* \ti\pi\|_{L^2(C_u\cap \widetilde{\D^+})}+\|\tir^{1-\frac{2}{p}} \D_*\ti\pi\|_{L_t^2 L_x^p(C_u\cap\widetilde{\D^+})}\les \la^{-\f12}\label{6.18.5.19}\\
 &\|\tir (\sn \pi,\sn_L  \pi), \pi\|_{L_t^2 L_\omega^p(C_u\cap \widetilde{\D^+})}\les \la^{-\f12}.\label{flux_3}
\end{align}
\end{proposition}
By using  Corollary \ref{eng_wave}, Proposition \ref{7.24.4.19} and Proposition \ref{9.10.2.19}, for the case that $\ti \pi=\bp \Phi$ with $\Phi=v, \varrho$,  the proof of the above results is similar to \cite[Lemma 5.5]{Wangrough}; and the results for the general form of $\ti\pi$ follow as a consequence as in \cite[Proposition 5.6]{Wangrough}.  The proof of \cite[Lemma 5.5]{Wangrough} is based on the proof in \cite[Proposition 2.6]{Wangricci} under the assumption (\ref{aux_1}).

 Recall that in Proposition \ref{9.10.2.19}, we did not provide the control on the flux of $\bT v$ upto the highest order. This slightly influences the proof of  (\ref{6.18.5.19}). The terms of  vorticity in (\ref{9.13.2.19}) and (\ref{9.13.3.19}) did not appear in the previous works on quasilinear wave equations. Thus we will focus on the proof of (\ref{9.13.2.19}), (\ref{9.13.3.19}) and (\ref{6.18.5.19}) for the case $\ti\pi=\bp (v, \varrho)$.  The proof of (\ref{flux_3}) can be found in \cite[Lemma 5.7]{Wangrough} under the assumption (\ref{aux_1}), and using the estimates (\ref{6.18.5.19}) and (\ref{pi.2}).
\begin{proof}
(\ref{9.13.3.19}) is the second estimate in (\ref{10.26.1.19}) after rescaling.
 The  estimates in (\ref{9.13.2.19}) are consequences of  (\ref{9.17.2.19}) after rescaling.


 To prove (\ref{6.18.5.19}), we recall from \cite[Proposition 2.5, Proposition 2.6]{Wangricci} that
  \begin{align}\label{10.2.3.19}
  \begin{split}
  \|\tir^{1-\frac{2}{p}}(\sn f, L f)\|_{L_t^2 L_x^p(C_u)}&\les\sF^\f12[f](C_u)+\sum_{l>0}l^{1-\frac{2}{p}}\big(\sF^\f12[P_\ell f](C_u)\big)\\
  &+\|f\|_{H^{s-1}(\Sigma_{\tau_*})}+\|f\|_{H^{s-1}(\Sigma_{t_\tmin})},
  \end{split}
  \end{align}
   and if the cone $C_u$ is initiated from $\Ga^+$, the last term of the last line actually vanishes.

We can apply the above inequality to $f=\p v, \bp \varrho$ to obtain  for $0\le 1-\frac{2}{p}< s-2$
\begin{equation}\label{10.27.1.19}
  \|\tir^{1-\frac{2}{p}} (\sn f, L f)\|_{L_t^2 L_x^p(C_u)}\les \la^{-\f12},
\end{equation}
which is due to Proposition \ref{7.24.4.19} and Proposition \ref{9.10.2.19}, and Corollary \ref{eng_wave}, followed with rescaling.

To complete the proof of (\ref{6.18.5.19}), we need to obtain the same control for $f=\bT v$. From  (\ref{10.2.2.19}), we bound
 \begin{align*}
 \|\tir( L \bT v, \sn \bT v)\|_{L_t^2 L_\omega^p(C_u\cap \widetilde{\D^+})}&\les \|\tir (L \p \varrho, \sn \p \varrho, (\bp \varrho)^2) \|_{L_t^2 L_\omega^p(C_u\cap \widetilde{\D^+})}.
 \end{align*}
 The first two terms on the right hand side can be bounded by (\ref{10.27.1.19}). For the quadratic term, we use the second inequality in (\ref{6.18.4.19}) to derive
 \begin{align*}
 \|\tir (\bp \varrho)^2) \|_{L_t^2 L_\omega^p(C_u\cap \widetilde{\D^+})}&\les \||\tir^\f12 \bp \varrho|^2\|_{L^\infty_t L_\omega^{2p}(C_u)}\tau_*^\f12\les\la^{-1+\f12-4\ep_0}\les \la^{-\f12-4\ep_0}.
 \end{align*}
 Hence, we conclude
 \begin{equation*}
 \|\tir( L \bT v, \sn \bT v)\|_{L_t^2 L_\omega^p(C_u\cap \widetilde{\D^+})}\les \la^{-\f12}.
 \end{equation*}
 The proof of (\ref{6.18.5.19}) is completed.

\end{proof}

\subsection*{The transport lemma} We will use transport equations to control the connection coefficients.
The following result can be derived in view of (\ref{lv}) and (\ref{comp_3_27}) and will be frequently used.

\begin{lemma}[The transport lemma]\label{tsp2}
For $C_u$ contained in $\widetilde{\D^+}$ let $t_\tmin =\max\{u, 0\}$. For any $S_{t,u}$-tangent tensor field $F$ satisfying
\begin{equation*}
\sn_L F+\frac{m}{2} {\emph\tr}\chi F= W
\end{equation*}
with a constant $m$, there holds
\begin{equation*}
v_t^{\frac{m}{2}} F(t)=\lim_{\tau\rightarrow t_{\tmin}}v_\tau^{\frac{m}{2}} F(\tau)+\int_{t_{\tmin}}^t v_{t'}^{\frac{m}{2}} Wdt'.
\end{equation*}
Similarly, for the transport equation
\begin{equation*}
\sn_L F+\frac{m}{t-u} F=G\c F+W
\end{equation*}
with a constant $m$, if $\|G\|_{L_\omega^\infty L_t^1}\le C$, then there holds
\begin{equation*}
\tir^{m} |F(t)|\les\lim_{\tau\rightarrow t_{\tmin}}(\tau-u)^{m}|F(\tau)|+\int_{t_\tmin}^t {(t'-u)}^{m} |W| d\tt.
\end{equation*}
The same result holds when $\frac{2}{t-u}$ in the transport equation is replaced by ${\emph\tr} \chi$.
The above integrals are taken along null geodesics on $C_u$.
\end{lemma}

We will also employ the Codazzi equations on the spheres $S_{t,u}$ for which we recall the following elliptic estimates, which  hold under the assumption  (\ref{ba3}).
\begin{lemma}\label{hdgm1}
Let $\D$ denote either $\D_1$ or $\D_2$ and let $p>2$ be the number in (\ref{ba3}). Then for
$2\le q\le p$ there holds
\begin{equation*}
\|\sn  F\|_{L^q(S_{t,u})}+\|\tir^{-1}F\|_{L^q(S_{t,u})}\les \| \D F\|_{L^q(S_{t,u})}
\end{equation*}
for any $S_{t,u}$-tangent tensor $F$ in the domain  of $\D$.
\end{lemma}
It follows from the above result, (\ref{sob}), (\ref{sobinf})  and the duality argument that 
\begin{proposition}\label{cz}
Let $F$ be a covariant symmetric traceless $2$-tensor satisfying the Hodge system
\begin{equation}\label{fe1}
{\sl{\div}} F=\sn G+e \qquad \mbox{on } S_{t,u}
\end{equation}
for some scalar function $G$ and $1$-form $e$. For $2<q<\infty$ and $\frac{1}{q'}=\frac{1}{2}+\frac{1}{q}$ there hold
\begin{equation}\label{lpp2}
\|F\|_{L^q(S_{t,u})}\les\|G\|_{L^q(S_{t,u})}+\|e\|_{L^{q'}(S_{t,u})};
\end{equation}
and 
\begin{equation}\label{cz0}
\|F\|_{L^\infty(S_{t,u})}\les\tir^{1-\frac{2}{q}}(\|\sn G\|_{L^q(S_{t,u})}+\|e\|_{L^q(S_{t,u})}).
\end{equation}
Similarly, for the Hodge system
\begin{equation}\label{divcurl}
\left\{
\begin{array}{lll}
{\sl{\div}} F= \sn\c G_1+e_1,\\
{\sl{\curl}} F=\sn\c G_2+e_2,
\end{array}
\right.
\end{equation}
with $1$-forms $G=(G_1, G_2)$  and scalar functions $e=(e_1, e_2)$, there hold (\ref{lpp2}) and (\ref{cz0}) for any $q>2$.
\end{proposition}

\begin{proposition}\label{cz.2}
Let $F$ and $G$ be $S_{t,u}$-tangent tensor fields of suitable type satisfying (\ref{fe1}) or
(\ref{divcurl}) with certain term $e$. Suppose $G$ is  a projection of a
tensor field $\ti G$ to tangent space of $S_{t,u}$ by $ \Pi_{\mu}^{\mu'}{\ti G}_{\mu'\cdots}$ or
takes the form of $f(\bb)\bN^\mu{\ti G}_{\mu\cdots}, f(\bb) L^\mu{\ti G}_{\mu\cdots}$, where $f$ is a smooth function of $\bb$. Under the assumption (\ref{aux_1}),  for $q>2$, $1\le c<\infty$ and
$\delta>0$ sufficiently close to $0$, there holds
\begin{align}
\|F\|_{L^\infty(S_{t,u})}\les \|\mu^{\delta}P_\mu \ti G\|_{l_\mu^c L^\infty(S_{t,u})}
+\|\ti G\|_{L^\infty(S_{t,u})}+\tir^{1-\frac{2}{q}}\|e\|_{L^q(S_{t,u})}.\label{cz2}
\end{align}
Here $\ti G$ is regarded as its components under the coordinate frame $\p_t, \p_1, \p_2, \p_3$.
\end{proposition}
We incorporate the additional factor $f(\bb)$ in the form of the definition for $G$.  Since we have $\|f(\bb)\|_{L^\infty}+\|\tir\sn( f(\bb))\|_{L_\omega^q}\les 1 $ for $0\le 1-\frac{2}{q}< s-2$  due to (\ref{aux_1}) and (\ref{bb_4}), this satisfies the condition used in \cite[Lemma 5.5]{Wangricci}.
The proof of the above result can follow the same as in \cite[Section 5]{Wangricci}. In application, we only use $f(\bb)=\bb,$ or $1$.

\subsection*{A sketch for the proof of Proposition \ref{cone_reg} and Proposition \ref{ricpr}.}
 
In comparison with the analysis in \cite[Section 5]{Wangrough},  due to the rough $\curl \Omega$, we have to carry out  normalizations by using (\ref{6.23.2.19}), (\ref{9.12.2.19}) and (\ref{7.04.8.19}) in order to obtain  the estimates of $z$ and $\sn z$. We also simplify  the estimate of $\zeta$   by using (\ref{6.17.7.19}) and Proposition \ref{cz.2}. It is used in controlling $\sY$ and also $\chih$.

(1)  We derive symbolic structure equations (\ref{lchi})-(\ref{ldz_2}) with the help of  (\ref{6.23.2.19}), (\ref{9.12.2.19}) and Proposition \ref{6.23.17.19}, then  achieve a set of preliminary estimates in Section \ref{prel_1} by using transport equations, which controls $\mho$ as desired and improves the auxiliary assumption in (\ref{aux_1}).

 (2) In Section \ref{causal_step1}, by using (\ref{ldz_2}), we
    obtain the bound $\|\tir^\frac{3}{2}\sn z\|_{L^p_\omega(S_{t,u})}$ with $2\le p< s'-2$, and with the help of the Codazzi equation (\ref{dchi1}) for $\sn \chih$ the bounds in (\ref{sna}).  Once the $L^p$ bound is obtained, we have the pointwise
     control on $\tir^\f12 z$ by Sobolev embedding. In Section \ref{causal_step2}, we obtain the sharp bound on $\zeta$ with the help of the Hodge system (\ref{6.17.7.19}). In view of the equation (\ref{trscoord2}) for propagating the metric component of $\ga$ and the Sobolev embedding (\ref{sobinf}), the proof of Proposition \ref{ricpr} is complete.

(3) In the region where $u\approx t$, the bound on  $\|z\|_{L_t^2 L^\infty(\D^+)}$ can not follow from the pointwise control of $\tir^\f12 z$, since $\tir$ can be close to $0$ in such region. This step is completed in  Proposition \ref{10.17.1.19} in Section \ref{10.4.1.19}.  To remove the potential signularity, we derive the  transport equation (\ref{10.26.2.19}) for $\sn \sY$, from which we see the potential singularity  arises exactly due to $\sn(k_{\bN\bN}-\f12 \Xi_4)$.
     We then derive the trace decomposition of $\sn(k_{\bN\bN}-\f12 \Xi_4)$ by using  (\ref{7.04.8.19}),  and hence use the structure to remove the singular term. With the help of the normalized transport equation (\ref{9.15.5.19}),  we can obtain the $L^p_\omega$ bound on $\tir\sn \sY$, and the favourable control on $\sY$ and  $z$ follow as consequences. With the bound on $\sn \sY$, we achieve the strong norms on $\chih$ in (\ref{ric1.1}) and  (\ref{ric1}) by normalizing the Codazzi equation (\ref{dchi1}) and applying Proposition \ref{cz} and Proposition \ref{cz.2}. The proof of Proposition \ref{cone_reg} hence can be completed.
    
\subsection{The preliminary estimates on $\chih$, $z$ and $\zeta$}\label{prel_1}
The goal of this subsection is to show the following preliminary estimates.
\begin{proposition}[$L^p$ and $L^{2p}$ estimates on $\widetilde{\D^+}$]\label{p1}
Let $0\le 1-\frac{2}{p}< s-2$, and let $\mho=\frac{\bb^{-1}-1}{\tir}$. There hold in $\widetilde{\D^+}$ the following estimates
\begin{align}
&\|\mho\|_{L_t^2 L^\infty_\omega}+\|\tir\D_*\mho\|_{L_t^2 L_\omega^p}+\|\tir^\f12\sn\log \bb\|_{L_\omega^p} +\|\tir^\f12\mho\|_{L_\omega^{2p}}\les \la^{-\f12},\label{9.14.1.19}
\end{align}
\begin{equation}\label{9.14.2.19}
\|\mho\|_{L_t^2 L^\infty(\D^+)}\les \la^{-\f12-4\ep_0}, \quad \|\mho\|_{L_t^\frac{q}{2} L^\infty(\widetilde{\D^+})}\les\la^{\frac{2}{q}-1-4\ep_0(\frac{4}{q}-1)}, 2<q<4.
\end{equation}
Let $\bA=\chih, \zeta, z, \pi$. There hold
\begin{align}
&\|\bA, \tir \sn_L \bA\|_{L_t^2 L_\omega^p(C_u)}\les \la^{-\f12}\label{pric1}\\
&\|\tir^\f12 \bA\|_{L_\omega^p}\les \la^{-\f12}\label{pric2}\\
&\|\tir^\f12 \bA \|_{L^{2p}_\omega}\les \la^{-\f12} \label{pric3}
\end{align}
\end{proposition}
\begin{remark}
Note due to $\zeta=\sn \log \bb+k_{A\bN}$ in (\ref{3.19.1}), (\ref{pric2}) improves the second estimates in (\ref{aux_1}). Due to $\tr\theta-\frac{2}{\tir}=z+\pi$ in view of (\ref{3.19.1}),  and  $\tir\les \la^{1-8\ep_0}T$, using (\ref{pric3}) gives $\|\bA\|_{L^3(\Sigma_t\cap \widetilde{\D^+})}\les \la^{-4\ep_0}\les 1$. Thus we can prove the assumption on $\tr\theta $ in (\ref{aux_1}).
\end{remark}

To prove the above proposition, we first give the symbolic version of their null structure equations.
\begin{lemma}
Let ${\bA}$ denote terms $\chih, z, \pi$. There hold, schematically,
\begin{align}
&\sn_L \chih+\f12 \tr\chi\chih=\pi\c {\bA}+\tir^{-1}\pi+(\sn, \sn_L) \pi+\exp\varrho \curl \Omega_\bN,  \label{lchi}\\
& Lz+\frac{2z}{t-u}=-\f12 \Xi_4^2+\left(\Xi_4-k_{\bN\bN}\right)\widetilde{\tr\chi}-|\chih|^2-\f12 z^2+\pi\c \pi\nn\\
&\qquad \qquad\qquad+\exp\varrho \curl \Omega_\bN, \label{lz}\\
&\sn_L \sn\log \bb+\f12 \tr\chi \sn \log \bb=-\chih\c \sn \log \bb-\sn(k_{\bN\bN})\label{lb1}\\
& {\sl{\div}} \chih=\f12 \left(\sn z-\sn \Xi_4 \right)+\sn \pi+\frac{\pi}{t-u}+{\bA}\c \pi.\label{dchi1}\\
& \sn_L (\sn z-e^\varrho\curl \Omega)_A+\frac{3}{t-u} (\sn z-e^\varrho(\curl \Omega))_A\label{ldz_2}\\
&\quad\quad\quad = {\bA}\c\sn z+\frac{1}{(t-u)}\left(\sn(\Xi_4-k_{\bN\bN})+e^\varrho(\curl \Omega)_A\right) + (z, \pi)\c (\sn  \pi+e^\varrho(\curl \Omega)_A)\nn\\
&\quad \quad \quad +e^{\varrho} \big({e_A}_i \Pi^{ij} \tensor{\ep}{_{jm}^l}(\curl^2\Omega)_l\bN^m+X\c(\chi+\pi)
\c \p \Omega \big)+\sn \chih\c \chih,\nn
\end{align}
where $\cdot X$ represents the contractions  with $L, \Lb, \Pi$.
\end{lemma}
Hence with $\fA$  an element of $\chih, z, \sn \log \bb$, there holds the symbolic formula
\begin{equation}\label{10.3.1.19}
\sn_L \fA+\frac{m}{\tir}\fA=\bA\c \bA+\D_*\pi+\tir^{-1}\pi+\exp\varrho \curl \Omega_{\bN}, \quad m= 1,2
\end{equation}
where on the right hand side $\bA=z, \chih, \zeta, \pi$, and all the possible terms which appear in the collection of the transport equation for $\fA$ are included. $m=1$ if $\fA=\chih, \sn \log \bb$ and $m=2$ if $\fA=z$.

\begin{proof}
(\ref{lb1}) can be directly obtained by using (\ref{lb}) and (\ref{cmu2}).

Combining (\ref{6.23.2.19}) with the equations (\ref{4.10.1.19}), (\ref{4.10.2.19}),  and  $\bd_L L=-k_{\bN\bN} L$ in (\ref{6.29.5.19}), we can obtain
\begin{equation}\label{6.30.5.19}
\bR_{44}=-\exp\varrho \curl \Omega_\bN+\pi \c \pi+L(\Xi_4).
\end{equation}

In view of the definition of $\widetilde{\tr\chi}$, substituting (\ref{6.30.5.19}) into (\ref{s1}) yields
\begin{equation}\label{9.15.1.19}
L\widetilde{\tr\chi}+\f12 (\widetilde{\tr\chi})^2=(\Xi_4-k_{\bN\bN})\widetilde{\tr\chi}-\f12 \Xi_4^2-|\chih|^2+\pi\c \pi+\exp\varrho \curl \Omega_\bN,
\end{equation}
which gives (\ref{lz}) by using the definition of $z$.
  (\ref{lchi}) can be obtained by the substitutions of (\ref{6.30.5.19}) and  Lemma \ref{6.23.17.19} (i) into (\ref{s2}). The proof of (\ref{dchi1}) can be obtained by substituting Lemma \ref{6.23.17.19} (ii) into (\ref{dchi}).

 To derive (\ref{ldz_2}), we directly take the covariant derivative on (\ref{lz})  and use the commutation formula (\ref{cmu2}) to obtain
\begin{equation}\label{9.12.4.19}
\sn_L  \sn z+\frac{3}{t-u} \sn z=- \chih\c \sn z+\f12 (\Xi_4-z) \sn z+\sn G,
\end{equation}
 where $G$ denotes the right hand side of (\ref{lz}).
 With the help of   (\ref{9.12.2.19}), we derive
 \begin{align}\label{9.16.1.19}
 \begin{split}
 &\sn_A(\exp\varrho \curl \Omega_\bN)\\
 &=\sn_L(e^\varrho\curl \Omega)_A+e^{\varrho} \big({e_A}_i \Pi^{ij} \tensor{\ep}{_{jm}^l}(\curl^2\Omega)_l\bN^m+(\pi+\chi)
 \c \p \Omega\c X).
 \end{split}
 \end{align}
 Note that the first term on the right is of the type $\bp \curl \Omega$, which is the second order derivative of $\Omega$. There is no direct bound for this term. Therefore we renormalize the equation (\ref{9.12.4.19}) by subtracting $(\sn_L+\frac{3}{t-u})(e^\varrho\curl \Omega)_A$ from both sides.
 This gives  (\ref{ldz_2}).
\end{proof}
\begin{proof}[Proof of (\ref{9.14.1.19}) and (\ref{9.14.2.19})]
Recall from Lemma \ref{inii}, $\lim_{t\rightarrow t_{\tmin}}(\bb-a)=0$ with $a=1$ if $u\ge 0$ and   the function $a$ in Proposition \ref{exten} otherwise.
By using (\ref{lb}), we have
\begin{equation}\label{9.16.4.19}
\frac{\bb^{-1}-a^{-1}}{\tir}=\frac{1}{\tir}\int_{t_{\tmin}}^t L(\bb^{-1})=\frac{1}{\tir}\int_{t_{\tmin}}^t \bb^{-1}k_{\bN\bN} dt',
\end{equation}
where $a=1$ if $u\ge 0$. While for $u\le 0$,  we derive in view of
   $\frac{\bb^{-1}-1}{\tir}=\frac{\bb^{-1}-a^{-1}}{\tir}+\frac{a^{-1}-1}{\tir}$,
\begin{equation}\label{9.16.5.19}
\frac{\bb^{-1}-1}{\tir}=\frac{1}{\tir}\int_{t_{\tmin}}^t \bb^{-1}k_{\bN\bN} dt'+\frac{a^{-1}-1}{\tir}.
\end{equation}
For both cases in the above, by using (\ref{bb_4}) and (\ref{w8.1.1}),
\begin{align*}
\|\frac{\bb^{-1}-1}{\tir}\|_{L_t^2 L_\omega^\infty(C_u\cap \widetilde{\D^+})}&\les \|k\|_{L_t^2 L_\omega^\infty}+\|\frac{1-a^{-1}}{\fv^\f12} \fv^\f12 \tir^{-1}\|_{L_t^2}\\
&\les \la^{-\f12}+\la^{-\f12}\|\fv^\f12 (t+\fv)^{-1}\|_{L_t^2}
\end{align*}
where we used (\ref{pi.2}) for the bound of $k$, and the last term vanishes unless $\fv=-u,$ for $ u<0$.
Hence by direct calculation,
\begin{equation*}
\|\frac{\bb^{-1}-1}{\tir}\|_{L_t^2 L_\omega^\infty(C_u\cap \widetilde{\D^+})}\les \la^{-\f12}.
\end{equation*}

Next, when $u\ge 0$, by using (\ref{pi.2}) and (\ref{hlm})
\begin{align*}
\|\frac{\bb^{-1}-1}{\tir}\|_{L_t^2 L^\infty(\D^+)}&\les  \|\tir^{-1}\int_u^t \|k_{\bN\bN}\|_{L_\omega^\infty}\|_{L_t^2 L_u^\infty}\les \|\M(\|k_{\bN\bN}\|_{L_\omega^\infty})\|_{L_t^2}\\
&\les \|k_{\bN\bN}\|_{L_t^2 L^\infty(\D^+)}\les\la^{-\f12-4\ep_0};
\end{align*}
similar to the above estimate, if $u\le 0 $, noting that  $\sup_{0\le \fv\le \fv_*}\fv^\f12 \tir^{-1}\approx t^{-\f12}$,
we can derive with the help of (\ref{9.16.5.19}), (\ref{w8.1.1}) and (\ref{pi.2}) that
\begin{align*}
\|\frac{\bb^{-1}-1}{\tir}\|_{L_t^\frac{q}{2} L^\infty}&\|k_{\bN\bN}\|_{L_t^\frac{q}{2} L^\infty}+\|\frac{1-a^{-1}}{\fv^\f12} \fv^\f12 \tir^{-1}\|_{L_t^\frac{q}{2} L_u^\infty}\\
&\les \|k\|_{L_t^\frac{q}{2} L^\infty}+\la^{\frac{2}{q}-1-4\ep_0(\frac{4}{q}-1)}\\
&\les \la^{\frac{2}{q}-1-4\ep_0(\frac{4}{q}-1)}.
\end{align*}
Next we prove the derivative estimate in (\ref{9.14.1.19}). By using (\ref{lb}),
$$\tir L\mho=L(\bb^{-1})-\tir^{-1}(\bb^{-1}-1)=\bb^{-1}k_{\bN\bN}-\mho,$$
on $C_u \cap \widetilde{\D^+}$, we derive  by using (\ref{pi.2}) and  the first estimate of (\ref{9.14.1.19}) which has been proved above
\begin{equation*}
\|\tir L\mho\|_{L_t^2 L_\omega^p}\les \|k_{\bN\bN}\|_{L_t^2 L_\omega^p}+\|\mho\|_{L_t^2 L_\omega^p}\les \la^{-\f12}.
\end{equation*}
Recall that by using  the first two estimates in (\ref{ba3.18.1}) and (\ref{pi.2})
\begin{align*}
\|\tr\chi-\frac{2}{\tir}, k_{\bN\bN}\|_{L_\omega^\infty L_t^1}\les \la^{-3\ep_0}.
\end{align*}
Hence we can apply Lemma \ref{tsp2} to (\ref{lb1}),
by using (i) in  Lemma \ref{inii} for $\sn \bb$ when $u\ge 0$. By using (\ref{flux_3}), this leads to
\begin{equation*}
\|\sn \log \bb\|_{L_t^2 L_\omega^p}+\tir^\f12 \|\sn\log \bb\|_{L_\omega^p}\les \|\tir\sn k_{\bN\bN}\|_{L_t^2 L_\omega^p}\les\la^{-\f12}.
\end{equation*}
Similarly, if $u\le 0$, we use the initial condition in (\ref{a_3}) to derive
\begin{align*}
\tir\|\sn\log \bb\|_{L_\omega^p}\les \int_{t_{\tmin}}^t \|\tir \sn k_{\bN\bN}\|_{L_\omega^p} dt'+\lim_{t\rightarrow 0} \|\tir\sn \log a\|_{L_\omega^p}.
\end{align*}
Noting that at $t=0$, (\ref{a_3}) implies $\tir^{-1}\fv^\f12\|\fv^\f12 \sn \log a\|_{L_\omega^p}\les \la^{-\f12} \tir^{-1} \fv^\f12$.
  Thus if $\fv$ is fixed,
  \begin{align*}
\|\tir^{-1}\fv^\f12\|\fv^\f12 \sn \log a\|_{L_\omega^p}\|_{L_t^2}\les \la^{-\f12},
  \end{align*}
  and
  \begin{align*}
  \tir^{-\f12} \fv^\f12 \|\fv^\f12 \sn \log a \|_{L_\omega^p}\les \la^{-\f12}.
  \end{align*}
  We then conclude by  using (\ref{flux_3}) on $C_u\cap \widetilde{\D^+}$
  \begin{align*}
  \tir^\f12 \|\sn \log \bb\|_{L_\omega^p}&+\|\sn \log \bb\|_{L_t^2 L_\omega^p}\les \la^{-\f12}+\|\tir \sn k_{\bN\bN}\|_{L_t^2 L_\omega^p}\les \la^{-\f12}.
  \end{align*}
  Thus the second and the third estimates in (\ref{9.14.1.19}) are proved. The last one follows by applying (\ref{7.04.20.19}) with the help of the first two estimates in (\ref{9.14.1.19}). The proofs for (\ref{9.14.1.19}) and (\ref{9.14.2.19}) are therefore complete.
\end{proof}

\begin{proof}[Proof  of (\ref{pric1})-(\ref{pric3})]
Note (\ref{pric1})-(\ref{pric3}) hold for $\bA=\pi$, which is a consequence of (\ref{flux_3}) and the last estimate of (\ref{6.18.4.19}). To prove (\ref{pric1})-(\ref{pric3}), we focus on the cases when $\bA=z, \chih, \zeta$.

We first note from (\ref{ba3.18.1}) and (\ref{pi.2})
\begin{equation}\label{a6.19}
\|\bA\|_{L_t^2 L_\omega^\infty}\les \la^{-\f12+\ep_0}.
\end{equation}
Since we have proved in (\ref{9.14.1.19}) that for $0\le 1-\frac{2}{p}\le s-2$  there holds
\begin{equation}\label{10.3.2.19}
  \tir^\f12 \|\sn \log \bb\|_{L_\omega^p}+\|\sn \log \bb\|_{L_t^2 L_\omega^p}\les \la^{-\f12}.
\end{equation}
With the help of (\ref{flux_3}), substituting the estimate and (\ref{a6.19}) in (\ref{lb1}) implies
\begin{align*}
\|\tir \sn_L \sn \log \bb\|_{L_t^2 L_\omega^p(C_u)}&\le \|\tir \bA \sn \log \bb\|_{L_t^2 L_\omega^p(C_u)}+\|\tir \sn \pi\|_{L_t^2 L_\omega^p(C_u)}\\
&\le \|\bA, \pi\|_{L_t^2 L_\omega^\infty}\|\tir \sn \log \bb\|_{L_t^\infty L_\omega^p}+\la^{-\f12}\les \la^{-\f12-3\ep_0}+\la^{-\f12}.
\end{align*}
Thus, by using the last identity in (\ref{3.19.1}) and (\ref{flux_3}),
$$
\|\tir \sn_L \zeta\|_{L_t^2 L_\omega^p(C_u)}\les \|\sn_L \pi\|_{L_t^2 L_\omega^p(C_u)}+\|\tir \sn_L \sn \log \bb\|_{L_t^2 L_\omega^p(C_u)}\les \la^{-\f12}.
 $$
 Similarly, by using (\ref{3.19.1}), (\ref{10.3.2.19}),  the estimates for $\pi$ in (\ref{pric1}) and (\ref{pric2}), we have
 \begin{align*}
& \|\zeta\|_{L_t^2 L_\omega^p}\les \|\sn \log \bb\|_{L_t^2 L_\omega^p}+\|k\|_{L_t^2 L_\omega^p}\les \la^{-\f12},\\
&\|\tir^\f12 \zeta\|_{L_\omega^p}\les \|\tir^\f12 \pi\|_{L_\omega^p}+\|\tir^\f12 \sn\log \bb\|_{L_\omega^p}\les \la^{-\f12}.
 \end{align*}
Thus  (\ref{pric1}) and (\ref{pric2}) hold for $\zeta$.

Note for any fixed point $p=(t, u, \omega)\in \widetilde{\D^+}$, there exists a unique null geodesic through the point such that  $p=\Upsilon(t, u, \omega)$,  which is either initiated from $ S_{0,-u}$ at the slice of $\{t=0\}$ or from the vertex $t=u$ at the time axis $\Ga^+$ if $u\ge 0$. Thus, by applying Lemma \ref{tsp2}, for both $\fA=\chih$ and $z$, with $m=1$ or $2$ we derive
\begin{equation}\label{10.3.5.19}
\tir^m |\fA(t)|\les \lim_{\tau\rightarrow t_\tmin}| (\tau-u)^m \fA(\tau)|+\int_{t_\tmin}^t\tir^m\big( |\bA\c\bA|+|\D_* \pi|+|\tir^{-1}\pi|+|\curl \Omega|\big)dt'.
\end{equation}

If $u\ge 0$,  for the data of $\fA=z, \chih$, we apply the result (i) in  Lemma \ref{inii}, the first term on the right of (\ref{10.3.5.19}) vanishes.

If $u< 0$, for $\fA=\chih$, we note that due to (\ref{a_3}), (\ref{3.19.1}) and applying the second estimate in (\ref{6.18.4.19}) for $\pi$,
\begin{equation*}
 \|\fv^{\f12-\frac{2}{q}}\hat\chi\|_{L^{q}(S_\fv)}\les \la^{-\f12},\quad 0\le 1-\frac{2}{q}<s-2;
\end{equation*}
for $\fA=z$, we apply (iii) in Lemma \ref{inii}. Thus, for $\fv>0$ fixed,  since $\tir=t+\fv$,
\begin{align}\label{10.10.1.19}
\begin{split}
&\|\tir^{-m+\f12}\fv^m \fA(0, \fv, \omega)\|_{L_t^\infty L_\omega^p}\les \|\fv^\f12 \fA(0, \fv, \omega)\|_{L_\omega^p}\les \la^{-\f12},\\
&\|\tir^{-m}\fv^m\fA(0,\fv, \omega)\|_{L_t^2 L_\omega^p}\le\|\tir^{-1}\fv\fA(0, \fv, \omega)\|_{L_t^2 L_\omega^p}\les \fv^\f12 \|\fA(0, \fv, \omega)\|_{L_\omega^p}\les \la^{-\f12}.
\end{split}
\end{align}
Therefore  by using (\ref{10.3.5.19}), in both cases, for $\fA=\chih, z$,
\begin{align*}
\|\tir^\f12 \fA\|_{L_\omega^p}&\les \la^{-\f12}+\|\tir^{-\f12}\int_{t_\tmin}^t \tir (|\D_*\pi|+|\bA\c \bA|+|\tir^{-1}\pi|+|\curl \Omega|) dt'\|_{L_\omega^p}\\
\|\fA\|_{L_t^2 L_\omega^p}&\les \la^{-\f12}+\|\tir^{-1} \int_{t_\tmin}^t \tir (|\D_*\pi|+|\bA\c \bA|+|\tir^{-1}\pi|+|\curl \Omega|) dt'\|_{L_\omega^p},
\end{align*}
 which lead to
 \begin{align}
 \|\tir^\f12 \fA\|_{L_\omega^p}+\|\fA\|_{L_t^2 L_\omega^p}&\les \la^{-\f12}+\|\tir \D_*\pi\|_{L_t^2 L_\omega^p(C_u)}+\|\pi\|_{L_t^2 L_\omega^p(C_u)}\nn\\
 &+\|\tir\curl \Omega\|_{L_t^2 L_\omega^p(C_u)}+\|\tir \bA \c\bA\|_{L_t^2 L_\omega^p(C_u)}\label{10.3.4.19}.
 \end{align}
 We apply (\ref{flux_3}) and (\ref{9.13.3.19}) to derive
  \begin{equation}\label{10.3.7.19}
  \|\tir \D_*\pi\|_{L_t^2 L_\omega^p(C_u)}+\|\pi\|_{L_t^2 L_\omega^p(C_u)}+\|\tir\curl \Omega\|_{L_t^2 L_\omega^p(C_u)}\les \la^{-\f12}.
  \end{equation}
  Thus by using (\ref{a6.19})
 \begin{align}
 \|\tir^\f12 \fA\|_{L_\omega^p}+\|\fA\|_{L_t^2 L_\omega^p}&\les \la^{-\f12}+\|\tir\bA\c\bA\|_{L_t^2 L_\omega^p(C_u)}\nn\\
 &\les \la^{-\f12}+\|\bA\|_{L_t^2 L_\omega^\infty}\|\tir \bA\|_{L_t^\infty L_\omega^p}\nn\\
 &\les \la^{-\f12}+\la^{-\f12+\ep_0}\tau_*^\f12 (\|\tir^\f12\chih\|_{L_t^\infty L_\omega^p}+\|\tir^\f12 z\|_{L_\omega^p}+\la^{-\f12}),\label{10.3.6.19}
 \end{align}
 where we used the proved estimate of (\ref{pric2}) for $\bA=\zeta, \pi$.
 Note $\tau_*^{\f12}\la^{-\f12+\ep_0}\les \la^{-3\ep_0}$, which can be sufficiently small with $\la\ge \La$ for sufficiently large $\La$. Therefore, we can conclude
 \begin{align*}
 \|\tir^\f12\chih\|_{L_t^\infty L_\omega^p}+\|\tir^\f12 z\|_{L_\omega^p}\les \la^{-\f12}
 \end{align*}
 which completes the proof of (\ref{pric2}). Substituting the above estimate to  (\ref{10.3.6.19})   yields the bound $\|\fA\|_{L_t^2 L_\omega^p}\les \la^{-\f12}$. Thus the first estimate of (\ref{pric1}) is complete.

 Note that we have shown by using (\ref{pric2}) and (\ref{a6.19})
 \begin{equation*}
 \|\tir \bA\c \bA\|_{L_t^2 L_\omega^p}\les \la^{-\f12-3\ep_0}.
 \end{equation*}
 Combining this estimate with (\ref{10.3.7.19}) and (\ref{10.3.1.19}), also using the first estimate in (\ref{pric1}), we can obtain the last estimate in (\ref{pric1}) for $\bA=\chih, z$. Since other cases have been proven, the estimate of (\ref{pric1}) is also proved.

For $\fA=\chih, z, \zeta$, by applying (\ref{7.04.20.19}) with the help of (\ref{ba3.18.1}), (\ref{pric1}) and Minkowski inequality, we can obtain
 \begin{equation}\label{10.3.8.19}
 \|\tir^\f12 \fA\|_{L_\omega^{2p}}^2\les \|\fA\|_{L_\omega^\infty L_t^2}(\|\tir \sn_L \fA\|_{L_\omega^p L_t^2(C_u)}+\|\fA\|_{L_\omega^p L_t^2(C_u)})\les\la^{-1}.
 \end{equation}
 This proof is completed.
\end{proof}

\subsubsection{ Control of $\sn z, \,\sn \chih$ and $\tir^\f12|z|$}\label{causal_step1}
In this subsection, we give the estimates of $\sn z$ and  $\sn\chih$ in (\ref{ricp}) and (\ref{sna}). We also  prove  the bound of   $|z|$ in (\ref{comp2}) as a consequence.

Since the right hand side of (\ref{lz}) is not bounded in $L_\omega^\infty L_t^1$,  the pointwise estimate for $z$ does not directly follow. There hold in view of the Sobolev embedding (\ref{sobinf}),
\begin{align}
&|\tir^\f12 z|\les \|\tir^\f12 (\tir\sn)\rp{\le 1} z\|_{L_\omega^p}, \label{9.14.8.19}\\
&\|z\|_{L_t^2 L^\infty(\D^+)}\les \| (\tir\sn)\rp{\le 1} z\|_{L_t^2 L_u^\infty L_\omega^p(\D^+)},\nn
\end{align}
where  $p>2$.

  It is natural to consider the bounds on  the right hand sides of the above two inequalities. Nevertheless, there is no direct estimate of $\|(\tir \sn)\rp{\le 1}z\|_{L_t^2 L_u^\infty L_\omega^p(\D^+)}$. Therefore, as the first step, we give the bound of $\tir^\f12 |z| $ via the first inequality. In the second step, we will derive the bound for $\|z\|_{L_t^2 L^\infty}$ by carrying out a further normalization on $\tr\chi$ in Section \ref{10.4.1.19}.

We first derive the  estimates in (\ref{ricp}) and (\ref{sna}), which are recast below.
\begin{proposition}
For $0\le 1-\frac{2}{p}< s'-2$, there hold
\begin{equation}\label{10.4.2.19}
\|\tir^\frac{3}{2}\sn z\|_{L_t^\infty L_u^\infty L_\omega^p(\widetilde{\D^+})}+\|\tir \big(\sn ( \chih,  z)\big)\|_{L_t^2 L_\omega^p(C_u\cap \widetilde{\D^+})}\les \la^{-\f12}.
\end{equation}
\end{proposition}
\begin{proof}
By using (\ref{ldz_2}), according to (\ref{a6.19}), we employ Lemma \ref{tsp2} to derive that
\begin{align}\label{9.14.6.19}
\begin{split}
\ti r^{3} |\sn_A z-e^\varrho (\curl\Omega)_A|&\les \lim_{\tau\rightarrow t_{\tmin}}\left| (\tau-u)^3(\sn z-e^\varrho\curl \Omega)_A(\tau)\right|\\
&+\int_{t_\tmin}^t \left(\tir^2|\sn \pi|+\tir^3 |\sn \chih\c \chih|+ \tir^3 |(z, \pi)\c \sn  \pi|\right)\\
&+\int_{t_\tmin}^t \tir^3\left( | \p \Omega \c (\bA+\tir^{-1})|+|\curl^2 \Omega|\right).
\end{split}
\end{align}
Hence, we bound
\begin{align*}
\|\tir&(\sn_A z-e^\varrho (\curl \Omega)_A)\|_{ L_\omega^p}\\
&\les\|\tir^{-2} |\min(u,0)|^3(|\sn z|+|\curl\Omega|)(0)\|_{L_\omega^p}+ \tir^{-1}\int_{t_\tmin}^t\|\tir\sn\pi\|_{L_\omega^p}dt'\\
&+\tir^{-1}\int_{t_\tmin}^t \|\tir\p \Omega, \tir^2 \curl^2 \Omega  \|_{L_\omega^p} dt'+\|\tir \sn \chih, \tir \sn \pi, \tir \p \Omega\|_{L_t^2 L_\omega^p}\|\bA\|_{L_t^2L_\omega^\infty}.
\end{align*}

Due to $\tir \les \tau_*$, by (\ref{9.13.3.19}),
\begin{equation}\label{lpcurl}
\|\tir\p \Omega\|_{L_\omega^p}\les \la^{-1-4\ep_0}.
\end{equation}
By (\ref{9.13.2.19}) and (\ref{lpcurl})
\begin{equation*}
\tir^{-1}\int_{t_{\tmin}}^t \|\tir\p \Omega, \tir^2 \curl^2 \Omega  \|_{L_\omega^p} dt'\les \la^{-1-4\ep_0};
\end{equation*}
and by using (\ref{flux_3}) and (\ref{lpcurl})
\begin{align*}
&\|\tir \sn \pi, \tir \p \Omega \|_{L_t^2 L_\omega^p}\les \la^{-\f12}.
\end{align*}

Hence by using the above three estimates,  in view of Lemma \ref{inii} (i) and (iii), and (\ref{a6.19}), we then obtain
\begin{equation}\label{9.14.7.19}
\|\ti r \sn z\|_{L_\omega^p} \les \tir^{-2}|\min(u,0)|^\frac{3}{2}\la^{-\f12}+\tir^{-1}\int_{t_\tmin}^t\|\tir \sn \pi\|_{L_\omega^p} dt'+\la^{-1+\ep_0}+\la^{-\f12+\ep_0}\| \tir\sn \chih\|_{L_t^2 L_\omega^p},
\end{equation}
which implies
\begin{align}\label{2psnz}
\|\ti r \sn z\|_{L_t^2 L_\omega^p(\widetilde{\D^+})} &\les \la^{-\f12}+\|\tir \sn \pi\|_{L_t^2 L_\omega^p(\widetilde{\D^+})}+\la^{-3\ep_0}\|\tir \sn \chih\|_{L_t^2 L_\omega^p(\widetilde{\D^+})} \nn \\
& \les \la^{-\f12}+\la^{-2\ep_0}\|\tir \sn \chih\|_{L_t^2 L_\omega^p(\widetilde{\D^+})},
\end{align}
where we used (\ref{hlm}) and (\ref{flux_3}).

By using (\ref{dchi1}) and Lemma \ref{hdgm1},
\begin{align}\label{9.14.5.19}
\begin{split}
\|\tir \sn \chih\|_{L_\omega^p}+\|\chih\|_{L_\omega^p}&\les \|\tir(\sn \pi,\tir^{-1}\pi, \sn z)\|_{L_\omega^p}+\|\tir\bA\c \pi\|_{L_\omega^p}\\
&\les \|\tir(\sn \pi,\tir^{-1}\pi, \sn z)\|_{L_\omega^p}+\la^{-1},
\end{split}
\end{align}
where we estimated by using (\ref{pric3})
\begin{align*}
\|\tir \bA\c \pi\|_{L_\omega^p}\les \|\tir^\f12 \bA\|_{L_\omega^{2p}}\|\tir^\f12 \pi\|_{L_\omega^{2p}}\les \la^{-1}.
\end{align*}
Taking $L_t^2$ norm of (\ref{9.14.5.19}) along  $C_u$ in $\widetilde{\D^+}$ yields
\begin{align*}
\|\tir\sn \chih\|_{L_t^2 L_\omega^p(C_u)}+\|\chih\|_{L_t^2 L_\omega^p(C_u)}\les \la^{-\f12}+\|\tir \sn z\|_{L_t^2 L_\omega^p(C_u)},
\end{align*}
where we used (\ref{flux_3}). Substituting the above estimate to the last term of (\ref{2psnz}) gives
\begin{equation*}
\|\tir\sn \chih, \tir \sn z\|_{L_t^2 L_\omega^p(C_u)}\les \la^{-\f12},
\end{equation*}
for all $C_u$ in $\widetilde{\D^+}$.

With the first estimate in the above,  we can use (\ref{9.14.7.19}) and  (\ref{flux_3}) to bound
\begin{align*}
\|\tir^\frac{3}{2}\sn z\|_{ L_\omega^p}&\les \la^{-\f12}+ \|\tir\sn\pi \|_{L_t^2 L_\omega^p(C_u)}+\tau_*^\f12.
\la^{-\f12+\ep_0}\|\tir \sn \chih\|_{L_t^2 L_\omega^p(C_u)}\les \la^{-\f12}.
\end{align*}
Thus we completed the proof of (\ref{10.4.2.19}), which are the estimates for $\sn \chih, \sn z$ in (\ref{ricp}) and (\ref{sna}).
\end{proof}
This also proved the second estimate in (\ref{comp2}) in view of (\ref{9.14.8.19}) and (\ref{pric2}) for $z$. The first estimate of (\ref{comp2}) follows as a consequence.

From (\ref{comp2}), in the region $u\le 0$, we can derive
\begin{equation*}
|z|\les (t+\fv)^{-\f12}\la^{-\f12}
\end{equation*}
and thus if $u<\frac{5t}{6}$
\begin{equation}\label{10.5.5.19}
|z|\les t^{-\f12}\la^{-\f12}.
\end{equation}
Hence, the estimate for $z$ in (\ref{ric1.1}) and (\ref{ric1})  hold in the region where $u\le  \frac{5t}{6}$, whereas  in the region that $u\ge \frac{5t}{6}$ we need to seek for a different approach. This is achieved in Section \ref{10.4.1.19} by considering the transport equations of $\sY$ and $\sn\sY$. Such analysis will be based on the  estimate of $\zeta$ in (\ref{ric1}).

Before proceeding to the control of $\zeta$,  we give a consequence of (\ref{sna}).
\begin{lemma}\label{larea}
Let $\varphi:=\log \sqrt{|\ga|}-\log \sqrt{|\ckk\ga|}$ on $S_{t, u}$ and  $0\le 1-\frac{2}{p}< s'-2$.  There hold on $\widetilde{\D^+}$ the estimates
\begin{equation*}
\|\tir^\f12\sn\varphi\|_{L_t^\infty L_\omega^p(C_u)}+\|\sn \varphi\|_{L_t^2 L_\omega^p(C_u)}\les \la^{-\f12}.
\end{equation*}
\end{lemma}
\begin{proof}
By using (\ref{cmu2}) and $L\varphi=\tr\chi-\frac{2}{t-u}$ we derive that
\begin{equation*}
\sn_L \sn\varphi+\frac{1}{2}\tr\chi\sn\varphi=-\chih \c\sn\varphi+\sn (\tr\chi-\frac{2}{t-u}).
\end{equation*}

By using (\ref{a6.19}),  with $m=1$ we  apply  Lemma \ref{tsp2} to obtain
\begin{align*}
\tir |\sn \varphi|\les \lim_{\tau\rightarrow t_\tmin} (\tau-u)|\sn \varphi|(\tau)+\int_{t_\tmin}^t \tir |\sn (\tr\chi-\frac{2}{\tir})| dt'.
\end{align*}
For null cones $C_u$ with $u\ge 0$ we  use (\ref{8.1.1}) to see the limit on the right hand side vanishes;  and for null cones $C_u$ with $u<0$ we use (\ref{4a_6}). Hence for $S_{t,u}$ and $C_u$ contained in $\widetilde{\D^+}$,
\begin{align*}
&\|\tir^\f12 \sn \varphi\|_{L_\omega^p(S_{t,u})}\les \|\tir^{-\f12} \min(u, 0)\sn \varphi\|_{L_\omega^p(S_{t,u})}+\|\tir \sn (\tr\chi-\frac{2}{\tir})\|_{L_t^2 L_\omega^p(C_u\cap \widetilde{\D^+})},\\
&\|\sn \varphi\|_{L_t^2 L_\omega^p(C_u)}\les \|\tir^{-1}\min(u,0)\sn \varphi\|_{L_t^2 L_\omega^p(C_u)}+\|\tir \sn (\tr\chi-\frac{2}{\tir})\|_{L_t^2 L_\omega^p(C_u\cap \widetilde{\D^+})}.
\end{align*}
Similar to  (\ref{10.10.1.19}), we can bound the first terms on the right hand side by $\la^{-\f12}$. Consequently,
\begin{align*}
\|\tir^\f12 \sn \varphi\|_{L_\omega^p(S_{t,u})}+\|\sn \varphi\|_{L_t^2 L_\omega^p(C_u)}\les \la^{-\f12} +\|\tir \sn (\tr\chi-\frac{2}{\tir})\|_{L_t^2 L_\omega^p(C_u\cap \widetilde{\D^+})}.
\end{align*}
Since $\tr \chi - \frac{2}{t-u}= z -\Xi_4$, we may use (\ref{sna}) and (\ref{flux_3}) to obtain
$
\left\|\tir\sn (\tr\chi-\frac{2}{\tir})\right\|_{L_t^2 L_\omega^p(C_u)}\les \la^{-\f12}.
$
Therefore
\begin{align*}
\|\tir^\f12 \sn \varphi\|_{L_\omega^p(S_{t,u})}+\|\sn \varphi\|_{L_t^2 L_\omega^p(C_u)}\les\la^{-\f12}.
\end{align*}
Hence the proof of Lemma \ref{larea} is complete.
\end{proof}

\subsubsection{Estimates of $\zeta$.}\label{causal_step2}
\begin{proposition}\label{7.04.1.19}
Let $0\le 1-\frac{2}{p}<s'-2$. There holds for $C_u \cap \widetilde{\D^+}$ that
\begin{equation} \label{7.03.5.19}
\|\tir\sn \zeta\|_{L_t^2 L^p_\omega(C_u)}+\|\zeta\|_{L_t^2 L^p_\omega(C_u)}\les\la^{-\f12}.
\end{equation}
\end{proposition}
\begin{proof}
 By using Lemma \ref{hdgm1}, (\ref{6.17.7.19}), (\ref{pric3}) and (\ref{sobinf}) on $S=S_{t,u}$,
\begin{align*}
\|\tir \sn(\zeta-\zb)\|_{L^p_\omega(S)}&+\|\zeta-\zb\|_{L^p_\omega(S)}\\
&\les \|\zeta-\zb\|_{L^\infty(S)}\|\tir \sn \varphi\|_{L_\omega^p(S)}+\|\tir |\bA|^2\|_{L^p_\omega(S)}+\|\tir\sn \zb\|_{L^p_\omega(S)}\\
&\les \|(\tir\sn)\rp{\le 1}(\zeta-\zb)\|_{L^p_\omega(S)}\|\tir \sn \varphi\|_{L_\omega^p(S)}+\la^{-1}+\|\tir\sn \zb\|_{L_\omega^p(S)}.
\end{align*}
Thus by using the first estimate in Lemma \ref{larea} and (\ref{flux_3}), we can obtain
\begin{equation*}
\|\tir\sn (\zeta-\zb)\|_{L_t^2 L^p_\omega(C_u)}+\|\zeta-\zb\|_{L_t^2 L^p_\omega(C_u)}\les\la^{-\f12},
\end{equation*}
which implies (\ref{7.03.5.19}) by applying (\ref{flux_3}) again to $\zb=-k_{A\bN}$.
\end{proof}

  Now apply (\ref{sobinf}) to $\chih$, $z$ and $\zeta$, together with using the first inequality in (\ref{pric1}), (\ref{10.4.2.19}) and (\ref{7.03.5.19}). This leads to the estimate for any $C_u$ contained in $\widetilde{\D^+}$,
\begin{equation}\label{9.14.9.19}
\|\chih, z, \zeta\|_{L_t^2 L_\omega^\infty(C_u)}\les \la^{-\f12},
\end{equation}
which gives (\ref{ric3.18.1}) and improves (\ref{ba3.18.1}).

As a consequence of the estimates of $\chih$ and $z$ in (\ref{ric3.18.1}) together with (\ref{sna}), (\ref{8.0.3}) can be proved by using the transport equation (\ref{trscoord2}) and its angular derivative.  (See the proof in \cite[Section 5.5.2]{Wangrough}.) Thus the proof of  Proposition \ref{ricpr} is complete.

Next we provide the control of $\zeta$ in (\ref{ric1.1}) and (\ref{ric1}). The following result is actually stronger than stated therein, and will be crucially used in the proof of estimate of $\chih$ in (\ref{ric1.1}) and (\ref{ric1}).
\begin{proposition}
There holds on $\widetilde{\D^+}$ that
\begin{equation}\label{9.14.10.19}
\|\zeta\|_{L_t^2 L^\infty_x(\widetilde{\D^+})}\les \la^{-\f12-4\ep_0}.
\end{equation}
\end{proposition}
\begin{proof}
Applying Proposition \ref{cz.2} to (\ref{6.17.7.19}) implies
\begin{align}\label{10.5.1.19}
\|\zeta\|_{L^\infty(S_{t,u})}&\les \|\mu^{\delta}P_\mu \ti\pi\|_{l_\mu^c L^\infty(S_{t,u})}
+\|\ti\pi\|_{L^\infty(S_{t,u})}+\tir^{1-\frac{2}{p}}\|(\zeta-\zb)\c\sn\varphi,\bA\c \bA\|_{L^p(S_{t,u})}
\end{align}
with $0<\delta<s'-2$  sufficiently small and  $0<1-\frac{2}{p}<s'-2$.

By using the first estimate in Lemma \ref{larea} and (\ref{pi.2})
\begin{align*}
\|\tir(\zeta-\zb)\c\sn\varphi\|_{L_t^2 L_u^\infty L^p_\omega(\widetilde{\D^+})}&\les \|\tir \sn \varphi\|_{L^\infty L^p_\omega(\widetilde{\D^+})}\| (\zeta-\zb)\|_{L_t^2 L_x^\infty(\widetilde{\D^+})}\\
&\les\la^{-4\ep_0} (\|\zeta\|_{L_t^2L_x^\infty(\widetilde{\D^+})}+\|\zb\|_{L_t^2 L_x^\infty(\widetilde{\D^+})})\\
&\les \la^{-4\ep_0}\|\zeta\|_{L_t^2 L_x^\infty(\widetilde{\D^+})}+\la^{-\f12-8\ep_0}.
\end{align*}
Substituting the above inequality into (\ref{10.5.1.19}) and using (\ref{pi.2}) and  (\ref{pric3}), we can obtain (\ref{9.14.10.19}) after taking $L_t^2 L_u^\infty$ norm.
\end{proof}


\subsection{Improved estimates of $z$ and  $\chih$}\label{10.4.1.19}
In order to complete the sets of estimates (\ref{ric1}) and (\ref{ric1.1}), due to $\tr\chi-\frac{2}{\tir}=z+\pi$, it only remains to provide the estimates of $z$ and $\chih$ therein. We first focus on deriving the improved estimates for $z$ below, for which we construct the quantity of $\sY$ whose angular derivative exhibits favourable structures. With the help of the control of $\sY$, we can derive the estimate on $\chih$ by applying estimate (\ref{cz0}) to a normalized equation based on (\ref{dchi}).

\begin{proposition}\label{10.17.1.19}
Let $0\le 1-\frac{2}{p}<s'-2$. There hold the following estimates, 
\begin{align}
&\|z\|_{L_t^2 L_u^\infty L_\omega^p(\D^+)}+\|\tir \sn \sY\|_{L_t^2 L_u^\infty L_\omega^p(\D^+)}\les \la^{-\f12-4\ep_0}\label{9.13.6.19}\\
&\|z\|_{L_t^2 L_x^\infty(\D^+)}\les \la^{-\f12-4\ep_0} \label{9.16.2.19}\\
&\|\tir \sn \sY\|_{L_t^{\frac{q}{2}}L_u^\infty L_\omega^p(\widetilde{\D^+})}\les \la^{\frac{2}{q}-1-4\ep_0(\frac{4}{q}-1)}, \quad 2<q<4\label{11.2.1.19}\\
&\|\chih\|_{L_t^2 L_x^\infty(\D^+)}\les\la^{-\f12-4\ep_0}\label{9.16.3.19}\\
&\| z, \chih\|_{L_t^{\frac{q}{2}}L_x^\infty(\widetilde{\D^+})}\les \la^{\frac{2}{q}-1-4\ep_0(\frac{4}{q}-1)}, \, 2<q<4.\label{10.5.2.19}
\end{align}
\end{proposition}
\begin{proof}
We first prove (\ref{9.13.6.19}). By integrating the transport equation (\ref{lz}) along null geodesics,
we have from Lemma \ref{tsp2} that
\begin{align}\label{tirz}
\begin{split}
\tir^2|z(t)|&\les \left|\lim_{\tau\rightarrow t_{\tmin}} (\tau-u)^2z(\tau)\right|\\
&+|\int_{t_\tmin}^t \tir^2 \left( |\bA\c\bA| + |\tir^{-1}\pi|+  |\curl \Omega| \right)dt'|.
\end{split}
\end{align}
For null cones $C_u$ with $u\ge 0$, the limit term  vanishes  due to  Lemma \ref{inii} (i).  Thus we can bound
\begin{align*}
|z(t)|&\les  \tir^{-1}\int_{t_\tmin}^t \tir \big( |\bA\c\bA| + |\curl \Omega|+\tir^{-1}|\pi|\big)dt'.
\end{align*}
Note that by using   (\ref{pric3})
\begin{equation*}
\|\tir \bA\c \bA\|_{ L_\omega^p}\les\|\tir^\f12 \bA\|_{L_t^\infty L_\omega^{2p}}^2\les \la^{-1}.
\end{equation*}
Using the above estimate and (\ref{lpcurl}), taking $L^p_\omega$ norm of $z$ gives
\begin{align}
\|z\|_{L_\omega^p}& \les  \|\tir \curl \Omega\|_{L_t^\infty L_\omega^p}+ \|\tir \bA\c\bA\|_{L_t^\infty L_\omega^p}+\tir^{-1}\int_u^t \|\pi\|_{L_\omega^p}dt'\nn\\
&\les \la^{-1}+\tir^{-1}\int_u^t \|\pi\|_{L_\omega^p}dt'.\label{9.14.3.19}
\end{align}

It follows by using (\ref{hlm}) and (\ref{pi.2}) that
\begin{align*}
(\int_0^{\tau_*}\sup_{0\le u'\le t}\|z\|_{L_\omega^p(S_{t,u'})}^2 dt)^\f12\les \la^{-\f12-4\ep_0}+\|\pi\|_{L_t^2 L^\infty(\D^+)}\les \la^{-\f12-4\ep_0}.
\end{align*}
Thus we obtained the first estimate in (\ref{9.13.6.19}).

Recall from the line of (\ref{9.14.6.19}), the first term of  $\sn \pi$ on the right hand side is a higher order linear term, which comes from differentiating
$(\Xi_4-k_{\bN\bN})\widetilde{\tr\chi}$ in (\ref{lz}). It becomes a singular term in particular in the region allowing $\tir$ to be close to $0$,  since $
\sn \pi$ is not sufficiently smooth. This is the main hurdle for us to achieve $\|\tir \sn z\|_{L_t^2 L_u^\infty L_\omega^p(\D^+)}$.
Our strategy is to construct the quantity $\sY$, whose major part is $z$. We will see the transport equation of $\sn \sY$ does not contain such higher order linear term, which gets around
of the potential singularity.

We recall from (\ref{7.04.8.19}) that
\begin{equation*}
\Xi_4 -2 k_{\bN\bN}=L(\log c+\varrho)+2 c^{-2} L(v^i) \bN^j \delta_{ij}.
\end{equation*}
By using (\ref{cmu2}),
\begin{align*}
\sn(\Xi_4-2k_{\bN\bN})&=\sn L(\log c+\varrho)+2\sn( L (v^i) \bN^j c^{-2}\delta_{ij})\\
&=\sn_L\sn(\log c+\varrho)+2\sn_L\sn (v^i)\bN^j c^{-2} \delta_{ij}+\chi\c \sn (\log c+\varrho)\\
&+2\chi\c\sn (v^i)\bN^j c^{-2}\delta_{ij}+2L(v^i) \sn(\bN^j g_{ij})\\
&=\sn_L\{\sn(\log c+\varrho)+2\sn (v^i) \bN^j c^{-2} \delta_{ij}\}-2\sn_L(\bN^j c^{-2} \delta_{ij}) \sn(v^i)\\
&+\chi\c (\sn (\log c+\varrho)+2\sn (v^i)\bN^j c^{-2}\delta_{ij})+2L(v^i) \sn(\bN^j g_{ij}).
\end{align*}
By using Proposition \ref{6.7con}, we can derive symbolically
\begin{equation}\label{10.5.3.19}
\sn_L(\bN^j c^{-2} \delta_{ij}) \sn(v^i), L(v^i) \sn(\bN^j g_{ij})=\pi\c \pi+\chi\c\pi,
\end{equation}
where $\pi\c \bg$ can be also regarded  of the type of $\pi$.
Indeed, from (\ref{6.29.5.19}) and $2\bN=L-\Lb$, we compute
 $$2\l\bd_L \bN, e_A\r=\l \bd_L L, e_A\r-\l\bd_L \Lb, e_A\r=-2\zb_A=2k_{A\bN}.$$
   Also by denoting  ${}\rp{\bg}\Ga\c X$ as $\pi$, we can write
   $
  \sn_L(\bN^j c^{-2} \delta_{ij}) \sn(v^i)=\pi\c \pi \c\bg$.

 The symbolic formula for the other term in (\ref{10.5.3.19}) can be obtained by noting $\sn_A \bN^i=\theta_{AB}e_B^i=(\chi_{AB}+k_{AB})e_B^i$ due to the first
 identity in (\ref{3.19.1}).

For convenience, we denote $\pi_1=\sn (\log c+\varrho)+2\sn (v^i) \bN^j c^{-2}\delta_{ij}$,  which can be symbolically regarded as $\pi$.
Thus
\begin{equation}\label{9.15.2.19}
\sn(\Xi_4-2k_{\bN\bN})=\sn_L(\pi_1)+\chi\c \pi+\pi\c \pi.
\end{equation}

In view of
\begin{equation}\label{9.15.4.19}
\chi_{AB}=\f12 \delta_{AB}\tr\chi+\chih_{AB}=\f12 \delta_{AB} (\widetilde{\tr\chi}-\Xi_4)+\chih_{AB},
\end{equation}
and in view of the definition that $z=\widetilde{\tr\chi}-\frac{2}{\tir}$, we derive from (\ref{9.15.2.19}) the symbolic form
\begin{equation}\label{9.15.3.19}
\sn(\Xi_4-2k_{\bN\bN})=\sn_L(\pi_1)+(\bA+\tir^{-1})\c \pi.
\end{equation}

On the other hand, we multiply (\ref{9.15.1.19}) by $\bb$ and apply (\ref{lb})
\begin{equation*}
L (\bb \widetilde{\tr\chi})+\f12 \bb (\widetilde{\tr\chi})^2=\frac{2}{\tir}(\Xi_4-2k_{\bN\bN})+(\Xi_4-2k_{\bN\bN})(\bb \widetilde{\tr\chi}-\frac{2}{\tir})+\bb\c G,
\end{equation*}
with
\begin{equation}\label{9.15.6.19}
G=\exp\varrho \curl \Omega_\bN-|\chih|^2+\pi\c \pi,
\end{equation}
where the last term of $G$ is a symbolic representation.

We differentiate the above transport equation with the help of (\ref{cmu2})
\begin{align}
\sn_L \sn (\bb \widetilde{\tr\chi})+\widetilde{\tr\chi}\sn (\bb\widetilde{\tr\chi})&=-\chi\c \sn(\bb \widetilde{\tr\chi})-\f12 \sn (\bb^{-1})(\bb\widetilde{\tr\chi})^2+\frac{2}{\tir}\sn(\Xi_4-2 k_{\bN\bN})\nn\\
&+\sn\{(\Xi_4-2k_{\bN\bN})(\bb \widetilde{\tr\chi}-\frac{2}{\tir})+\bb\c G\}.\label{10.26.2.19}
\end{align}
By using (\ref{9.15.4.19})
\begin{align*}
\sn_L  \sn (\bb \widetilde{\tr\chi})&+\frac{3}{2}\widetilde{\tr\chi}\sn (\bb\widetilde{\tr\chi})-\frac{2}{\tir}\sn(\Xi_4-2 k_{\bN\bN})-\bb \sn G\\
&=(-\chih+\pi)\sn(\bb\widetilde{\tr\chi})+\sn\pi(\bb \widetilde{\tr\chi}-\frac{2}{\tir})+\sn \bb (G+( \widetilde{\tr\chi})^2),
\end{align*}
where the right hand side is a symbolic expression.
In view of $z=\widetilde{\tr\chi}-\frac{2}{\tir}$, applying (\ref{9.15.3.19}) to the above identity gives
\begin{align*}
&\sn_L (\sn(\bb\widetilde{\tr\chi})-\frac{2}{\tir}\pi_1)+\frac{3}{2}\widetilde{\tr\chi}(\sn(\bb\widetilde{\tr\chi})-\frac{2}{\tir}\pi_1)\\
&=\bA(\sn(\bb\widetilde{\tr\chi})-\frac{2}{\tir}\pi_1)+\tir^{-1}\pi(\bA+\tir^{-1})+\sn\pi(\bb \widetilde{\tr\chi}-\frac{2}{\tir})\\
&+\sn \bb (G+(\widetilde{\tr\chi})^2)+\bb \sn G.
\end{align*}
For the term $\bb\sn G$, with $\ti G=-|\chih|^2+\pi\c \pi$, we substitute the trace decomposition in (\ref{9.16.1.19}) to treat
the leading term from $G$ in (\ref{9.15.6.19}).
This leads to
\begin{equation}\label{9.15.5.19}
\begin{split}
&\sn_L (\sn_A\sY-\bb e^\varrho(\curl \Omega)_A-\frac{2}{\tir} \pi_1)+\frac{3}{2}\widetilde{\tr\chi}((\sn_A \sY-\bb  e^\varrho(\curl \Omega)_A-\frac{2}{\tir}\pi_1)\\
&=\bA (\sn \sY-\bb  e^\varrho(\curl \Omega)_A-\frac{2}{\tir}\pi_1)+\bb(\tir^{-1}+\bA) e^\varrho \curl \Omega+\sn\pi\c \sY\\
&+\tir^{-1}\pi(\bA+\tir^{-1})+\sn \bb (\ti G+(\widetilde{\tr\chi})^2)+\bb \sn \ti G\\
&+\bb e^{\varrho} \big({e_A}_i \Pi^{ij} \tensor{\ep}{_{jm}^l}(\curl^2\Omega)_l\bN^m+(\pi+\chi) \p \Omega\c X \big).
\end{split}
\end{equation}

We note in general for $C_u \cap \widetilde{\D^+}$,   combining (\ref{ric3.18.1}) and (\ref{pi.2})  gives
\begin{equation}\label{10.5.4.19}
\|\bA\|_{L_t^2 L^\infty_\omega(C_u)}\les \la^{-\f12}.
\end{equation}
By using Lemma \ref{inii} (i), when $u\ge 0$, $\sn \sY\rightarrow 0$ as $t\rightarrow u$,  $\tir \pi_1\rightarrow 0$ and $\tir |\curl \Omega|\rightarrow 0$.
Due to  the fact that $\|\bA\|_{L_\omega^\infty L_t^1}\les 1$ derived from (\ref{10.5.4.19}), we apply  Lemma \ref{tsp2} to (\ref{9.15.5.19}) to derive
\begin{align*}
&\ti r^{3} |\sn\sY-\bb e^\varrho (\curl \Omega)_A-\frac{2}{\tir} \pi_1|\\
&\les \int_{t_\tmin}^t \tir^3 \left|\sn\pi\c \sY
+\tir^{-1}\pi(\bA+\tir^{-1})+\sn \bb (\ti G+(\widetilde{\tr\chi})^2)+\bb \sn \ti G\right| dt'\\
&+\int_{t_{\tmin}}^t\tir^3 \bb e^{\varrho} \left|{e_A}_i \Pi^{ij} \tensor{\ep}{_{jm}^l}(\curl^2\Omega)_l\bN^m+(\tir^{-1}+\bA) e^\varrho \curl \Omega
+(\pi+\chi) \p \Omega\c X\right| dt',
\end{align*}
where $t_\tmin=u$, since $u\ge 0$.

 By using $|\bb-1|\le \f12$ and (\ref{comp2}), $\zeta=\sn\log \bb+\pi$ and $z=\widetilde{\tr \chi}-\frac{2}{\tir}$, we can derive symbolically,
\begin{align*}
\ti r^{3} &|\sn\sY-\bb e^\varrho (\curl \Omega)_A-\frac{2}{\tir} \pi_1|\\
&\les\int_{t_\tmin}^t \tir^3 \left(|\sn\pi\c \sY|+\tir^{-1}|\bA\c \bA|+(|\zeta|+|\pi|) \c (|\ti G|+\tir^{-2})+ |\sn \ti G|\right)dt'\\
&+\int_{t_{\tmin}}^t\tir^3 \big(|\curl^2\Omega|+\tir^{-1} | \p \Omega|+|\bA \c \p \Omega|\big)dt'.
\end{align*}
Hence
\begin{align}
\tir\|\sn \sY&-\bb e^\varrho (\curl \Omega)_A-\frac{2}{\tir}\pi_1\|_{L_\omega^p}\label{9.18.2.19}\\
&\les  \int_{t_{\tmin}}^t \tir(\|\sn \pi\c \sY\|_{L_\omega^p}+\|\curl^2 \Omega\|_{L_\omega^p}) dt'+\tir^{-1}\int_{t_{\tmin}}^t  (\|\zeta\|_{L_\omega^p}+\|\pi\|_{L_\omega^p}+\|\tir \p \Omega\|_{L_\omega^p}) dt'\nn\\
&+\int_{t_{\tmin}}^t\tir \left(\||\bA|(\tir^{-1}|\bA|+|\ti G|)\|_{L_\omega^p}+\|\sn \ti G\|_{L_\omega^p}+\|\bA \c \p \Omega\|_{L_\omega^p}\right)\nn.
\end{align}
We can compute
$$|\sn \ti G|\les  |\sn \chih| \c |\chih|+|\sn \pi| |\pi|.$$

 We then use (\ref{10.5.4.19}), (\ref{lpcurl}), the first estimate in (\ref{sna}) and (\ref{flux_3}) to obtain
\begin{align}\label{9.18.3.19}
\|\tir \sn \ti G\|_{L_t^1 L_\omega^p(C_u)}&+\|\tir \bA\c \p \Omega\|_{L_t^1 L_\omega^p(C_u)}\\
&\les \|\bA|\|_{L_t^2 L_\omega^\infty(C_u)}\|\tir( \sn \chih, \sn \pi, \p \Omega)\|_{L_t^2 L^p_\omega(C_u)}\les \la^{-1}.\nn
\end{align}

And for the lower order terms  since $\ti G=\bA\c \bA$,
\begin{align}\label{9.18.4.19}
\begin{split}
\|\tir \bA &\c \bA\c(\bA+\tir^{-1})\|_{L_t^1 L_\omega^p(C_u)}\\
&\les\|\bA\c\bA\|_{L_t^1 L_\omega^p(C_u)}+  \|\bA\|_{L_t^2 L_\omega^\infty(C_u)}\|\tir \bA\c \bA\|_{L_t^2 L_\omega^p(C_u)}\\
&\les \|\bA\|_{L_t^2 L_\omega^\infty(C_u)}(\|\bA\|_{L_t^2 L_\omega^p(C_u)}+\|\tir \bA\c \bA\|_{L_t^2 L_\omega^p(C_u)})\les \la^{-1},
\end{split}
\end{align}
where we used (\ref{10.5.4.19}), (\ref{pric1}) and (\ref{pric3}).

Combining  $z=\bb^{-1}\sY+2\mho$, the first estimate in (\ref{9.14.1.19}) and the estimate for $z$ in (\ref{9.14.9.19}) gives
\begin{equation*}
\|\sY\|_{L_t^2 L_\omega^\infty(C_u\cap \widetilde{\D^+})}\les \la^{-\f12}.
\end{equation*}
Using (\ref{flux_3}) and the above estimate implies
\begin{align*}
\|\tir \sn \pi \c \sY\|_{L_t^1 L_\omega^p(C_u)}&\les \|\sY\|_{L_t^2 L_\omega^\infty(C_u)}\|\tir \sn \pi\|_{L_t^2 L_\omega^p(C_u)}\les \la^{-1}.
\end{align*}
By using the above estimate together with (\ref{9.18.3.19}), (\ref{9.18.4.19}) and the second estimate in  (\ref{9.13.2.19}), we can bound on $S_{t,u}$ contained in $\D^+$,
\begin{align*}
\tir\|\sn \sY&-\bb e^\varrho (\curl \Omega)_A-\frac{2}{\tir}\pi_1\|_{L_\omega^p}\les\la^{-1}+ \tir^{-1}\int_{t_\tmin}^t (\|\zeta\|_{L_\omega^p}+\| \pi\|_{L_\omega^p}) dt'.
\end{align*}
It then follows by using (\ref{hlm}), (\ref{pi.2}) and (\ref{9.14.10.19}) that
\begin{align*}
\tir\|\sn \sY-\bb e^\varrho (\curl \Omega)_A-\frac{2}{\tir}\pi_1\|_{L_t^2 L_u^\infty L_\omega^p(\D^+)}&\les \|\zeta, \pi\|_{L_t^2 L_x^\infty(\D^+)}+\la^{-\f12-4\ep_0}\\
&\les \la^{-\f12-4\ep_0}.
\end{align*}
Using (\ref{pi.2}) and (\ref{lpcurl}),
\begin{equation*}
\|\tir \sn \sY\|_{L_t^2 L_u^\infty L_\omega^p(\D^+)}\les \la^{-\f12-4\ep_0}.
\end{equation*}
This is the second estimate of (\ref{9.13.6.19}).

Using (\ref{9.14.2.19}), $\sY=\bb(z-2\mho)$ and the first estimate in (\ref{9.13.6.19}),
\begin{align*}
\|\sY\|_{L_t^2 L_u^\infty L_\omega^p(\D^+)}&\les \la^{-\f12-4\ep_0}.
\end{align*}
We combine the above two estimates and use the Sobolev embedding (\ref{sobinf}) to conclude
\begin{equation*}
\|\sY\|_{L_t^2 L^\infty(\D^+)}\les \la^{-\f12-4\ep_0}.
\end{equation*}
Using $ z=\bb^{-1}\sY+2\mho$  and the first estimate of (\ref{9.14.2.19})  gives (\ref{9.16.2.19}). For the estimate of $z$ in (\ref{10.5.2.19}), it remains to consider the region $u<0$, which can be derived immediately by integrating (\ref{10.5.5.19}).

Next, we consider $\sn\sY$ in the region where  $-\fv_*\le u\le \frac{5t}{6}$. For $S=S_{t,u}$ contained in this region, we derive directly from the definition of $\sY$ by using the first estimate in (\ref{10.4.2.19}), the second estimate in (\ref{ric4}) and (\ref{pric3})  that
\begin{align*}
\|\tir \sn \sY\|_{L_\omega^p(S)}&\les \|\tir \sn z\|_{L_\omega^p(S)}+\|\tir \sn \bb \c (z-2\mho)\|_{L_\omega^p(S)}+\|\tir \sn \mho\|_{L_\omega^p(S)}\\
&\les t^{-\f12}\|\tir^\frac{3}{2} \sn z\|_{L_\omega^p(S)}+\la^{-1}+\|\sn \bb\|_{L_\omega^p(S)}\\
&\les t^{-\f12} \la^{-\f12}+\la^{-1}+\|\zeta\|_{L_\omega^p(S)}+\|\pi\|_{L_\omega^p(S)}.
\end{align*}
As a consequence of the above estimate, (\ref{9.13.6.19}), (\ref{9.14.10.19}) and (\ref{pi.2}), we conclude
\begin{align*}
\|\tir \sn \sY\|_{L_t^{\frac{q}{2}}L_u^\infty L_\omega^p(\widetilde{\D^+})}&\les \|\zeta, \pi\|_{L_t^\frac{q}{2} L^\infty(\widetilde{\D^+})}+\la^{(1-8\ep_0)(\frac{2}{q}-\f12)-\f12}\\
&\les \la^{\frac{2}{q}-1-4\ep_0(\frac{4}{q}-1)}.
\end{align*}
This shows (\ref{11.2.1.19}).

Next, consider the estimate for $\chih$ in (\ref{9.16.3.19}). In view of (\ref{dchi1}), we can derive
\begin{equation*}
\sl{\div}(\bb\chih) = \f12 \sn \sY+  + \bb\sn \pi + \tir^{-1}(\bb\pi+\sn \bb) +\bb \bA\c \bA.
\end{equation*}
 Note $|\bb-1|\le \f12$ due to (\ref{bb_4}).  Applying (\ref{cz0}) and Proposition \ref{cz.2} yields
\begin{align*}
\|\chih\|_{L^\infty(S_{t,u})}&\les \|\D_2^{-1} (\sn \sY)\|_{L^\infty(S_{t, u})} + \|\D_2^{-1} \left(\sn (\bb\pi) + \tir^{-1}(\bb\pi+\sn \bb) +\bb \bA\c \bA\right)\|_{L^\infty(S_{t, u})}\\
& \les  \|\tir^{1-\frac{2}{p}}\sn \sY\|_{L^p(S_{t,u})}+ \|\mu^{0+}P_\mu \ti\pi\|_{l_\mu^2 L^\infty (S_{t,u})}+\|\tir^{1-\frac{2}{p}}(\bA\c \bA, \tir^{-1}\pi, \tir^{-1}\zeta) \|_{L^{p}(S_{t,u})}\\
& \les \|\tir^{1-\frac{2}{p}}\sn \sY\|_{L^p(S_{t,u})}+ \|\mu^{0+}P_\mu \ti\pi\|_{l_\mu^2 L^\infty (S_{t,u})}+\|\pi, \zeta\|_{L^\infty(S_{t,u})}+\la^{-1},
\end{align*}
where $0<1-\frac{2}{p}<s'-2$, and  for the last inequality we used (\ref{pric3}).  By virtue of (\ref{pi.2}), (\ref{9.14.10.19}) and the estimate for $\sY$ in (\ref{9.13.6.19}),
we have
$$
\|\chih\|_{L_t^2 L_x^\infty(\D^+)} \les  \la^{-\f12-4\ep_0};
$$
and  by using (\ref{pi.2}), (\ref{9.14.10.19}) and (\ref{11.2.1.19}) with $2<q<4$, we have
$$
\|\chih\|_{L_t^\frac{q}{2} L_x^\infty(\widetilde{\D^+})}\les \la^{\frac{2}{q}-1-4\ep_0(\frac{4}{q}-1)}.
$$
We therefore have obtained the estimates for $\chih$ in (\ref{9.16.3.19}) and (\ref{10.5.2.19}). Thus the proof of Proposition \ref{10.17.1.19} is complete.
\end{proof}

In summary,  (\ref{9.14.10.19}) is stronger than the estimates of $\zeta$ stated in (\ref{ric1}) and (\ref{ric1.1}), the estimates (\ref{9.16.3.19}) and  (\ref{10.5.2.19}) are included in (\ref{ric1}) and (\ref{ric1.1}). Since $\tr\chi-\frac{2}{\tir}=z+\pi$, using (\ref{pi.2}) and combining the estimates for $z$ in  (\ref{ric1}) and (\ref{ric1.1}), the estimates of $\tr\chi-\frac{2}{\tir}$ therein can be obtained immediately. Thus,  the proof of Proposition \ref{cone_reg} is complete.

\section{Regularity of the conformal metric and the mass aspect function}\label{conf}

Let us  set  in $\D^+$
\begin{equation}\label{6.22.17.19}
L \sigma=\f12 \Xi_L, \qquad \sigma(\Ga^+)=0
\end{equation}
where $\Ga^+$ is the time axis.
In order to prove  Theorem \ref{BT}, we need to carry out the conformal change of metric in the spacetime region $\D^+$ by  introducing the  metric $\ti \bg=e^{2\sigma} \bg$.
 The conformal method was introduced in \cite{Wangrough} to treat  the term $-\Xi_L$ in $\tr\chi-\widetilde{\tr\chi}$. It is in particular crucial for solving the significant difficulty caused by the weak regularity on $\sn \tr\chi$ and $\mu$ due to the rough data in \cite{Wangrough} for the equation (\ref{4.10.4.19}). The rough $\Xi_4$ derivative causes the same hurdle  in the general acoustic spacetime as for the irrotational case. Therefore, in this section, we provide the necessary control on $\sigma$ and $\mu$ for proving Theorem \ref{BT}. This theorem is the main building block to complete the dispersive estimate as indicated in Section \ref{10.30.1.19}. By the completion of the section, we will be able to achieve the complete set of estimates for the geometric quantities  required by reproducing the proof of Theorem \ref{BT} in \cite[Section 7]{Wangrough}. 
 
In this section, instead of bounding $\ckk\mu$ directly as in \cite{Wangrough}, we introduce a further normalization on the mass aspect function in (\ref{10.5.6.19}) to cope with the issue of the rough vorticity derivative, with the help of the transport equation of $\curl \Omega$ in (\ref{4.25.4.19}). It turns out that the influence of the rough vorticity can be reduced to be  lower order, which can be seen from Proposition \ref{Trans_mu} and the resulting estimates in Proposition \ref{6.23.9.19}. We also take advantage of the $\sl{\curl}$ structure for the vorticity on $S_{t,u}$ in (\ref{9.12.1.19}) to prove the decomposition of $\sn \sigma$ in Proposition \ref{dcmpsig}.

We recall the preliminary estimates, which can be obtained in the exact same way as in \cite[Lemma 6.1]{Wangrough}.
\begin{lemma}
Let $0\le 1-\frac{2}{p}<s-2$. Within $\D^+$, there hold
\begin{align}
\|\tir^\f12 L \sigma\|_{L_\omega^{2p}(C_u)}+\|r^{\f12-\frac{2}{p}}\sn \sigma\|_{L_x^p L^\infty(C_u)}+\|\sn \sigma\|_{L_t^2 L_\omega^p(C_u)}\les \la^{-\f12}\label{6.22.18.19}\\
\|\sigma\|_{L^\infty}\les \la^{-8\ep_0}, \quad \|\tir^{-\f12} \sigma\|_{L^\infty} \les \la^{-\f12-4\ep_0},\nn
\end{align}
\end{lemma}
The above estimates of $\sn \sigma$ are much weaker for our purpose, the improved estimate on $\sn\sigma$ will be achieved with the help of the estimate of a normalized mass aspect function $\ckk \mu$ (see (\ref{6.23_mu}) for the definition). Let us first recall the null transport equation for the mass aspect function $\mu$ from \cite[(6.12)]{Wangrough}.
\begin{lemma}
\begin{align}\label{6.22.19.19}
L\mu&+\tr\chi\mu=R(\mu)-k_{\bN\bN}L\tr\chi +2(\zb_A-\ze_A)\sn_A\tr\chi+\f12\left( \tr\chi\chih\c\chibh+\tr\chib|\chih|^2\right)\nn\\
& +\tr\chi\left(\sl{\div}\zb+|\zb|^2+\f12\delta^{AB} \bR_{A34B}+L(k_{\bN\bN})-(\sl{\div} \pi + \sl \bE)\right) \nn\\
& -2\chih_{AB} \left(2\sn_A\zeta_B+k_{\bN\bN}\chih_{AB}+2\zeta_A\zeta_B+\bR_{A43B}\right),
\end{align}
where $\sl\bE= \bA\c \pi+\tr\chi\c \pi$, and
$$
R(\mu):=-\Lb \bR_{44}-\tr\chi \bR_{34}-\f12\tr\chib \bR_{44}.
$$
\end{lemma}
By virtue of (\ref{5.27.2.19}), (\ref{6.23.2.19}),  (\ref{6.7con}) and the fact that $S=\pi\c \pi$, we can derive that
\begin{align*}
\bR_{44} & = L(\Xi_L)+k_{\bN\bN} \Xi_4 + \pi\c\pi-e^{\varrho}\delta_{ij}\bN^j \curl \Omega^i,\\
\bR_{34}&=\f12( \Lb(\Xi_L)+L(\Xi_\Lb))+\Xi\c (\zeta+\zb)+k_{\bN\bN}\c \Xi+\pi\c\pi,\\
-\Lb \bR_{44}&=-\Lb L (\Xi_L)-\Lb (k_{\bN\bN}) \Xi_L +k_{\bN\bN}\c (\zeta+k)\c \Xi+k \c\bd \Xi\\
&+\delta_{ij}\Lb (e^\varrho\bN^j \curl \Omega^i)+\Lb (\pi\c\pi),
\end{align*}
for which we used (\ref{4.10.2.19}), (\ref{6.29.5.19}) and (\ref{6.29.6.19}) for simplifying the right hand sides of the above formulas.

Hence
\begin{equation}\label{6.23.3.19}
\begin{split}
R(\mu)&=-L \Lb (\Xi_L)-\f12 \tr\chi\Lb (\Xi_L)-\f12 \tr\chib L (\Xi_L)-\f12 \tr\chi L (\Xi_\Lb)\\
&+\delta_{ij}(\Lb(e^{\varrho}\bN^j \curl \Omega^i)+\f12 e^\varrho\tr\chib\bN^i \curl \Omega^j)+\tr\chi \pi\c\pi+\bA \c \bd \ti \pi+\bA^3.
\end{split}
\end{equation}
The first term of the last line can be written as
\begin{equation}\label{6.23.6.19}
\delta_{ij}\Lb(e^\varrho\bN^j \curl \Omega^i)= \delta_{ij}(-L+2\bT)(e^\varrho\bN^j \curl \Omega^i).
\end{equation}

Meanwhile in view of (\ref{tran1}) and (\ref{6.22.17.19}), we have
\begin{align}
L \sD \sigma+ \tr\chi \sD \sigma&=\f12 \sD (\Xi_4)-2\chih_{AC} \sn_A \sn_C \sigma- \sn_A\tr\chi \sn_A \sigma-2 \delta^{AB}\bR_{CA4B}\c \sn_C\sigma
\nn\\
&-\left(\f12 k_{A\bN} \tr\chi-\chih_{AB} k_{B\bN}\right)\sn_A\sigma-\chi_{AB} \zb_A \sn_B\sigma\label{6.23.4.19}.
\end{align}
Applying (\ref{6.30.2.19}) to $\Xi_4$ gives  the null decomposition of $\Box_\bg (\Xi_4)$
\begin{equation}\label{6.23.5.19}
\Box_\bg (\Xi_4) =- L \Lb (\Xi_4)+\sD (\Xi_4)-\f12 \tr\chi \Lb (\Xi_4)-\f12 \tr\chib L (\Xi_4)+2 \zb^A \sn_A (\Xi_4)+k_{NN} \Lb (\Xi_4).
\end{equation}
Apart from the terms contributed by vorticity, we observe that the leading terms of $R(\mu)$ and  the right hand side of (\ref{6.23.4.19}) contain all the second order terms of (\ref{6.23.5.19}). As in \cite[Section 6]{Wangrough} we can derive a transport equation for the renormalized mass aspect function
\begin{equation}\label{6.23_mu}
\ckk \mu= 2\sD \sigma +\mu- \tr\chi k_{\bN\bN}+\f12\tr\chi \Xi_\Lb.
\end{equation}
To further cancel the higher order terms of vorticity  in $R(\mu)$,  we will derive the transport equation with the help of the decomposition of (\ref{6.23.6.19}) for the following quantity
 \begin{equation}\label{10.5.6.19}
 \ti \mu=\ckk\mu+\delta_{ij}e^\varrho\bN^i \curl \Omega^j.
 \end{equation}
 Similar to the calculation in \cite[(6.15) and Lemma 6.2]{Wangrough} for $\ckk\mu$,  we derive for $\ti\mu$ that
\begin{align}
L\ti \mu+\tr\chi \ti\mu&=\Box_\bg( \Xi_4)-2\left(2 \delta^{AB} \bR_{CA4B}+\f12 k_{A\bN} \tr\chi-\chih\c k+\chi\c \zb\right)\sn\sigma\label{mu2}\\
& +2(\zb-\ti\zeta) \sn \tr\chi -4 \chih\c\sn\ti\ze-2\chih_{AB} \big(k_{\bN\bN}\chih_{AB}+2\zeta_A\zeta_B+\bR_{A43B}\big)\nn\\
&+\tr\chi \left(\sl{\div}\zb+|\zb|^2+\f12\delta^{AB} \bR_{A34B}\right)+\f12 \left( \tr\chi\chih\c\chibh+\tr\chib|\chih|^2\right)\nn\\
&+2k_{\bN\bN} \left(|\chih|^2+k_{\bN\bN} \tr\chi+\bR_{44}\right)+ \bA \c(\bd \bpi+\bE)+\tr\chi \left(\sl{\div} \pi+\bE\right)\nn\\
&+2\delta_{ij}\bT(e^\varrho\curl \Omega^i\bN^j)+(\tr\chi+\f12 \tr\chib)\delta_{ij}e^\varrho\bN^i \curl\Omega^j,\label{6.23.7.19}
\end{align}
 where $\ti\zeta=\sn\sigma+\zeta$ and
$\bE= \bA\c \bA +\tr\chi \c \bA$.

Next we simplify the above equation in two steps.

{\bf Step 1.}
Recall  (\ref{4.25.4.19}) with $\fC=e^{-\varrho} \curl \Omega$. Also using the first equation in (\ref{4.23.1.19}), we  derive
\begin{equation}\label{6.23.8.19}
\bT(\curl\Omega^i)=\p v \p \Omega.
\end{equation}
By using Proposition \ref{6.7con} and (\ref{3.19.2}), we can compute  $\bd_\bT \bN=\frac{1}{4}(\bd_L L-\bd_\Lb \Lb+[\Lb,L])=\zeta_A e_A$. Therefore,
\begin{equation*}
\bT (\bN^i)=\bd_\bT \bN+\pi=\zeta_A e_A^i+\pi.
\end{equation*}
 Denote  the line of (\ref{6.23.7.19}) by $\sI$ for which we derive
\begin{align}
\sI&= e^\varrho\big(\bT \varrho \bN^i \curl \Omega^j \delta_{ij}+(\zeta+\pi) \curl \Omega+ \p v \c \p \Omega\c \bN^i\big)+(\tr\chi+\f12 \tr\chib)\delta_{ij}e^\varrho\bN^i \curl\Omega^j\nn\\
&=e^\varrho \bA \p \Omega \c X+\tir^{-1}\delta_{ij}e^\varrho \bN^i \curl \Omega^j.\label{6.23.19.19}
\end{align}

{\bf Step 2.}
We next compute the term $\Box_\bg (\Xi_L)$ by using (\ref{6.23.10.19}),
\begin{equation}\label{6.23.12.19}
\Box_\bg (\Xi_L)=2  \Box_\bg \left(\bT(\log c+\varrho)\right)+\Box_\bg\left(L(\log c-\varrho)\right).
\end{equation}
Note that there hold for scalar functions $\phi$ the following commutation formula,
\begin{equation}\label{6.23.11.19}
\begin{split}
[\Box_\bg, \bT] \phi&=-\bT \Tr k \bT \phi+[\Delta_g , \bT]\phi\\
&=-\bT \Tr k \bT \phi+\nab_g (k\nab_g \phi)+k\c \nab_g^2 \phi+\tensor{\bR}{^\mu_{i\bT i}}\bd_\mu \phi\\
&=\bg(\bp \ti \pi\c \bp \phi+\p^2 \phi\c \ti\pi).
\end{split}
\end{equation}
Combining the equation (\ref{4.10.2.19}) with the above commutation formula for $\phi=\log c+\varrho$ gives
\begin{equation}\label{6.23.13.19}
\Box_\bg \bT (\log c+\varrho)=f(\varrho)\bp \ti \pi \c \ti\pi+f(\varrho)\c{\ti\pi}^3,
\end{equation}
where $f(\varrho)$ represents smooth functions of $\varrho$, and we used the fact that $\ti\pi\c \bg$ still can be denoted by  $\ti\pi$.

Next we employ the null decomposition of the operator $\Box_\bg$ in (\ref{6.30.2.19}),  (\ref{3.19.2}) and (\ref{tran1}) to compute for scalar functions $\phi$
\begin{equation}\label{comL}
\begin{split}
[\Box_\bg, L] \phi&=L([L,\Lb]\phi)+[\sD, L] \phi+\f12 L \tr\chib L \phi+2\zb\c [\sn, L]\phi-2\sn_L \zb\c \sn \phi\\
&+(k_{\bN\bN}-\f12 \tr \chi)[\Lb, L]\phi-\sn_L (k_{\bN\bN}-\f12 \tr\chi)\Lb \phi \\
&=\sn_L(2(\zb-\zeta)\c\sn \phi-2 k_{\bN\bN}\bN \phi)+\tr\chi \sD \phi+2\chih\c \sn^2 \phi+\sl{\div}\chi_C \sn_C \phi\\
&-(\tr\chi\zb-\chi_{AC} \zb_A-\delta^{AB}\bR_{CA4B})\sn_C \phi+\f12 L \tr\chib L\phi+2\zb\c \chi\c \sn \phi-2\sn_L \zb\c \sn \phi\\
&+(k_{\bN\bN}-\f12 \tr\chi)\big(2(\zeta-\zb)\c \sn \phi+2 k_{\bN\bN} \bN \phi\big)-L(k_{\bN\bN}-\f12 \tr\chi)\Lb \phi,
\end{split}
\end{equation}
where we employed (\ref{3.19.2}) and (\ref{cmu_2}) to derive the last identity.

We now claim that
\begin{align}
&L \tr\chi=-2\tir^{-2}+ \D_* \pi+\bE,\label{10.10.2.19}\\
&L \tr\chib=\sl{\div}\pi+\sn \pi+\bE+2\tir^{-2},\label{6.23.14.19}\\
&\sn_L \zeta= \tr\chi \bA+\D_* \pi+\bE,\label{6.23.15.19}\\
&\sl{\div} \chi=\sn \tr\chi+\sn \pi+\bE\label{6.23.16.19},
\end{align}
where $\bE=\bA\c \bA+\tr\chi\c \bA$ and $\bA=\chih, z, \pi, \zeta$.

To see (\ref{10.10.2.19}), we recast (\ref{s1}) by using Lemma \ref{6.23.17.19} (i),
\begin{equation*}
L \tr\chi=-\f12(z+\pi+\frac{2}{\tir})\tr\chi+ \D_* \pi+\bE=-\tir^{-1}\tr\chi+\D_* \pi +\bE.
\end{equation*}
Note
\begin{align}\label{10.31.1.19}
\tir^{-1}\tr\chi=(\tir^{-1}-\f12 \tr\chi)(\tr\chi-\frac{2}{\tir})+\frac{2}{\tir^2}+\f12 \tr\chi(\tr\chi-\frac{2}{\tir})=2\tir^{-2}+\bE.
\end{align}
Combining the above two identities, (\ref{10.10.2.19}) follows as a consequence.

By using (\ref{10.31.2.19}), we can derive
\begin{align*}
\tr\chib\tr\chi=-\tr\chi^2-2\tr k \c \tr\chi=-(\tr\chi-\frac{2}{\tir}+\frac{2}{\tir})\tr\chi+\bE=-\frac{2}{\tir} \tr\chi+\bE.
\end{align*}
Combining the above identity with (\ref{10.31.1.19}), (\ref{6.23.14.19}) can be obtained in view of  (\ref{mub}) and the decomposition of curvature in Lemma \ref{6.23.17.19} (iv). (\ref{6.23.15.19}) can be obtained in view  of (\ref{tran2}) and Lemma \ref{6.23.17.19} (i). (\ref{6.23.16.19}) can be obtained in view of (\ref{dchi}) and Lemma \ref{6.23.17.19} (ii).

Now we set $\phi=\log c-\varrho$ in (\ref{comL}), then substitute the symbolic formulas (\ref{10.10.2.19})-(\ref{6.23.16.19}) to (\ref{comL}). Also by using Lemma \ref{6.23.17.19} (ii), we can derive
\begin{equation}\label{6.23.18.19}
[\Box_\bg, L](\log c-\varrho)=\tir^{-1}\sl{\div}\pi+\tir^{-2}\pi+(\sn \tr\chi, \D_*\pi, \bE)\c\bA.
\end{equation}

By using (\ref{4.10.2.19})
\begin{equation*}
L \Box_\bg(\log c-\varrho)= f(\varrho)\D_* \ti\pi\c \ti\pi+f(\varrho)\c (\ti \pi)^3.
\end{equation*}
Combining the above two identities yields
\begin{equation*}
\Box_\bg (L(\log c-\varrho))=f(\varrho)\D_* \ti\pi\c \ti\pi+f(\varrho)\c(\ti \pi)^3 +\tir^{-1}\sl{\div}\pi+\tir^{-2}\pi+(\sn \tr\chi, \D_*\pi, \bE)\c\bA.
\end{equation*}
Now combining the above identity with (\ref{6.23.13.19}) in view of (\ref{6.23.12.19}) gives
\begin{equation}\label{11.1.3.19}
\Box_\bg (\Xi_L)=f(\varrho)\bd \ti\pi\c \ti\pi+f(\varrho)\c(\ti \pi)^3 +\tir^{-1}\sl{\div}\pi +\tir^{-2}\pi+(\sn \tr\chi, \bd\pi, \bE)\c \bA.
\end{equation}
Since we will carry out $L^p$ estimate instead of the derivative estimates of the right hand side of the above equation, we can drop the smooth function $f(\varrho)$ since $|f(\varrho)|\les 1$.

\begin{proposition}\label{Trans_mu}
For $\ti\mu$ defined in (\ref{10.5.6.19}), there holds the transport equation
\begin{align}
(L\ti\mu +\tr\chi\ti\mu) -(\tir^{-1} (\sl{\div} \pi+\ep^{AB}\sn_A  \fw_B)  +\tir^{-2}\pi)&=\chih \c\sn \ti \zeta+\sn \sigma\c (\bE+\sn \pi+\sn \tr\chi)\nn\\
&+\bA\c(\sn \tr\chi, \bE, \bd \ti\pi)+e^\varrho\p \Omega\bA \label{mutrans},
\end{align}
where $\bA= \chih, \zeta, \pi, z$, $\bE= \bA\c \bA + \tr\chi \c \bA$,  and ``$\c X$" on the right hand side has been omitted since $X$ are all bounded frames \begin{footnote}{This means $|X^\mu|\les 1$ where $X=X^\mu\p_\mu$ is the decomposition relative to the Cartesian coordinates.} \end{footnote} and the equation  will not be further differentiated.
  For $\pi$ and $\ti \pi$ we refer to Section \ref{causal_reg} for their meaning.
\end{proposition}

\begin{proof}
Note due to Proposition \ref{6.7con}, $\bd \pi=\bd \ti \pi \c X+\bE$. Hence substituting  (\ref{11.1.3.19}) and  (\ref{6.23.19.19}) to (\ref{6.23.7.19}),
also using Lemma \ref{6.23.17.19}, we can conclude
\begin{align}
(L\ti\mu +\tr\chi\ti\mu) -(\tir^{-1} \sl{\div} \pi  +\tir^{-2}\pi)-\tir^{-1} \delta_{ij}\bN^i e^\varrho \curl \Omega^j& =\chih \c\sn \ti \zeta+\sn \sigma\c (\bE+\sn \pi+\sn \tr\chi)\nn\\
&+\bA\c(\sn \tr\chi, \bE, \bd \ti\pi)+e^\varrho\p \Omega\bA. \label{11.1.1.19}
\end{align}
Now we recast (\ref{9.12.1.19}) as
\begin{align}
e^\varrho \bN^i \curl \Omega_i&=\ep^{AB} \sn_A (e^\varrho\Omega_B)-e^\varrho \ep^{AB} \sn_A \varrho \Omega_B=\ep^{AB} \sn_A \fw_B-\ep^{AB} \sn_A \varrho \fw_B\nn\\
&=\ep^{AB} \sn_A \fw_B+\pi \c \pi,\label{11.1.2.19}
\end{align}
where we used the definition of $\fw$ for giving the symbolic form of the second term.
Using the above formula for $\tir^{-1} e^\varrho \bN^i \curl \Omega_i$ and  since we can regard $\tir^{-1} \pi\c \pi=\bE\c \bA$, (\ref{mutrans}) follows
as a consequence of the above calculation and (\ref{11.1.1.19}).
\end{proof}
 Let $\ti \zeta = \zeta +\sn \sigma$. We recall the Hodge operator $\D_1$
which sends an $S_{t,u}$-tangent tensor $F$ to $(\sl{\div} F, \sl{\curl} F)$. Thus we can write
$
\sn \tilde \zeta = \sn \D_1^{-1} (\sl{\div}\ti \zeta, \sl{\curl} \ti \zeta)
$
and use the Hodge system
\begin{align}
\sl{\div}\ti\zeta&=\f12(\ckk \mu-2|\zeta|^2-|\chih|^2-2 k_{AB} \chih_{AB})+\sl{\div} \pi_2+{\bE} \nn\\
&=\frac{1}{2} (\ti \mu-\delta_{ij}e^\varrho\bN^i \curl \Omega^j) +\sl{\div} \pi_2 +\bE, \label{hodge1} \\
\sl{\curl} \ti\zeta&=\f12 \ep^{AB} k_{AC} \chih_{CB}+\f12 \ep^{AB}\bR_{B43A}=\sl{\curl} \pi_3+\bE\label{hodge2},
\end{align}
which are directly derived from (\ref{dze}) and  (\ref{dcurl}) together with the curvature decomposition Lemma \ref{6.23.17.19} (iv), where $\pi_2$ and $\pi_3$ are 1-forms of type $\pi$.

By using Proposition \ref{Trans_mu} and the Hodge system (\ref{hodge1}) and (\ref{hodge2}), we will prove
\begin{proposition}\label{6.23.9.19}
For any $p$ satisfying $0\le 1-\frac{2}{p}<s'-2$ there hold
\begin{align}
&\|\sn \sigma\|_{L_u^2 L_t^2 L_\omega^\infty(\D^+)}+\|\tir\ckk \mu,\tir \ti \mu, \tir\sn \ti \zeta\|_{L_u^2 L_t^2 L_\omega^{p}(\D^+)}\les \la^{-4\ep_0},\label{ckmu2}\\
&\|\tir^{\frac{3}{2}} \ckk \mu, \tir^\frac{3}{2} \ti \mu \|_{L_u^2 L_t^\infty L_\omega^{p}(\D^+)}\les \la^{-4\ep_0}. \label{ckmu3}
\end{align}
\end{proposition}
\begin{proof}
We prove the estimate on $\|\sn \sigma\|_{L_u^2 L_t^2 L_\omega^\infty}$ by making the bootstrap assumption
\begin{equation}\label{ba.1}
\|\sn \sigma\|_{L_t^2 L_u^2 L_\omega^\infty(\D^+)}\le 1
\end{equation}
and improving it to
\begin{equation}\label{con.1}
\|\sn \sigma\|_{L_u^2 L_t^2 L_\omega^\infty(\D^+)}\les \la^{-4\ep_0}.
\end{equation}
The other estimates will be established during the course of the derivation.

We first consider the estimates for $\ckk \mu$. In view of (\ref{10.5.6.19}), using (\ref{lpcurl}), on any $S_{t,u}$  contained in $\D^+$ we bound
\begin{equation}\label{10.11.1.19}
\|\tir\ckk \mu\|_{L_\omega^p}\les \|\tir\ti \mu\|_{L_\omega^p}+\|\tir e^\varrho\delta_{ij} \bN^i\curl \Omega^j\|_{L_\omega^p}\les \|\tir \ti \mu\|_{L_\omega^p}+\la^{-1}.
\end{equation}
The estimates of $\ckk\mu$ in (\ref{ckmu2}) and (\ref{ckmu3}) will follow from  integrating (\ref{10.11.1.19})  and the estimates of $\ti\mu$ in the corresponding norms.
Note that $\ti \mu$ verifies the transport equation (\ref{mutrans}). By using Lemma \ref{inii} (i), $\lim_{t\rightarrow u}|\tir^2 \ti \mu|=0$.
 Thus,  by Lemma \ref{tsp2} and (\ref{comp_3_27}), we have
\begin{equation}\label{gronmu}
|\tir^2\ti\mu|\les\left|\int_u^t  {\ti r}^2 |L\ti\mu+\tr\chi\ti \mu| d t'\right|.
\end{equation}
We first compute
\begin{align*}
\left\|\frac{1}{\tir}\int_u^t \tir^2 \sn \sigma(\sn \pi +\sn\tr\chi+\bE)\right\|_{L_u^2 L_t^\infty L_\omega^{p}(\D^+)}
\les \|\sn \sigma\|_{L_u^2 L_t^2 L_x^\infty(\D^+)} \|\tir \sn \pi, \tir\sn \tr\chi, \tir \bE\|_{L_u^\infty L_t^2 L_\omega^{p}(\D^+)}.
\end{align*}
Recall that $\bE = \bA\c \bA + \tr\chi \c \bA$ and $\tr\chi=\widetilde{\tr\chi}-\Xi_4$. By using (\ref{comp2}), (\ref{pric1}) and (\ref{pric3}),
\begin{equation*}
\|\tir \bE\|_{L_t^2 L_\omega^p(C_u\cap \D^+)}\le \|\tir \widetilde{\tr\chi} \bA\|_{L_t^2L_\omega^p(C_u\cap \D^+)}+\|\tir \bA\c \bA\|_{L_t^2 L_\omega^p(C_u\cap \D^+)}\les \la^{-\f12}.
\end{equation*}
Using  $\tr\chi=z-\Xi_4+\frac{2}{\tir}$ again,  also using the  estimates in (\ref{sna}) and (\ref{flux_3}), we have
\begin{equation}\label{10.11.4.19}
\|\tir\sn \tr\chi, \tir \sn \pi\|_{L_t^2 L_\omega^p(C_u)}\les \la^{-\f12}.
\end{equation}
Combining the above estimates with the estimate of $\bE$ and (\ref{ba.1}) implies
\begin{equation*}
\left\|\frac{1}{\tir}\int_u^t \tir^2 \sn \sigma(\sn \pi +\sn\tr\chi+\bE)\right\|_{L_u^2 L_t^\infty L_\omega^{p}(\D^+)}\les \la^{-\f12}\|\sn \sigma\|_{L_u^2 L_t^2 L_\omega^\infty(\D^+)}\les \la^{-\f12}.
\end{equation*}
Moreover, by the following consequence of  (\ref{ric1}),
\begin{equation}\label{10.11.2.19}
\|\bA\|_{L_t^2 L_x^\infty(\D^+)}\les \la^{-\f12-4\ep_0},
\end{equation}
also by using (\ref{comp2}) and (\ref{pric3}), it is easily seen that
\begin{equation}\label{err_4_4}
\|\tir \bE\|_{L_t^2 L_u^\infty L_\omega^p(\D^+)}\les \la^{-\f12-4\ep_0}, \quad\|\tir^\frac{3}{2} \bE\|_{L_u^2 L_t^\infty L_\omega^p(\D^+)}\les \la^{-4\ep_0}.
\end{equation}
It follows from (\ref{err_4_4}), (\ref{6.18.3.19}), the first estimate in (\ref{10.11.4.19}) and  (\ref{lpcurl}) that
\begin{equation}\label{10.11.5.19}
 \|\tir (\sn \tr\chi,\bd \ti\pi, e^\varrho\p\Omega, \bE)\|_{L_u^2 L_t^2 L_\omega^{p}(\D^+)}\les \la^{-4\ep_0}.
\end{equation}
 Consequently, by using  (\ref{10.11.3.19})
\begin{align*}
&\left\|\frac{1}{\tir}\int_u^t \tir^2 \bA\c(\sn \tr\chi, \bd \ti\pi, e^\varrho \p \Omega, \bE)\right\|_{L_u^2 L_t^\infty L_\omega^{p}(\D^+)}\\
&\quad\quad \quad\les\|\bA\|_{L_u^\infty L_t^2 L_\omega^\infty(\D^+)} \|\tir (\sn \tr\chi,\bd \ti\pi, e^\varrho\p\Omega, \bE)\|_{L_u^2 L_t^2 L_\omega^{p}(\D^+)}\\
 &\quad\quad \quad\les \la^{-\f12-4\ep_0}.
\end{align*}
For the term $\chih\c \sn\ti\zeta$, we may use (\ref{hodge1}), (\ref{hodge2}), Lemma \ref{hdgm1}, (\ref{10.11.5.19}) and  (\ref{10.11.3.19})
to derive that
\begin{align*}
\left\|\frac{1}{\tir} \int_u^t \tir^2 \chih\c \sn \ti\zeta dt'\right\|_{L_u^2 L_t^\infty L_\omega^p(\D^+)}
&= \left\|\frac{1}{\tir}\int_u^t \tir'^2 \chih\c \sn \D_1^{-1} (\sl{\div} \ti \zeta, \sl{\curl} \ti \zeta)d\tt\right\|_{L_u^2 L_t^\infty L_\omega^{p}(\D^+)}\\
&\les \|\chih\|_{L_u^\infty L_t^2 L_\omega^\infty(\D^+)} \|\tir (\ti\mu,\sn \pi, \bE, e^\varrho\curl \Omega)\|_{L_u^2 L_t^2 L_\omega^p(\D^+)}\\
&\les (\la^{-4\ep_0}+\|\tir\ti\mu\|_{L_u^2 L_t^2 L_\omega^p(\D^+)})\la^{-\f12}.
\end{align*}
Finally, for the second part on the left hand side of (\ref{mutrans}), we use (\ref{hlm}), (\ref{flux_3}), (\ref{pric3}) and (\ref{lpcurl}) to derive that
\begin{equation*}
\left\|\tir^{-1} \int_u^t (\tir'\sl{\div} \pi, \pi, \tir' \ep^{AB}\sn_A \fw_B )\right\|_{L_t^2 L_\omega^p(C_u\cap \D^+)}\les \|\tir \sn \pi, \pi, \tir\pi\c \pi,  \tir \p \Omega\|_{L_t^2 L_\omega^p(C_u\cap \D^+)} \les \la^{-\f12},
\end{equation*}
where we used (\ref{11.1.2.19}) to obtain $|\ep^{AB}\sn_A \fw_B|\les |\p \Omega|+|\pi\c \pi|$.

Now we may divide (\ref{gronmu}) by $\tir$ and use the above estimates to derive that
\begin{equation*}
\|\tir\ti\mu\|_{L_u^2 L_t^2 L_\omega^p(\D^+)}\les \la^{-4\ep_0}\|\tir\ti\mu\|_{L_u^2 L_t^2 L_\omega^p(\D^+)}+\la^{-4\ep_0},
\end{equation*}
which gives
\begin{equation}\label{mu_44}
\|\tir \ti\mu\|_{L_u^2 L_t^2 L_\omega^p(\D^+)}\les \la^{-4\ep_0}.
\end{equation}
We may divide (\ref{gronmu}) by $\tir^{\f12}$ and employ the similar argument as above to derive
\begin{equation*}
\|\tir^\frac{3}{2}\ti \mu\|_{L_u^2 L_t^\infty L_\omega^p(\D^+)}\les \la^{-\f12}.
\end{equation*}
We thus have obtained the estimates for $\ti \mu$ in (\ref{ckmu2}) and (\ref{ckmu3}).
By (\ref{10.11.1.19}), the estimates of $\ckk \mu$ in (\ref{ckmu2}) and (\ref{ckmu3}) are also proved.

By making use of (\ref{hodge1}), (\ref{hodge2}), Lemma \ref{hdgm1}, (\ref{mu_44}) and (\ref{10.11.5.19}), we can obtain that
\begin{align*}
\|\tir\sn \ti \zeta\|_{L_u^2 L_t^2 L_\omega^p(\D^+)}&\les \|\tir \ti\mu\|_{L_u^2 L_t^2 L_\omega^p(\D^+)}
+\|\tir (\bE+\sn\pi+e^\varrho \p \Omega)\|_{L_u^2 L_t^2 L_\omega^p(\D^+)}\les \la^{-4\ep_0}.
\end{align*}
Using the above estimate, in view of $\sn \sigma=\ti\zeta-\zeta$, (\ref{con.1}) can be obtained by using (\ref{sobinf}), the first estimate in (\ref{7.03.5.19}) and the last estimate in the line of (\ref{6.22.18.19}).
The proof is thus complete.
 \end{proof}

\begin{proposition}\label{dcmpsig}
$\sn \sigma$ can be decomposed as
$$
\sn \sigma=\bA+\bA^\dag+\mu^\dag,
$$
where $\bA^\dag$ and $\mu^\dag$ are 1-forms satisfying the estimates
$$
\|\bA^\dag\|_{L_t^2 L_x^\infty(\D^+)}\les \la^{-\f12-4\ep_0} \quad \mbox{and} \quad \|\mu^\dag\|_{L_u^2 L^\infty(\D^+)}\les \la^{-\f12-4\ep_0}.
$$
\end{proposition}
\begin{proof}
Now using $\ti \mu$, we introduce an $S_{t,u}$-tangent 1-tensor field $\sl \mu$. On each $S_{t,u}$, $ \sl\mu$ is defined by
\begin{equation}\label{hdg_1}
\sl{\div} \slashed{\mu}=\f12(\ti \mu-\overline{\ti\mu}),  \qquad \sl{\curl} \slashed{\mu}=0,
\end{equation}
where $\overline{f}:=\frac{1}{|S_{t,u}|} \int_{S_{t,u}} f d\mu_\ga$.

By using Lemma \ref{hdgm1}, (\ref{hdg_1}), (\ref{ckmu2}) and (\ref{sobinf}), we derive with $0\le 1-\frac{2}{p}< s'-2$
\begin{equation}\label{smu2}
\|\tir\sn\smu, \smu\|_{L_t^2 L_u^2 L_\omega^p} + \|\smu\|_{L_t^2 L_u^2 L_x^\infty}\les \la^{-4\ep_0}.
\end{equation}
In order to represent $\sn\sigma$, we first derive in view of (\ref{hodge1}) and (\ref{hodge2}),
\begin{equation*}
\zeta +\sn \sigma -\sl \mu = \D_1^{-1} \left(\sl{\div} \ti \zeta -\f12 (\ti \mu-\overline{\ti \mu}), \sl{\curl} \ti\zeta\right).
\end{equation*}
We further decompose $\sl\mu$ into two parts. For this purpose, we will derive a transport equation of $\sl \mu$.

By direct calculation, we have
\begin{equation*}
L \bar f=\overline{L f}+\overline{\tr\chi f-\overline{\tr\chi}{\overline f}}.
\end{equation*}
Let $G$ denote the right hand side of the identity in Proposition \ref{Trans_mu} and note that $\overline{\tir^{-1} \sn_A \pi_A}=0$ and $\overline{\tir^{-1}\ep^{AB} \sn_A \fw_B}=0$.
    We thus can obtain
    \begin{align*}
&L(\ti\mu-\overline{\ti \mu})+\tr \chi (\ti \mu -\overline{\ti\mu})\\
&= -(\tr\chi-\overline{\tr\chi}) \overline{\ti \mu} + G-\overline{G} + \tir^{-1}(\sl{\div} \pi+\ep^{AB} \sn_A \fw_B)+ \tir^{-2} (\pi -\overline{\pi}).
\end{align*}
By using (\ref{cmu2}), similar to \cite[Proposition 6.4]{Wangrough}, we obtain
\begin{align}\label{g1_1}
\begin{split}
\sl{\div} (L\smu+\f12 \tr\chi\smu)&=G_1+\f12\left(\tir^{-1} (\sl{\div} \pi+\ep^{AB}\sn_A\fw_B)+\tir^{-2}(\pi-\overline\pi)\right),\\
\sl{\curl} (L\smu+\f12 \tr\chi \smu)&=G_2,
\end{split}
\end{align}
where
\begin{align*}
G_1&=\f12 \sn\tr\chi \c\smu+\chih \c \sn \smu+\f12( G-\overline{G}) +\smu\c( \sn \pi + \bE)-\f12(\tr\chi-\overline{\tr\chi})\overline{\ti \mu},\\
G_2&=\hat\chi \c \sn \smu+\f12\sn \tr\chi\c\smu+\smu\c (\sn \pi + \bE).
\end{align*}
Consequently
\begin{equation}\label{3_21_1}
L\smu+\f12 \tr\chi \smu =\f12\D_1^{-1} \left(\tir^{-1}(\sl{\div} \pi+\ep^{AB}\sn_A \fw_B)+\tir^{-2} (\pi -\overline\pi),0\right)+\D_1^{-1}(G_1, G_2),
\end{equation}
which together with Lemma \ref{tsp2} implies that
\begin{align*}
\smu=v_t^{-\f12}\int_u^t v_\tt^\f12 \D_1^{-1}(G_1, G_2) d \tt
+ v_t^{-\f12}\int_u^t v_\tt^\f12 \D_1^{-1} \left(\tir^{-1}(\sl{\div} \pi+\ep^{AB} \sn_A \fw_B)+\tir^{-2}(\pi-\overline{\pi}),0\right) d\tt.
\end{align*}
Therefore $\sn \sigma = \bA + \mu^\dag + \bA^\dag$ with $\bA=-\zeta$ and
\begin{align*}
\mu^\dag & = v_t^{-\f12}\int_u^t v_\tt^\f12 \D_1^{-1}(G_1, G_2) d \tt, \displaybreak[0]\\
\bA^\dag & =  v_t^{-\f12}\int_u^t v_\tt^\f12 \D_1^{-1} \left(\tir^{-1}(\sl{\div} \pi+\ep^{AB}\sn_A \fw_B)+\tir^{-2}(\pi-\overline{\pi}),0\right) d\tt\\
 &+ \D_1^{-1} \left(\sl{\div} \ti \zeta -\f12 (\ti\mu-\overline{\ti \mu}), \sl{\curl} \ti\zeta\right).
\end{align*}
Note the first line on the right hand side of  $\bA^\dag$ term contains vorticity.  The control of the last term in $\bA^\dagger$ and the term  $G_1$ also involves the estimate of $\p \Omega$. Now we show
\begin{equation}\label{10.11.6.19}
\|\tir(G_1, G_2)\|_{L_u^2 L_t^1 L_\omega^p(\D^+)}\les \la^{-\f12-4\ep_0}.
\end{equation}
To prove the above estimate, we first treat  the last term in $G_1$ symbolically by
\begin{equation}\label{10.11.8.19}
  (\tr\chi-\overline{\tr\chi})\overline{\ti \mu}=z\c\ti\mu,
\end{equation}
where we have ignored the operator of taking average.

Hence, we schematically recast the terms of $G_1$ and $G_2$ as
\begin{align*}
G_1, G_2&=(\smu, \sn\sigma)\c (\sn\widetilde{\tr\chi}, \sn \pi, \bE)+\chih\c (\sn \ti\zeta, \sn \smu)
+\bA\c(\sn \widetilde{\tr\chi}, \bE, \bd \ti\pi, e^\varrho\p \Omega^j, \ti \mu ).
\end{align*}
By using (\ref{smu2}), (\ref{10.4.2.19}), (\ref{7.03.5.19}), (\ref{10.11.5.19}), (\ref{10.5.4.19}) and  (\ref{ckmu2}), we derive
\begin{align*}
\|\tir( G_1, G_2)\|_{L_u^2 L_t^1 L_\omega^p(\D^+)}&\les \|\smu, \sn \sigma\|_{L_u^2 L_t^2 L_\omega^\infty}\|\tir(\sn \widetilde{\tr\chi}, \sn \pi, \bE)\|_{L_u^\infty L_t^2 L_\omega^p}\\
& \quad \, + \|\bA\|_{L_u^\infty L_t^2 L_\omega^\infty} \|\tir (\sn \zeta, \sn\sl \mu, \sn \widetilde{\tr\chi}, \bE, \bd \ti\pi,\ti\mu, e^\varrho\p \Omega^j )\|_{L_u^2L_t^2 L_\omega^p}\\
& \les \la^{-\f12-4\ep_0},
\end{align*}
as desired in (\ref{10.11.6.19}). As its consequence, in view of (\ref{comp_3_27}), we can obtain $\|\mu^\dagger \|_{L_u^2 L^\infty}\les \la^{-\f12-4\ep_0}$.

Note that by using (\ref{dze}) and Lemma \ref{6.23.17.19} (iv), $\bar \mu=\bar \bE$.
Hence, by the definitions of $\ckk \mu$ and $\ti \mu $ in (\ref{6.23_mu}) and (\ref{10.5.6.19}) respectively,
\begin{equation*}
\bar{\ti \mu}=\overline{\bE}+\overline{\delta_{ij}e^\varrho \bN^i \curl \Omega^j}.
\end{equation*}
Thus by using (\ref{cz2}), (\ref{hodge1}) and  (\ref{hodge2})
\begin{align*}
\|\D_1^{-1}& \left(\sl{\div} \ti \zeta -\f12 (\ti\mu-\overline{\ti \mu}), \sl{\curl} \ti\zeta\right)\|_{L_t^2 L_u^\infty L_\omega^\infty}\\
&\les \|\tir (\bE-\overline{\bE})\|_{L_t^2 L_u^\infty L_\omega^p}+\|\tir\delta_{ij}e^\varrho \bN^i \curl \Omega^j\|_{L_t^2 L_u^\infty L_\omega^p}+\|\ell^c P_{\ell} \ti \pi\|_{L_t^2 l_\ell^2 L_x^\infty}
+\|\ti \pi \|_{L_t^2 L_x^\infty}\\
& \les \|\tir \bE\|_{L_t^2 L_u^\infty L_\omega^p}+\|\ell^c P_{\ell} \ti \pi\|_{L_t^2 l_\ell^2 L_x^\infty}
+\|\ti \pi \|_{L_t^2 L_x^\infty}+\|\tir\p \Omega\|_{L_t^2 L_u^\infty L_\omega^p} \\
&\les \la^{-\f12 -4\ep_0},
\end{align*}
where we used (\ref{err_4_4}), (\ref{pi.2})  and (\ref{lpcurl}) to derive the last inequality.

Finally, by using (\ref{cz2})  with $0<c<s'-2$, (\ref{pi.2}) and (\ref{comp_3_27})
\begin{align*}
\|v_t^{-\f12}\int_u^t &v_\tt^\f12 \D_1^{-1} (\tir^{-1}(\sl{\div} \pi+\ep^{AB}\sn_A \fw_B)+\tir^{-2}(\pi-\overline{\pi}),0) d\tt\|_{L_t^2 L_x^\infty}\\
&\les \|\ell^c P_\ell \ti\pi\|_{L_t^2 l_\ell^2 L_x^\infty}+\|\ti\pi\|_{L_t^2 L_x^\infty}\les \la^{-\f12-4\ep_0}
\end{align*}
where we regarded $\fw=\p v=\ti \pi$. Therefore the proof of  Proposition \ref{dcmpsig} is complete.
\end{proof}



\section{Appendix}\label{app}
In this section, we rely on the trichotomy (\ref{9.08.3.19}) of the Littlewood-Paley projections to derive  commutator estimates and product estimates, which are the basic analytic tools to treat the analysis in the fractional Sobolev spaces.
\begin{lemma}
(1) For smooth scalar functions $F$ and $G$,  $1\le p, q, r\le \infty$ satisfying $\frac{1}{p}+\frac{1}{r}=\frac{1}{q}$,
\begin{equation}\label{5.12.1.19}
\|[P_\mu, F] G\|_{L_x^q}\les \mu^{-1} \|\p F\|_{L_x^p}\|G\|_{L_x^r}
\end{equation}
(2)
For smooth scalar functions $F$ and $G$, with $0<\a<1$,
\begin{align}
&\|\mu^\a[P_\mu, F]\p G\|_{l_\mu^2 L_x^2}\les \|\p F\|_{L_x^\infty} \|G\|_{H^\a_x},\label{4.13.3.19}\\
&\|\mu^\a[P_\mu,  F]\p G\|_{l_\mu^2 L_x^2}\les \| F\|_{L^\infty_x}\|\p G\|_{H^\a_x}+\|F\|_{H^{1+\a}_x} \|G\|_{L_x^\infty},\label{4.13.2.19}\\
&\|\mu^\a\p[P_\mu, F] \p G\|_{l_\mu^2 L_x^2}\les \|\p F\|_{L_x^\infty}\|\p G\|_{H^\a_x}+\|\p F\|_{H^{\a+1}_x}\|G\|_{L^\infty_x}.\label{4.13.4.19}
\end{align}
\end{lemma}
\begin{proof}
We recall from \cite[(6.195)]{WangCMCSH} that for smooth scalar functions $f$ and $W$,
\begin{equation}\label{4.13.5.19}
[P_\mu, f]W=[P_\mu, f]W_{\le \mu}+\sum_{\la>\mu}P_\mu(f_\la W_\la).
\end{equation}
We consider (\ref{4.13.3.19}) by applying (\ref{4.13.5.19}) to $(f,W)=(F, \p G)$.
\begin{equation*}
\|\mu^\a[P_\mu, F] (\p G)_{\le \mu}\|_{L_x^2}\les  \| \p F\|_{L_x^\infty}\sum_{l\le \mu} (\frac{l}{\mu})^{1-\a} \|l^\a G_l\|_{L_x^2}.
\end{equation*}
Taking $l_\mu^2$ norm gives
\begin{equation*}
\|\mu^\a[P_\mu, F](\p G)_{\le \mu}\|_{l_\mu^2 L_x^2}\les \|\p F\|_{L_x^\infty}\|G\|_{H^\a_x}.
\end{equation*}
For the high-high interaction term  in (\ref{4.13.5.19}), by using the finite band property
\begin{equation*}
\|\mu^\a \sum_{\la>\mu}P_\mu(F_\la (\p G)_\la)\|_{L_x^2}\les \sum_{\la>\mu}(\frac{\mu}{\la})^\a\|(\p F)_\la\|_{L_x^\infty}\|\la^\a G_\la\|_{L_x^2}.
\end{equation*}
Taking $l_\mu^2$ norm implies
\begin{equation*}
\|\mu^\a \sum_{\la>\mu}P_\mu(F_\la (\p G)_\la) \|_{l_\mu^2 L_x^2}\les \| \p F\|_{L_x^\infty} \|G\|_{H^\a_x}.
\end{equation*}
(\ref{4.13.3.19}) follows by combining the estimates for both parts.

For (\ref{4.13.2.19}), by  applying (\ref{4.13.5.19}) to $(f, W)=(F, G)$, we first treat the term
by using the trichotomy and orthogonality property of the Littlewood-Paley projection,
\begin{equation*}
 [P_\mu, F](\p G)_{\le \mu}=P_\mu(F_\mu (\p G)_{\le \mu})-F P_\mu (\p G).
\end{equation*}
For this term, we compute
\begin{equation*}
\|\mu^\a [P_\mu, F](\p G)_{\le\mu}\|_{l_\mu^2 L_x^2}\les \| G\|_{L^\infty}\|F\|_{H^{\a+1}_x}+\|F\|_{L^\infty_x} \|G\|_{H^{\a+1}_x}.
\end{equation*}
The high-high interaction term can be controlled by
\begin{equation*}
\mu^\a\|\sum_{\la>\mu}P_\mu(F_\la (\p G)_\la)\|_{L_x^2}\les \sum_{\la>\mu}(\frac{\mu}{\la})^\a\|F_\la\|_{L^\infty} \la^{1+\a}\|G_\la\|_{L_x^2}.
\end{equation*}
Thus
\begin{equation*}
\|\mu^\a \sum_{\la>\mu}P_\mu(F_\la (\p G)_\la)\|_{l_\mu^2 L_x^2}\les \|F\|_{L^\infty} \|G\|_{H^{\a+1}_x}.
\end{equation*}
(\ref{4.13.2.19}) follows by  combining the two terms.

For (\ref{4.13.4.19}), we first note that
\begin{equation*}
\p [P_\mu, F] \p G=[P_\mu, \p F]\p G+[P_\mu, F]\p^2 G.
\end{equation*}
Applying (\ref{4.13.3.19}) with $(F,G)$ replaced by $(F, \p G)$ to the second term, (\ref{4.13.2.19}) with $(F,G)$ replaced by $(\p F, G)$ for the first term, (\ref{4.13.4.19}) follows immediately.
\end{proof}

Next we give the first set of the product estimates.
\begin{lemma}\label{prod}
For $\a>0$, there hold for scalar functions $F$ and $G$ that
\begin{align}
\|F\c  G\|_{\dot{H}^\a_x}&\les \|F\|_{L_x^\infty}\|G\|_{\dot{H}^\a_x}+\|F\|_{\dot{H}^\a_x}\|G\|_{L_x^\infty}  \label{4.14.11.19}\\
\|F \c \p G\|_{\dot{H}^\a_x}&\les\|F\|_{\dot{H}^{1+\a}_x}\|G\|_{L_x^\infty}+\|F\|_{L_x^\infty}\|\p G\|_{H^\a_x}\label{5.05.2.19}\\
\|\p(F\c G)\|_{\dot{H}^\a_x}&\les \|\p F\|_{H^\a_x}\|G\|_{L_x^\infty}+\|F\|_{L_x^\infty}\|\p G\|_{H^\a_x}\label{9.20.5.19}\\
\|F\c \p^2 G\|_{\dot{H}^\a_x}&\les \|F\|_{\dot{H}^{2+\a}_x}\|G\|_{L_x^\infty}+\|F\|_{L_x^\infty}\|\p^2 G\|_{H^\a_x}\label{5.05.5.19}\\
\|\p(F\c G)\|_{\dot{H}^\a_x}&\les \|F\|_{\dot{H}^{1+\a}_x}\|G\|_{L_x^\infty}+\|F\|_{H^1_x}\|G\|_{\dot{H}^{\frac{3}{2}+\a}_x}\label{5.05.6.19}\\
 \|F\c \p G\|_{\dot{H}^\a_x}&\les \|F\|_{\dot{H}^{1+\a}_x}\|G\|_{L_x^\infty}+\|F\|_{H^1_x}\|G\|_{\dot{H}^{\frac{3}{2}+\a}_x}\label{5.05.7.19}\\
 \|F \c G\|_{\dot{H}^\a_x}&\les \|F\|_{\dot{H}^{\f12+\a}_x}\|G\|_{H^1_x}+\|F\|_{L_x^\infty}\|G\|_{\dot{H}^\a_x},\label{9.22.7.19}
\end{align}
where $\dot{H}^{\a}$ denotes the Sobolev norm of $H^\a$ with the  $L^2$ norm excluded.
\end{lemma}
\begin{proof}
Now we prove (\ref{4.14.11.19}). By trichotomy,
\begin{align*}
\mu^\a P_\mu(F\c G)&=\mu^\a  P_\mu[F\c G]_{HL}+\mu^\a P_\mu[F\c G]_{LH}+\mu^\a P_\mu[F\c G]_{HH}\\
&=I_\mu+J_\mu+K_\mu.
\end{align*}
For the three terms, by using H\"{o}lder's inequality, we can compute
\begin{align*}
\|I_\mu\|_{l_\mu^2 L_x^2}&\les \|F\|_{\dot{H}^a_x}\|G\|_{L_x^\infty}; \quad \|J_\mu\|_{l_\mu^2 L_x^2}\les \|F\|_{L_x^\infty}\|G\|_{\dot{H}^\a_x}\\
\|K_\mu\|_{l_\mu^2 L_x^2}&\les\|\sum_{\la>\mu}(\frac{\mu}{\la})^{\a} \|F_\la \|_{L_x^\infty}\|\la^\a G_\la\|_{L_x^2} \|_{l_\mu^2}\les \|F\|_{L_x^\infty}\|G\|_{\dot{H}^\a_x}.
\end{align*}
This gives (\ref{4.14.11.19}).

Next we prove (\ref{5.05.2.19}). Again by using trichotomy,
\begin{align*}
\mu^\a P_\mu(F\c \p G)&=\mu^\a P_\mu[F\c \p G]_{HL}+\mu^\a P_\mu[F \c \p G]_{L H}+\mu^\a P_\mu[F\c \p G]_{HH}\\
&=I_\mu+J_\mu+K_\mu.
\end{align*}
By the finite band property,
\begin{align}
\|I_\mu\|_{L_x^2}&\les \|\mu^{1+\a}F_\mu\|_{L_x^2}\sum_{l\le \mu}\frac{l}{\mu}\|G_l\|_{L_x^\infty}\nn\\
\|J_\mu\|_{L_x^2}&\les \|\mu^\a \p G_\mu\|_{L_x^2}\|F_{\le \mu}\|_{L_x^\infty}\les \|\mu^\a \p G_\mu\|_{L_x^2}\|F\|_{L_x^\infty}\nn\\
\|K_\mu\|_{L_x^2}&\les \sum_{\la>\mu}(\frac{\mu}{\la})^\a\|\la^{\a}\p G_\la\|_{L_x^2}\|F_\la\|_{L_x^\infty}\nn.
\end{align}
Summing the above terms in terms of $l_\mu^2$, we can conclude (\ref{5.05.2.19}). (\ref{5.05.5.19}) can be similarly proved.



Next we consider (\ref{5.05.6.19}) by using trichotomy. For simplicity, we set $\I_\mu=\mu^{1+\a}\|P_\mu(F\c G)\|_{L_x^2}$.
 \begin{align*}
 \I_\mu&\les \mu^{1+\a}\left(\|P_\mu[F\c G]_{HL}\|_{L_x^2}+\|P_\mu[F\c G]_{LH}\|_{L_x^2}+\|P_\mu[F\c G]_{HH}\|_{L_x^2}\right).
 \end{align*}
 By using the  finite band property and Bernstein inequality, we obtain
 \begin{align*}
 \mu^{1+\a}\|P_\mu[F\c G]_{HL}\|_{L_x^2}&\les \mu^{1+\a}\| P_\mu F\|_{L_x^2}\|G_{\le \mu}\|_{L_x^\infty}\\
\mu^{1+\a} \|P_\mu[F\c G]_{L H}\|_{L_x^2}&\les\|P_{\le \mu} F\|_{L^{6}_x}\|P_\mu G\|_{L_x^3}\mu^{1+\a}\\
 &\les \|F\|_{H^1_x}\|P_\mu G\|_{\dot{H}^{\frac{3}{2}+\a}_x}\\
 \mu^{1+\a}\|P_\mu[F \c G]_{HH}\|_{L_x^2}&\les\sum_{\la >\mu}\|P_\la  F\|_{L_x^6}(\frac{\mu}{\la})^{1+\a}\|\la^{\frac{3}{2}+\a} P_\la G\|_{L_x^2},
 \end{align*}
which yields in view of (\ref{4.12.2.19}) that
\begin{equation*}
\|\I_\mu\|_{l_\mu^2}\les \|F\|_{H^{1+\a}_x} \|G\|_{L^\infty_x}+\|F\|_{H^1_x}\|G\|_{\dot{H}^{\frac{3}{2}+\a}_x}
\end{equation*}
as desired.

Next we prove (\ref{5.05.7.19}). Let $\J_\mu= \mu^\a\|P_\mu(F \c \p G)\|_{L_x^2}$. In view of the trichotomy in  (\ref{9.08.3.19}), we estimate
 by using the finite band property and Bernstein inequality that
\begin{align*}
\mu^\a\|P_\mu[F \c \p G]_{HL}\|_{L_x^2}&\les\sum_{\la<\mu}\frac{\la}{\mu}\|\mu^{1+\a} P_\mu F\|_{L_x^2}\|P_\la G\|_{L_x^\infty},\\
\mu^\a\|P_\mu[F \c \p G]_{LH}\|_{L_x^2}&\les \|P_{\le \mu}F\|_{L_x^6}\mu^{\a}\|P_\mu \p G\|_{L_x^3}\les \|F\|_{L_x^6} \|\mu^{\a+\f12} P_\mu \p G\|_{L_x^2},\\
\mu^\a\|P_\mu[F \c\p G]_{HH}\|_{L_x^2}&\les \mu^{\a}\sum_{\la>\mu} \|P_\la F\|_{L_x^6}\|P_\la \p G\|_{L_x^3}\\
&\les \sum_{\la>\mu} (\frac{\mu}{\la})^\a \|P_\la F\|_{H^1_x}\|\la^{\f12+\a}P_\la \p G\|_{L_x^2}.
\end{align*}
Taking $l_\mu^2$ norm for the above three terms implies
\begin{align*}
\|\J_\mu\|_{l_\mu^2}\les \|F\|_{H^{1+\a}_x}\|G\|_{L_x^\infty}+\|F\|_{H^1_x}\|\p G\|_{H^{\frac{1}{2}+\a}_x}
\end{align*}
as desired.

At last we consider (\ref{9.22.7.19}). We estimate the terms in (\ref{9.08.3.19}) as follows
\begin{align*}
\|\mu^\a P_\mu[F \c G]_{HL}\|_{L_x^2}&\les \mu^\a \|F_\mu\|_{L_x^2}\sum_{l\le \mu}l^\frac{3}{2}\|P_l G\|_{L_x^2}\\
&\les \mu^{\a+\frac{1}{2}}\|F_\mu\|_{L_x^2}\sum_{l\le \mu} (\frac{l}{\mu})^\f12 \|l P_l G\|_{L_x^2},
\end{align*}
\begin{align*}
\|\mu^\a P_\mu[F\c G]_{LH}\|_{L_x^2}&\les \|F_{\le \mu}\|_{L_x^\infty}\mu^\a\|G_\mu\|_{L_x^2},
\end{align*}
\begin{align*}
\|\mu^\a P_\mu[F\c G]_{HH}\|_{L_x^2}&\les \|F\|_{L_x^\infty}\sum_{\la>\mu}(\frac{\mu}{\la})^\a \|\la^\a G_\la \|_{L_x^2}.
\end{align*}
Summing up the three inequalities implies
\begin{equation*}
\|\mu^\a P_\mu(F G)\|_{l_\mu^2 L_x^2}\les \|F\|_{\dot{H}^{\f12+\a}}\|G\|_{H^1_x}+\|F\|_{L_x^\infty}\|G\|_{\dot{H}^\a_x}.
\end{equation*}
The proof of (\ref{9.22.7.19}) is completed.
\end{proof}
\begin{lemma}
Let $0<\a<1$ be fixed.
\begin{align}
&\|\La^\a(F\c G)\|_{L_x^2}\les\|F\|_{B_{\infty, 2, x}^\a}\|G\|_{L_x^2}+\|F\|_{L_x^\infty}\|G\|_{\dot{H}^\a_x},\label{9.08.1.19}\\
&\|\La^\a(F\c G)\|_{L_x^2}\les \|F\|_{H^{\f12+\a}_x}\|G\|_{H^1_x}+\|G\|_{H^{\f12+\a}_x}\|F\|_{H^1_x}, \label{9.08.8.19}\\
&\|\La^\a(G_1 G_2 G_3)\|_{L_x^2}\les \sum_{j=1}^3\big(\|\La^\a G_j\|_{H^1_x}\Pi_{l\neq j}\|G_l\|_{H^1_x}\big).\label{9.08.17.19}
\end{align}
\end{lemma}
\begin{proof}
(\ref{9.08.8.19}) and (\ref{9.08.17.19}) are \cite[(6.188) and Lemma 18]{WangCMCSH} respectively.

By using the trichotomy in (\ref{9.08.3.19})
\begin{align*}
&\mu^\a\|P_\mu[F\c G]_{HL}\|_{L_x^2}\les\mu^\a \|F_\mu\|_{L_x^\infty}\|G_{\le \mu}\|_{L_x^2}\les \mu^\a \|F_\mu\|_{L_x^\infty} \|G\|_{L_x^2},\\
&\mu^\a\|P_\mu[F\c G]_{LH}\|_{L_x^2}\les \mu^\a \|F_{\le\mu} \|_{L_x^\infty}\|P_\mu G\|_{L_x^2}\les \|F\|_{L_x^\infty}\mu^\a \|P_\mu G\|_{L_x^2},\\
&\mu^\a\|P_\mu[F\c G]_{HH}\|_{L_x^2}\les\sum_{\la \ge\mu} (\frac{\mu}{\la})^\a \| F_\la\|_{L_x^\infty}\|\la^\a P_\la G\|_{L_x^2}.
\end{align*}
(\ref{9.08.1.19}) follows by taking $l_\mu^2$ norms on the above  inequalities.
\end{proof}

\begin{lemma}\label{lem2}
For $0<\a<1/2$ there hold
\begin{align}
\mu^{-\f12+\a}\|\p [P_\mu, F]G\|_{L^2_x}
&\les \|\p F\|_{L^6_x} \sum_{\la\le \mu} \left(\frac{\la}{\mu}\right)^{1/2-\a} \|\la^\a G_\la\|_{L^2_x} \nonumber\\
&+ \|\p F\|_{L^6_x} \sum_{\la>\mu} \left(\frac{\mu}{\la}\right)^{1/2+\a} \|\la^\a G_\la\|_{L^2_x}\nn
\end{align}
and
\begin{align}
\mu^{\f12+\a} \|[P_\mu, F]G\|_{L^2_x}
& \les \|\p F\|_{L^6_x} \sum_{\la\le \mu} \left(\frac{\la}{\mu}\right)^{1/2-\a} \|\la^\a G_\la\|_{L^2_x} \nonumber\\
& +\|\p F\|_{L^6_x} \sum_{\la\ge \mu} \left(\frac{\mu}{\la}\right)^{1+\a} \|\la^\a G_\la\|_{L^2_x}.\label{lem4eq}
\end{align}
\end{lemma}
This is \cite[Lemma 23]{WangCMCSH}.\\

\noindent\textbf{Acknowledgement} The author is indebted to  Pin Yu, who drew her attention to the regularity problem of compressible fluids in 2016, for a helpful discussion in 2018; and would like to thank Jared Speck for a couple of enlightening discussions in January 2019.
 
\end{document}